\documentclass[a4paper,12pt]{book}
\usepackage[utf8]{inputenc}
\usepackage{amssymb}
\usepackage{array}
\usepackage{appendix}
\usepackage{fancyhdr}
\usepackage[english]{babel}
\usepackage{enumitem}
\usepackage{hyperref}
\usepackage[figurename=Figure,labelsep=period, labelfont={sc}]{caption}
\usepackage{subcaption}
\usepackage{setspace}
\usepackage{xcolor}
\usepackage[reqno]{amsmath}
\usepackage{amsthm}
\usepackage{tikz}
\usetikzlibrary{decorations.pathreplacing}
\usepackage[all]{xy}
\usepackage{chngcntr}
\usepackage{anyfontsize}
\usepackage{titlesec}
\titleformat{\chapter}[hang]{\vspace{-1cm}\bfseries\raggedright}{\fontsize{35}{40}\selectfont \thechapter\hspace{.35cm}$\cdot$}{14pt}{\fontsize{22}{28}\selectfont}
\newif\ifhide
\hidefalse
\definecolor{lblue}{RGB}{230,230,250}
\definecolor{dblue}{RGB}{180,180,250}
\definecolor{ddblue}{RGB}{50,50,140}
\definecolor{lgrey}{RGB}{180,180,180}

\hypersetup{
pdfhighlight=/N,
linktocpage=true,
    unicode=false,          
    pdftoolbar=true,        
    pdfmenubar=true,        
    pdffitwindow=false,     
    pdfstartview={FitH},    
    pdftitle={The power of random information for numerical approximation and integration},    
    pdfauthor={Mathias Sonnleitner},     
    colorlinks=true,       
		linkcolor=black,
    linkbordercolor={1 1 1},          
    citecolor=black,        
    filecolor=black,      
    urlcolor=black           
  ,breaklinks=true,raiselinks=true,nesting=true
}

\newtheoremstyle{mytheorem} 
	{}{}
	{\itshape}{}
	{\scshape}{.}
	{.5em}{}

\newtheoremstyle{mydef} 
	{}{}
	{}{}
	{\scshape}{.}
	{.5em}{}

\theoremstyle{mytheorem}
\newtheorem{theorem}{Theorem}[chapter]
\newtheorem{lemma}[theorem]{Lemma}
\newtheorem{corollary}[theorem]{Corollary}
\newtheorem{proposition}[theorem]{Proposition}
\newtheorem{prop}[theorem]{Proposition}
\newtheorem{conj}[theorem]{Conjecture}

\theoremstyle{mydef}
\newtheorem{definition}[theorem]{Definition}
\newtheorem{question}{Question}[chapter]
\newtheorem{example}[theorem]{Example}
\newtheorem{remark}[theorem]{Remark}

\newcommand{\id}{\mathrm{id}}
\newcommand{\crn}{\mathrm{cr}_n}
\newcommand{\rad}{\mathrm{rad}}
\newcommand{\decay}{\mathrm{decay}}
\newcommand{\diam}{\mathrm{diam}}

\newcommand{\elp}{\mathcal{E}_{p,\sigma}^m}
\newcommand{\elz}{\mathcal{E}_{\sigma}^m}
\newcommand{\lall}{\Lambda^{\rm all}}
\newcommand{\lstd}{\Lambda^{\rm std}}

\newcommand{\intmu}{\mathrm{INT}_{\mu}}
\newcommand{\enran}{E_n^{\mathrm{ran}}}
\newcommand{\Pn}{P_n}
\newcommand{\pnran}{P_n^{\mathrm{ran}}}

\newcommand{\Uin}{U_{\rm in}}
\newcommand{\Uout}{U_{\rm out}}
\newcommand{\chole}{c_{\rm hole}}
\newcommand{\cgood}{c_{\rm good}}
\newcommand{\cpoly}{c_{\rm poly}}
\newcommand{\lip}{\mathrm{Lip}(\mathbb{R}^d)}
\newcommand{\lipd}{{\rm Lip}(D)}

\newcommand{\IB}{\mathbb{B}}
\newcommand{\IBo}{\mathbb{B}}
\newcommand{\IE}{\mathbb{E}}
\newcommand{\IN}{\mathbb{N}}
\newcommand{\IP}{\mathbb{P}}

\newcommand{\IR}{\mathbb{R}}
\newcommand{\IS}{\mathbb{S}}
\newcommand{\TT}{\mathbb{T}}
\newcommand{\ZZ}{\mathbb{Z}}

\newcommand{\cA}{\mathcal{D}}
\newcommand{\cB}{\mathcal{B}}
\newcommand{\cC}{\mathcal{C}}
\newcommand{\cH}{\mathcal{H}}
\newcommand{\cK}{\mathcal{K}}
\newcommand{\cP}{\mathcal{P}}
\newcommand{\cQ}{\mathcal{Q}}

\newcommand{\fP}{\mathfrak{P}}

\newcommand{\dd}{{\rm d}}
\newcommand{\dist}{\mathrm{dist}}
\newcommand{\vol}{\mathrm{vol}}
\newcommand{\bfone}{\mathbf{1}}
\newcommand{\dia}{\mathrm{diam}(\mathfrak{P})}

\newcommand{\gs}{\gtrsim}
\newcommand{\ls}{\lesssim}
\newcommand{\chg}[1]{\textcolor{black}{#1}}

\DeclareMathOperator*{\esssup}{ess\,sup}

\makeatletter
\newcommand{\leqnomode}{\tagsleft@true}
\newcommand{\reqnomode}{\tagsleft@false}
\makeatother

\title{Thesis}
\author{Mathias Sonnleitner}

\date{}

\begin{document}

\setstretch{1.3}
\frontmatter
\pagestyle{empty}
\setlength{\voffset}{-2cm}
\begin{flushright}
\includegraphics[scale=.22]{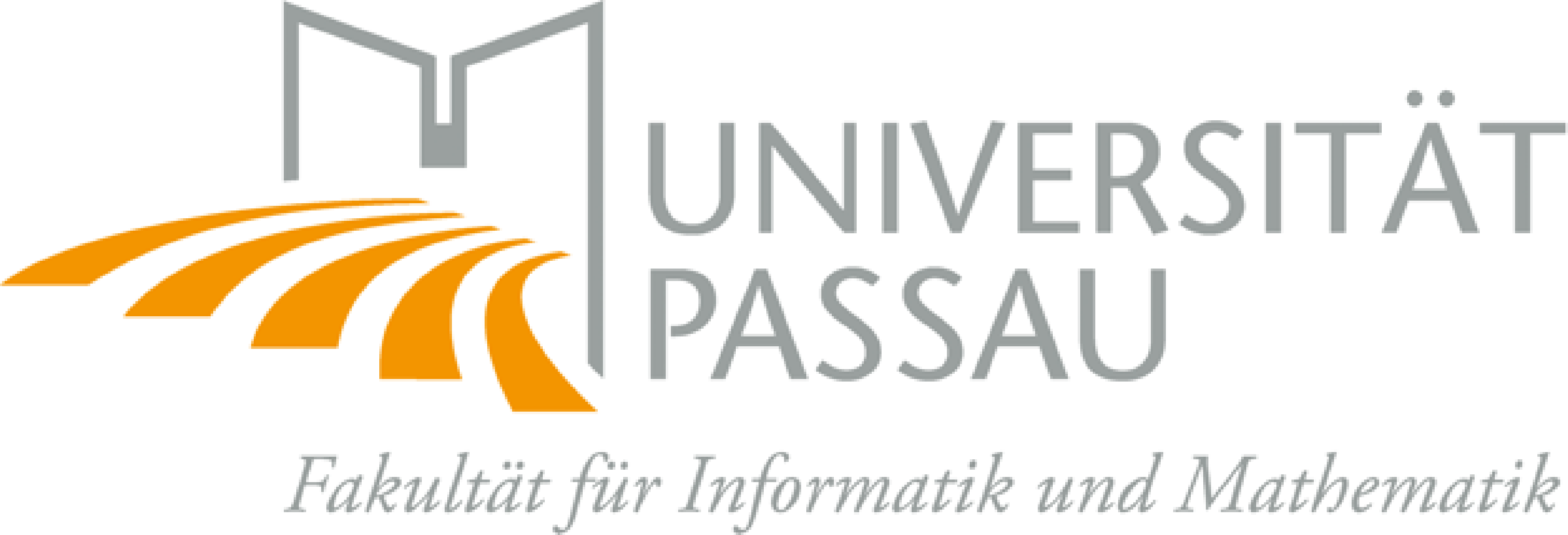}
\end{flushright}
\vspace{1cm}
\begin{center}
	\huge{\textbf{The power of random information for numerical approximation}}\\ 
	\huge{\textbf{and integration}}\\
	\vspace{1cm}
	\Large{Dipl.-Ing.~Mathias Sonnleitner}\\
	\vspace{3cm}
	\Large{A dissertation submitted to the \\ faculty of computer science and mathematics \\ in partial fulfillment of the requirements for the degree of\\ doctor of natural sciences}\\
\end{center}
\vspace{2cm}
\begin{flushleft}
	\begin{tabular}{ll}
		\Large{Advisors:} & \Large{Prof.~Dr.~Joscha Prochno}\\[.1cm]
				& \Large{Univ.-Prof.~Dr.~Aicke Hinrichs}	
	\end{tabular}
\end{flushleft}
	\vspace{1cm}
	\begin{center}
	\end{center}

\newpage
\setlength{\voffset}{-1.8cm}
\setlength{\textheight}{23cm} 

\begin{flushleft}
	\phantom{a}

\vfill

\chg{
	\begin{tabular}{ll}
		\large{Reviewers:} & \large{Prof.~Dr.~Joscha Prochno}\\[.1cm]
				& \large{Univ.-Prof.~Dr.~Aicke Hinrichs}\\[.1cm]
				& \large{Prof.~Dr.~Alexander Litvak}	\\[.5cm]
				\large{Date of oral exam:} & \large{2022-05-18}
	\end{tabular}
	}
\end{flushleft}

\pagestyle{fancy}

\newpage
\begin{center}
	\Large{\textbf{Abstract}}
\end{center}
\vspace{.5cm}

This thesis investigates the quality of randomly collected data by employing a framework built on information-based complexity, a field related to the numerical analysis of abstract problems. The quality or power of gathered information is measured by its radius which is the uniform error obtainable by the best possible algorithm using it. The main aim is to present progress towards understanding the power of random information for approximation and integration problems. 

In the first problem considered, information given by linear functionals is used to recover vectors, in particular from generalized ellipsoids. This is related to the approximation of diagonal operators which are important objects of study in the theory of function spaces. We obtain upper bounds on the radius of random information both in a convex and a quasi-normed setting, which extend and, in some cases, improve  existing results. We conjecture and partially establish that the power of random information is subject to a dichotomy determined by the decay of the length of the semiaxes of the generalized ellipsoid.

Second, we study multivariate approximation and integration using information given by function values at sampling point sets. We obtain an asymptotic characterization of the radius of information in terms of a geometric measure of equidistribution, the \chg{distortion, which} is well known in the theory of quantization of measures. This holds for isotropic Sobolev as well as Hölder and Triebel-Lizorkin spaces on bounded convex \chg{domains}. We obtain that for these spaces, depending on the parameters involved, typical point sets are either asymptotically optimal or worse by a logarithmic factor, again extending and improving existing results. 

Further, we study isotropic discrepancy which is related to numerical integration using linear algorithms with equal weights. In particular, we analyze the quality of lattice point sets with respect to this criterion and obtain that they are suboptimal compared to uniform random points. This is in contrast to the approximation of Sobolev functions and resolves an open question raised in the context of a possible low discrepancy construction on the two-dimensional sphere.

\newpage
\phantom{}

\newpage
\begin{center}
	\Large{\textbf{Acknowledgements}}
\end{center}
\vspace{.5cm}

At this point, I want to acknowledge that this work would not have been created without the support of several important people.

First, I want to express my gratitude to my supervisors Aicke Hinrichs and Joscha Prochno who provided me with the opportunity to conduct my own research and supported me in the process. 

This thesis relies on joint publications with kind colleagues. In addition to my supervisors, these are David Krieg, with whom I exchanged innummerable e-mails about our joint projects and whose comments improved the presentation of our joint work in this thesis, and  Friedrich Pillichshammer, \chg{who guided} me during my first publication. 

Special thanks go to Mario Ullrich, who served as a mentor during my time in Linz, for many (non-)mathematical discussions and helpful advice.

I was fortunate to have met many people during my studies, especially at the Institute of Analysis and the neighboring Institute of Financial Mathematics and Applied Number Theory at Johannes Kepler University (JKU) Linz and the Institute of Mathematics and Computer Science at the University of Graz. The conversations with them enriched my life and made it easy to take short (and also longer) breaks from mathematics.

Even more so, I am glad to have friends who provided most welcome distraction from symbols and formulas.

Many of my teachers deserve my gratitude for increasing my joy in learning which made composing this thesis an overall satisfying experience.

Finally, let me state the obvious fact that I would not exist without my parents. They provided me with much more than mere existence and I want to thank them for their unfailing support. Returning home always has been a source of strength because of them and my brother.

\vspace{1cm}

The support by the Austrian Science Fund (FWF) through the project grants FF5513-N26, which is part of the Special Research Program `` Quasi-Monte Carlo Methods: Theory and Applications'', P32405 ``Asymptotic geometric analysis and applications'', and P34808 ``Information-Based Complexity: Beyond the standard settings'' is gratefully acknowledged. 

\chapter*{Preface}
\addcontentsline{toc}{chapter}{\protect\numberline{}Preface}
\fancyhead{}
\fancyfoot{}
\fancyhead[CO]{\nouppercase{\textsc{\rightmark}}}
\fancyhead[CE]{\nouppercase{\textsc{\leftmark}}}
\fancyfoot[CO,CE]{\thepage}

This thesis is the result of several projects conducted at the Institute of Analysis at JKU Linz and the Institute of Mathematics and Scientific Computing at the University of Graz. It is based on the following four research works (ordered chronologically by first-announcement date) and aims to put them into a common framework.

\begin{enumerate}
	\item[]
		with A.~Hinrichs and J.~Prochno: Random sections of $\ell_p$-ellipsoids, optimal recovery and Gelfand numbers of diagonal operators, 2021. (submitted) \cite{HPS21} (see Chapter~\ref{ch:ell})
	\item[] with F.~Pillichshammer: On the relation of the spectral test to isotropic discrepancy and $L_q$-approximation in Sobolev spaces. \textit{J.~Complex.}, \textbf{67}, Article 101576, 2021. \cite{SP21} (see Chapter~\ref{ch:iso})
	\item[] with D.~Krieg: Random points are optimal for the approximation of Sobolev functions, 2020.  (submitted) \cite{KS20} (see Chapter~\ref{ch:sob})
	\item[] with F.~Pillichshammer: A note on isotropic discrepancy and spectral test of lattice point sets. \textit{J.~Complex.}, \textbf{58}, Article 101441, 2020. \cite{PS20} (see Chapter~\ref{ch:iso})
\end{enumerate}

\noindent The following additional research articles by the author are not included in this thesis.
\begin{enumerate}
	\item[] with D.~Krieg: Function recovery on manifolds using scattered data, 2021. \cite{KS21}
	\item[] with D.~Krieg and E.~Novak: Recovery of Sobolev functions restricted to iid sampling. \chg{\textit{Math.~Comp.}, 91(338):2715-2738, 2022.} \cite{KNS21}
	\item[] with A.~Baci, Z.~Kabluchko, J.~Prochno and C.~Th\"ale: Limit theorems for random points in a simplex. \textit{J.~Appl.~Probab.} (accepted), 2021+. \cite{BKP+20}
\end{enumerate}

\tableofcontents

\mainmatter

\chapter{Introduction} \label{ch:intro}
\fancyhead[CO]{\nouppercase{\textsc{\leftmark}}}
\fancyhead[CE]{\nouppercase{\textsc{\leftmark}}}

One could say that mathematics is an endless process in which the following sequence repeats itself indefinitely. Asking a question, searching for an answer which is often (to the possible discomfort of non-mathematicians) understood as a proof of the mere existence of a solution, discovering the structures behind and raising questions about them. It may come as no surprise that this work is no different as it is guided by a question which arose in previous research and in turn triggers new open problems. In order to give more details, let us introduce the setting and give the motivation behind. The subject of study belongs to theoretical numerical analysis, which deals less with the numerics behind implementable algorithms than with the analysis of the given numerical problem itself. As we understand it, this branch of mathematics lies at a crossroads between approximation theory, functional analysis, geometry, information-based complexity (IBC) and probability. In what follows, we draw upon these fields to set the stage, and fill in details in Chapter~\ref{ch:prelim}. 

Consider the abstract numerical problem of finding a solution $S(f)$ for each object $f$ (or problem instance) in a class $F$. For one reason or another, it might be the case that we cannot solve the problem exactly for each problem instance. For one, the quantity $S(f)$ might be difficult to compute, say it is the integral of a complicated function $f$ belonging to some class of functions $F$. Second, we may only have incomplete knowledge about \chg{$f$, meaning} it cannot be described through the finite amount of available data. 

In either case, we might try to use finitely many observations of any particular $f\in F$ to approximate the solution $S(f)$. Let $L_1(f),\dots,L_n(f)$ be real numbers denoting the $n\in\IN$ observations we collect about $f$, where the measurements $L_1,\dots,L_n\colon F\to \IR$ are known and may be applied to any $f\in F$ to produce observations. To avoid writing them in full, we use the notation $N_n(f)=\big(L_1(f),\ldots,L_n(f)\big)$ for \emph{information} about $f$, which we shall use to construct an approximation $A(f)$ to the true solution $S(f)$. Necessarily, if we are only to use this information, the approximation has to be of the form $A(f)=\varphi\big(N_n(f)\big)$, where $\varphi$ takes $n$ real numbers as input. 

\newpage

To illustrate, let us give two examples which will accompany us throughout this work. In the integration problem, the object $f$ is a function belonging to a class $F$ of functions and has integral $S(f)$ (with respect to some fixed measure). In this case, information may be given for example by function evaluations $N_n(f)=\big(f(x_1),\ldots, f(x_n)\big)$ and $\varphi\big(N_n(f)\big)$ should approximate the integral of $f$. In the approximation problem, the object $S(f)=f$ is to be approximated itself from information $N_n(f)$ given by function evaluations or from linear measurements such as Fourier coefficients.

We view $N_n$ as a map from $F$ to $\IR^n$, returning for each $f$ information $N_n(f)$ about $f$. We then speak of the \chg{\emph{information map}} $N_n$. Similarly, we can interpret $\varphi$ as a map from $\IR^n$ to the set of possible approximants and call it the \chg{\emph{reconstruction map}}. The composition $A=\varphi\circ N_n$ is then an algorithm using the information map $N_n$. In practice, such an algorithm may represent an actual (future) implementation which runs on a computer (or any other Turing machine) and handles finitely many inputs to produce an output in finite time. 

In general, finite information $N_n(f)$ is insufficient to determine $f\in F$ and, even though we may know $f$ itself, we assume it is an unknown element of $F$ except for information $N_n(f)$. Therefore, any, even the best,  algorithm using the information map $N_n$ is bound to incur a non-zero error on some $f\in F$. \chg{The minimal worst-case} error over all possible algorithms using the information map $N_n$ is called the \emph{radius of information} $N_n$ and measures its quality or power. The smaller its radius, the better the information map.  

Naturally, one is interested in solving the given numerical problem of approximating $S(f)$ with an error as small as possible by using $n$ measurements. A lower bound on this error is given by the smallest possible radius over all information maps using at most $n$ measurements. This quantity is called the $n$-th minimal radius of information. Then our main aim is to understand whether (near-)optimal information, whose radius is close to this lower bound (in an asymptotic sense), has to be something singular, arcanely constructed, or rather is typical in the sense that most information maps using $n$ measurements have a radius comparable to the $n$-th minimal radius of information. We phrase this as two related questions.
\begin{center}
How typical is optimal information? How optimal is typical information?
\end{center}

As a shortcoming of our rather abstract approach this pair of questions is not yet well-posed. Although we shall make it more precise in Chapter~\ref{ch:prelim}, we understand it rather as a theme with many facets and do not expect a definite answer but a whole spectrum of them. In the following, we wish to give some idea \chg{how} possible solutions may look like. 

First, we need to clarify what ``typical'' and ``most'' mean in this context. This can be resolved by putting a measure on the set of allowed measurements and thus considering random measurements. In many situations, a canonical choice of measure suggests itself but it may be specific to the problem. The random measurements can be collected to form random information. Then, determining the optimality of typical information amounts to studying the radius of random information and \chg{therefore} we will use the terms ``typical'' and ``random'' interchangeably. 

From a practical point of view, there are several issues we will ignore in this work. We are interested in the radius of information as a theoretical quantity and are not so much concerned with explicit algorithms. Also, we do not consider the often non-trivial process of generating the required random objects. Further, in most cases, we will not be interested in constants but only asymptotic behaviour. Instead, we wish to gain insight into the qualititative behaviour of the typical radius compared to the minimal radius. In this way, we hope to increase our understanding of the problems themselves. For further motivation we draw in the following from previous work in this direction.

To the best of our knowledge, the pair of questions above has not been studied in the literature in such a general context except by Hinrichs, Krieg, \chg{Novak, Prochno and} M.~Ullrich in the articles \cite{HKN+20} and \cite{HKN+21} which initiated part of the research presented in this thesis. The work \cite{HKN+20} asks about the \emph{power of random information}, that is, the typical quality of random information coming in by independent and identically distributed (i.i.d.) observations for various numerical problems. The mentioned works contain results on integration and approximation in Sobolev spaces as well as recovery in ellipsoids and associated Hilbert spaces. This thesis extends these insights about the power of random information and gives possible answers to the two questions above.

In \cite{HKN+20} at least two arguments were put forward to justify the use, and consequently the study, of random instead of optimal information.  First, one can easily increase the number of measurements and thus decrease the error without problem if there is a method of sampling random measurements. This allows for some flexibility compared to specifically chosen information which may need to be re-computed if one adds an additional measurement. Second, optimal information is often tailored to the specific problem, \chg{whereas} typical information can be used for many different problems. This behaviour reflects a kind of \emph{universitality}.

To mention potential applications, the assumption of i.i.d.\ observations is common in (statistical) learning theory, which infers functional dependencies from empirical data from real-world applications, see, e.g., Vapnik~\cite{Vap98}. This field is behind many developments associated with the nowadays popular concept of a neural network which should replace the object it is modeled on, the (human) brain, in modern applications such as pattern recognition or autonomous driving. Therefore, the study of the power of random information may contribute to the understanding of the effectiveness of algorithms used in these fields. 

As mentioned, we expect a spectrum of possible ways how random information might compare to optimal information. Essentially, we distinguish the following ends of the spectrum. Random information may either be close to optimal or completely useless. In fact, these are the only two possibilities appearing in a problem treated in Chapter~\ref{ch:ell}. Let us describe what we can conclude from either one. 

If typical information is close to optimal or, equivalently, the power of random information is best possible with high probability, then one might use random information if optimal information is not available or costly to obtain. Sometimes it is even the case that the best known information is a realization of random information. This is related to the \emph{Probabilistic method} using which one proves the existence of an object with a desired property by establishing it with positive probability for a randomly chosen object from a suitably constructed probability space. For an illustration and examples ranging from random graphs to combinatorial discrepancy we refer to the book of the same name \cite{AS16} by Alon and Spencer. Let us add here that this approach is behind the \chg{construction of suitable} point sets for numerical integration or the use of random matrices in the field of compressed sensing. The latter is a toolbox to efficiently process signals from applications such as facial recognition or magnetic resonance imaging, where only a few large wavelet coefficients are necessary to obtain a good approximation.

On the other hand, a large gap between the power of random information and the optimal behaviour may suggest that measurements have to be carefully selected in order to obtain good algorithms. On a practical note, we may deduce from this how \emph{not} to gather information. It appears that the required lower bounds on the power of random or typical information, which also exhibit the limitations of the Probabilistic method, are far and few in between. Although it is sensible to focus on upper bounds as guarantees for random algorithms, we believe that the other side of the coin deserves attention too. 

Apart from its contribution to solving numerical problems, the study of random information combines different fields of mathematics and may increase our knowledge about random structures. For example, if information is collected via function samples, then the radius of information is related to geometric features of the point set such as the distortion or the discrepancy, see Chapters~\ref{ch:sob} and \ref{ch:iso}, respectively. If the point set is drawn randomly, this is connected to the study of Voronoi tessellations induced by a Poisson point process, \chg{see, e.g.,} Yukich~\cite{Yuk08}. Further, if the underlying set is a sphere, the radius of the largest hole in a random point set determines the Hausdorff distance between the ball and its approximation by the convex hull of the points. These topics belong to stochastic geometry as described for example in the book by Schneider and Weil \cite{SW08}.  

As an example of the power of random information, we would like to mention the famous result on the polynomial tractability of star-discrepancy by Heinrich, Novak, Wasilkowski and Wo\'zniakowski \cite{HNW+01} drawing from progress in empirical process theory due to Talagrand \cite{Tal94}, see also Novak and Wo\'zniakowski \cite[Sec.~3.1.5]{NW08}. The concept of tractability quantifies high-dimensional behaviour and this example shows that random information can be an antidote to the so-called \emph{curse of dimensionality}, which occurs if the problem complexity depends exponentially on the dimension. 

In fact, random information is part of the \emph{blessing of dimensionality} as observed for example by Kainen \cite{Kai97}. Put briefly, this term summarizes regularizing effects as the dimension increases such as the \emph{concentration of measure phenomenon} for which we refer to the book by Ledoux \cite{Led01} of the same name.

Concentration of measure has been extensively used in the local theory of Banach spaces which evolved the field of asymptotic geometric analysis at the midpoint between geometry and functional analysis. There, properties of normed spaces as the dimension tends to infinity are studied. Famous questions like the thin-shell or the hyperplane conjecture are still open and the progress towards their solution shows how unintuitive high-dimensional objects may be to the imagination of us three-dimensional beings, see the books by Artstein-Avidan, Giannopoulos and V.~D.~Milman~\cite{AGM15} as well as V.~D.~Milman and Schechtman \cite{MS86} for an extensive overview of the subject.

A related phenomenon specific to high dimensions has been called \emph{existence versus prevalence} by Vershynin~\cite{Ver06}, see also Giannopoulos, V.~D.~Milman and Tsolomitis~\cite{GMT05}, Litvak, Pajor and Tomczak-Jaegermann \cite{LPT06} as well as the references given there. Roughly speaking, it claims that the existence of one structure with a certain desired property implies that actually most structures satisfy this property. In the works mentioned, the diameter of a random sections of convex bodies has been considered but this also occurs also in related situations. This is a further example of the power of random information and will be discussed in Chapter~\ref{ch:ell}.

\newpage

\textbf{Outline. }
We give an overview of the structure of the remainder of this thesis. First, Chapter~\ref{ch:prelim} gives definitions of the concepts introduced above and lays the foundation for the remaining chapters. In particular, the notion of random information is discussed for linear information given by functionals (Section~\ref{sec:lin}) and standard information given by function evaluations (Section~\ref{sec:std}).  

The main body is composed of three Chapters~\ref{ch:ell}, \ref{ch:sob} and \ref{ch:iso} which are modeled onto published research during work on this thesis, plus an interlude Chapter~\ref{ch:interlude} connecting Chapters~\ref{ch:sob} and \ref{ch:iso}. Let us first describe the common structure of the triple (\ref{ch:ell},\,\ref{ch:sob},\,\ref{ch:iso}). Each opens with a detailed introduction which embeds the related publication(s) into a wider context and serves as a motivation. This is followed by a discussion of the results together with existing work. Afterwards, concepts of the proofs behind the results are highlighted in order to make the proofs more accessible and to prepare for follow-up work. Remaining open questions will be posed and discussed at the end in order to provide points of departure for subsequent research and to continue the mathematical process. Throughout these chapters, we will use the optional notation ``\textsc{Theorem X.x }([XYZ00, Thm.~X])'' to clarify where the corresponding result is derived from. 

Chapter~\ref{ch:ell} treats linear information for the recovery of vectors belonging to general convex bodies and also to generalized ellipsoids. The results obtained there extend the work from the Hilbert space case in \cite{HKN+21} and are related to Gelfand numbers of diagonal operators. We present two upper bounds on the radius of random information, one relying on asymptotic geometric analysis, and another relying on compressed sensing. Then we discuss implications for the important special case of polynomial semiaxes, where we conjecture a dichotomy for the power of random information. This is based on the joint work~\cite{HPS21} with Hinrichs and Prochno.

In Chapter~\ref{ch:sob} we study standard information for the approximation and integration of functions, mainly from Sobolev but also from Hölder and Triebel-Lizorkin spaces. Motivated by the study of the power of random information, we derive a characterization of the radius of information given by a point set in terms of the distortion, which measures the size of an average or the largest hole. As a consequence, this yields the optimality of random information for a certain range of parameters including integration. We discuss connections to quantization theory which we use to deduce our result. This chapter is derived from the joint work \cite{KS20} with Krieg.

The subsequent Chapter~\ref{ch:interlude} does not contain any new results but instead collects material on optimal transport and quantization theory with the aim of discussing the weights used by (near-)optimal algorithms for integration problems. We find that these weights can be given explicitly in terms of the Voronoi cells of the underlying point set. Further, we argue why in some cases weights can be assumed to be normalized and sometimes even of equal size, which yields a (somewhat improvised) connection to the geometrical concept of discrepancy.

Finally, in Chapter~\ref{ch:iso} standard information is discussed in the context of equal weight cubature rules for numerical integration and discrepancy, which is introduced in a general fashion. The main result is a lower bound on the isotropic discrepancy of lattice point sets which answers a question raised by Aistleitner, Brauchart and Dick~\cite{ABD12} and shows that structured point sets may be at a disadvantage compared to random ones. We derive this from a characterization of the isotropic discrepancy of a lattice point set in terms of the so-called spectral test and derive an asymptotic equivalence for the radius of information given by a lattice point set. The results are taken from joint work with Pillichshammer, namely, \cite{PS20} and the follow-up work \cite{SP21}.

Almost all of the proofs are outsourced to the Appendices \ref{ch:ell-app} - \ref{ch:iso-app} at the end of this thesis, which mirror Chapters \ref{ch:ell} - \ref{ch:iso}.

\bigskip

\textbf{Notation. } At this point we wish to clarify some notation. The positive integers are denoted by $\IN=\{1,2,\ldots\}$ and we write $\IN_0=\IN\cup \{0\}$. Given a sequence $x=(x_i)_{i\in\IN}$ and $0<p\le\infty$, write $\|x\|_p:=(\sum_{i=1}^{\infty}|x_i|^p)^{1/p}$ for $p<\infty$ and $\|x\|_{\infty}:=\sup_{i\in\IN}|x_i|$ for $p=\infty$. The unit ball of $\ell_p$, the space of sequences $x$ with $\|x\|_p<\infty$, is $\IB_p:=\{x\colon \|x\|_p\le 1\}$. We write $\ell_p^m$ for the $m$-dimensional subspace spanned by the first $m\in\IN$ coordinates and $\IB_p^m$ for its (closed) unit ball. If $1\le p\le\infty$, then $p^*$ is the Hölder conjugate with $\frac{1}{p}+\frac{1}{p^*}=1$, where $a/\infty:=0$ for any $a\in\IR$. If $p=2$, the space $\ell_2^m$ is equipped with the standard inner product $\langle\cdot,\cdot\rangle$ and $\IS^{m-1}:=\partial \IB_2^m$ is the sphere bounding the unit ball. 

For any real $a\in \IR$ we use the notation $(a)_+:=\max\{a,0\}$ and $a=\lfloor a\rfloor+\{a\}$, where $\lfloor a\rfloor\in \ZZ$ is the integral and $0\le \{a\}<1$ the fractional part of $a$. Logarithms are always taken to basis ${\rm e}=2.71828\dots$.

If $A$ is any set, then $\bfone_A$ denotes the indicator function of $A$ with $\bfone_A(x)=1$ if $x\in A$ and zero else. If $B$ is another set, we do not exclude the case of equality in the notation $A\subset B$. If $A$ is finite, we use $\#A$ to denote its cardinality. 

We assume that all random variables are defined on a common probability space and use $\IP$ and $\IE$ for probability and expectation, respectively. The asymptotic notations $\asymp,\lesssim,\gtrsim$ are explained in Definitions~\ref{def:asymp} and \ref{def:asymp-std}. Other notation is explained at its first point of appearance and collected in the \hyperref[ch:los]{List of symbols}.

\chapter{Preliminaries} \label{ch:prelim}
\fancyhead[CO]{\nouppercase{\textsc{\rightmark}}}
\fancyhead[CE]{\nouppercase{\textsc{\leftmark}}}

In this chapter, we introduce the framework supporting the remainder of the thesis. We wish to clarify the notions used in the introduction and for the sake of completeness allow for some redudance. First, in Section~\ref{sec:ibc}, we explain fundamental concepts such as the radius or minimal radius of information. Section~\ref{sec:lin} introduces the setting of recovery using linear information on which Chapter~\ref{ch:ell} is based. In Section~\ref{sec:std} we discuss standard information which is obtained by function \chg{samples and} a special case of linear information. This is the basis for Chapter~\ref{ch:sob}, and in parts also Chapters~\ref{ch:interlude} and \ref{ch:iso}. 

\section{The radius of information}
\label{sec:ibc}

Recall that for each $f$ in a class of objects $F$ we want to compute a solution $S(f)$ to a numerical problem. Here, the solution operator $S\colon F\to G$ maps into a set $G$ containing the set of possible outcomes $S(F):=\{S(f)\colon f\in F\}$. We use $n$ measurements $L_1(f),\ldots,L_n(f)\in\IR$, where $L_i\colon F\to \IR$ for $i=1,\ldots,n$, and combine them into information $N_n:=(L_1,\ldots,L_n)\colon F\to\IR^n$. Using a reconstruction map $\varphi\colon \IR^n\to G$, we construct from the given information $N_n(f)$ an approximation $\varphi\big(N_n(f)\big)\approx S(f)$. Note that we assume the knowledge of the information map $N_n$, and the map $\varphi$ may (and will) depend on the information map $N_n$. The following diagram illustrates the situation.
\[
\xymatrix{
F\ar[rr]^{S}\ar[rd]_{N_n} & & G\\
& \mathbb{R}^n \ar[ur]_{\varphi}  & }
\]
We impose additional structure and assume that $F$ and $G$ are normed vector spaces equipped with norms $\|\cdot\|_F$ and $\|\cdot\|_G$, respectively. In this way, we arrive at the setting of information-based complexity, as treated for example by Traub, Wasilkowski and Wo\'zniakowski~\cite{TWW88} and in the book series by Novak and Wo\'zniakowski~\cite{NW08,NW10,NW12}. 

\newpage

The problem formulation suggests measuring the closeness $\approx$ between the approximation $A(f)=\varphi\big(N_n(f)\big)$ and the true value $S(f)$ with $\|S(f)-A(f)\|_G$, which is called the \emph{error} of the algorithm $A$ at $f\in F$. It is useful to have a uniform error guarantee for the algorithm $A$ as the instance $f\in F$ producing information $N_n(f)$ may be unknown. Therefore, consider the \emph{worst-case error}
\begin{equation} \label{eq:wce}
e(S\colon F\to G,A)
:=\sup_{f\in F_0} \|S(f)-A(f)\|_G
\end{equation}
of $A$ over the unit ball $F_0:=\{f\colon \|f\|_F\le 1\}$ of $F$. Abusing notation, we frequently replace the supremum by $\sup_{\|f\|_F\le 1}$. 

One reason to consider the worst-case error over the unit ball is that the norm $\|\cdot\|_F$ measures the size or complexity of the input and elements with larger norm may be more difficult to approximate. More precisely, if both $A$ and $S$ are linear maps, we can use the worst-case error to obtain a bound on the error of $A$ at an arbitrary $f\in F$ by normalizing it. This provides some motivation to consider this error criterion and results from 
\begin{align*}
\|S(f)-A(f)\|_G
=\|f\|_F\,\big\|S(f/\|f\|_F)-A(f/\|f\|_F)\big\|_G
\le \|f\|_F\, e(S\colon F\to G,A).
\end{align*}
To ensure a finite worst-case error, it is convenient to suppose that $S$ is bounded, that is, 
\[
\sup_{\|f\|_F\le 1}\|S(f)\|_G
=e(S\colon F\to G,0)<\infty.
\]
In other words, when we do not have any information, we take the algorithm equal to zero and the resulting initial error is supposed to be finite.  

Formally, the radius of information is defined as follows.

\begin{definition}\label{def:radius}
The radius of an information map $N_n\colon F\to \IR^n$, where $n\in\IN$, is
\begin{equation} \label{eq:radius-info}
r(S\colon F\to G,N_n)
:=\inf_{\varphi}e(\chg{S\colon F\to G},\varphi\circ N_n),
\end{equation}
where the infimum is taken over all reconstruction maps $\varphi\colon\IR^n\to G$. 
\end{definition}
Equivalently, the radius $r(S\colon F\to G,N_n)$ is the minimal worst-case error over all algorithms $A=\varphi\circ N_n$ using the information map $N_n$. The term ``radius'' is of geometric origin and is related to the fact that the algorithm which chooses a center of the set of elements of equal information has minimal worst-case error. The concept of a Chebyshev radius from numerical analysis is similar, see, e.g., Novak \cite[Sec.~A.3]{Nov88}. 

\label{loc:lall}
In the following, we shall consider only continuous and linear measurements belonging to a class $\Lambda\subset \lall:=F'$, where $F'$ is the dual space of $F$ and contains all continuous and linear functionals from $F$ to $\IR$. Sections~\ref{sec:lin} and \ref{sec:std} will discuss the two choices of $\Lambda$ considered in this thesis, namely the case of linear information $\Lambda=\lall$ and the case of standard information $\Lambda=\lstd$ consisting of function evaluations.

In general, the class $\Lambda$ gives rise to information maps $N_n$ of the form 
\begin{equation} \label{eq:info-lambda}
N_n=(L_1,\dots,L_n),\quad
\text{where }L_i\in \Lambda \text{ for all }i=1,\dots,n,
\end{equation}
where $n\in\IN$ is the cardinality of $N_n$. The minimal radius is a benchmark for the quality of any information $N_n$ of the form \eqref{eq:info-lambda} and is given as follows.
\begin{definition}\label{def:minimal-radius}
The $n$-th minimal radius of information with respect to $\Lambda$ is
\begin{equation} \label{eq:def-min-info}
r(S\colon F\to G,\Lambda,n)
:=\inf_{N_n}r(S\colon F\to G,N_n),\quad n\in\IN,
\end{equation}
where the infimum is over all information maps $N_n\colon F\to \IR^n$ of the form \eqref{eq:info-lambda}.
\end{definition}
This is the best we can do using $n$ optimal measurements from the class $\Lambda$. By definition, the $n$-th minimal radius will be a non-increasing function of $n$ and, under suitable conditions, will tend to zero as the amount of information increases. \chg{This means} that we can obtain arbitrarily good approximations using finite information, see also Remark~\ref{rem:compact} below. Further, for any $n\in\IN$ and $N_n$ as in \eqref{eq:info-lambda}, we have
\begin{equation} \label{eq:min-info}
r(S\colon F\to G,N_n)
\ge r(S\colon F\to G,\Lambda,n).
\end{equation}
We shall be interested in the question whether \eqref{eq:min-info} can be reversed up to a constant for most information maps.
\begin{question}\label{que:question}\hfill
\begin{center}
	Does it hold that for most information maps $N_n$ with measurements from $\Lambda$ that
\[
r(S\colon F\to G,N_n)
\lesssim r(S\colon F\to G,\Lambda,n)?
\]
\end{center}
\end{question}

This is a very general formulation of the problem studied in this thesis. In order to make sense of it, it is necessary to clarify what ``most'' shall mean. This shall be done in the remaining sections of this chapter. For now, we shall be concerned with explaining the notation $\lesssim$.

\begin{definition}\label{def:asymp}
	Given two sequences $(a_n)_{n\in\IN}$ and $(b_n)_{n\in\IN}$ of non-negative real numbers, we write $a_n\lesssim b_n$ if there exists a constant $C>0$ such that $a_n \le C\,b_n$ for all $n\in\IN$. In this case, we may also write $b_n\gtrsim a_n$. If there exists another constant $c>0$ such that also $a_n \ge c\, b_n$ for all $n\in\IN$ ($a_n\gtrsim b_n$) holds, then we write $a_n\asymp b_n$ and speak of asymptotic equivalence. We shall use this notation also for non-negative real numbers indexed for example by $(n,m)$ instead of $n$ and then demand that the implicit constants are independent of $n$ and $m$. Further we may use subindices as in $\lesssim_p$ to indicate that the implicit constant may depend on an additional parameter $p$.
\end{definition}

The notation makes clear that we are mainly interested in the asymptotic behaviour, that is, in $n$ tending to $\infty$. If Question~\ref{que:question} were true, then the asymptotic equivalence
\[
r(S\colon F\to G,N_n)
\asymp r(S\colon F\to G,\Lambda,n)
\]
would hold for ``most'' information maps $N_n$ with measurements belonging to $\Lambda$. In this case we say that the radius of typical information is asymptotically equivalent to the minimal radius. As we are mostly concerned only with this asymptotic formulation, we will often be careless about constants, which may be deduced from our proofs. 

At the end of this section, we define the numerical problems central to this thesis.

\bigskip

\textbf{Approximation. }
\label{loc:embedding}
In the approximation problem, the map $S$ is equal to the embedding $\id\colon F\to G$ which takes $f\in F$ to $\id(f)=f$ itself. We assume that $F\subset G$ and further that the embedding $\id\colon F\to G$ is bounded, i.e., there exists a constant $C$ such that $\|f\|_G\le C\|f\|_F$ for all $f\in F$. In this case we speak of a continuous embedding of $F$ into $G$ and write $F\hookrightarrow G$. We then use the notation
\[
r(F\hookrightarrow G, N_n):=r(\id\colon F\to G, N_n)
\]
and analogously for the minimal radius. The upcoming section will introduce random linear information for approximation as a preparation for Chapter~\ref{ch:ell} and Section~\ref{sec:std} will introduce standard information used in Chapter~\ref{ch:sob} for the approximation problem.

\bigskip

\textbf{Integration. }
In the integration problem, the space $F$ is a space of real-valued functions on some measure space and $S$ takes a function $f\in F$ to its integral, so that $G=\IR$. Here, we assume that that $F$ is continuously embedded into the space of integrable functions. Then $S$ is a continuous linear functional itself and therefore it does not make sense to consider linear information but only standard information for this problem which will be treated in Chapters~\ref{ch:sob}, \ref{ch:interlude} and \ref{ch:iso} in various guises.

\begin{remark}
	In this thesis, we consider only non-adaptive information maps of the form \eqref{eq:info-lambda} and use the same measurements for every $f\in F$. This is in \chg{contrast} to adaptive information, where one may use the measurement $L_n(f)$ to select a suitable $L_{n+1}\in\Lambda$. Note that if $S$ is linear between normed spaces $F$ and $G$, as in the majority of the problems we will consider, then adaption can only yield an improvement of at most a constant factor, see, e.g., \cite[Thm.~4.4]{NW08}. This justifies our assumption.
\end{remark}

\section{Linear information and Gelfand widths}
\label{sec:lin}

In this section, we consider linear information $N_n$, which is of the form 
\begin{equation} \label{eq:info-lin}
N_n=(L_1,\dots,L_n),\quad
\text{where }L_i\in \lall \text{ for all }i=1,\dots,n,
\end{equation}
with $n\in\IN$. Recall that $\lall$ is the set of all continuous linear functionals on $F$. In the following, we shall assume that $S$ is linear and bounded as this is the case for the approximation and the integration problems we consider. Then the radius of information (Definition~\ref{def:radius}) satisfies the following well-known relation. 

\begin{proposition}\label{pro:radius-general}
	Let $S\colon F\to G$ be a bounded linear map between normed spaces. Then we have, for all linear information maps $N_n$ as in \eqref{eq:info-lin},
\[
\sup_{f\in F_0\cap \ker N_n} \|S(f)\|_G
\le r(S\colon F\to G,N_n)
\le 2\, \sup_{f\in F_0\cap \ker N_n} \|S(f)\|_G,
\]
where $\ker N_n:=\{f\in F\colon N_n(f)=0\}$ is the kernel of $N_n$.
\end{proposition}

This states that the radius of any linear information $N_n$ is equivalent to the radius of the smallest ball in $G$ containing the symmetric set $\{S(f)\colon f\in F_0\cap \ker N_n\}$ of possible solutions having zero information. For convenience a proof of Proposition~\ref{pro:radius-general} will be provided at the end of this section. Before \chg{we apply} it to random information, we discuss its implications for the minimal radius. 

Derived from Definition~\ref{def:minimal-radius}, the $n$-th minimal radius of linear information is given by
\begin{equation} 
r(S\colon F\to G,\lall,n)
:=\inf_{N_n}r(S\colon F\to G,N_n),\quad n\in\IN,
\end{equation}
where the infimum is over all linear information maps $N_n$ of the form \eqref{eq:info-lin}. This is closely related to the following concept.

\begin{definition}\label{def:gelfand-number}
The $n$-th Gelfand number of $S\colon F\to G$ is given by
\[
c_{n}(S\colon F\to G)
:= \inf_{E_{n-1}} \sup_{f\in F_0\cap E_{n-1}}\|S(f)\|_G,
\]
where the infimum is over all linear subspaces $E_{n-1}\subset F$ of codimension $n-1$ and  for convenience we set $c_n(S\colon F\to G)=0$ if $n\ge \dim(F)$. 
\end{definition}
It is well known that the Gelfand numbers are comparable to the minimal radius in the sense of
\begin{equation}\label{eq:gelfand-radius}
	c_{n+1}(S\colon F\to G)
\le r(S\colon F\to G,\lall,n)
\le 2\, c_{n+1}(S\colon F\to G),\quad
n\in\IN.
\end{equation}
A proof can be found for example in \cite[Prop.~3.1]{Hei94} by Heinrich but may deduced also from Proposition~\ref{pro:radius-general} by taking the infimum over all linear information maps $N_n$ of cardinality $n$, since continuity implies that a subspace $E_n$ of $F$ is of codimension $n$ precisely when it is the kernel of such an $N_n$. 

In the following, let us consider the approximation problem introduced at the end of the previous section. Our assumptions imply that Proposition~\ref{pro:radius-general} applies to the embedding $F\hookrightarrow G$ which we deal with for the approximation problem. In view of this, we may also define, for every linear subspace $E\subset F$,
\begin{equation} \label{eq:radius-def}
\rad_G(F_0,E)
:= \sup_{f\in \chg{F_0\cap E}}\|f\|_G,
\end{equation}
which is the radius of the section of the unit ball $F_0$ with the subspace $E$ measured in the norm of $G$. Then, for all linear information maps $N_n$ we have
\begin{equation} \label{eq:radius-radius}
	\rad_G(F_0,\ker N_n)
	\le r(F\hookrightarrow G,N_n)
	\le 2\,\rad_G(F_0,\ker N_n).
\end{equation}

To lay the foundation for Chapter~\ref{ch:ell}, let us consider now the finite-dimensional case with $F=(\IR^m,\|\cdot\|_K)$ and $G=(\IR^m,\|\cdot\|_G)$ of dimension $m\in\IN$ and equipped with norms $\|\cdot\|_K$ and $\|\cdot\|_G$, respectively. The unit ball $K:=\{x\in\IR^m\colon \|x\|_K\le 1\}\subset \IR^m$ is a convex body, that is a compact convex set with non-empty interior. We will focus mainly on the special case of $\|\cdot\|_G=\|\cdot\|_2$ such that $G=(\IR^m,\|\cdot\|_2)=\ell_2^m$ becomes the Euclidean space equipped with its usual distance. In this case, we drop $G$ from the notation and replace \eqref{eq:radius-radius} by
\begin{equation} \label{eq:rad-r-finite}
\rad(K,\ker N_n)
\le r(K,N_n)
\le 2\,\rad(K,\ker N_n),
\end{equation}
where $r(K,N_n):=r\big((\IR^m,\|\cdot\|_K)\hookrightarrow \ell_2^m,N_n\big)$. This yields the equivalence between the radius of any linear information to the radius of the section of $K$ with the subspace $\ker N_n$. Most of the time, we shall assume that $n<m$ as otherwise vectors may be reconstructed exactly by simply measuring their coordinates. 

By means of \eqref{eq:gelfand-radius}, the minimal radius of information is then up to a constant of two equivalent to the Gelfand numbers of id$\colon (\IR^m,\|\cdot\|_K)\to \ell_2^m$, which we denote by
\begin{equation} \label{eq:gelfand-width}
c_{n}(K)
:= \inf_{E_{n-1}} \sup_{f\in K\cap E_{n-1}}\|x\|_2,\quad n\in \IN,
\end{equation}
where the infimum is again over all linear subspaces $E_{n-1}\subset \IR^m$ of codimension $n-1$. The number $c_n(K)$ is also called the $n$-th Gelfand width of $K$ in $\ell_2$. 

\begin{remark}\label{rem:compact}
Generally speaking, the Gelfand numbers/widths quantify the degree of compactness and are studied as part of the theory of $s$-numbers and $n$-widths, see Pietsch \cite{Pie87} and Pinkus \cite{Pin85} for further information. It is known (see \cite[Prop.~2.4.10]{Pie87}) that a bounded linear map $S$ is compact if and only if the sequence of Gelfand numbers $(c_n(S\colon F\to G)\big)_{n\in\IN}$ decays to zero. By Proposition~\ref{pro:radius-general} this transfers to the sequence of minimal radii and justifies assuming compactness of embeddings.
\end{remark}

Now that we have elaborated sufficiently on the minimal radius and thus optimal linear information we would like to consider random linear information in this finite-dimensional setting. To this end, let us note that any linear functional $L\colon F\to \IR$ is automatically continuous and of the form $L(x)=\langle y,x\rangle,x\in \IR^m$, for some $y\in \IR^m$, where $\langle y,x\rangle = \chg{\sum_{i=1}^{m}y_i x_i}$ is the standard inner product. In this case, we say that $L$ is represented by the vector $y$. In this way, a linear information map $N_n=(L_1,\ldots,L_n)\colon \IR^m\to \IR^n$ can be identified with a $n\times m$-matrix with the representer of $L_i$ forming \chg{the} $i$-th row. 

Via the correspondence between linear functionals and their representers, a random measurement amounts to a random vector $X$ in $\IR^m$. Arguably, random information should not give priority to some particular direction. This amounts to supposing that the direction $X/\|X\|_2$ is uniformly distributed on the sphere $\IS^{m-1}:=\{x\in\IR^m\colon \|x\|_2=1\}$. To explain, the sphere $\IS^{m-1}$ can be equipped with a unique normalized rotation invariant measure $\sigma^{(m-1)}$ and we call a random vector $X'$ uniformly distributed on it if $\IP(X'\in A)=\sigma^{(m-1)}(A)$ for all Borel sets $A\subset \IS^{m-1}$. There is a simple and elegant way to construct such an object. 

Take $g_1,\ldots,g_m$ independent standard Gaussians, that is random variables with density $(2\pi)^{-1/2}\exp(-x^2/2)$, and form the random vector $G_1:=(g_1,\ldots,g_m)$. Then the rotation invariance of $G_1$ implies that its direction $G_1/\|G_1\|_2$ is uniformly distributed on the unit sphere. This motivates the following definition of random linear information.

\begin{definition}\label{def:lin-ran}
	A random linear measurement is given by $x\mapsto \langle G_1,x\rangle$ and random linear information of cardinality $n$ consists of the measurements represented by $G_1,\dots,G_n$, where $G_2,\dots,G_n$ are i.i.d.\ copies of $G_1$. It can be identified with the $n\times m$-matrix $G_{n,m}$ filled with independent standard Gaussian entries and acts by mapping $x\in\IR^m$ to $(\sum_{j=1}^{m}g_{ij} x_j)_{i=1}^{n}\in \IR^n$. Here and in the following, we simplify notation by dropping the arguments of random variables. 
\end{definition}

The radius of such random (Gaussian) information is the random variable $r(K,G_{n,m})$. By relation \eqref{eq:rad-r-finite}, we have, for all realizations,  
\[
\rad\big(K,\ker(G_{n,m})\big)
\le r(K,G_{n,m})
\le 2\, \rad\big(K,\ker(G_{n,m})\big).
\]
Here, the kernel $\enran:=\ker(G_{n,m})$ is a random subspace of $\IR^m$ of codimension $n$ in the following sense, see for example Mattila \cite[Ch.~3]{Mat95} and Meckes~\cite{Mec19}. Let $\mathcal{G}_{m,m-n}$ be the Grassmannian parametrizing all subspaces of $\IR^m$ which are of dimension $m-n$ and thus of codimension $n$. There is a unique probability measure on $\mathcal{G}_{m,m-n}$ which is invariant under rotation and $\enran$ is distributed according to this measure. Note that the probability that $\enran$ has codimension smaller than $n$ is zero. 

In analogy to the radius of information, the radius of any section is at least the radius of a minimal section and thus, for every realization,
\[
\rad(K,n):=c_{n+1}(K)\le \rad(K,\enran),\quad 1\le n < m.
\]
The equivalence \eqref{eq:rad-r-finite} between the power of random (Gaussian) information for the recovery of vectors from an $m$-dimensional normed space with unit ball $K\subset \IR^m$ and the radius of its intersection with a random subspace allows to pose Question~\ref{que:question} as follows.

\begin{question}\label{que:lin}
What is the power of random standard information? 
\hfill
\begin{center}
Do we have for random sections that
$
\rad(K,\enran) \lesssim \rad(K,n)
$,
or equivalently \\for Gaussian information that
$
r(K,G_{n,m}) \lesssim r(K,n)
$
with high probability?
\end{center}
\end{question}
This question will be taken up again in Chapter~\ref{ch:ell} where we present results for a special class of sets, the $\ell_p$-ellipsoids and study the asymptotics as $n$ becomes large but remains much smaller than $m$. As we shall also consider unit balls of quasi-normed spaces, we conclude this section by commenting on them.

A quasi-normed space $F$ is a linear space equipped with a quasi-norm $\|\cdot\|_F$ instead of a norm. In replacement of the triangle inequality it satisfies
\[
\|f_1+f_2\|_F
\le C_F\big(\|f_1\|_F+\|f_2\|_F\big)\quad \text{for some constant }C_F\ge 1 \text{ and all }f_1,f_2\in F.
\]
If norms are replaced by quasi-norms, Proposition~\ref{pro:radius-general} and relation \eqref{eq:gelfand-radius} remain valid with changed constants. For the convenience of the reader we state this in the following proposition and provide a proof.

\begin{proposition}\label{pro:radius-general-+}
	Let $S\colon F\to G$ be a bounded linear operator between quasi-normed spaces, that is, $F$ and $G$ are equipped with quasi-norms with constant $C_F,C_G\ge 1$, respectively.  Then we have, for all linear information maps $N_n$ as in \eqref{eq:info-lin},
\[
C_G^{-1}\sup_{f\in F_0\cap \ker N_n} \|S(f)\|_G
\le r(S\colon F\to G,N_n)
\le 2\, C_F\, \sup_{f\in F_0\cap \ker N_n} \|S(f)\|_G,
\]
where $\chg{\ker N_n=}\{f\in F\colon N_n(f)=0\}$.
\end{proposition}

\begin{proof}
	For the lower bound on the radius of information we have to prove that
\[
\inf_{\varphi}\sup_{f\in F_0} \big\|S(f)-\varphi\big(N_n(f)\big)\big\|_G
\ge C_G^{-1}\sup_{f\in F_0\cap\ker N_n}\|S(f)\|_G
\]
where $F_0=\{f\in F\colon \|f\|_F\le 1\}$ and the infimum is over all $\varphi\colon \IR^n\to G$. Thus fix an arbitrary $\varphi\colon \IR^n\to G$. For every $f\in F_0\cap \ker N_n$ also $-f\in F_0\cap \ker N_n$ and, since $S$ is linear,
\begin{equation}\label{eq:radzero}
\|S(f)-\varphi(0)\|_G \ge C_G^{-1}\|S(f)\|_G \quad\text{ or }\quad \|S(-f)-\varphi(0)\|_G \ge C_G^{-1}\|S(f)\|_G
\end{equation}
holds. Moreover, the symmetry of $F_0\cap \ker N_n$ implies
\[
\sup_{f\in F_0}\|S(f)-\varphi(N_n f)\|_G
\ge \sup_{f\in F_0\cap \ker N_n}\max\big\{\|S(f)-\varphi(0)\|_G,\|S(-f)-\varphi(0)\|_G\big\}.
\]
Together with \eqref{eq:radzero} this proves the lower bound.

For the upper bound we specify a map $\varphi$ by $\varphi(y)=S(g)$ for any $y\in N_n(F_0)$, where $g\in F_0$ with $N_n g=y$ is arbitrary. Then 
\[
\sup_{f\in F_0}\|S(f)-\varphi(N_n f)\|_G
\le \sup_{\substack{f_1,f_2\in F_0\\ N_n f_1 =N_n f_2}}\|S(f_1)-S(f_2)\|_G
\le \sup_{f\in 2C_F F_0 \cap \ker N_n}\|S(f)\|_G
\]
since $f_1-f_2\in 2 C_{F} F_0$ if $f_1,f_2\in F_0$ and $N_n (f_1-f_2)=0$ if $N_nf_1 = N_nf_2$. Since $S$ is linear, a rescaling concludes the proof.
\end{proof}

\begin{remark}
	Note that Proposition~\ref{pro:radius-general-+} may be seen as a replacement for \cite[Lem.~3]{HPS21} which applies even to sets not necessarily being unit balls of quasi-normed spaces \chg{and} is more general than what we shall need. The proof of \cite[Lem.~3]{HPS21} is very similar to the one of Proposition~\ref{pro:radius-general-+}.
\end{remark}

\begin{remark}\label{rem:subgaussian}
	One might also consider other distributions on the space of linear functionals giving rise to other notions of typical information. For example, Mendelson, Pajor and Tomczak-Jaegermann~\cite{MPT07} studied recovery of vectors using information given by i.i.d.\ \chg{copies of a linear measurement} $x\mapsto \langle X,x\rangle$, where $X$ is an isotropic and subgaussian random vector, that is, for some $\alpha\ge 1$,
\[
\chg{\IE X =0,\quad }\IE |\langle X,x\rangle|^2 = \|x\|_2^2 \quad \text{and}\quad \|\langle X,x\rangle \|_{\psi_2}\le \alpha \|x\|_2, \quad x\in \IR^m,
\]
where $\|\cdot\|_{\psi_2}$ is an Orlicz norm with function $\psi_2(t)=e^{t^2}-1$, see, e.g., \cite[Ch.~3.5.2]{AGM15}. This includes, aside from the Gaussian measure, the uniform measure on $\{-1,1\}^m$ and the normalized volume on various symmetric convex bodies such as $\ell_p$-balls, $2\le p\le \infty$. 
\end{remark}

\section{Standard information and sampling algorithms}
\label{sec:std}

In the following, we will discuss standard information where the measurements correspond to function evaluations at a point set. In particular, the objects of study will be functions and $F$ will be a normed space of functions $f\colon D\to \IR$ on some (infinite) set $D$, which may be a subset of $\IR^d$, a manifold or an even more general space. As in Section~\ref{sec:ibc}, let $S\colon F\to G$, where $G$ is a normed space. 

The class of standard information is $\lstd:=\{\delta_x\colon x\in D\}$, where for all $x\in D$ the delta functional $\delta_x$ is the point evaluation $f\mapsto \delta_{x}(f):=f(x)$ for all $f\in F$. Then a standard information map is given by
\begin{equation} \label{eq:info-std}
N_n=(\delta_{x_1},\ldots,\delta_{x_n}),\quad
\text{where }x_i\in D \text{ for all }i=1,\dots,n.
\end{equation}

We can identify such an information map $N_n$ with the point set $\Pn=\{x_1,\ldots,x_n\}$ it relies on and write $N_n(f)=\big( f(x_1),\ldots,f(x_n)\big)$. Here and in the following, we always assume that such point sets are non-empty and finite, and that their points are pairwise distinct. For every $x\in D$ the evaluation functional $\delta_x$ is linear and we shall also assume that it is continuous or bounded, that is, there is a constant $C>0$ such that, for every $f\in F$, we have $|f(x)|\le C\|f\|_F$. Then $\lstd\subset \lall$ holds. 

The definition of the radius of the information map $N_n=(\delta_{x_1},\ldots,\delta_{x_n})$ can be written in terms of the point set $\Pn=\{x_1,\ldots,x_n\}$ and we use the notation
\[
r(S\colon F\to G,\Pn)
:=r(S\colon F\to G,N_n)
=\inf_{\varphi}\big\|S(f)-\varphi\big(f(x_1),\ldots,f(x_n)\big)\big\|_G,
\]
where the infimum is over all maps $\varphi:\IR^n\to G$. From now on we shall always replace standard information $N_n$ as in \eqref{eq:info-std} by its underlying point set and refer to the radius of information given by a point set also as the quality of the point set itself.

A special case of \eqref{eq:def-min-info}, the $n$-th minimal radius of standard information is
\[
r(S\colon F\to G,\lstd,n)
=\inf_{\#\Pn= n}
r(S\colon F\to G,\Pn), \quad n\in\IN,
\]
where the infimum is over all $n$-point sets $\Pn\subset D$. Here, the number $\#\Pn=n$ denotes the cardinality of $\Pn$. The $n$-th minimal radius is sometimes called $n$-th sampling number and may be expressed as
\[
r(S\colon F\to G,\lstd,n)
=\inf_{S_{\Pn}}e(S\colon F\to G,S_{\Pn})
\]
where the infimum is over all sampling operators of the form
\begin{equation}\label{eq:sampling-operator}
	\chg{S_{\Pn}\colon F \to G}, 
 \qquad S_{\Pn}(f)=\varphi\big(f(x_1),\hdots,f(x_n)\big)
\end{equation}
where $\Pn=\{x_1,\ldots,x_n\}\subset D$ consists of $n$ points and $\varphi\colon\IR^n\to G$ is arbitrary. By Proposition~\ref{pro:radius-general}, for any bounded linear $S$ and any point set $\Pn\subset D$,
\begin{equation}
	\label{eq:radius-sampling}
	\sup_{\substack{f\in F_0\\ f|_{\Pn}=0}}\|S(f)\|_G
\le r(S\colon F\to G,\Pn)
\le 2 \sup_{\substack{f\in F_0\\f|_{\Pn}=0}}\|S(f)\|_G,
\end{equation}
and by taking the infimum over all $n$-point sets this holds for the sampling numbers. Here, $f|_{\Pn}$ is the restriction of $f$ to $\Pn$. Via Proposition~\ref{pro:radius-general-+} this transfers to the quasi-Banach case, see also Novak and Triebel~\cite[Prop.~19]{NT06}.

Since the infimum in the definition of the sampling numbers is over the smaller set of function evaluations, we have
\[
r(S\colon F\to G,\lstd,n)\ge r(S\colon F\to G,\lall,n),\quad n\in \IN,
\]
and in general comparing the power of standard and linear information is of much interest in IBC, see for example \cite{NW12}.

To formalize Question~\ref{que:question} on the power of random information for standard information, we need to give meaning what ``most'' means and for this need a ``uniform'' measure on the set of evaluation functionals $\lstd$. As each functional $\delta_x\in\lstd$ is in one-to-one correspondence with a point of $x\in D$, this amounts to putting a measure on $D$. In general, let $D=(D,\Sigma,\lambda)$ be a finite measure space, where the measure $\lambda$ will be assumed to be a probability measure. Often, a suitable choice of $\lambda$ can be thought of as uniform measure and may be used to define random standard information as follows. 

Recall that for linear information on $\IR^m$ the space itself parametrizes the linear functionals and we defined a random linear functional as a Gaussian vector. For standard information we can identify measurements with points and use random points instead.

\begin{definition}\label{def:std-ran}
A random point evaluation is given by $\delta_X$, where $X$ is a random point $X$ in $D$ distributed according to $\lambda$, that is $\IP(X\in A)=\lambda(A)$ for every $A\in\Sigma$. For random standard information of cardinality $n\in\IN$ we take the map $N_n=(\delta_{X_1},\ldots,\delta_{X_n})$ where $X_1,\ldots,X_n$ are independent copies of $X$. It can be identified with the random point set $\pnran:=\{X_1,\dots,X_n\}$. 
\end{definition}

In many cases there are obvious choices for such a measure $\lambda$. For example, we will deal with $D$ being some $d$-dimensional open set equipped with the normalized Lebesgue measure defined by $\vol(A\cap D)/\vol(D)$ for all Borel sets $A\subset \IR^d$, where $\vol=\vol_d$ denotes the $d$-dimensional Lebesgue measure. Further examples include Riemannian manifolds, on which the metric induces a volume measure, or the Grassmannian  with its rotation invariant probability measure, which was discussed in the previous section. 

For each realization, the random point set $\pnran=\{X_1,\ldots, X_n\}$ is a point set and we ignore the possibility that points of it coincide as this event will have probability zero in the cases we consider. The radius of random standard information is the random quantity	$r(S\colon F\to G,\pnran)$ and it is at least the $n$-th minimal radius of standard information. In other words, random points can not be better than optimal points. We translate Question~\ref{que:question} to standard information as follows. 

\begin{question}\label{que:std}
What is the power of random standard information?  \hfill
\begin{center}
Do we have for the radius of random information that 
\[
r(S\colon F\to G, \pnran) \lesssim r(S\colon F\to G,\lstd, n)
\]
with high probability?
\end{center}
\end{question}

In order to give an answer to this question, we will study the radius of information of arbitrary realizations and derive expressions which can be estimated for the random point set $\pnran$. The following notation which is a variant of the asymptotic notation in Definition~\ref{def:asymp} will be useful.
\begin{definition}\label{def:asymp-std}
	Given two real numbers $a(\Pn)$ and $b(\Pn)$ which depend on a point set $\Pn\subset D$, we write $a(\Pn)\asymp b(\Pn)$ if there exist $c,C>0$ such that 
\[
c\, a(\Pn) \le b(\Pn) \le C\,a(\Pn)\quad \text{for all }\Pn=\{x_1,\dots,x_n\}\subset D,n\in\IN,
\]
and analogously for $\lesssim$ and $\gtrsim$.	
\end{definition}

\label{loc:cb}
We will deal with Question~\ref{que:std} in Chapters~\ref{ch:sob}, \ref{ch:interlude} and \ref{ch:iso}. We want to describe the approximation and the integration problems for which we study it in more detail. Let $D=(D,\Sigma,\lambda)$ be a finite measure space. If $D$ is additionally a topological space, then the continuity of the evaluation functionals is ensured by assuming that $F\hookrightarrow C_b(D)$, i.e., $F$ is continuously embedded into $C_b(D)$, the space of bounded continuous real-valued functions on $D$ equipped with the supremum norm $\|f\|_{\infty}:=\sup_{x\in D}|f(x)|$. In this case, there exists a constant $C>0$ such that $\sup_{x\in D}|f(x)|\le C\|f\|_F$ for all $f\in F$. This assumption, which is slightly stronger than necessary, is satisfied by the function spaces considered in Chapter~\ref{ch:sob}.

Let further $\mu$ be a probability measure on $D$ which is absolutely continuous with respect to $\lambda$ and has bounded density $\varrho\colon D\to\IR$.

\bigskip

\textbf{$L_q$-approximation.} This is the approximation problem considered at the end of Section~\ref{sec:ibc} with $(G,\|\cdot\|_G)=(L_q(D,\mu),\|\cdot\|_{L_q(D,\mu)})$ for some $0< q \le \infty$. Here, $L_q(D,\mu)$ is the space of (equivalence classes of) measurable functions $f\colon D\to\IR$ with finite $L_q$-norm  
\begin{equation} \label{eq:lqnorm}
\|f\|_{L_q(D,\mu)}
:=\begin{cases}
	\Big(\int_D |f(x)|^q \dd \mu(x)\Big)^{1/q}&\text{if }\chg{0< q<\infty,}\\
	\esssup_{x\in D} |f(x)|&\text{if }\chg{q =\infty}.
\end{cases}
\end{equation}
The embedding into $C_b(D)$ ensures, for all $0< q\le\infty$, the embedding into $L_q(D,\varrho)$ because of $\chg{\|f\|_{L_q(D,\mu)}^q\le C\|\varrho\|_{\infty}\lambda(D)\|f\|_F^q}$, for all $f\in F$. If $\mu=\lambda$ we will omit it in the notation and simply write $L_q(D)$.

\bigskip

\textbf{Integration.} \label{loc:int} Here, we choose $S={\rm INT}_{\mu}$ and $G=\IR$, where 
\[
{\rm INT}_{\mu}\colon F\to \IR,\quad
{\rm INT}_{\mu}(f)
:=\int_D f(x) \dd\mu(x), \quad f\in F,
\]
is the integration functional. Note that ${\rm INT}_{\mu}(f)\le \|f\|_{L_1(D,\mu)}$ and thus $F\hookrightarrow C_b(D)$ also ensures the continuity of $\intmu$. In general, we omit the space $G=\IR$ in the notation and write $e(F,\intmu,A)$ for the worst-case error of an algorithm $A$ and $r(F,\intmu,\Pn)$ for the radius of information given by $\Pn\subset D$. If $\mu=\lambda$, we shall use ${\rm INT}$ instead of $\intmu$. 

\bigskip

Both problems will be studied in Chapter~\ref{ch:sob} for Sobolev spaces and also other function spaces defined on subsets of $\IR^d$. For simplicity, only the case of $\mu=\lambda=\vol$, the Lebesgue measure, will be considered there. At this point, we remark that on $\IR^d$ measurability and integrability are understood in the sense of Lebesgue, except mentioned otherwise. The integration problem is also central to Chapter~\ref{ch:interlude}, where it is treated with respect to general probability measures on $\IR^d$, and Chapter~\ref{ch:iso}, where mainly the cube equipped with $\vol$ is treated.

\begin{remark}
Since the measure $\mu$ influences the contribution of subsets of $D$ to the approximation error, one may use $\mu$ instead of $\lambda$ to define random standard information. However, in this thesis we will not consider sampling from distributions different from the uniform one. 
\end{remark}

\chapter{Gaussian information for generalized ellipsoids}
\label{ch:ell}

In this chapter, we will present results from \cite{HPS21} on the power of random linear information for the recovery of vectors from $\ell_p$-ellipsoids, and thus generalize the work \cite{HKN+21} for the usual ellipsoids. To give some idea, the $\ell_p$-ellipsoids can be defined as the images of $\ell_p$-balls $\IB_p^m$ under diagonal operators and their Gelfand widths are \chg{thus} related to the Gelfand numbers of diagonal operators. In the upcoming Section~\ref{sec:ell-intro} we present their connection to certain function spaces after surveying the literature on random sections of convex bodies.

Let us describe the remainder of this chapter. In Section~\ref{sec:ell-res} we present our main results from \cite{HPS21} and discuss connections with existing work. Then in Section~\ref{sec:ell-poly} we shall look at the important example of $\ell_p$-ellipsoids with polynomial semiaxes and assess the power of random linear information for those. Afterwards in Sections~\ref{sec:ell-rounding}, \ref{sec:ell-lower} and \ref{sec:ell-sparse} we explain the methods behind the proofs, essentially covering topics ranging from asymptotic geometric analysis to sparse approximation. Finally, in Section~\ref{sec:ell-open} we collect points of departure for further research. 

\section{Introduction and motivation}
\label{sec:ell-intro}

Consider the $m$-dimensional approximation problem from Section~\ref{sec:lin}, where linear information is used to recover vectors in the $\ell_2^m$-norm from the unit ball of an $m$-dimensional normed space, which is a symmetric convex body $K$. Motivated by Question~\ref{que:lin}, we study the radius of Gaussian information which is equivalent to the radius of the intersection of $K$ with a random subspace of codimension $n$. We will be interested in the case of $m$ being much larger than $n$, say at least $m>2n$.

The study of the radius, or equivalently the diameter, of the section of a convex body with a linear subspace has its roots in the local theory of Banach spaces, see V.~D.~Milman~\cite{Mil85a,Mil85b}. Today, it is part of the field of asymptotic geometric analysis, where it is connected with low $M^*$-estimates, which will be discussed in Section~\ref{sec:ell-rounding} as an important tool behind our bounds. With particular focus on subspace dimensions proportional to the dimension of the convex body, this was investigated, among others, by Giannopoulos and Milman~\cite{GM97,GM98}, also together with Tsolomitis \cite{GMT05}, and quite recently by Litvak, Pajor and Tomczak-Jaegermann~\cite{LPT06}. 

In \cite{LPT06} a random matrix perspective was employed to show that, on the scale of proportional subspaces, typical intersections of a symmetric convex body $K\subset \IR^m$ are not much larger than minimal intersections. More precisely, typical intersections of $K$ with $N$-codimensional random subspaces are at most a factor $C(\varepsilon)$ larger than the $n$-th Gelfand width $c_n(K)$ provided that $\chg{N/m>\varepsilon}$ and $N/n>1+\varepsilon$ for some fixed $\varepsilon>0$. This is an instance of the ``existence versus prevalence'' phenomenon mentioned in the introduction, see also \cite{GMT05,Ver06}.

In contrast to \cite{LPT06}, where the knowledge of one minimal section suffices, Litvak and Tomczak-Jaegermann~\cite{LT00} obtained a bound on the radius of random section in terms of a suitable average over the tail of the sequence of minimal radii. To make this precise, define the random Gelfand numbers of the operator $\id\colon (\IR^m,\|\cdot\|_K)\to \ell_2^m$, that is the random Gelfand widths of $K$, for $1\le n\le m$, by
\begin{equation} \label{eq:crn-def}
\crn(K)
:=\inf\{a>0\colon \rad(K,E_{n-1}^{\rm ran})<a \},
\end{equation}
where in this relation a property holds for a random subspace $E_{n-1}^{\rm ran}$ of codimension $n-1$ if it is satisfied with exponentially high probability in \chg{$n$}, that is, with probability at least $1-\chg{\exp\big(-\alpha_0(n+1)\big)}$ for a suitable $\alpha_0>0$ fixed beforehand.  Then, for not necessarily symmetric $K$, \cite[Thm.~3.2]{LT00} yields constants $c,C>0$ such that the random Gelfand widths satisfy
\begin{equation}\label{eq:general-upper}
\crn(K)
\le C \frac{1}{\sqrt{n}}\sum_{j=\lfloor c\, n\rfloor}^{m} \frac{c_j(K)}{\sqrt{j}}, \quad 1\le n \le m.
\end{equation}
Recall that the Gelfand width $c_{n+1}(K)=\rad(K,n)$ is the minimal radius over $n$-codimen\-sional sections of $K$, which is in turn equivalent to $r(K,n)$, the $n$-th minimal radius of linear information. It can be deduced from this bound that if the Gelfand widths decay at a polynomial rate faster than $n^{-1/2}$, then the random Gelfand widths are asymptotically equivalent to the Gelfand widths and Question~\ref{que:lin} can be answered in the positive. We will say more about this in the next section when we compare this bound to our results. 

The bound in \eqref{eq:general-upper} is very general but we do not have a matching lower bound. To increase our understanding, we will study concrete examples such as $\ell_p$-ellipsoids which generalize the usual ellipsoid 
\begin{equation} \label{eq:ellipsoid}
\elz
:= \Big\{x\in\IR^m\colon \sum_{i=1}^{m} \frac{|x_i|^2}{\sigma_i^2}\le 1\Big\},
\end{equation}
with semiaxes $\sigma=(\sigma_1,\sigma_2,\ldots)$ of positive length. For the sake of \chg{motivation, we} provide some details about the connection between the sections of ellipsoids and linear information for Hilbert space embeddings before introducing $\ell_p$-ellipsoids. 

In the context of function spaces, ellipsoids appear in the study of a compact embedding of a Hilbert space into an $L_2$-space. This is well known and for example elaborated in \cite{HKN+21}, see also Pinkus~\cite[Ch.~IV]{Pin85}. To give the details, let $S=\id\colon H\to K$ be a compact embedding between separable Hilbert spaces $H,K$. Then one can form an orthonormal basis $e_1,e_2,\ldots$ of $K$ using the normalized eigenfunctions of the compact selfadjoint operator $S^* S$ defined on $K$, where $S^*$ is the adjoint. The basis is assumed to be ordered such that the singular numbers $\sigma_j=\sqrt{\lambda_j}$ with $S^* S e_j=\lambda_j e_j$, where $j\in\IN$, satisfy $\sigma_1\ge \sigma_2\ge \cdots$. The unit ball in $H$ is then given by all functions $f=\sum_{j=1}^{\infty}a_j e_j$ with
\[
\|f\|_H
=\sum_{j=1}^{\infty}\frac{a_j^2}{\sigma_j^2}
\le 1,
\]
that is an infinite dimensional ellipsoid. 

Linear functionals on $H$ can then be identified with sequences and random information becomes a sequence of independent standard Gaussian random variables. It is known that optimal information using $n$ linear functionals is given by evaluating the first coefficients $a_j=\langle f,e_j\rangle_H$ and that the optimal algorithm is $A(f)=\sum_{j=1}^{n}\langle f,e_j\rangle_H e_j$, which is the projection \chg{onto} the coordinates in which the ellipsoid is largest. Then the $n$-th minimal radius of linear information is given by the singular value $\sigma_{n+1}$, which is also the radius of the minimal section of $\elz$, i.e., 
\[
\rad(\elz,n)=r(\elz,n)=\sigma_{n+1}, \quad 0\le n <m.
\]
This is because \chg{sections of an ellipsoid are again ellipsoids and} the diameter of an ellipsoid is twice its radius. Thus, we have equality in Proposition~\ref{pro:radius-general} and consequently relation \eqref{eq:gelfand-radius} instead of the constant two, see, e.g., \cite{TWW88} or \cite[Ch.~4]{NW08}.

In \cite{HKN+21} the relation between random and minimal sections of an ellipsoid has been investigated and by \cite[Thm.~3]{HKN+21} there exists an absolute constant $C>0$ such that 
\begin{equation}\label{eq:upper-2}
\rad(\elz,\enran)
\leq C \frac{1}{\sqrt{n}}\Big(\sum_{j\ge \lfloor n/4\rfloor}\sigma_j^2\Big)^{1/2},\quad \text{for }1\le n<m,
\end{equation}
with exponentially high probability (in $n$). Similar to the bound \eqref{eq:general-upper} it yields that the typical radius is asymptotically equivalent to the minimal radius provided the latter decays fast enough for increasing $n$.

However, if the semiaxes $\sigma=(\sigma_j)_{j\in\IN}$ decay too slowly and are not square-summable, that is $\|\sigma\|_2=\infty$, then the behaviour of the radii of typical sections changes completely. In this case, \cite[Thm.~5]{HKN+21} yields that, for any $\varepsilon>0$,
\begin{equation} \label{eq:lower-2}
\rad(\mathcal{E}_{\sigma}^m,\enran)
\geq  (1-\varepsilon)\sigma_1
\end{equation}
with exponentially high probability, provided that $m$ is large enough compared to $n$. 

Taken together, these bounds yield the following dichotomy on the behaviour of the radius of random information compared to the minimal radius of information. Namely, if $\|\sigma\|_2<\infty$, then the radius of random information, $r(\elz,G_{n,m})$, is close to optimal for all $m \, (>n)$. On the other hand, if not, then random information is asymptotically useless if $m \,(>n)$ is large enough. Both statements hold with exponentially high probability.
\begin{remark}
	It should be noted that in \cite{HKN+21} also an upper bound more suitable for exponentially decaying sequences of semiaxes was obtained, where one looses at most a polynomial factor. We will not consider this but focus our attention on polynomially decaying semiaxes later on. 
\end{remark}

We wish to extend the results about ellipsoids and investigate the above dichotomy for the more general $\ell_p$-ellipsoids, defined as follows. 

\begin{definition}\label{def:elp}
An $m$-dimensional $\ell_p$-ellipsoid, $0< p\le \infty$, with semiaxes of length $ \sigma_1\ge \cdots \ge \sigma_m >0$ is given by
\[
\elp:=\Big\{x\in\IR^m\colon \sum_{i=1}^{m} \frac{|x_i|^p}{\sigma_i^p}\le 1\Big\},
\]
where the sum over the $p$-th powers is replaced by the maximum if $p=\infty$. 
\end{definition}

The $\ell_p$-ellipsoid $\elp$ is the unit ball of $\IR^m$ equipped with the (quasi-)norm
\begin{equation} \label{eq:psigma-norm}
\|x\|_{p,\sigma}
:=\|(x_i/\sigma_i)_{i=1}^m\|_p,\quad
\text{}x\in\IR^m,
\end{equation}
and therefore it is convex if and only if $p\ge 1$. See Figure~\ref{fig:ellipsoids} for an illustration.
\begin{figure}[ht]
	\begin{subfigure}{.22\textwidth}
		\centering
		\includegraphics[width=\linewidth,trim={.5cm 1cm 1cm .5cm},clip]{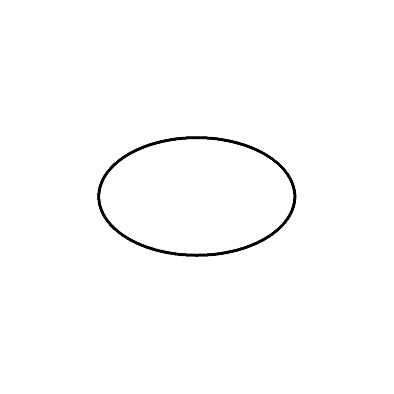}
		\caption*{\phantom{aaa}$p=2$}
	\end{subfigure}
	\begin{subfigure}{.22\textwidth}
		\centering
		\includegraphics[width=\linewidth,trim={.5cm 1cm 1cm .5cm},clip]{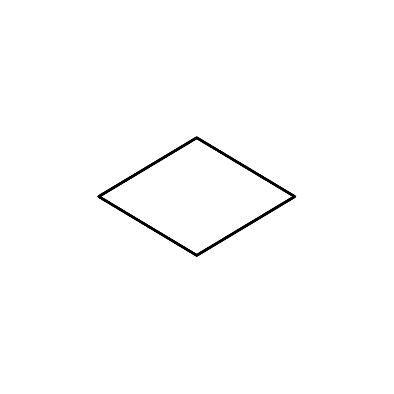}
		\caption*{\phantom{aaa}$p=1$}
	\end{subfigure}
	\begin{subfigure}{.22\textwidth}
		\centering
		\includegraphics[width=\linewidth,trim={.5cm 1cm 1cm .5cm},clip]{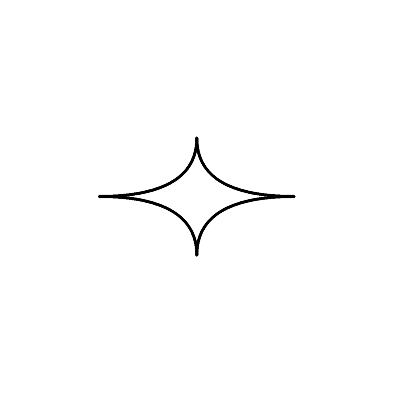}
		\caption*{\phantom{aaa}$p=\frac{1}{2}$}
	\end{subfigure}
	\begin{subfigure}{.22\textwidth}
		\centering
		\includegraphics[width=\linewidth,trim={.5cm 1cm 1cm .5cm},clip]{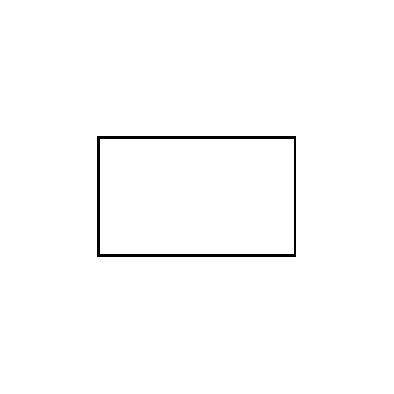}
		\caption*{\phantom{aaa}$p=\infty$}
	\end{subfigure}
	\caption{Two-dimensional $\ell_p$-ellipsoids with semiaxes $\sigma_2=0.6\,\sigma_1$.}
	\label{fig:ellipsoids}
\end{figure}

As a first motivation to study $\ell_p$-ellipsoids, consider the diagonal operator $D_{\sigma}$ with multiplier sequence $\sigma=(\sigma_1,\ldots,\sigma_m)$ taking $x\in\IR^m$ to $D_{\sigma}(x)=(\sigma_1x_1,\ldots,\sigma_m x_m)$. Then $\elp$ is the image of the $\ell_p$-unit ball $\IB_p^m$ under $D_{\sigma}$, and a change of variables shows that, for any linear subspace $E_n\subset \IR^m$,
\[
\rad(\elp,E_n)
=\sup_{\elp\cap E_n}\|x\|_2
=\sup_{\IB_p^m\cap E_n}\|D_{\sigma}x\|_2.
\]
By taking the infimum over all subspaces of codimension $n$ we obtain
\begin{equation} \label{eq:gelfand-diag}
\rad(\elp,n)
=c_{n+1}(D_{\sigma}\colon \ell_p^m\to \ell_2^m) \quad \text{for all }n\in\IN.
\end{equation}
This relation is used in Appendix~\ref{sec:ell-gelfand} to determine the minimal radius for the $\ell_p$-ellipsoids. 

In the study of operators between (quasi-) Banach spaces, diagonal operators are important prototypes, in part due to their particular structure, and in part because some interesting operators act like them. For example, differential operators on periodic functions may be analyzed via diagonal operators applied to their Fourier coefficients, and also embeddings between Besov and Sobolev spaces can be modeled using them. We refer to Carl~\cite{Car81}, König~\cite[Ch.~3.c]{Koe86}, Kühn~\cite{Kue05}, Linde~\cite{Lin85}, Pietsch~\cite[Ch.~7.7.1]{Pie87} and Schütt~\cite{Sch84} for detailed information on this topic.	

Additionally, we would like to mention that $\ell_p$-ellipsoids have been studied quite recently by Juhos and Prochno \cite{JP21} with focus on the asymptotic volume distribution of the intersection of two such generalized ellipsoids as the dimension of the underlying space tends to infinity, and by van Handel \cite{van18} as examples to investigate the intricacies of entropy bounds on the expected supremum of Gaussian processes. 

To give more details, it has been observed by Talagrand in \cite[Ch.~2.5]{Tal14} that ellipsoids $\elz$ are examples where the generic chaining approach is superior to Dudley's bound on the expected supremum of a Gaussian process $(Z_t)_{t\in K}$ with $Z_t=\sum_{i=1}^{m}g_j t_j$, which is indexed by a convex body $K$, in terms of entropy numbers of $K$ is not optimal. Intuitively, this is because an ellipsoid is smaller than its entropy numbers suggest. As noted by van Handel~\cite{van18}, this extends to $\ell_p$-ellipsoids. We will discuss this in the upcoming section. 

At the end of this section, let us briefly present how $\ell_p$-ellipsoids relate to the following spaces of periodic functions, which were introduced by Sprengel~\cite{Spr00} in order to generalize error estimates for interpolation using sparse grids.

\begin{definition}\label{def:aaq}
For $1\le q \le\infty$ and $\alpha\in\IR$ define
\[
A_q^{\alpha}(\TT^d)
:=\Big\{f\in D'(\TT^d)\colon \|f\|_{A_q^{\alpha}(\TT^d)}:=\Big(\sum_{k\in\ZZ^d}\big|(1+\|k\|_2^2)^{\alpha/2}c_k(f)\big|^q\Big)^{1/q}<\infty\Big\},
\]
where $\big(c_k(f)\big)_{k\in\ZZ^d}$ are the Fourier coefficients of $f$ belonging to $D'(\TT^d)$, the space of periodic distributions on the $d$-dimensional torus $\TT^d$. 
\end{definition}

This scale of function spaces encompasses the Wiener algebra of periodic functions with absolutely summable Fourier coefficients ($q=1$ and $\alpha=0$) and periodic Sobolev spaces ($q=2$ and $\alpha\in\IN$). In one dimension, the case $q=\infty$ and $\alpha>1$ corresponds to Korobov spaces, whose multi-dimensional generalizations are integral to the study of lattice rules, see, e.g., Niederreiter~\cite[Ch.~4]{Nie92}. In general, a description of the space $A_q^{\alpha}(\TT^d)$ in terms of smoothness is \chg{probably} out of reach as already in one dimension there is no precise relation between smoothness and decay conditions on the Fourier coefficients. 

To make precise how the function spaces of Definition~\ref{def:aaq} are connected to $\ell_p$-ellipsoids, enumerate the integer points $k\in\ZZ^d$ by $\big(k(j)\big)_{j=1}^{\infty}$ such that their Euclidean norm is non-decreasing. Then, for every $f\in A_q^{\alpha}(\TT^d)$,	the sum defining its norm can be regrouped to
\[
\|f\|_{A_q^{\alpha}(\TT^d)}^q 
=\sum_{j=1}^{\infty} \frac{|c_{k(j)}|^q}{\sigma_j^q}
\]
where $\sigma_j:=(1+\|k(j)\|_2^{2})^{-\alpha/2}$ for $ j\in \IN$. Thus, via the correspondence $f\leftrightarrow \big(c_k(f)\big)_{k\in\ZZ^d}$, the unit ball of $A_q^{\alpha}(\TT^d)$ may be identified with an infinite-dimensional $\ell_q$-ellipsoid. Then a straightforward volume estimate implies the asymptotic relation $\sigma_j\asymp j^{-\alpha/d}$, and if $\alpha/d>1-1/q=1/q^*$ or $\alpha=0$ and $q=1$, then $A_q^{\alpha}(\TT^d)$ is embedded into the continuous functions. Curiously enough, we will encounter this condition on the semiaxes in our study of the behaviour of the radius of a random section.

\section{Discussion of results}
\label{sec:ell-res}

Having sufficiently motivated the study of sections of $\ell_p$-ellipsoids with random subspaces, we discuss our results obtained in \cite{HPS21}. We first state a general upper bound on the radius $\rad(\elp,\enran)$ of a random section for $1\le p\le\infty$ (Theorem~\ref{thm:ell-main}), which we compare with the bound of \eqref{eq:general-upper} given in \cite{LT00}. Then we discuss Corollary~\ref{cor:mstar-estimate} which extends a bound on the supremum of a Gaussian process indexed by an $\ell_p$-ellipsoid obtained by van Handel~\cite{van18}. Afterwards an upper bound on the typical radius (Theorem~\ref{thm:ell-main-quasi}) obtained in the range $0<p<1$ together with a consequence (Corollary~\ref{cor:gelfand-quasi}) for the Gelfand numbers of diagonal operators is presented. 

In Section~\ref{sec:ell-poly} we shall discuss ramifications for polynomial semiaxes. We shall also state a lower bound (Proposition~\ref{pro:lower}) and discuss a dichotomy for the optimality of Gaussian information depending on the decay of the semiaxes. This leads us to Conjecture~\ref{conj:threshold}.

The following main result bounds the typical radius from above via the $p^*$-norm of the tail of the sequence of semiaxes. Here and in the following, the parameter $p^*$ is the Hölder conjugate of $p$, that is $\frac 1p +\frac{1}{p^*}=1$.

\begin{theorem}[{\cite[Thm.~A]{HPS21}}]\label{thm:ell-main}
	There exists a constant $C>0$ such that, for all $m\in\IN$ and $1\le n<m$, we have
	\[
	\rad(\elp,\enran)
		\le C
		\begin{cases}
			n^{-1/2}\sup\limits_{k\le j\le m}\sigma_{j}\sqrt{\log(j)+1} &\text{if } p=1, \\
		\sqrt{p^*}n^{-1/2} \Big(\sum\limits_{j=k}^{m}\sigma_{j}^{p^{*}}\Big)^{1/p^{*}} &\text{if } 1<p \leq \infty,
		\end{cases}
    \]
	with probability at least $1-c_1 \exp(-c_2 n)$, where $c_1,c_2\in(0,\infty)$ are absolute constants, and $k=k(n)\asymp n$ for $p=1$ while $k=k(n)\asymp \frac{n}{p^*}$ for $p>1$.
\end{theorem}

The proof will be discussed in Section~\ref{sec:ell-rounding}.

\begin{remark}\label{rem:ell-quasi}
	The notation we use here is different from \cite{HPS21}, where we used $\rad$ also for the radius of information and the radius of Gaussian information $r(\elp,G_{n,m})$. Here, the radius of a typical section $\rad(\elp,\enran)=\rad(\elp,\ker G_{n,m})$ is used to state the results. However, by Proposition~\ref{pro:radius-general-+} these notions are equivalent up to a constant of two if $p\ge 1$ and a constant depending on $p$ otherwise. To see this, note that $\elp$ is the unit ball of $\IR^m$ equipped with $\|\cdot\|_{p,\sigma}$, which satisfies, by the $p$-triangle inequality, 
\[
\|x+y\|_{p,\sigma}\le 2^{(1/p-1)_+}\big(\|x\|_{p,\sigma}+\|y\|_{p,\sigma}\big)\quad\text{for every }x,y\in\IR^m.
\]
\end{remark}

Theorem~\ref{thm:ell-main} canonically extends the bound \eqref{eq:upper-2} obtained in \cite{HKN+21} for the case $p=2$ and shows that if $\sigma\in \ell_{p^*}$ we get a bound on the radius of most sections. In order to compare this with the more general upper bound \eqref{eq:general-upper} from above, obtained by Litvak and Tomczak-Jaegermann in \cite{LT00}, we introduce the following concept. Similar to \cite{HKN+21} we will be interested in bounds which are independent of the dimension $m$ and may be extended to infinite $m$. 
\begin{definition}\label{def:lorentz}
Let $0<p,q\le \infty$. The Lorentz (quasi-)norm of a (possibly finite) sequence $x=(x_1,x_2,\dots)$ is given by
\[
\|x\|_{p,q}
:=\|i^{1/p-1/q}x_i^*\|_q,
\]
where $x_1^*\ge x_2^*\ge\cdots$ is the non-increasing rearrangement of $(|x_i|)_{i=1}^{\infty}$, see, e.g., \cite[Ch.~2.1]{Pie87}. 
\end{definition}

For convenience let us recall that \eqref{eq:general-upper} yields
\begin{equation*}
\crn(\elp)
\le C \frac{1}{\sqrt{n}}\sum_{j=\lfloor c\, n\rfloor}^{m} \frac{c_j(\elp)}{\sqrt{j}}, \quad \text{for all }n\in\IN
\end{equation*}
with exponentially high probability in \chg{$n$.}

To compare this with the upper bound of Theorem~\ref{thm:ell-main}, we need to know the behaviour of the Gelfand widths $\big(c_n(\elp)\big)_{n\le m}$ or, equivalently, the minimal radius of information. Due to equality \eqref{eq:gelfand-diag}, we can use results about Gelfand numbers of diagonal operators to infer the behaviour of $	\big(c_n(\elp)\big)_{n\le m}$ in terms of the sequence of semiaxes $\sigma$. We collect these results in Appendix~\ref{sec:ell-gelfand} and note that for $p<2$ one knows this behaviour only to polynomial order.

The bound in \eqref{eq:general-upper} is finite for all $m\in\IN$ if and only if $\big(c_n(D_{\sigma}\colon \ell_p\to\ell_2)\big)_{n\in\IN}\in \ell_{2,1}$. From Proposition~\ref{pro:lorentzeq} in Appendix~\ref{sec:ell-gelfand} one can deduce that for $p\ge 2$ this holds if and only if $\sigma\in \ell_{p^*,1}$. This is a stronger requirement than $\sigma\in \ell_{p^*}$, which is needed by Theorem~\ref{thm:ell-main} for a finite upper bound independent of $m$.  Note that this observation is already true for $p=2$ and thus could have been made after the statement of \eqref{eq:upper-2}. 

Let us elaborate. For $p=2$ the semiaxes coincide with the minimal radii, i.e., we have $\sigma_{n+1}=\rad(\mathcal{E}_\sigma^m,n)$. If $r(\elz,n)\asymp n^{-\lambda}(\log n)^{\alpha}$ for $\lambda>1/2$ and $\alpha\in \IR$ both \eqref{eq:general-upper} and \eqref{eq:upper-2} imply that the radius of a random section satisfies 
\begin{equation} \label{eq:ell-answer}
\rad(\elz,\enran)\lesssim \rad(\elz,n)
\end{equation}
with exponentially high probability, i.e., we have asymptotic optimality of random information. If $\lambda=1/2$ and $\alpha<-1/2$, then the bound \eqref{eq:upper-2} gives \eqref{eq:ell-answer} with an additional factor of $\sqrt{\log n}$ on the right-hand side.  This improves upon the general bound~\eqref{eq:general-upper} by a factor of $\sqrt{\log n}$. A slightly smaller improvement of $(\log n)^{1/p}$ may be deduced in the case of $p>2$ and $\alpha<-1+1/p$ from Theorem~\ref{thm:ell-main} and Proposition~\ref{pro:gelfandtail}.

Summarizing, we may say that, applied to $\ell_p$-ellipsoids $\elp$ with $1\le p\le \infty$ and $ m\in\IN$, the bound \eqref{eq:general-upper} is in general slightly weaker than Theorem~\ref{thm:ell-main}. However, we wish to emphasize that the bounds differ only at an logarithmic scale and for comparing the polynomial order, as done in Section~\ref{sec:ell-poly}, they are of equal strength.

As a consequence of the proof of Theorem~\ref{thm:ell-main}, more precisely of Proposition~\ref{pro:mstar-estimate} in Section~\ref{sec:ell-rounding} below, we obtain the following corollary on the supremum of a Gaussian process indexed by an $\ell_p$-ellipsoid.
\begin{corollary}[{\cite[Cor.~2]{HPS21}}]\label{cor:mstar-estimate}
Let $g_1,g_2,\dots$ be independent standard Gaussians. Then there exists a constant $C>0$ such that, for all $m\in\IN$, 
\[
\IE\sup_{y\in \elp}\sum_{j=1}^{m}g_j y_j
\le C
\begin{cases}
	\sup\limits_{1\le j\le m}\sigma_{j}\sqrt{\log(j)+1} &\text{if } p=1, \\
\sqrt{p^*} \big(\sum\limits_{j=1}^{m}\sigma_{j}^{p^{*}}\big)^{1/p^{*}} & \text{if }1<p\le\infty.
\end{cases}
\]
\end{corollary}

In view of the dependence on the parameter $p$, this bound extends the mentioned work by van Handel \cite{van18} who deduced it for $1\le p <\infty$ but with an unspecified constant in $p$ which cannot be obtained using his approach, see \cite[Remark 3.4]{van18}. In Corollary \ref{cor:mstar-estimate}, we obtain an upper bound on the behavior in $p$ and thus complement his result.  

\begin{remark}
Employing estimates for entropy numbers of diagonal operators from \cite{Car81} as in \cite{van18}, it can be deduced that the $\ell_p$-ellipsoid $\elp$ with semiaxes satisfying
\[
	\sigma\in \ell_{p^*} \text{ but } \sigma\not\in\ell_{p^*,1} \text{ for }1<p<\infty \quad\text{or}\quad
	\quad \sup_{j\in\IN}\sigma_{j}\sqrt{\log(j)+1}<\infty\text{ but }\sigma\not\in\ell_{\infty,1}
\]
is an example where Dudley's bound, in contrast to Corollary~\ref{cor:mstar-estimate}, fails to give a finite bound if the dimension becomes large. In the case $1<p<\infty$, van Handel explains this with the fact that the entropy numbers need not be estimated for the whole set but for certain thin subsets, in this case dilations of $\ell_{2p-2}$-ellipsoids. For $p\ge 2$ an alternative explanation can be given in terms of $p$-convexity, see \cite[Ch.~4]{Tal14}. Note that in finite dimension Dudley's bound can be off by at most $\log m$, see \cite[Ch.~2]{Tal14}.
\end{remark}

Regarding the case $0<p<1$, it seems that little is known in general about random sections of unit balls of quasi-Banach spaces and different techniques are necessary than in the case of convex bodies.  Relying on methods commonly used in the field of compressed sensing, e.g., in a work of Foucart, Pajor, Rauhut and T.~Ullrich \cite{FPR+10} on the Gelfand widths of $\ell_p$-balls in the quasi-Banach regime, see also Donoho~\cite{Don06b}, we deduce the following upper bound for the radius of random information when the semiaxes have polynomial decay.

\begin{theorem}[{\cite[Thm.~B]{HPS21}}] \label{thm:ell-main-quasi}
Let $0<p\leq 1$ and $\sigma_j=j^{-\lambda}$, $j\in\IN$, for some $\lambda>0$. Then there exist constants $c,C,D>0$ such that, for all $m\in\IN$ and all $1\le n <m$ with
\[
n\ge D\log({\rm e}m/n),
\]
we have
\[
\rad(\elp,\enran)
\le C\,\Big(\frac{\log({\rm e}m/n)}{n}\Big)^{\lambda+1/p-1/2}
\]
with probability at least $1-2\exp(-c\, n)$.
\end{theorem}
To the best of our knowledge, this seems to be the best known upper bound for the minimal radius of sections of $\ell_p$-ellipsoids with $0<p<1$ and thus the best information available is given by random Gaussian measurements. Below, we shall prove the more general Theorem~\ref{thm:ell-main-quasi-general}, where the radius is measured in the $\ell_q$-quasi-norm for $p<q\le 2$. This yields the following bound on Gelfand numbers of diagonal operators in the quasi-Banach regime and extends known results on the identity between $\ell_p^m$ and $\ell_q^m$ from the mentioned works \cite{Don06b} and \cite{FPR+10}.

\begin{corollary}[{\cite[Cor.~4]{HPS21}}]\label{cor:gelfand-quasi}
Let $0<p\leq 1$ and $p<q\le 2$. Assume that $\sigma_j=j^{-\lambda}$, $j\in\IN$, for some $\lambda>0$. Then there exist constants $C,D>0$ such that, for all $m\in\IN$ and all $1\le n <m$ with
\[
n\ge D\log({\rm e}m/n),
\]
we have
\[
c_{n+1}(D_{\sigma}\colon \ell_p^m\to \ell_q^m)
\le C\, \Big(\frac{\log({\rm e}m/n)}{n}\Big)^{\lambda+1/p-1/q}.
\]
\end{corollary}

\newpage

\section{Discussion of results -- polynomial semiaxes}
\label{sec:ell-poly}

In the following, we discuss the consequences of our results for $\ell_p$-ellipsoids with $1\le p\le \infty$ and polynomially decaying semiaxes. Our main reason for doing so is to be able to compare the minimal radius $\rad(\elp,n)$ with the radius of typical sections $\rad(\elp,\enran)$ and to make progress towards extending the dichotomy for the ellipsoids $\elz$ mentioned in Section~\ref{sec:ell-intro}. We give some further motivation to look at polynomial semiaxes.

An incentive can be seen in the function spaces $A_q^{\alpha}(\TT^d)$ of Definition~\ref{def:aaq}, whose unit ball can be identified with an infinite-dimensional $\ell_q$-ellipsoid with semiaxes $\sigma_j\asymp j^{-\alpha/d}$. Let us mention here that Theorem~\ref{thm:ell-main} can be applied to this sequence if and only if $\alpha/d>1/p^*$. In this case, one could use an isomorphism to obtain a bound on random linear information for the spaces $A_q^{\alpha}(\TT^d)$.

For the sake of further motivation, let us mention that the unit ball of $\IR^m$ equipped with the Lorentz (quasi-)norm $\|\cdot\|_{p,q}$ bears some similarity to an $\ell_q$-ellipsoid with semiaxes $\sigma_j=j^{1/q-1/p}$. Under the assumption that the coefficients are ordered non-increasingly, one may interpret recovery of functions whose Fourier or wavelet coefficients belong to a Lorentz ball as a possible application of our results. This includes for example nonlinear approximation classes, in part related to Besov spaces, see for example \cite{DeV98}. 

Let now $\sigma_j=j^{-\lambda}$ for all $j\in\IN$ and some $\lambda>0$. Equivalent results hold if this is weakened to $\sigma_j\asymp j^{-\lambda}$. In order to assess the size of the radius of a typical section of $\elp$, we need to know the minimal radius $\rad(\elp,n)$ for all $n\in\IN$.  Its behavior is known exactly when $p\ge 2$ but can only be deduced up to subpolynomial factors when $1\le p<2$, see Appendix~\ref{sec:ell-gelfand}. Therefore, we will ignore subpolynomial terms in the following discussion and use only the rate of polynomial decay to compare.

To make this precise, we define the rate of polynomial decay (in $n$) of an infinite array $a=(a_{n,m})_{m\in\IN,1\le n <m}$ of real numbers by
\label{loc:decay-det}
\[
\decay(a)
:=\sup\{\varrho\ge 0\colon \exists C>0 \text{ with }a_{n,m}\le C\, n^{-\varrho} \text{ for all }m\in\IN \text{ and }1\le n<m\}.
\]
This definition is derived from the concept of a diagonal limit order used to study Gelfand numbers of diagonal operators arising from a polynomially decaying sequence, see, e.g., Pietsch \cite[6.2.5.3]{Pie07}. 

Further, we need to suppose that $\sigma_j=j^{-\lambda}$, $j\in\IN$, for some 
\[
\lambda>\frac{1}{s}:=\Big(\frac{1}{2}-\frac{1}{p}\Big)_+=\max\Big\{\frac{1}{2}-\frac{1}{p},0\Big\},
\]
since the minimal radius does not decay if $\lambda\le \frac{1}{s}$. For the decay rate of the minimal radius Corollary~\ref{cor:min-order} implies
\begin{equation}\label{eq:min-order}
\decay\big(\rad(\elp,n)\big)
=	\begin{cases}
	\lambda\cdot\frac{p^{*}}{2}&\text{if } 1\le p<2 \text{ and }\lambda<\frac{1}{p^{*}},\\
		\lambda+\frac{1}{p}-\frac{1}{2}&\text{otherwise.}
	\end{cases}
\end{equation}

By Theorem~\ref{thm:ell-main}, if $\sigma$ decays polynomially with exponent $\lambda>\frac{1}{p^*}$, then there is $C>0$ such that the radius of a random section satisfies, for every $m\in\IN$ and $1\le n<m$,  
\begin{equation}\label{eq:bound-polynomial}
\rad(\elp,\enran)
\le C
\begin{cases}
	n^{-\lambda-1/2}\sqrt{\log n}&\text{if }p=1,\\
	n^{-\lambda+1/2-1/p}&\text{if }1<p\leq \infty,
\end{cases}
\end{equation}
with probability $1-c_1\exp(-c_2 n)$. 

\label{loc:decay-ran}
Similar to the above, let us define $\decay\big(\rad(\elp,\enran)\big)$ to be the supremum over all $\varrho\ge 0$ such there exist $C,C_1,C_2\in (0,\infty) $ such that
$
\rad(\elp,\enran)\le C n^{-\varrho}
$
holds with probability at least $1-C_1\exp(-C_2n)$ for all $m\in\IN$ and $1\le n <m$. This notion is somewhat related to the notion of the random Gelfand widths.

Then the combination of \eqref{eq:min-order} and \eqref{eq:bound-polynomial} implies, for $1\le p\le\infty$ and $\lambda>\frac{1}{p^*}$, 
\begin{equation}\label{eq:ran-order}
	\decay\big(\rad(\elp,\enran)\big)
=	\lambda+\frac{1}{p}-\frac{1}{2} 
=	\decay\big(\rad(\elp,n)\big).
\end{equation}
That is, the rate of decay of random information is equal to the one of optimal information which provides a positive partial answer to Question~\ref{que:lin} for $\ell_p$-ellipsoids with polynomially decaying semiaxes.

Therefore, the bound of Theorem~\ref{thm:ell-main} yields an optimal bound on the rate of polynomial decay of typical sections in the case of $\lambda>\frac{1}{p^*}$. However, if $\lambda\le \frac{1}{p^*}$, it does not yield a useful result. Instead, we have the following lower bound on the radius of a random section if $1< p\le 2$ and $m$ is large enough compared to $n$.

\begin{proposition}[{\cite[Prop.~1]{HPS21}}]\label{pro:lower}
	Let $1<p\le 2$ and $\sigma_j=j^{-\lambda},j\in\IN$, for some $\lambda$ with $0<\lambda<\frac{1}{p^*}$. Then, for any $\varepsilon\in (0,1),$ $n\in\IN$ and $m>n$ large enough, we have
\[
	\IP\Big[\rad( \elp,\enran)\ge \frac{1}{2}\Big]
	\ge 1-\varepsilon.
\]
\end{proposition}
 
In other words, if $1< p\le 2$ and the semiaxes decay too slowly compared to $\frac{1}{p^*}$, random information is asymptotically as good as no information at all, that is, its radius does not decay as the following corollary summarizes in terms of random sections.

\begin{corollary}\label{cor:lower}
	Let $1\le p\le 2$ and $\sigma_j=j^{-\lambda},j\in\IN$, for some $\lambda$ with $0<\lambda<\frac{1}{p^*}$. Then
\[
\decay\big(\rad(\elp,\enran)\big)=0.
\]
\end{corollary}

We visualize $\decay\big(\rad(\elp,\enran)\big)$, or equivalently the rate of decay of the radius of random information, in Figure~\ref{fig:poly} and discuss it in the following.
\begin{figure}[ht]
\begin{center}
	\begin{tikzpicture}[scale=5]
		\fill[lgrey] (0,0) -- (0,.35) -- (0.35,0) -- (0,0);
		\draw[->] (0,-.1) -- (0,1) node [midway,xshift=-10pt,yshift=-10pt]{$\frac{1}{2}$} node [near end,xshift=-10pt,yshift=0pt]{$1$} node [at end,xshift=-10pt,yshift=0pt]{$\lambda$} ; 	
		\draw[->] (-.1,0) -- (1,0) node [midway,yshift=-10pt,xshift=-10pt]{$\frac{1}{2}$} node [near end,yshift=-10pt,xshift=0pt]{$1$} node [at end,xshift=3pt,yshift=-10pt]{$\frac{1}{p}$}; 	
		\draw (0,.7) -- (.7,0); 	
		\draw (0,0) rectangle (.35,.35); 	
		\draw[dashed] (0,.35) -- (.35,0); 	
		\draw[dashed] (.7,0) -- (.7,1) node [at end,yshift=5pt]{\small{$p=1$}}; 	
		\node at (.35,.75) {$\lambda+\frac{1}{p}-\frac{1}{2}$};
		\node at (.12,.44) {$?$};		
		\node at (.27,.27) {0};
		\node at (.47,.10) {0};
		\clip (0,0) rectangle (0.35,.35);
	\end{tikzpicture}
\end{center}
\caption{This diagram depicts our knowledge about $\decay\big(\rad(\elp,\enran)\big)$ in the $\frac{1}{p}-\lambda$-plane, where $0< p,\lambda\le \infty$ and $\sigma_j=j^{-\lambda},j\in\IN$.}
\label{fig:poly}
\end{figure}

Above the line $\lambda=1-\frac{1}{p}$, where $ 1\le p\le\infty,$ we just deduced that Theorem \ref{thm:ell-main} yields that random information is optimal up to an additional logarithmic factor if $p=1$. The decay rate is equal to $\lambda+\frac{1}{p}-\frac{1}{2}$, see \eqref{eq:ran-order}. As noted above, below and including the line $\lambda=\frac{1}{2}-\frac{1}{p}$, where $ 2\le p\le\infty,$ optimal information does not decay at all, in other words, information is useless and does not help to recover vectors. Geometrically, this corresponds to the fact that, no matter how large the codimension $n\,(<m)$ of a subspace is, the section with $\elp$ has a radius bounded from below. 

In the square, that is for $p\ge 2$ and $\lambda\le\frac{1}{2}$, it follows via $\elz\subset \elp$ from the lower bound \eqref{eq:lower-2} for the case $p=2$ that, no matter how large we choose $n$, if $m$ is large enough, then with high probability $\rad(\elp,\enran)$ is bounded from below by a constant.  That is, $\decay\big(\rad(\elp,\enran)\big)=0$ and so random information is useless. By Corollary~\ref{cor:lower} this also holds for the triangle given by $1< p<2$ and $0<\lambda<1-\frac{1}{p}$.

Finally, note that on the right-hand side of the dashed line where $p=1$, that is, where $0<p<1$, Theorem \ref{thm:ell-main-quasi} provides an upper bound with decay rate $\lambda+\frac{1}{p}-\frac{1}{2}$, which depends on $m$. We do not have a corresponding lower bound for optimal information in this region. 

We pose the following conjecture claiming that there is a threshold of decay separating regimes of completely different behavior of random information.

\begin{conj}\label{conj:threshold}
	Let $1\le p\le\infty$ and $\sigma_j=j^{-\lambda},j\in\IN,$ with $\lambda>\frac{1}{s}$. Then, 
\[
\decay\big(\rad(\elp,\enran)\big)=
\begin{cases}
	\decay\big(\rad(\elp,n)\big)
&\text{if }\lambda>\frac{1}{p^*},\\
	0&\text{if }\lambda\le\frac{1}{p^*}.
\end{cases}
\]
\end{conj}

In fact, this dichotomy may also be true more generally for non-polynomial semiaxes depending on whether $\|\sigma\|_{p^*}<\infty$ or not, see Question~\ref{que:conj} in Section~\ref{sec:ell-open}.

\section{Low $M^*$-estimates and the rounding of ellipsoids}
\label{sec:ell-rounding}

In this section we will explain how Theorem~\ref{thm:ell-main} follows from a (low(er)) $M^*$-estimate using a rounded version of the $\ell_p$-ellipsoid $\elp$. To illustrate the ideas behind this approach, which has been used already in \cite{HKN+21}, and in a similar form also in \cite{LT00} (and many other papers), we will consider a general convex body $K\subset \IR^m$.

As mentioned, unit balls of normed spaces and in particular the $\ell_p$-ellipsoids, where $1\le p\le \infty$, are examples of convex bodies, and additionally symmetric. In the following, $K$ need not be symmetric but for convenience we assume that it contains the origin in its interior.

A well-known type of result in the field of asymptotic geometric analysis is the low $M^*$-estimate, which consists of an estimate of the radius of the intersection of $K$ with a random subspace in terms of its mean width and a factor depending on the proportion of the dimension of subspace. Let us introduce the necessary notions and give a formal statement together with some history.

The mean width of $K$ measures the average width of $K$ over all directions belonging to the sphere $\IS^{m-1}=\{x\in \IR^m\colon \|x\|_2=1\}$. To formally define it, let $h_{K}\colon \mathbb S^{m-1}\to\IR$, $u\mapsto\sup_{y\in K}\langle u,y\rangle$ be the support function of $K$, which measures the distance from the origin of the supporting hyperplane in direction $u$. That is, the width into direction $u$ is given by $h_K(u)+h_K(-u)$. The (half) mean width of $K$ is then
\begin{equation} \label{eq:mean-width}
	M^{*}(K)
	:=\int_{\mathbb S^{m-1}}h_{K}(u)\,\dd\sigma^{m-1}(u),
\end{equation}
where $\sigma^{m-1}$ is the normalized surface measure on $\mathbb S^{m-1}$. Intuitively, this is the expected radius of a one-dimensional random subspace.

Using the probabilistic representation of the uniform measure on $\mathbb S^{m-1}$ in terms of standard Gaussians (see Section~\ref{sec:lin}) shows that $M^*(K)$ can be written as the supremum of a Gaussian process (see, e.g., \cite[Lem.~9.1.3]{AGM15}), namely,
\begin{equation}\label{eq:mstargaussian}
	M^{*}(K)
	= \frac{1}{c_m} \IE \sup_{t\in K}\sum_{j=1}^{m}t_j g_j,
\end{equation}
where $g_{1},g_{2},\ldots$ are independent standard Gaussian random variables and $c_m\asymp \sqrt{m}$. 

We will state now a form of the lower $M^*$-estimate as presented in \cite[Thm.~9.3.8]{AGM15} with $\gamma=\frac{1}{2}$ there. 
\begin{proposition}\label{pro:escape}
	Let $K\subset\IR^{m}$ be a convex body containing the origin in its interior. For any $1\le n<m$ there exists a subset of the Grassmannian $\mathcal{G}_{m,m-n}$ with Haar measure at least $1-\frac{7}{2} \exp(-\frac{1}{72}a_{n}^{2})$ such that for any subspace $E_{n}$ in this set and all $x\in K\cap E_{n}$ we have
\[
	\|x\|_{2}\le 2\frac{a_{m}}{a_{n}}M^{*}(K),
\]
where, for each $k\in\IN$,
\begin{equation}\label{eq:asymptotic ak}
	a_{k}
	:=\IE\Big(\sum_{j=1}^{k}g_{j}^{2}\Big)^{1/2}
	=\frac{\sqrt{2}\Gamma((k+1)/2)}{\Gamma(k/2)}
	\asymp \sqrt{k}.
\end{equation}
\end{proposition}

Thus it implies 
\[
\rad(K,\enran)
\le 2\frac{a_{m}}{a_{n}}M^{*}(K),
\]
for the random subspace $\enran$ with the claimed probability.

The statement of Proposition~\ref{pro:escape} is adapted from Gordon~\cite{Gor88} who proved it using a min-max principle to derive a result on subspaces escaping through a mesh. In \cite{LT00} and \cite[Ch.~7]{AGM15} one may find some history starting with V.~D.~Milman~\cite{Mil85a,Mil85b}. In particular, Gordon's result improves upon the work \cite{PT86} due to Pajor and Tomczak-Jaegermann who obtained a larger constant. In fact, this would be sufficient for our purposes as we do not care about the size of the absolute constants. 

As in the mentioned works \cite{LT00}, \cite{LPT06} and \cite{MPT07}, see  also \cite[Ch.~7.5]{AGM15}, we use the technique of rounding to improve the power of the low $M^*$-estimates as follows. Define the \emph{rounded body} by
\[
K_{\varrho}:=K\cap \varrho\IBo_2^m=\{x\in K\colon \|x\|_2\le \varrho\},
\]
where $\varrho>0$ is some cutoff threshold. The idea to obtain a bound on the radius of the section with the random subspace $\enran$ is then to derive a bound on $M^*(K_{\varrho})$ for an unspecified $\varrho>0$, which gives by the $M^*$-estimate a bound on $\rad(K_{\varrho}, E_n)$ for most realizations $E_n$ of $\enran$, and to afterwards minimize over all $\varrho>0$ such that 
\begin{equation} \label{eq:rounding}
\rad(K_{\varrho}, E_n)<\varrho\quad\text{and consequently}\quad \rad(K, E_n)< \varrho
\end{equation}
holds for these subspaces. The improvement in this case comes from the fact that we can decrease the mean width $M^*(K_{\varrho})$ by decreasing $\varrho$ and thus cut off the peaky parts of $K$, which are somewhat irrelevant to random subspaces.

For the proof of Theorem~\ref{thm:ell-main} we apply this technique to the $\ell_p$-ellipsoid $\elp$ which results in cutting away the contribution of the first $cn$ semiaxes, where  $0<c<1$.  This will be a consequence of the following proposition.  
\begin{proposition}[{\cite[Prop.~3]{HPS21}}]\label{pro:mstar-estimate}
Let $m\in\IN$. For any $0\le k< m$ and $\varrho>0$,
\[
M^{*}(\elp\cap \varrho \IB_{2}^{m})
	\lesssim
	\begin{cases}
		m^{-1/2}\Big(\varrho\sqrt{k}+\sup\limits_{k+1\leq j \leq m}\sigma_{j}\sqrt{\log(j)+1}\Big) &\text{if } p=1, \\
	\sqrt{p^*}m^{-1/2} \Big(\varrho\sqrt{k} + \big(\sum\limits_{j=k+1}^{m}\sigma_{j}^{p^{*}}\big)^{1/p^{*}}\Big) &\text{if } 1< p\le \infty.
	\end{cases}
\] 
\end{proposition}

This follows from an application of Hölder's inequality together with estimates on the norm of a Gaussian vector. A similar approach was used by Gordon, Litvak, Mendelson and Pajor in \cite{GLM+07} for estimating the supremum of a Gaussian process indexed by the intersection of the $\ell_p$-ball $\IB_p^m$ with a scaled $\ell_q$-ball $\varrho\IB_q^m$. Using a suitable cutoff $\varrho>0$, which we will not attempt to optimize, one can obtain from Proposition~\ref{pro:mstar-estimate} the proof of Theorem~\ref{thm:ell-main}. The details are deferred to Appendix~\ref{sec:ell-proof-convex}.

Now, Corollary~\ref{cor:mstar-estimate} follows from Proposition~\ref{pro:mstar-estimate} by setting $k=0$ and $\varrho>0$ large enough such that $\elp\subset \varrho \IB_{2}^m$ together with the representation \eqref{eq:mstargaussian} of the expected supremum of a Gaussian process indexed by a convex body. 

\bigskip

\section{A lower bound}
\label{sec:ell-lower}

The idea behind the proof of the lower bound of Proposition~\ref{pro:lower} for polynomial semiaxes in the case $1< p\le 2$ is as follows. To bound the radius of $\elp$ intersected with a random subspace from below, we have to find with a certain probability a vector in it which has large enough norm. In fact, it is often even possible to find a vector having a single large coordinate with high probability as the following observation made in \cite{HKN+21} shows.

\begin{lemma}[{\cite[Lem.~25]{HKN+21}}]\label{lem:large-coord}
	For any $\varepsilon\in (0,1)$ it holds that, for all $m\in\IN$ and $1\le n< m$,
	\[
	\IP\Big[\sup\big\{ x_1^2\colon \|x\|_2=1, G_{n,m}x=0\big\}\ge 1-\frac{n}{\varepsilon m}\Big]
	\ge 1-\varepsilon.
	\]
\end{lemma}

Using this, we can prove the following more general version of Proposition~\ref{pro:lower}.

\begin{proposition}[{\cite[Prop.~4]{HPS21}}]\label{pro:lower-general}
Let $1< p\le 2$. For any $\varepsilon\in (0,1)$, all $m\in\IN$ and $1\le n< m$ with $n\le \varepsilon\sigma_m^2 m^{2/p^*}$ we have
	\[
	\IP\Big[\rad(\elp,\enran)\ge \frac{\sigma_1}{1+\sigma_1}\Big]
	\ge 1-\varepsilon.
	\]
\end{proposition}

In order to prove Proposition~\ref{pro:lower-general} we take a random vector as given by Lemma~\ref{lem:large-coord} which belongs to the intersection of the unit sphere with the kernel of $G_{n,m}$  and has a large first coordinate. Then we show using Hölder's inequality that the renormalization by the factor $1+1/\sigma_1$ lies in $\elp$ provided that the dimension $m$ is large enough compared to $n$ when $\sigma_m m^{1/p^*}$ is increasing. For details we refer to Appendix~\ref{sec:ell-proof-convex}.

\chg{In the case $p=2$ this method yields} bounds which are optimal with respect to polynomial order, see \cite[Prop.~24]{HPS21}. \chg{The restriction to} $1\le p\le 2$ is due to Hölder's inequality and different ideas seem to be necessary for the case $p>2$. We do not know if also there it is sufficient to exhibit vectors with a single large coordinate but we believe that this may be true. A possible approach could consist in finding a variant of Lemma~\ref{lem:large-coord} where the sphere is replaced by the surface of the cube.

\begin{remark}
	From Proposition~\ref{pro:lower-general} it can be deduced that the statement of Proposition~\ref{pro:lower} remains true in the case of $\sigma_j=j^{-1/p^*}a_j, j\in\IN,$ with $a_j\to \infty$ as $j\to \infty$. 
\end{remark}
\begin{remark}
	In \cite{HKN+21} both bounds \eqref{eq:upper-2} and \eqref{eq:lower-2} were proven using random matrices. In particular, knowledge on the singular values of structured Gaussian matrices was employed and the Hilbert space structure was used in a crucial way. This is the reason why we could not extend the method they used without the loss of additional factors in the dimension $m$. Perhaps, new results on structured Gaussian random matrices may be of help, see, e.g., Gu\'edon, Hinrichs, Litvak and Prochno~\cite{GHL+17}.
\end{remark}

\section{Sparse approximation}
\label{sec:ell-sparse}

In order to describe the ideas behind the proof of Theorem~\ref{thm:ell-main-quasi}, we give a short introduction to a subfield of compressed sensing, known as sparse approximation or recovery. We refer to the book of Foucart and Rauhut~\cite{FR13} for a mathematical background on the subject. 

It is in part due to random matrices that sparse recovery experienced a relatively recent surge of interest. The seminal papers of Cand\`es, Romberg and Tao~\cite{CRT06} and Donoho~\cite{Don06b} were among the first to investigate this and found that random matrices satisfy with high probability the restricted isometry property which enables exact reconstruction of sparse vectors. Further, an universality principle has been discovered, namely, random matrices can be used to gather measurements which are effective independently of the basis with respect to which a signal is sparse.

The fact that random information is the best information available for many problems in sparse recovery is reflected in the statement of Theorem~\ref{thm:ell-main-quasi} which will be deduced from the more general Theorem~\ref{thm:ell-main-quasi-general} below. To this end, let us give some background on sparse approximation.

Let $m,s\in\IN$ with $1\le s \le m$. A vector $z\in\IR^m$ is called $s$-sparse if at most $s$ of its coordinates are non-zero.  Given information $N_n z =y$, where $N_n\in \IR^{n\times m}$ is a suitable linear information map, sparse vectors can be reconstructed via $\ell_p$-minimization, where $0<p\le 1$, that is,
\begin{equation} \label{eq:ellp-min}
\Delta_p(y):= \text{arg min} \|z\|_p \quad \text{subject to }N_n z=y.
\end{equation}
Note that $\Delta_p$ is a map from $\IR^n$ to $\IR^m$ which depends on $N_n$. If the matrix $N_n$ satisfies the restricted isometry property with a small restricted isometry constant $\delta_{2s}(N_n)$ of order $2s$, which is the smallest $\delta>0$ such that
\begin{equation} \label{eq:rip}
(1-\delta)\|z\|_2^2 
\le \|Az\|_2^2
\le (1+\delta)\|z\|_2^2 \quad \text{for all }2s\text{-sparse }z\in\IR^m,
\end{equation}
then $s$-sparse vectors $z\in\IR^m$ can be recovered exactly, i.e., $z=\Delta_p(N_n z)$. This \chg{is} a consequence of the next proposition which gives a bound on the error of reconstruction using $\ell_p$-minimization in terms of the following concept.
\begin{definition}
Given $0<p\le \infty$, the error of best $s$-term approximation of $x\in \IR^m$ in the $\ell_p$-(quasi-)norm is
\[
\sigma_s(x)_p := \inf\{\|x-z\|_p \colon z \text{ is }s\text{-sparse}\}.
\]
\end{definition}

\begin{proposition}[{\cite[Thm.~1.3]{FPR+10}}]\label{pro:rip-sparse}
If $0<p\le 1$ and $N_n$ satisfies the restricted isometry property with constant $\delta_{2s}<\sqrt{2}-1$, then, for all $x\in\IR^m$,
\[
\|x-\Delta_p(N_nx)\|_p\le C^{1/p} \sigma_s(x)_p,
\]
where $C>0$ is a constant that depends only on $\delta_{2s}$. In particular, the reconstruction of $s$-sparse vectors is exact.
\end{proposition}

The restricted isometry property can be interpreted as a condition of uniformity on the information map $N_n$, which is satisfied by a rescaling of the Gaussian matrix $G_{n,m}=(g_{ij})_{i=1,j=1}^{n,m}$ with high probability as the following result shows.  
\begin{proposition}[{\cite[Thm.~9.2]{FR13}}] \label{pro:rip-gaussian}
	For every $\delta\in(0,1)$ there exist $C_1,C_2>0$ such that for all $m\in\IN$ the restricted isometry property $\delta_{s}(n^{-1/2}G_{n,m})\le \delta$ holds with probability at least $1-2\exp(-C_2 n)$ provided that
$	
	n\ge C_1 s \log({\rm e}m/s).
$
\end{proposition}
In fact, the statement is more general and applies also to subgaussian random matrices, see also Remark~\ref{rem:subgaussian}. Taken together, Propositions~\ref{pro:rip-sparse} and \ref{pro:rip-gaussian} show that, provided that enough, in fact not much more than the number of its non-zero entries if $m$ is not too large compared to $n$,  Gaussian measurements are available, a sparse vector can be reconstructed exactly with high probability using $\ell_p$-minimization. Together with the following lemma this completes the proof strategy for Theorem~\ref{thm:ell-main-quasi}.  
\begin{lemma}[{\cite[Lem.~4]{HPS21}}] \label{lem:nonlin-app}
	Let $m\in\IN$, $0<p,q\leq\infty$ and $\sigma_j=j^{-\lambda}$, $1\leq j \leq m$,  for some $\lambda>(1/q-1/p)_+$. Then, for all $1\le s\le m/2$,
\[
\sup_{x\in\elp}\sigma_s(x)_q
\asymp s^{-\lambda+1/q-1/p},
\]
where the implicit constant is independent of $s$ and $m$.
\end{lemma}

Its proof, given in Appendix~\ref{sec:ell-proof-quasi}, is inspired by the proof for the $\ell_p$-balls which can be found in Vyb\'iral~\cite{Vyb12}. 

To state the afore mentioned generalization of Theorem~\ref{thm:ell-main-quasi}, where we measure the radius in the quasi-norm of $\ell_q^m$ instead of $\ell_2^m$, we define the radius $\rad_G(K,E)$ of the section of any set $K\subset \IR^m$ with a linear subspace $E\subset\IR^m$ in $G$ by 
\[
\rad_G(K,E):=\sup_{x\in K\cap E}\|x\|_G,
\]
where $G$ is equipped with a quasi-norm $\|\cdot\|_G$. By Proposition~\ref{pro:radius-general-+} it is related to the radius of information for recovery of vectors in $K$ in $\|\cdot\|_G$. 
\begin{theorem}[{\cite[Thm.~C]{HPS21}}] \label{thm:ell-main-quasi-general}
Let $0<p\leq 1$ and $p<q\le 2$. Assume that $\sigma_j=j^{-\lambda}$, $j\in\IN$, for some $\lambda>0$. Then there exist constants $c,C,D>0$ such that, for all $m\in\IN$ and all $1\le n <m$ with
\[
n\ge D\log({\rm e}m/n),
\]
we have 
\[
\rad_{\ell_q^m}(\elp,\enran)
\le C\, \Big(\frac{\log({\rm e}m/n)}{n}\Big)^{\lambda+1/p-1/q}
\]
with probability at least $1-2\exp(-c\, n)$.
\end{theorem}

The proof of this theorem is a straightforward adaption of the proof of \cite[Thm.~3.2]{FPR+10} and relies on Lemma~\ref{lem:nonlin-app} as well as other techniques from sparse recovery. Again, details are provided in Appendix~\ref{sec:ell-proof-quasi}. 

\begin{remark}
	Our original proof of Theorem~\ref{thm:ell-main-quasi-general} as given in the preprint \cite{HPS21} used $\ell_r$-minimization with $r=p$ which does not yield the claimed bound without additional modifications. Instead of giving a corrected version of this variant, we can set $r=\min\{1,q\}$ as in the proof of \cite[Thm.~3.2]{FPR+10} since Lemma~\ref{lem:nonlin-app} is more general than \cite[Lem.~4]{HPS21}. To keep the presentation self-contained, we give a complete proof instead of only stating the required modifications.
\end{remark}

\section{Open questions}
\label{sec:ell-open}

\begin{question}\label{que:existence}
	As mentioned in Section~\ref{sec:ell-intro}, knowledge of a single minimal section may imply a bound on typical sections of similar dimension which is in contrast to \eqref{eq:general-upper}, derived from \cite{LT00}, where knowledge of the whole sequence $\big(c_n(K)\big)_{1\le n\le m}$ is needed. This difference has been observed at the end of Section 1.2 in \cite{MY21} by E.~Milman and Yifrach, who stated that it can be expected that knowledge of the whole sequence of minimal radii is needed to obtain good upper bounds, \chg{but a precise formulation is unavailable}.
\begin{center}
	Is the knowledge of the whole sequence of Gelfand widths necessary to obtain an optimal bound on the radius of a random section?
\end{center}
\end{question}

\bigskip

\begin{question}
	To the best of our knowledge, general lower bounds on the radius of typical information or the random Gelfand widths, which could complement the upper bound \eqref{eq:general-upper}, are not available. Related to Question~\ref{que:existence} we pose the following one.
\begin{center}
	Can the upper bound~\eqref{eq:general-upper} be improved? What about lower bounds?
\end{center} 
It would be interesting to have for example a negative statement if the Gelfand widths are not square-summable. We hope that our study of $\ell_p$-ellipsoids sheds some light on these questions.
\end{question}

\bigskip

\begin{question}\label{que:conj}
	As discussed in Section~\ref{sec:ell-poly}, Conjecture~\ref{conj:threshold} is verified except for the two cases
\begin{enumerate}
	\item $1< p<2$ and $\lambda=\frac{1}{p^*}$,
	\item $p> 2$ and $\frac{1}{2}<\lambda\le \frac{1}{p^*}$.
\end{enumerate}
\begin{center}
	Can one prove Conjecture~\ref{conj:threshold} or find counterexamples to it?
\end{center}
It seems reasonable to conjecture even that $\decay\big(\rad(\elp,\enran)\big) = \decay\big(\rad(\elp,n)\big)$ as long as $\|(\sigma_j)_{j\in\IN}\|_{p^*}<\infty$, while $\decay\big(\rad(\elp,\enran)\big)=0$ whenever $\|(\sigma_j)_{j\in\IN}\|_{p^*}=\infty$. With an eye on the case $p=2$, it may also be true that the threshold is determined by the square-summability of the sequence of minimal radii or Gelfand widths. 
\end{question}

\bigskip

\begin{question}
In Corollary~\ref{cor:mstar-estimate} we obtained an upper bound on the dependence on $p$ of the supremum of a Gaussian process indexed by $\elp$. As we do not have a corresponding lower bound, the following would be interesting to know.
\begin{center}
	What is the correct dependence in Corollary~\ref{cor:mstar-estimate} on $p$?
\end{center}
A possible approach could be based on the equality case in Hölder's inequality used for the proof. See also \cite{GLM+07}, where lower bounds on the mean width are provided.
\end{question}

\bigskip

\begin{question}
If $0<p<1$, we do not have a lower bound on the minimal radius of information and therefore cannot compare typical with optimal information. 
\begin{center}
	How sharp is Theorem~\ref{thm:ell-main-quasi} or its extension Theorem~\ref{thm:ell-main-quasi-general}?
\end{center}
We comment on a possible approach to a lower bound to Theorem~\ref{thm:ell-main-quasi-general} or Theorem~\ref{thm:ell-main-quasi} and more generally a lower bound on the Gelfand numbers of diagonal operators in the quasi-Banach regime. In the case of the $\ell_p$-balls, that is  $\sigma_1=\cdots=\sigma_m$, a matching lower bound is provided in \cite[Prop.~2.1]{FPR+10}. Its proof relies on a combinatorial lemma providing a partially overlapping covering of the index set $\{1,\ldots,m\}$ and does not extend to $\ell_p$-ellipsoids with decaying semiaxes. 

In \cite{FPR+10} the authors had to find a different method of proof than the one used by Donoho~\cite{Don06b} who relied on the validity of Carl's inequality, which compares entropy numbers with Gelfand numbers, in quasi-Banach spaces. However, this was only shown several years later by Hinrichs, Kolleck and Vyb\'iral \cite{HKV16}. With this additional knowledge one could possibly prove a lower bound to Theorem~\ref{thm:ell-main-quasi-general} using a lower bound on the entropy numbers of diagonal operators as shown by Kühn~\cite{Kue01} for the case of the identity operator. 

\end{question}

\bigskip

\begin{question}
An important tool in the proof of Theorem~\ref{thm:ell-main-quasi}, Lemma~\ref{lem:nonlin-app} itself is proven only for polynomial semiaxes. Naturally, the following question arises.
\begin{center}
 Can Lemma~\ref{lem:nonlin-app} be generalized to other types of semiaxes?
\end{center}
A generalization to other semiaxes would directly generalize Theorem~\ref{thm:ell-main-quasi-general} and thus Theorem~\ref{thm:ell-main-quasi} to the same class of semiaxes. 

A possible approach to a lower bound may be to bound the supremum from below by the average of $\sigma_s(x)_q$ over $x\in \elp$. Then one can use an expression of a uniform vector in $\elp$ in terms of asymptotically independent vectors, see \cite{JP21}, and use results from Gordon, Litvak, Schütt and Werner \cite{GLS+06} on $k$-th minima ($\chg{(m-k)}$-th maxima) of independent random variables.
\end{question}

\bigskip

\chapter{Function recovery and quantization} \label{ch:sob}

Passing from linear to standard information, this chapter is concerned with the quality of random samples compared to optimal samples. We present the results obtained in \cite{KS20} together with Krieg for the $L_q$-approximation problem and the integration problem for Sobolev but also Hölder and Triebel-Lizorkin spaces. Compared to the original presentation, the proof technique employed here is slightly different, however, and allows for an extension to integration in Triebel-Lizorkin spaces.

Subsequently, we will motivate the study of random standard information, in particular for the important example of Sobolev spaces. Afterwards, in Section~\ref{sec:sob-res} we will present the obtained results for Sobolev spaces and extensions to other spaces. In particular, we will characterize the radius of information given by a point set in terms of a geometric quantity\chg{. This allows} us to compare the expected radius of random information to the minimal radius of information \chg{and consequently to determine} the power of random information in terms of the given parameters. 

Then we will give some details on the proof strategy in Section~\ref{sec:sob-mls}, where we discuss an asymptotically optimal construction for any point set, Section~\ref{sec:sob-fool-int}, where a lower bound and the relation to the integration problem is established, and Section~\ref{sec:sob-quant}, where the geometric quantity will be discussed. Finally, in Section~\ref{sec:sob-open} we collect remaining open questions.

\section{Introduction and motivation}\label{sec:sob-intro}

Given a function $f$ belonging to some function space $F$ defined on a set $D$ we can use standard information, that is, function evaluations $f(x_1),\ldots,f(x_n)$ at some point set $\Pn=\{x_1,\dots,x_n\}\subset D$, in order to solve a numerical problem. We are interested in comparing the radius of random points with the sampling numbers in order to give possible answers to Question~\ref{que:std}. We refer to Section~\ref{sec:std} for a background. 

As function evaluations are in particular linear measurements, linear information can lead to better algorithms than relying on standard information alone. In general, however, evaluations are more easily obtained than, say, Fourier coefficients and thus there is an interest in algorithms using standard information. Further, as we will briefly discuss in the following, recent research shows that for an important type of problem function evaluations can be almost as good as optimal linear information.

In the case of $L_2$-approximation in a function space $F$, it has been obtained by Krieg and M.~Ullrich for Hilbert spaces \cite{KU21}, and for more general spaces in \cite{KU20}, that a least squares algorithm using random points is up to a logarithmic factor as good as the best linear algorithm. There, the sampling density depends on $F$ and is allowed to be uniform in some cases. Together with an improvement by M.~Ullrich~\cite{Ull20} this shows that typical information (with respect to this density) is close to optimal. There was an improvement by Nagel, Schäfer and T.~Ullrich \cite{NSU20} based on progress in the related field of sampling discretization due to Nitzan, Olevskii and Ulanovskii \cite{NOU16} in turn building on the solution of the Kadison-Singer conjecture by Marcus, Spielman and Srivastava~\cite{MSS15}. \chg{Recently, Dolbeault, Krieg and M.~Ullrich \cite{DKU22} provided sharp upper bounds.} However, in general, the logarithmic gap between the \chg{power of standard and linear information cannot be dispensed with} as an example of Sobolev spaces with low smoothness given by Hinrichs, Krieg, Novak and Vyb\'iral~\cite{HKN+21b} shows. 	

It seems difficult to improve the current results on the radius of typical information for general function spaces. We are positive that the methods and results from \cite{KS20}, which we discuss in this chapter, carry over to other settings. Further, we hope to shed light on possible conditions for asymptotic optimality of (sequences of) point sets in more general function spaces, and on the power of random standard information for those.

A starting point for our study was an open question by Hinrichs, Krieg, Novak, Prochno and M.~Ullrich~\cite{HKN+21}, who obtained for the $L_q$-approximation problem for Sobolev spaces $W^1_p([0,1])$ that uniform random points are asymptotically on average optimal if $q<p$ and worse by a logarithmic factor if $q\ge p$. They made use of known asymptotics for averages of spacings between uniformly distributed points on an interval, which is related to the norm of a random point in a simplex as studied by Baci, Kabluchko, Prochno, Thäle and the author in \cite{BKP+20}. In \cite{KS20} we extended the work \cite{HKN+21} to bounded convex sets \chg{in arbitrary dimension}.

Another motivation was to improve a related suboptimal bound on the radius of random information for the approximation of Sobolev functions on manifolds obtained by Ehler, Gräf and Oates~\cite{EGO19} using known asymptotics for the radius of the largest hole amidst a set of uniformly distributed independent random points. Together with D.~Krieg we could close the logarithmic gap from \cite{EGO19} in almost all cases in the work \cite{KS21}, where we transferred the results from \cite{KS20} via charts to manifolds. 

In general, one reason to study Sobolev spaces is their ubiquity in the theory of function spaces, which is in part due to their fundamental role as solution spaces of partial differential equations (PDEs). See for example Evans \cite{Eva10} for a gentle introduction and Maz'ja~\cite{Maz85} for a comprehensive account. Let us give their definition. 

\begin{definition}\label{def:sobolev}
Let $\Omega\subset\IR^d$ be a domain, that is an open and non-empty set. For a smoothness parameter $ s\in\mathbb{N} $ and an integrability parameter $ 1\leq p\le \infty $ we consider the Sobolev space
\[
W^s_p(\Omega):=\Big\{f\in L_p(\Omega)\colon \|f\|_{W^s_p(\Omega)}:=\Big(\sum_{|\alpha|\leq s} \|D^{\alpha}f\|_{L_p(\Omega)}^p\Big)^{1/p}<\infty\Big\},
\]
where the sum is over all multi-indices $ \alpha\in \mathbb{N}_0^d $ with $ |\alpha|=\alpha_1+\ldots+\alpha_d \le s $ and the expression $D^\alpha f = \frac{\partial^{|\alpha|}}{\partial x_1^{\alpha_1}\cdots\partial x_d^{\alpha_d}} f$ denotes a weak partial derivative of order $|\alpha|$. 
\end{definition}

By definition, the Sobolev space $W^s_p(\Omega)$ is a subspace of $L_p(\Omega)$ in which two elements are identified if they differ only on a set of measure zero. Without any additional assumptions, function evaluation is therefore not well defined. The embedding $W^s_p(\Omega)\hookrightarrow C_b(\Omega)$ into the bounded continuous functions eliminates this concern. It holds if 
\begin{equation} \label{eq:embedding}
	s>d/p \quad \text{if }1<p\le \infty \quad \text{or}\quad s\ge d \quad \text{if } p=1,
\end{equation}
under additional assumptions on $\Omega$ which will be satisfied for the domains we consider, see, e.g., \cite[Sec.~1.4.5]{Maz85}. In fact, we shall need the slightly stronger condition of $s>d/p$ for all $1\le p\le\infty$, which we assume from now on. In this case we have a compact embedding $W^s_p(\Omega)\hookrightarrow C_b(\Omega)$, see, e.g., \cite[Sec.~1.4.6]{Maz85}.

Several $s$-numbers of embeddings between Sobolev spaces have been studied, see, e.g., Pinkus~\cite[Ch.~VII]{Pin85} for an overview. There, and also for example in the book by Lorentz, von Golitschek and Makovoz~\cite[Ch.~14]{LGM96}, discretization techniques allow to use bounds from embeddings between sequence spaces. This is connected to the study of diagonal operators as mentioned in the introduction to Chapter~\ref{ch:ell}.

Let us describe the situation for sampling numbers or equivalently the minimal radius of standard information. Recall that the $n$-th minimal radius of information is given in the case of $L_q$-approximation by
\[
r\big(W_p^s(\Omega)\hookrightarrow L_q(\Omega),n\big)
=\inf_{S_{\Pn}}\sup_{\|f\|_{W^s_p(\Omega)}\le 1}\|f-S_{\Pn}(f)\|_{L_q(\Omega)},
\]
where the infimum is over all sampling operators of the form
\begin{equation}\label{eq:sampling-operator1}
	S_{\Pn}\colon W^s_p(\Omega) \to L_q(\Omega), 
	\qquad S_{\Pn}(f)=\varphi\big(f(x_1),\dots,f(x_n)\big)
\end{equation}
for some $n$-point set $\Pn=\{x_1,\ldots,x_n\}\subset \Omega$ and $\varphi\colon \IR^n\to L_q(\Omega)$. Here, we dropped the $\lstd$ in the notation as we will only be concerned with standard information from now on. 

On a bounded Lipschitz domain $\Omega\subset\IR^d$ (see Definition~\ref{def:bounded-lipschitz} below), the minimal radius is known to satisfy 
\begin{equation} \label{eq:minrad-app}
 r\big(W_p^s(\Omega)\hookrightarrow L_q(\Omega),n\big) 
 \,\asymp\, n^{-s/d+(1/p-1/q)_+},
\end{equation}
where $(1/q-1/p)_+=1/q-1/p$ if $q<p$ and zero otherwise, see the discussion below. Note that the rate of decay improves if the smoothness $s$ is large compared to the dimension $d$ but deteriorates otherwise. This is in contrast to anisotropic function spaces of mixed smoothness where the \chg{dependence on the dimension is through a logarithmic factor}. We refer to Temlyakov \cite[Ch.~6.9]{Tem18} for more information and note that for an important class of function spaces of mixed smoothness the \chg{asymptotics of the sampling numbers in $L_2$ are resolved by \cite{DKU22}.}

For the integration problem, which corresponds to $q=1$ above, it is known that 
\[
 r\big(W^s_p(\Omega),{\rm INT},n\big) 
 \,\asymp\, n^{-s/d}.
\]
Recall that, in analogy to the above, the $n$-th minimal radius of information is given by
\[
r\big(W_p^s(\Omega),{\rm INT},n\big)
=\inf_{S_{\Pn}}\sup_{\|f\|_{W^s_p(\Omega)}\le 1}\Big|\int_{\Omega}f(x)\dd x-S_{\Pn}(f)\Big|,
\]
where the infimum is over all sampling operators \chg{as in} \eqref{eq:sampling-operator1} with $L_q(\Omega)$ replaced by $\IR$.

In fact, it is known that for integration as well as $L_q$-approximation the sampling operators achieving the optimal asymptotic rate can be chosen linear, i.e., of the form 
\begin{equation} \label{eq:linear-algo}
	A_{\Pn}(f)
=\sum_{i=1}^{n}u_i f(x_i),
\end{equation}
where $\Pn=\{x_1,\ldots,x_n\}$ is a suitably chosen point set and $u_i$ are real numbers in the case of the integration problem and bounded real-valued functions on $\Omega$ in the case of the approximation problem. In general, the optimal algorithm for the integration \chg{problem} given any standard information is linear. This is a well-known result due to Smolyak and Bakhvalov, see, e.g., \cite[Thm.~4.7]{NW08}. 

The above asymptotics for the minimal radius are classical for special domains like the cube and we do not know their origin. They have been obtained, for example, in the context of finite elements, which are used for the numerical approximation of solutions of PDEs, see, e.g., Heinrich~\cite{Hei94} who refers to the treatise of Ciarlet~\cite{Cia78}. 

Using the moving least squares method, Wendland~\cite{Wen01} proved a result on the local reproduction of polynomials, see Section~\ref{sec:sob-mls}. This was then used by Novak and Triebel in \cite{NT06} to prove the asymptotic behaviour of the radius of minimal information for approximation \eqref{eq:minrad-app} for bounded Lipschitz domains. For sufficiently high smoothness, this may also be deduced from an earlier result due to Narcowich, Ward and Wendland~\cite{NWW04} also relying on \cite{Wen01}. At this point, it seems worthwhile noting that in \cite{NT06} the authors also compared the minimal radius of standard information to the one of linear information and concluded that, if one restricts to linear methods, linear information can be asymptotically better than standard information if and only if $p<2<q$. 

As standard information is given by point sets, it is perhaps natural to look at their geometry to find (asymptotically) optimal point sets. In the works above it has been found that (sequences of) point sets covering the domain well, so that no point is far from the point set, are asymptotically optimal. That is, it was supposed that the covering radius
\begin{equation} \label{eq:covering-rad}
h_{\Pn,\Omega}
:=\sup_{x\in \Omega}\dist(x,\Pn)
\end{equation}
which is the supremum of the distance function
\begin{equation}\label{eq:dist-function}
\dist(\cdot, \Pn)\colon \mathbb{R}^d\to [0,\infty), 
\qquad \dist(x,\Pn):=\min_{y\in \Pn}\|x-y\|_2
\end{equation}
to the $n$-point sampling set $\Pn\subset \Omega$ has to be of the best possible order $n^{-1/d}$, see also \eqref{eq:dist-optimal} below. This assumption of an optimal covering radius is very common in numerical analysis; to give an impression, we refer to  Arcang\'eli, L\'opez de Silanes and Torrens \cite{ALdST07}, Brauchart, Dick, Saff, Sloan, Wang and Womersley \cite{BDS+15}, Duchon~\cite{Duc78}, Mhaskar \cite{Mha10}, as well as Wendland and Schaback~\cite{WS93}, while being well aware of the \chg{incompleteness of this enumeration}. In the following section, we will see that the covering radius is not sufficient to describe which (sequences of) point sets achieve the optimal rate. This allows for an asymptotic description of the expected power of random standard information.

\section{Discussion of results}
\label{sec:sob-res}

Motivated by the question whether random points or typical information is optimal, we derived a geometric characterization for the radius of information given by a point set, Theorem~\ref{thm:sob-main} below. This yields a criterion (Corollary~\ref{cor:sob-char}) of optimality given a sequence of point sets. Using results from quantization theory we show Proposition~\ref{pro:random-covering} and derive Corollary~\ref{cor:sob-ran} which settles Question~\ref{que:std} on the power of random information for the $L_q$-approximation and the integration problem above. Finally, Theorem~\ref{thm:sob-main-ext} extends the results to other spaces, including Hölder classes, Sobolev spaces of fractional smoothness and Triebel-Lizorkin spaces.

In order to investigate the power of random standard information, we analyzed the quality of arbitrary point sets and arrived at the following theorem.

\leqnomode
\begin{theorem}[{\cite[Thm.~1]{KS20}}]\label{thm:sob-main}
	Let $\Omega\subset \mathbb{R}^d$ be a bounded convex domain, $1\le p,q \le \infty$ and $s\in\mathbb{N}$ \chg{as in \eqref{eq:embedding}}.  Then, for any point set $\Pn\subset \Omega$,
	\begin{align}
	\tag{\textit a}
	&r\big(W^s_p(\Omega)\hookrightarrow L_q(\Omega),\Pn\big)\, \asymp \,
	\begin{cases}
\big\|\dist(\cdot, \Pn)\big\|_{L_{\infty}(\Omega)}^{s-d(1/p-1/q)} & \text{if } q\ge p,\vphantom{\bigg|}\\
\big\|\dist(\cdot, \Pn)\big\|_{L_{\gamma}(\Omega)}^s & \text{if } q<p,
\end{cases}\\[9pt]
\tag{\textit b}
	 &r\big(W^s_p(\Omega),{\rm INT},\Pn\big)\, \asymp\, r\big(W^s_p(\Omega)\hookrightarrow L_1(\Omega),\Pn\big),
	\end{align} 
	where $\gamma=s(1/q-1/p)^{-1}$ and $\asymp$ is to be interpreted as in Definition~\ref{def:asymp-std}. The algorithm achieving the upper bounds can be chosen linear, that is, of the form~\eqref{eq:linear-algo}.
\end{theorem}
\reqnomode

The proof and the assumption of convexity will be discussed in Sections~\ref{sec:sob-mls} and \ref{sec:sob-fool-int}. 

Theorem~\ref{thm:sob-main} implies for a sequence of point sets $(P_n)_{n\in\IN}$, each consisting of $n$ points, that the radius of information of $\Pn$ is asymptotically determined by the radius of the largest hole amidst the points if $q\ge p$ and by an average of the distance to the point set if $q<p$. The case $q\ge p$, where the covering radius determines the radius of information is included for completeness and can be deduced from the results of \cite{NWW04} and \cite{NT06} mentioned before. Our contribution was to show that the covering radius is too large to capture the radius of information for $q<p$, and in particular, for numerical integration. In general, the assumption of a small covering radius is unnecessarily strong as we shall see in Corollary~\ref{cor:sob-char} below.

Theorem~\ref{thm:sob-main} is a tool to analyze the asymptotic optimality of arbitrary (sequences of) point sets and, in particular, random or typical ones. Before we will come to the quality of random points, let us comment on the optimal behaviour of the minimal $L_{\gamma}$-norm of the distance function to any point set with at most $n$ points. It is well known that this satisfies
\begin{equation} \label{eq:dist-optimal}
\inf_{\#\Pn \le n} \|\dist(\cdot, \Pn)\|_{L_{\gamma}(\Omega)} \,\asymp\, n^{-1/d} 
\qquad \text{for every} \quad
0<\gamma\leq \infty,
\end{equation}
where the infimum is taken over all point sets $\Pn\subset \Omega$ with at most $n$ points. This can be deduced from a standard volume argument using balls centered at the points. For completeness a proof is provided in Appendix~\ref{sec:sob-proof-discuss}.

Thus, asymptotically optimal points, such as for example a suitably scaled lattice (see Chapter~\ref{ch:iso}), can be said to have holes of order $n^{-1/d}$. Plugging this into Theorem~\ref{thm:sob-main}, one can see that the rates for the minimal radius described in the introduction are achieved by such point sets. This \chg{yields} the following characterization of (asymptotically) optimal point sets. 

\begin{corollary}[{\cite[Cor.~1]{KS20}}]\label{cor:sob-char}
	Let $\Omega\subset \mathbb{R}^d$ be a bounded convex domain, $1\le p,q \le \infty$ and $s\in\mathbb{N}$ \chg{as in \eqref{eq:embedding}}. Assume that for each (or at least infinitely many) $ n\in\mathbb{N}$ an $ n $-point set $ P_n\subset \Omega $ is given.  These point sets are asymptotically optimal, i.e.,
	\[
	r\big(W^s_p(\Omega)\hookrightarrow L_q(\Omega),n\big)
	\,\asymp \,r\big(W^s_p(\Omega)\hookrightarrow L_q(\Omega),P_n\big),
	\]
	if and only if 
	\[
	\|\dist(\cdot, P_n)\|_{L_\gamma(\Omega)} \,\asymp\, n^{-1/d},
	\]
	where $\gamma=s(1/q-1/p)^{-1}$ for $q<p$ and $\gamma=\infty$ for $q\ge p$.
\end{corollary}

By Theorem~\ref{thm:sob-main}b, we clearly get the same characterization for numerical integration as for the problem of $L_1$-approximation. The relation between these two problems will be commented on later in Section~\ref{sec:sob-fool-int}.

Our results for $q<p$ seem to be novel already for $d=1$, where recent results were obtained in \cite{HKN+21}. To the best of our knowledge, similar results only have been known for the spaces $W_\infty^s(\Omega)$ with $s\le 2$, see Sukharev \cite{S79} and Pag\`es \cite{P98}. The latter author uses concepts from the theory of quantization of measures which will be discussed later on.

To illustrate the advantage of the characterization of optimal points in Corollary~\ref{cor:sob-char} compared to conditions involving the covering radius, consider the following example depicted in Figure~\ref{fig:big-hole}. 

\begin{example}[{\cite[Sec.~2.4]{KS20}}]\label{ex:big-hole}
Take for each $n\in\IN$ an $n$-point set $\Pn$ on a bounded convex domain $\Omega\subset\IR^d$. With regard to asymptotic behaviour, \chg{we ask:
\begin{center}
How large can the largest hole admist $\Pn$ be for this to be still optimal as a sampling set for $L_q$-approximation of $W^s_p(\Omega)$-functions? 
\end{center}
}

By Corollary~\ref{cor:sob-char} the sequence of point sets $(P_n)$ can be (but not necessarily is) asymptotically optimal for this problem if and only if  
\begin{equation} \label{eq:condition-largest-hole}
h_{P_n,\Omega}\lesssim n^{-1/d+1/(\gamma+d)}.
\end{equation}
This means that the radius of the largest hole is allowed to exceed the optimal covering radius $n^{-1/d}$ by the polynomial factor $n^{1/(\gamma+d)}$. A proof is provided in Appendix~\ref{sec:sob-proof-discuss}. We also refer to \cite[Thm.~1.2]{BDS+15}, where the necessity of condition~\eqref{eq:condition-largest-hole} (for $q=1$) has been observed for numerical integration on the sphere.
\end{example}

\begin{figure}[ht]
	\begin{center}
		\includegraphics[width=8cm]{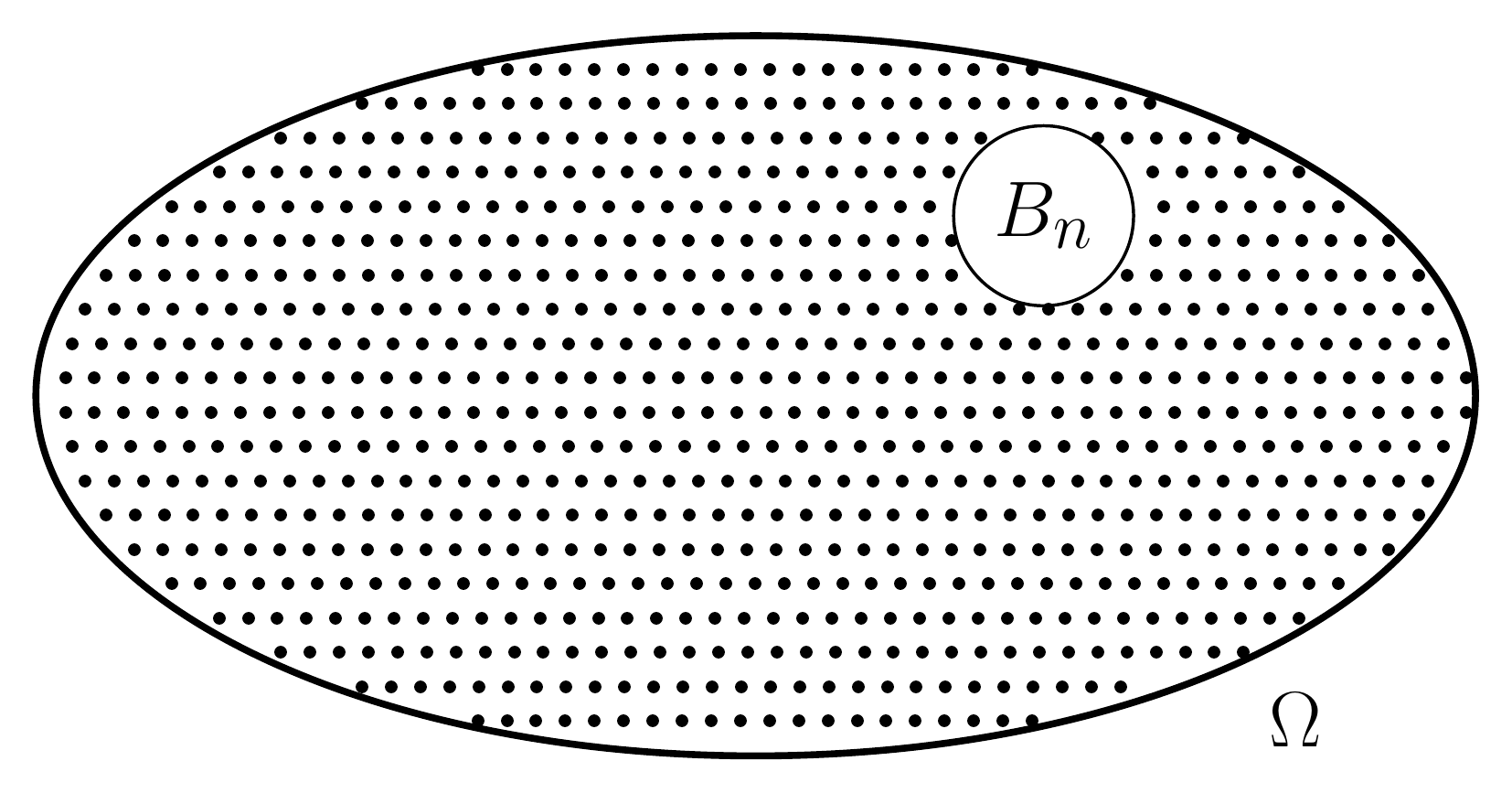}
	\end{center}
	\caption{An otherwise well-spaced point set containing a single large hole $B_n$ whose radius may be larger than the other holes by a polynomial factor without destroying the asymptotic optimality. Courtesy of D.~Krieg.}
	\label{fig:big-hole}
\end{figure}

Now that we have looked at optimal points, let us come to random points which have almost optimal covering properties as the following proposition shows. 

\begin{prop}[{\cite[Prop.~1]{KS20}}]\label{pro:random-covering}
Let $X_1,X_2,\ldots$ be independent and uniformly distributed on a bounded convex domain $\Omega\subset \mathbb{R}^d$. Consider the random $n$-point set $\pnran=\{X_1,\ldots,X_n\}$. Then, for any $\alpha\in (0,\infty)$,
\[
\mathbb{E}\,\|\dist(\cdot,\pnran)\|_{L_{\gamma}(\Omega)}^{\alpha}\,\asymp\, 
\begin{cases}
n^{-\alpha/d} & \text{if } 0<\gamma<\infty,\\
n^{-\alpha/d}(\log n)^{\alpha/d} & \text{if } \gamma=\infty. 
\end{cases}
\]
\end{prop}

Essentially, the proposition claims that the average hole size of random points is on average of optimal order $n^{-1/d}$, whereas the largest hole is logarithmically larger compared to optimal points. Its proof will be discussed in Section~\ref{sec:sob-quant}.

Together with Theorem~\ref{thm:sob-main} we obtain the following corollary on the power of random information for $L_q$-approximation and integration in Sobolev spaces. 

\begin{corollary}[{\cite[Cor.~2]{KS20}}]\label{cor:sob-ran}
Let $X_1,X_2,\ldots$ be independent and uniformly distributed on a bounded convex domain $\Omega\subset \mathbb{R}^d$, $1\le p,q \le \infty$ and $s\in\mathbb{N}$ \chg{as in \eqref{eq:embedding}}. Consider the random $n$-point set $\pnran=\{X_1,\ldots,X_n\}$. Then
\[
 \IE\, r\big(W_p^s(\Omega)\hookrightarrow L_q(\Omega),\pnran\big) 
 \,\asymp\, \begin{cases}
 	r\big(W^s_p(\Omega)\hookrightarrow L_q(\Omega),n/\log n\big) & \text{if } q\ge p,\\
	r\big(W^s_p(\Omega)\hookrightarrow L_q(\Omega),n\big) & \text{if } q< p.\vphantom{\Big|}
	\end{cases}
\]
\end{corollary}

Summarizing, we get that, on average, $n$ random points are as asympotically as good as $n/\log n$ optimal ones if $q\ge p$ and asymptotically optimal if $q<p$. This generalizes the one-dimensional result in \cite[Thm.~4.1]{HKN+20} ($s$ has to be replaced by one there) and gives an answer to Question~\ref{que:std} for the integration and the $L_q$-approximation problem.

Before discussing the extension to other function spaces, let us note a few things.

\begin{remark}
	We stated Corollary~\ref{cor:sob-ran} for the expected value of the radius of random information but \chg{in fact it} also holds with arbitrarily high probability by adjusting the constants and using Markov's inequality giving that, for $q<p$, 
\[
\IP\big[r\big(W_p^s(\Omega)\hookrightarrow L_q(\Omega),\pnran\big)> C\,\varepsilon^{-1}\,r\big(W^s_p(\Omega)\hookrightarrow L_q(\Omega),n\big)\big]\le \varepsilon,
\]
where $C>0$ is some constant independent of $n$ and $\varepsilon>0$ can be made arbitrarily small.
\end{remark}

\begin{remark}
	Proposition~\ref{pro:random-covering} is a slightly modified version of Proposition 1 in \cite{KS20}. There, we overlooked that estimates on arbitrary large moments are necessary for deducing Corollary~\ref{cor:sob-ran}. Thus, the proof of \cite[Cor.~2]{KS20} is flawed and it is necessary to carry out some minor amendments. The details will be explained at the end of Section~\ref{sec:sob-quant}. 
\end{remark}
\begin{remark}
	Instead of the worst-case error criterion one can consider a randomized error criterion for a general random algorithm, where one first takes the expected error on an individual function and then the supremum over all functions in the unit ball.  Together with Krieg and Novak we studied in \cite{KNS21} the power of random information in the randomized setting by adapting the techniques from Section~\ref{sec:sob-mls} below. See also the references cited there for further information. 
\end{remark}

In the remainder of this section we will present generalizations of the above results to other classes of sets $\Omega$ and also to other function spaces on $\Omega$ including Sobolev-Slobodeckij spaces of fractional smoothness, Triebel-Lizorkin spaces, H\"older-Zygmund spaces and Bessel potential spaces.

Let $\Omega\subset\IR^d$ be a bounded convex domain. Several of the function spaces mentioned are part of the scale of Triebel-Lizorkin spaces $F^s_{p\tau}(\Omega)$, where $0<p<\infty$, $0<\tau\le \infty$ and $s>d/p$, which ensures $F^s_{p\tau}(\Omega)\hookrightarrow C_b(\Omega)$. \chg{As for $W^s_p(\Omega)$ the condition $s>d/p$ can be weakened for $p\le 1$ but for simplicity we will always assume that $s>d/p$ holds for the $F$-spaces we consider.} These are closely related to Besov spaces $B^s_{p\tau}(\Omega)$ for which also $\tau=\infty$ is possible and we provide a more information on both scales in Appendix~\ref{sec:sob-proof-ext}. Here, we simply note that $F^s_{p\tau}(\Omega)$ covers a variety of interesting spaces:
\begin{itemize}
	\item For $\tau=2$ and $1<p<\infty$, we obtain the fractional Sobolev space (or Bessel potential space) $H^s_p(\Omega)$, see e.g.\  \cite{NT06}. If additionally $s\in\mathbb{N}$, we arrive at the classical Sobolev spaces $W^s_p(\Omega)$ from Definition~\ref{def:sobolev}.
\item For $s\not\in\IN$, $1\le p<\infty$ and $\tau=p$, we obtain the Sobolev-Slobodeckij space $W^s_p(\Omega)$ of fractional smoothness, see e.g.~Dupont and Scott~\cite{DS93}, and note that $ F^s_{pp}(\Omega) $ is equal to the Besov space $B^s_{pp}(\Omega)$.
\end{itemize}
We also want to discuss the H\"older spaces $C^s(\Omega)$, which are not included in this scale.  For $s\in\IN$, the H\"older space $C^s(\Omega)$ is the space of all $s$ times continuously differentiable functions with
\[
  \Vert f \Vert_{C^s(\Omega)} := \max_{|\alpha| \le s}\, \sup_{x\in\Omega}\, \vert D^{\alpha}f (x) \vert < \infty,
\]
where $\alpha$ is a multi-index as in Definition~\ref{def:sobolev}. For $ s\not\in\IN $, the space $C^s(\Omega)$ is defined to be the space of functions $f \in C^{\lfloor s\rfloor}(\Omega)$ with
\begin{equation}\label{eq:semi-norm-hoelder}
	|f|_{C^s(\Omega)}
	:=\max_{|\alpha|=\lfloor s\rfloor}\,\sup_{\substack{x,y\in\Omega\\x\ne y}}\, \frac{|D^{\alpha}f(x)-D^{\alpha}f(y)|}{\|x-y\|^{\{s\}}} < \infty,
\end{equation}
where $ s=\lfloor s\rfloor + \{s\} $ with $ 0< \{s\}<1 $ and $C^{0}(\Omega)=C(\Omega)$.  It is equipped with the norm $\|\cdot\|_{C^{\lfloor s\rfloor}(\Omega)}+|\cdot|_{C^s(\Omega)}$. 

\begin{theorem}[{\cite[Thm.~2]{KS20}}]\label{thm:sob-main-ext}
Theorem~\ref{thm:sob-main}, Corollary~\ref{cor:sob-char} and Corollary~\ref{cor:sob-ran} remain valid for arbitrary real parameters $0<p,q,\tau\le \infty$ and $s\in\IR$ with $s>d/p$ if we replace $W_p^s(\Omega)$ either by $C^s(\Omega)$ for $p=\infty$ or by $F^s_{p\tau}(\Omega)$ for $p<\infty$.
\end{theorem}

We defer the proof to Appendix~\ref{sec:sob-proof-ext}. In particular, Theorem~\ref{thm:sob-main-ext} implies that the fine index $\tau$, measuring smoothness on a finer scale, does not influence the asymptotic quality of point sets. 

In order to discuss general domains, we define the function spaces on a general bounded measurable set $\Omega\subset\IR^d$ by restriction of the corresponding function spaces on $\IR^d$. Let
\begin{equation} \label{eq:restriction}
A(\Omega):=\{ f\vert_\Omega \colon f\in A(\IR^d)\},
\qquad \Vert g \Vert_{A(\Omega)} := \inf_{f \in A(\IR^d)\colon f\vert_\Omega=g} \Vert f \Vert_{A(\IR^d)}
\end{equation}
for $A\in \{W_p^s, F^s_{p\tau}, C^s\}$ with parameters as above. If there is a bounded extension operator ${\rm ext}\colon A(\Omega)\to A(\IR^d)$ with $\chg{{\rm ext}(f\vert_\Omega)} = f$ these spaces coincide with the (intrinsically) defined ones above. Here and everywhere else in this thesis, equality between function spaces is meant up to equivalence of norms and the equivalence constants vanish in the asymptotic notation.

\begin{remark}\label{rem:domains}
	It can be deduced from the proofs below that Theorem~\ref{thm:sob-main}, Corollary~\ref{cor:sob-char} and Corollary~\ref{cor:sob-ran} and also Theorem~\ref{thm:sob-main-ext} remain valid in the case
	\begin{itemize}
		\item $q\ge p$ if $\Omega$ satisfies an interior cone condition, see Definition~\ref{def:cone-condition} in the following section, and if there is a bounded extension operator. This includes all bounded Lipschitz domains. For more information see Appendix~\ref{sec:sob-proof-poly} for Sobolev spaces as well as Appendix~\ref{sec:sob-proof-ext} for the other function spaces. 
		\item $q<p$ if $\Omega$ is a bounded convex set with non-empty interior. 
	\end{itemize}
 \end{remark}

\begin{remark}\label{rem:f-int}
	In our paper \cite{KS20} Theorem~\ref{thm:sob-main-ext} did not apply to Theorem~\ref{thm:sob-main}b, the equivalence of the integration to the $L_1$-approximation problem. Here, we use a novel approach not published yet in order to remedy this. The idea is to use a different way to construct fooling functions and estimate their norm. We want to thank M.~Ullrich for pointing us to the concept of atomic decompositions which we, together with my coauthor D.~Krieg, employed to prove this extension. See also the discussion in Section~\ref{sec:sob-fool-int}.
\end{remark}

\begin{remark}
	As mentioned, our paper \cite{KS21} relies on the work of \cite{KS20} discussed here. There, we showed an analogon of Theorem~\ref{thm:sob-main} for Bessel potential spaces on the sphere. In particular, the lower bounds in \cite[Thm.~2]{KS21} (see also Remark~2 there) partially rely on the lower bound discussed in the previous Remark~\ref{rem:f-int}, which we provide in this thesis. 
\end{remark}

\newpage

\section{Local moving least squares on good cubes} \label{sec:sob-mls}

In the following, we will discuss the strategy behind the proof of the upper bound of the characterization given in Theorem~\ref{thm:sob-main}a (and its extension, Theorem~\ref{thm:sob-main-ext}). The goal is to find, for each point set $\Pn$ in the domain $\Omega$, a reconstruction map $\varphi$ using information given by the function values at $\Pn$ such that the corresponding algorithm has a worst-case error over the unit ball of the Sobolev space $W^s_p(\Omega)$ which is bounded by a constant times the $L_{\gamma}(\Omega)$-norm of the distance function to $\Pn$.

For finding a suitable algorithm we use the technique of moving least squares which enables to reconstruct polynomials exactly and then use the fact that Sobolev functions can be well approximated by polynomials via Taylor's theorem. Let us give some details.

Building on earlier work, moving least squares approximation was developed by Lancaster and Salkauskas \cite{LS81} in 1981 and in recent decades attracted attention due to the increase in computing power, see Wendland~\cite[Sec.~4.4]{Wen04} for additional references. It is a special case of weighted least squares approximation, a variant of the usual least squares method, working as follows. 

Given a point set $\Pn=\{x_1,\dots,x_n\}$ and a function $f$ belonging to a space $F$, one tries to find a function $g_m$ in a fixed $m$-dimensional subspace $X_m$ of $F$ which minimizes the weighted sum of squares
\[
\sum_{i=1}^{n}|f(x_i)-g(x_i)|^2\, w(x,x_i), \quad w\ge 0,
\]
for the point $x$ and then returns $g_m(x)$ as an approximation to $f(x)$. This procedure reproduces elements from $X_m$ exactly provided they are uniquely defined by the function values at $\Pn$. In general, increasing the number of points $n$ allows to increase the dimension $m$ and thus for a better approximation. On a side note, weighted least squares approximation is the backbone of the results \chg{in \cite{DKU22,KU20,KU21}} described in the introduction to this chapter and effective for recovery, see for example Cohen and Migliorati~\cite{CM17}.

In the case of moving least squares on a $d$-dimensional domain, the finite-dimensional space is the space of polynomials
\[
\cP_m^d:=\Big\{\sum_{|\alpha|\le m}c_{\alpha}x^{\alpha}\colon c_{\alpha}\in\IR\text{ for all }\alpha\in \IN_0^d\Big\}
\]
of degree at most $m\in\IN$ in $d$ variables, where $x^{\alpha}=x_1^{\alpha_1}\cdots x_d^{\alpha_d}$ is the monomial associated to the multi-index $\alpha$. Further, the weight function $w$ is a bump function depending on $(x-x_i)/\delta$, where $\delta>0$ is a suitably chosen scaling parameter.  If the point set $\Pn$ is locally dense enough, then polynomials can be reconstructed exactly using moving least squares. This will be made precise by the following lemma which follows directly from results by Wendland~\cite{Wen01} (see also \cite[Thm.~4.7]{Wen04} in his book).

To measure the local density of a point set $\Pn\subset \IR^d$ on a set $\Omega\subset\IR^d$, we use the covering radius as given in \eqref{eq:covering-rad}, which measures the radius of the largest hole of the point set $\Pn$ in $\Omega$. In the following, a \emph{hole} in a point set $\Pn\subset \Omega$ will be a ball contained in $\Omega$ empty of $\Pn$, i.e., not containing a point of $\Pn$.

\begin{lemma}\label{lem:wendland}
	Let $K \subset\mathbb{R}^d$ be a compact set satisfying an interior cone condition with parameters $r$ and $\theta$, and let $m\in \mathbb{N}$. There are constants $c_0,c_1,c_2>0$ depending solely on $\theta,m$ and $d$ such that for any $\Pn=\{x_1,\dots,x_n\}\subset K$ with covering radius $h_{\Pn,K}\le c_1 r$ there exist continuous functions $u_1,\dots,u_n\colon K\to\mathbb{R}$ with
\begin{enumerate}[label={\emph{(\roman*)}}]
\item $\displaystyle \pi(y)=\sum_{i=1}^n \pi(x_i) u_i(y)$ for all $y\in K$
and $\pi\in \mathcal{P}_m^d $,
\item $\displaystyle \sum_{i=1}^n |u_i(y)| \le c_0$ for all $y\in  K$ and
\item $u_i(y)=0$ for all $i\le n$ and $y\in K$ with $\Vert x_i-y \Vert \ge c_2 h_{\Pn, K}$.
\end{enumerate}
\end{lemma}

Before we elaborate on the interior cone condition, let us give some background on how this lemma follows from \cite[Thm.~4.7]{Wen04}. There, it is required that the point set $\Pn$ (or rather an underlying sequence of $n$-point sets) is quasi-uniform, that is there exists a constant $c_{\rm qu}>0$ with
\begin{equation} \label{eq:quasi-uniform}
h_{\Pn,K}
\le q_{\Pn}
\le c_{qu} h_{\Pn,K},
\end{equation}
where $q_{\Pn}:=\frac{1}{2}\min_{i\neq j}\|x_i-x_j\|_2$ is the separation distance of $\Pn=\{x_1,\dots,x_n\}$. In Appendix~\ref{sec:sob-proof-ge} we will deduce Lemma~\ref{lem:wendland} \chg{by throwing away points too close to others in order to arrive at a quasi-uniform point set without affecting the local density by much.} 

\begin{definition}\label{def:cone-condition}
We say that a set $\Omega\subset \mathbb{R}^d$ satisfies an interior cone condition with radius $r>0$ and angle $\theta\in (0,\pi/2)$ if, for all $x\in \Omega$, there is a direction $\xi(x)\in \mathbb{S}^{d-1}$ such that the cone
\[
C(x,\xi(x),r,\theta):=\left\{x+\lambda y\colon y\in \mathbb S^{d-1}, \langle y, \xi(x)\rangle \geq \cos\theta,\lambda\in [0,r]\right\}
\]
with apex $x$ is contained in $\Omega$. See Figure~\ref{fig:cone-condition} for an illustration.
\end{definition}

\begin{figure}[ht]
	\centering
	\begin{tikzpicture}[path fading=fade right,line cap=round,line join=round,x=1cm,y=1cm, scale=1]
			\node at (4,1) {$\Omega$};
			\draw [shift={(4.98,1.85)},color=lgrey,fill=lgrey,fill opacity=0.10000000149011612] (0,0) -- (-80:0.8888888888888862) arc (-80:-35:0.8888888888888862) -- cycle;
		\draw[thick] (0,-1.5) --(1.06,0.59);
		\draw[thick] (0,-1.5) -- (2.5,-1.7);
		\draw[thick]  (1.06,0.59)-- (2.52,1.85);
		\draw[thick]  (2.52,1.85)-- (5.5,2.9);
		\draw[thick]  (8,2)-- (5.5,2.9);
		\draw[thick] (8,2)--(9,1);
		\draw[thick]  (2.5,-1.7)-- (6.8,-1.39);
		\draw[thick]  (9,1)-- (6.8,-1.39);

		\draw  (4.98,1.85)-- (7.437456132866975,0.12927069094686194);
		\draw (4.98,1.85)-- (7.437456132866975,0.12927069094686194) node [midway,yshift=7pt,sloped]{\small{$r$}};
		\draw  (4.98,1.85)-- (5.500944533000792,-1.1044232590366234);
		\draw [shift={(4.98,1.85)}]  plot[domain=4.886921905584122:5.672320068981571,variable=\t]({1*3*cos(\t r)+0*3*sin(\t r)},{0*3*cos(\t r)+1*3*sin(\t r)});
		\draw [->,thick] (4.98,1.85) -- (5.724243218734402,0.6817721536851105) node [very near end,xshift=14pt,yshift=-4pt]{\small{$\xi(x)$}};
		\fill (4.98,1.85) circle (.04) node [yshift=7pt] at (4.98,1.85) {\small{$x$}};
		\node [xshift=15pt,yshift=-16pt] at (4.98,1.85) {\small{$\theta$}};
		\draw  (5.724243218734402,0.6817721536851105)-- (6.591898825040471,-0.6801743374386566);
		\draw [shift={(0.88,-0.81)},color=lgrey,fill=lgrey,fill opacity=0.10000000149011612] (0,0) -- (0:0.8888888888888862) arc (0:45:0.8888888888888862) -- cycle;
		\draw  (0.88,-0.81)-- (3.88,-0.81);
		\draw [shift={(0.88,-0.81)}]  plot[domain=0:0.7853981633974483,variable=\t]({1*3*cos(\t r)+0*3*sin(\t r)},{0*3*cos(\t r)+1*3*sin(\t r)});
		\draw  (0.88,-0.81)-- (3.0013203435596427,1.3113203435596423);
		\draw (0.88,-0.81)-- (3.0013203435596427,1.3113203435596423) node [midway,yshift=7pt,sloped]{\small{$r$}};
		\draw [->,thick] (0.88,-0.81) -- (2.2670836530974685,-0.2354511387410111) node [very near end,xshift=3pt,yshift=-8pt]{\small{$\xi(x)$}};
		\fill (0.88,-0.81) circle (.04) node [yshift=-7pt] at (0.88,-0.81) {\small{$x$}};
		\node [xshift=16pt,yshift=11pt] at (0.88,-0.81) {\small{$\theta$}};
		\draw  (2.267,-0.235)-- (3.6516385975338603,0.33805029709526946);

  	\end{tikzpicture}
	\qquad
	\caption{For each $x\in\Omega$ the cone $C(x,\xi(x),r,\theta)$ may point in a different direction but is of the same shape.}
	\label{fig:cone-condition}
\end{figure}
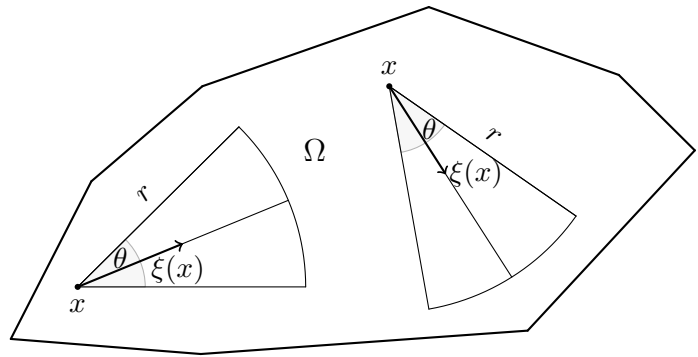

The interior cone condition above is related to similar regularity conditions, in part proposed to characterize the class of domains for which certain (Sobolev) function space embedding and extension theorems hold, see for example the books by Adams and Fournier~\cite{AF03} and Grisvard~\cite{Gri85}. Let us mention here that bounded convex domains satisfy an interior cone condition as above. The following lemma is a consequence of \cite[Prop.~11.26]{Wen04}. 

\begin{lemma}\label{lem:convex-cone}
Any bounded convex set $\Omega\subset \IR^d$ containing a closed ball of radius $r$ satisfies an interior cone condition with radius $r$ and angle $\theta=2\arcsin\bigl(r/2\,\diam(\Omega)\bigr)$.
\end{lemma}
Here and in the following, balls and cubes are always assumed full-dimensional. For applying Lemma~\ref{lem:wendland} to a non-compact set, and in particular a bounded convex domain, we shall need the following lemma, proven in Appendix~\ref{sec:sob-proof-ge}.
\begin{lemma} \label{lem:closure-CC}
	If a set $\Omega\subset\IR^d$ satisfies an interior cone condition, then its closure $\overline{\Omega}$ satisfies an interior cone condition with the same parameters $r$ and $\theta$.
\end{lemma}

Let us continue \chg{by} explaining the ideas behind the proof of the upper bound of Theorem~\ref{thm:sob-main}a. The strategy is different depending on whether $q\ge p$, where we basically extend the proof of \cite[Prop.~21]{NT06} to point sets with non-optimal covering radius and employ Lemma~\ref{lem:wendland} on $\overline{\Omega}$, or $q<p$, where it will be applied to small parts of the domain, which themselves satisfy an interior cone condition. This case distinction is also reflected by the structure of Appendix~\ref{ch:sob-app}, where Sections~\ref{sec:sob-proof-ge} and \ref{sec:sob-proof-le} contain the proofs needed for the case $q\ge p$ and $q<p$, respectively. The proof for the already known case $q\ge p$ is meant as a preparation for the novel approach developed in \cite{KS20} for the case $q<p$ on the basis of the mentioned works \cite{NT06,NWW04,Wen01}. In the following, we will explain the ideas behind starting with a brief description of the simpler case. 

In the case $q\ge p$, we apply Lemma~\ref{lem:wendland} to the closure of the bounded convex domain $\Omega$, which satisfies an interior cone condition by Lemma~\ref{lem:convex-cone} and Lemma~\ref{lem:closure-CC}. To bound the approximation error of the resulting algorithm $f\mapsto \sum_{i=1}^{n}f(x_i)u_i$, we cover $\Omega$ by parts of diameter $h_{\Pn,\Omega}$ and use error estimates on the local approximation of Sobolev functions by polynomials. For more details see Section~\ref{sec:sob-proof-ge}.

In the case $p>q$, we cannot achieve the desired bound by applying Lemma~\ref{lem:wendland} to the whole domain. Instead, we have to suitably cover it by pieces adapted to the point set $\Pn$ which satisfy an interior cone condition in such a way such that the constants involved in the error bounds do not depend on the point set. Our approach requires that the radius of the interior cone condition satisfied by any piece of the cover should be proportional to the diameter of the piece. This is where the convexity of the domain comes into play.  If the cover is constructed with cubes as it is the case in our work \cite{KS20}, the following lemma supplies the required interior cone condition. To dispose of the convexity assumption, a different approach seems to be required as Figure~\ref{fig:dent} shows, see Question~\ref{que:sob-convexity} below. 

\begin{lemma} \label{lem:local-CC}
Let $\Omega\subset\mathbb{R}^d$ be a bounded convex set containing a closed ball of radius $r>0$. Then there exist constants $c_{\theta}>0$ and $\theta'\in (0,\pi/2)$ depending on $\Omega$ and $r$ such that closure of the intersection of $\Omega$ with any cube $Q(x, \varrho )$ of radius $0< \varrho \le r$ centered at some $x\in \Omega$  satisfies an interior cone condition with radius $c_\theta  \varrho $ and angle $\theta'$.
\end{lemma}
Note that the exact choice of $r>0$ does not matter as this lemma is applied to very small cubes if the points become more and more dense. We use the notation
\begin{equation} \label{eq:cubes-balls}
Q(x,\varrho):=\{y\in\IR^d\colon\|x-y\|_{\infty}<\varrho\}\quad \text{and}\quad B(x,\varrho):=\{y\in\IR^d\colon\|x-y\|_2<\varrho\}
\end{equation}
for an open cube or ball centered at $x\in\IR^d$ and of radius $\varrho>0$.

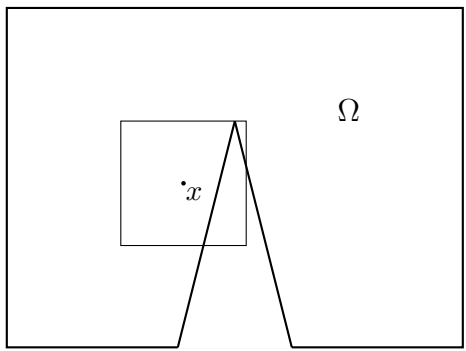
\begin{figure}[ht]
	\centering
    \begin{tikzpicture}[scale=1.5]
		\draw[thick] (0,0) rectangle (4,3);
		\draw[white,very thick] (1.5,0) -- (2.5,0);
		\fill[white] (1.5,0) -- (2.5,0)--(2,2)--(1.5,0);
		\draw[thick] (1.5,0)--(2,2);
		\draw[thick] (2.5,0)--(2,2);
		\draw (1.0,0.9) rectangle (2.1,2);
		\fill (1.55,1.45) circle (.02) node[xshift=4pt,yshift=-4pt] {\small{$x$}};
		\node at (3.0,2.1) {$\Omega$};
	\end{tikzpicture}
	\qquad
	\caption{The depicted bounded domain $\Omega$ is not convex and its intersection with the small cube has two connected components. By translating the cube to the left, one of the components becomes arbitrarily small and therefore cannot contain a single cone which has radius proportional to the cube. This can happen for an arbitrarily small cube and shows that Lemma~\ref{lem:local-CC} does not apply to $\Omega$. This remains unchanged by suitably rounding the corners of $\Omega$ such that its boundary is smooth.}
	\label{fig:dent}
\end{figure}

One reason why we choose cubes in order to construct the covering is that they form a family which is invariant under translations and rescalings. This is needed to ensure the independence of the constants in the Sobolev embedding theorem, which is used for the estimates on the local error of polynomial approximation by Sobolev functions, see Lemma~\ref{lem:sob-poly}. However, other \chg{set systems such as balls} may work as well, see Remark~\ref{rem:non-cubes} below.

Combined with Lemma~\ref{lem:wendland}, the local cone condition from Lemma~\ref{lem:local-CC} yields the following local approximation result, which will be applied for each cube in a suitable cover of the domain.

\begin{lemma} \label{lem:local-estimate}
	Let $\Omega\subset\IR^d$ be a bounded convex set containing a ball of radius $r>0$, $1\le p\le\infty$ and $s\in\IN$ \chg{as in \eqref{eq:embedding}}. There are constants $\cgood \in(0,1)$ and $C>0$, such that for any cube $Q(x,\varrho)\subset\IR^d$ of radius $0< \varrho \le  r $ centered at $x\in\Omega$ and any point set $\Pn=\{x_1,\dots,x_n\}\subset \Omega$ with
\begin{equation}\label{eq:good-cube}
 \sup_{y\in \Omega\cap Q(x,\varrho)} \dist(y,\Pn) \le \cgood\, \varrho ,
\end{equation}
there are bounded continuous functions $u_1,\dots,u_n\colon \Omega\cap Q(x,\varrho)\to \IR$ such that for any $f\in W^s_{p}(\IR^d)$, we have
\[
\sup_{y\in \Omega\cap Q(x,\varrho)} \Big\vert f(y) - \sum_{i=1}^n f(x_i) u_i(y) \Big\vert
\le C\,  \varrho ^{s-d/p}|f|_{W^s_{p}(Q(x,\varrho))}.
\]
\end{lemma}

The local upper bound provided by this lemma is in terms of a Sobolev seminorm $|\cdot|_{W^s_p(Q(x,\varrho))}$ which is essentially the sum in the definition of the norm (Definition~\ref{def:sobolev}) restricted to $|\alpha|=s$, see also \eqref{eq:sob-seminorm}. 

Lemma~\ref{lem:local-estimate} applies if condition \eqref{eq:good-cube} is satisfied, that is, if the covering radius $h_{\Pn,\Omega\cap Q(x,\varrho)}$ with respect to the point set $\Pn$ and the domain $\Omega\cap Q(x,\varrho)$ is small compared to the radius $\varrho$ of the cube. This motivates the following definition.

\begin{definition}
	We say that $Q(x,\varrho)$ with $x\in\Omega$ and $0<\varrho\le r$ is a good cube if it satisfies condition \eqref{eq:good-cube}.
\end{definition}

The smaller the radius of a good cube, the better the bound provided by Lemma~\ref{lem:local-estimate} \chg{if} $s-d/p>0$. Therefore, we would like to work for every $x\in\Omega$ with a good cube $Q_{\Pn}(x)$ of smallest radius possible as given by the following definition for which we provide some justification below.

\begin{definition}\label{def:small-cube}
	Define for each $\Pn\subset\Omega$ and $x\in\Omega$ the radius $r_{\Pn}(x)$ as the infimum over all $ \varrho \in(0, r )$ such that $Q=Q(x,\varrho)$ satisfies \eqref{eq:good-cube} and let $Q_{\Pn}(x):=Q\big(x,r_{\Pn}(x)\big)$.
\end{definition}

The radius $r_{\Pn}(x)$ is well-defined since we may assume that, without loss of generality, for any $\Pn\subset \Omega$ and $x\in \Omega$ there is some $ \varrho \in(0, r )$ such that \eqref{eq:good-cube} holds for the cube $Q(x,\varrho)$. Otherwise, we have $\|\dist(\cdot,\Pn)\|_{L_{\gamma}(\Omega)}\ge c$ for some $c>0$, see Lemma~\ref{lem:sob-trivial} in Appendix~\ref{sec:sob-proof-le}. Then the upper bound of Theorem~\ref{thm:sob-main}a would become trivial.

The infimum in Definition~\ref{def:small-cube} is in fact attained, see Lemma~\ref{lem:inf-min} in Appendix~\ref{sec:sob-proof-le}. Consequently, also $r_{\Pn}(x)$ satisfies \eqref{eq:good-cube} and $Q_{\Pn}(x)$ is a good cube, the smallest of all good cubes centered at $x$. Note that we could also have chosen $r_{\Pn}(x)$ to be the minimizer up to a constant factor and the proof still works.

Subsequently, we consider only the collection $\{Q_{\Pn}(x)\colon x\in\Omega\}$ of good cubes of minimal radius and we will cover the domain by a suitable finite subset of them. The following proposition yields a suitable cover. It follows from a Besicovitch-type covering theorem for cubes as presented by de Guzman \cite[Thm.~1.1]{DeG75}.

\begin{proposition}\label{pro:good-cover}
	Let $\Omega\subset\IR^d$ be bounded. There exist $M_1,M_2\in\mathbb{N}$ depending only on $d$ such that for any point set $\Pn\subset\Omega$ there are centers $y_1,\ldots, y_N\in \Omega$ such that the cubes $Q_i:=Q_{\Pn}(y_i)$, where $i=1,\dots,N$,
\begin{enumerate}
\item form a cover, i.e., $\Omega\subset \bigcup_{i=1}^N Q_i$,
\item can be distributed in at most $M_1$ families, i.e., $\{Q_1,\ldots,Q_N\}=\bigcup_{j=1}^{M_1}\cQ_j,$ with each family $\cQ_j$ consisting of pairwise disjoint cubes,
\item cover efficiently, i.e., no point of $\IR^d$ is contained in more than $M_2$ of them.
\end{enumerate}
\end{proposition}

To complete the proof of the upper bound of Theorem~\ref{thm:sob-main}, apply Lemma~\ref{lem:local-estimate} to each of the cubes $Q_i$ given by Proposition~\ref{pro:good-cover} to get, for every $i=1,\ldots,N$, a linear algorithm $S_{\Pn,i}\colon W^s_p(\Omega\cap Q_i)\to L_q(\Omega\cap Q_i)$. Then select for each point $x\in\Omega$ a cube $Q_i$ containing $x$ and let $S_{\Pn}(f)(x):=S_{\Pn,i}(f)(x)$. In this way, we form a global linear sampling algorithm $f\mapsto S_{\Pn}(f)=\varphi\big(f(x_1),\ldots,f(x_n)\big)\in L_{q}(\Omega)$, defined on $\Omega$. In fact, since the components $S_{\Pn,i}(f)$ are bounded, the algorithm returns a bounded function. It remains to combine the estimates given by Lemma~\ref{lem:local-estimate} with Hölder's inequality and the efficiency of the covering. A detailed proof will be provided in Appendix~\ref{sec:sob-proof-le}. We have not used the second property of the covering in Proposition~\ref{pro:good-cover}; it will be needed for the lower bound described in the next section.

\begin{remark}\label{rem:non-cubes}
	Note that Proposition~\ref{pro:good-cover} also works if the cubes are replaced by the familiy of all translations and rescalings of a fixed bounded set satisfying an interior cone condition, see \cite[Sec.~I.1, Rem.~(4)]{DeG75} where this is attributed to A.~P.~Morse. It should be possible to replace everywhere in our proofs the cubes by such a family.
\end{remark}

\begin{remark}
	Concerning Proposition~\ref{pro:good-cover}, we proved in \cite{KS20} the existence of a similar cover directly, see Lemma~11 and 12 as well as Proposition~2 there. During a seminar talk on our work by David Krieg in Linz, Simon Hackl suggested that the required cover may be obtained using \cite[Thm.~1.1]{DeG75}. Together with D.~Krieg, we succeeded in thus simplifying the proof in this way and the author wants \chg{to} thank S.~Hackl. This is one reason why the proof of the upper bound in the case $q<p$, which is given here, is slightly different from the one given in \cite{KS20}. 
\end{remark}

\section{Fooling functions and the integration problem} \label{sec:sob-fool-int}

We come now to the lower bound of Theorem~\ref{thm:sob-main}a which relies on the well-known technique of fooling functions and also discuss the related equivalence between $L_1$-approximation and integration stated in Theorem~\ref{thm:sob-main}b, and its extension Theorem~\ref{thm:sob-main-ext}. 

Given a sampling set $\Pn$, a \emph{fooling function} forces any algorithm using information provided by $\Pn$ to \chg{produce a large worst-case} error. Exhibiting such a function is a popular \chg{approach} to lower bounds for sampling algorithms. To elaborate a bit, let $\Omega$ be some domain and $f_\ast\in W^s_p(\Omega)$ with norm at most one and $f_\ast(x_i)=0$ for all $i=1,\ldots,n$. Then an algorithm of the form $S_{\Pn}\colon f\mapsto \varphi\big(f(x_1),\ldots,f(x_n)\big)$ cannot distinguish $f_\ast$ from $-f_\ast$, i.e., it satisfies $S_{\Pn}(f_*)=S_{\Pn}(-f_*)$. Therefore, the radius of information satisfies
\begin{align*}
	r\big(W_p^s(\Omega)\hookrightarrow L_q(\Omega),\Pn\big) &\,=\, \inf_{S_{\Pn}} \sup_{\|f\|_{W^s_p(\Omega)}\le 1} \Vert S_{\Pn}(f) - f \Vert_{L_q(\Omega)}\\
 &\ge\, \inf_{u\in L_q(\Omega)}\, \max\left\{ \Vert u - f_* \Vert_{L_q(\Omega)}, \Vert u + f_* \Vert_{L_q(\Omega)} \right\}\\
&\ge\, \|f_\ast\|_{L_q(\Omega)}.
\end{align*}
This reproves the lower bound in relation \eqref{eq:radius-sampling}. Any $f_{\ast}$ with sufficiently large $\|f_{\ast}\|_{L_q(\Omega)}$ can be a fooling function. In particular, it is meaningful to rescale such that $\|f_{\ast}\|_{W^s_p(\Omega)}=1$. We describe how to construct such a function.  

In the case $q\ge p$ it is enough to consider a smooth non-negative function $f_\ast$ which is supported \chg{in a large hole} with radius close to $h_{\Pn,\Omega}$. Scaling properties of the Sobolev and the $L_q$-norm can then be used to obtain the required estimates.

Again, the case $q<p$ demands a different approach. There, we have to take multiple disjoint holes in which the fooling function is supported. Intuitively, this is because we are interested in an average error. This distinction is well-known in the literature, \chg{see, e.g., the proofs of the lower bounds in \cite{BDS+15} and in \cite{NT06}}. To obtain the disjoint holes we will use the cover of Proposition~\ref{pro:good-cover} and select a family of disjoint cubes as in Definition~\ref{def:small-cube} with a large sum of radii. The following lemma \chg{then provides} a suitable hole of proportional size in every such cube and makes clear why we did choose good cubes of small radius.

\begin{lemma} \label{lem:good-hole}
	There exists a constant $\chole\in (0,\infty)$ such that for every point set $\Pn\subset \Omega$ and any $x\in\Omega$ the intersection $\Omega\cap Q_{\Pn}(x)$, contains a ball of radius $\chole\, r_{\Pn}(x)$.
\end{lemma}

To obtain the required estimates in the case $q<p$, we need some control over the Sobolev norm of a sum of functions with disjoint supports. Implicitly we used a property of Sobolev spaces, see for \chg{example} \cite[eq.~(9)]{Hei94}, which states that the Sobolev norm of a sum of a linear combination of functions $f_1,\ldots,f_M$ with pairwise disjoint support, which are obtained by scaling and translating a fixed function, can be described using an $\ell_p$-norm of the coefficients. This property is shared by other function spaces to a certain extent as discussed below.

Before this, we comment on the proof of Theorem~\ref{thm:sob-main}b which states that for the Sobolev spaces $W^s_p(\Omega)$ the radius of information for the integration problem is asymptotically equivalent to the one for the $L_1$-approximation problem. 

In order to obtain a suitable sampling algorithm to approximate the integral, we use the fact that if we have an algorithm $S_{\Pn}$ of the form $S_{\Pn}\colon f\mapsto \varphi\big(f(x_1),\ldots,f(x_n)\big)$ for the problem of $L_1$-approximation, then $S_{\Pn}(f)\in L_1(\Omega)$ and
\[
\Big|\int_{\Omega} f(x)- \int_{\Omega} S_{\Pn}(f)(x)\dd x\Big|
\le \|f-S_{\Pn}(f)\|_{L_1(\Omega)}.
\]
Regarding the lower bound, the above approach using fooling functions also applies to the integration problem if $f_\ast\ge 0$ since then $\|f_\ast\|_{L_1(\Omega)}={\rm INT}(f_\ast)$. Thus, Theorem~\ref{thm:sob-main}b is proven if we can find non-negative fooling functions.

Indeed, we shall use non-negative fooling functions in the proofs in Appendix~\ref{ch:sob-app} below. In Section~\ref{sec:sob-proof-le} there, we employ techniques which will also be suitable for the extension to other function spaces mentioned in Theorem~\ref{thm:sob-main-ext}. In our original work \cite{KS20} we extended the lower bound for the $L_q$-approximation problem to Triebel-Lizorkin spaces, which do not behave as well as Sobolev spaces with respect \chg{to} sums of functions of disjoint supports of different sizes, by using a wavelet characterization pointed out to us by W. Sickel.

However, the generalization of Theorem~\ref{thm:sob-main}b dealing with the integration problem did not succeed using the wavelets as these have to satisfy moment conditions which require them to have an integral of zero. Using atomic decompositions instead gives non-negative fooling functions and thus \chg{proves} also the extension of Theorem~\ref{thm:sob-main}b to the Triebel-Lizorkin spaces, see also Remark~\ref{rem:f-int} above. This will be done in the Appendix, Section~\ref{sec:sob-proof-le}.

At this point, we want to say a few more words on the proof of the generalization to other spaces, that is the proof of Theorem~\ref{thm:sob-main-ext}. As a crucial ingredient, we need a bound on polynomial approximation in terms of a seminorm that scales well. The work \cite{NT06} contains the required tools, even for spaces on the more general bounded Lipschitz domains defined below. Further, our method requires that the seminorms on the cubes can be glued together in a suitable manner. In this regard, our attempt at an extension of Theorem~\ref{thm:sob-main} to Besov spaces failed. We believe that this is solely a shortcoming of our technique and that Theorem~\ref{thm:sob-main-ext} holds also for Besov spaces, see also Question~\ref{que:besov}. For a definition of Triebel-Lizorkin and Besov spaces and details of the proof see Appendix~\ref{sec:sob-proof-ext}.

\section{Optimal quantizers and the distortion of random points} \label{sec:sob-quant}

In the following, we shall prove Proposition~\ref{pro:random-covering} on the expected average hole size of random points, which together with Theorem~\ref{thm:sob-main}, leads to the main result of Corollary~\ref{cor:sob-ran}. 

At the heart of the study of the radius of information of an arbitrary, and thus in particular also a typical or random, point set $\Pn=\{x_1,\ldots,x_n\}$ for the problems considered in this chapter lies the quantity 
\begin{equation} \label{eq:distortion}
D_{\mu,\Pn,\gamma}:=
\|\dist(\cdot,\Pn)\|_{L_{\gamma}(\IR^d,\mu)}^{\gamma}
=\int_{\IR^d}\dist(x,\Pn)^{\gamma}\dd \mu(x),
\end{equation}
where $0<\gamma<\infty$ and $\mu$ is a Borel measure on $\IR^d$ which we assume to be absolutely continuous with respect to the Lebesgue measure. In the context of quantization theory, which originated from signal processing, the number $D_{\mu,\Pn,\gamma}$ is also called the \emph{distortion}, see the monograph \cite{GL00} by Graf and Luschgy for more information on this subject. Before, we considered only the Lebesgue measure restricted to a domain $\Omega$ but in the following we shall provide some background from a more general perspective.

Roughly speaking, quantization means replacing a continuous object by a discrete one. In quantization theory one seeks to replace a random vector $X$ in $\IR^d$ with law $\mu$ by a quantized version which assumes only finitely many values. For this, one uses a Borel measurable map $T\colon\IR^d\to\IR^d$ with an image of cardinality $\#T(\IR^d)=n\in\IN$ and forms the random variable $T(X)$. Such a map is called an $n$-quantizer. To measure how well the quantized version $T(X)$ represents $X$ itself, one considers the quantization error $ \IE \|X-T(X)\|_2^r $, where $1\le r <\infty$ is such that $\IE \|X\|_2^r<\infty$.

If the image of an $n$-quantizer is some $n$-point set $\Pn\subset \IR^d$, then the optimal quantizer $T_{\Pn}$ maps a realization of $X$ to the closest point of $\Pn$. This means that it is of the form
\begin{equation} \label{eq:opt-quant}
	T_{\Pn}=\sum_{i=1}^{n}x_i\bfone_{C(x_i,\Pn)},
\end{equation}
where 
$
C(x_i,\Pn):=\{x\in\IR^d\colon \|x-x_i\|_2=\dist(x,\Pn)\}
$
are the Voronoi cells with respect to $\Pn$. By deciding ties arbitrarily, the collection $\{C(x_i,\Pn)\colon i=1,\ldots,n\}$ becomes a disjoint partition of $\IR^d$. A moment of thought reveals that the quantization error of $T_{\Pn}$ is exactly the distortion, i.e.,
\begin{equation} \label{eq:dist-quant}
	\IE\|X-T_{\Pn}(X)\|_2^r
=\sum_{i=1}^{n}\int_{C(x_i,\Pn)}\dist(x,\Pn)^r\dd\mu(x)
= D_{\mu,\Pn,r}.
\end{equation}

In quantization theory, one is often interested in the minimal quantization error over all $n$-quantizers, or equivalently in the infimum of the distortion over all $n$-point sets $\Pn$. In \eqref{eq:dist-optimal} we gave rough asymptotics for this quantity but in fact exact asymptotics are available. Also for random quantizers precise limit theorems have been obtained, for example, by Cohort~\cite{Coh04} and Yukich~\cite{Yuk08} who build on work of Zador~\cite{Zad82}. We deduce from \cite[Thm.~1]{Coh04} the following statement on random distortion.

\begin{proposition}\label{pro:limittheorem}
Let $0<\gamma<\infty$ and $X_1,X_2,\ldots$ be independent and uniformly distributed on a bounded measurable set $\Omega\subset \mathbb{R}^d$ satisfying an interior cone condition. Consider the random $n$-point set $\pnran=\{X_1,\ldots, X_n\}$. Then, for every $p=2,3,\ldots$,
\[
n^{\gamma/d}\frac{1}{\vol(\Omega)}\int_{\Omega}\dist(x,\pnran)^{\gamma}{\rm d}x
\,\xrightarrow[n\to\infty]{L_p}\, 
\left(\frac{\vol(\Omega)}{\vol\bigl(\IB_2^d\bigr)}\right)^{\gamma/d}\Gamma\Bigl(1+\frac{\gamma}{d}\Bigr).
\]
\end{proposition}

In order to deduce this from \cite[Thm.~1]{Coh04}, it remains to observe that the interior cone condition satisfied by $\Omega$ yields a constant $c_{\Omega}>0$ such that for every $x\in\Omega$ and every $0<\varrho\le 1+\sup_{x\in\Omega}\|x\|_2$ we have $\vol\bigl(\Omega\cap B(x,\varrho)\bigr)\ge c_{\Omega} \vol\bigl(B(x,\varrho)\bigr)$. 

We are now ready to give the proof of Proposition~\ref{pro:random-covering}. If $0<\gamma<\infty$, the lower bound is clear from the behaviour of optimal points as mentioned in \eqref{eq:dist-optimal}. For the upper bound, an application of Jensen's inequality shows that for $p= p(\alpha,\gamma)\in \IN$ chosen large enough such that $p\ge \alpha/\gamma$ we have 
\[
\Big(\IE \,n^{\alpha/d} \|\dist(\cdot,P)\|_{L_{\gamma}}^{\alpha}\Big)^{\gamma/\alpha}
\le 
\Big(\IE \Big[n^{\gamma/d} \int_{\Omega}\dist(x,P)^{\gamma}\dd x \Big]^{p}\Big)^{1/p}.
\]
By Lemma~\ref{lem:convex-cone} the bounded convex domain $\Omega$ satisfies the assumptions of Proposition~\ref{pro:limittheorem}, and thus the right-hand side converges to some finite limit. This concludes the proof of Proposition~\ref{pro:random-covering} for this choice of $\gamma$. 

The case $\gamma=\infty$ follows from asymptotics for the size of the largest hole in a random point set which were given by  Reznikov and Saff~\cite{RS16}. More precisely, we obtain the statement of Proposition~\ref{pro:random-covering} in this case from Corollary 2.3 there applied to $\Phi(r)=r^d$ together with the fact that convex domains satisfy an interior cone condition. 

\begin{remark}
	The form of the limit theorem in Proposition~\ref{pro:limittheorem} is different from the one used in \cite[Sec.~3]{KS20}. There, we relied on \cite[Thm.~2]{Coh04} which gives an almost sure convergence together with convergence in $L_2$ but we neglected that convergence of arbitrary high moments is necessary. The above proof remedies this.

\end{remark}

\begin{remark}
Actually, we used only a special instance of Cohort's result which is valid for different measures $\mu$, too. Further, the arguments behind the proof of Proposition~\ref{pro:limittheorem} transfer also to points sampled from other measures which are absolutely continuous with respect to the Lebesgue measure and have a density bounded from below. One probably needs to employ different methods and arrives at qualitatively distinct results if more general measures are considered. 
\end{remark}
\begin{remark}
	Proposition~\ref{pro:limittheorem} also holds in more general spaces. For example, a simplified proof working on general Riemannian manifolds was presented by us in \cite{KS21} and shows that one may easily obtain convergence in $L_p$ for $0<p<\infty$ using Jensen's inequality. 
\end{remark}
\begin{remark}
	One can also show Corollary~\ref{cor:sob-ran} without Theorem~\ref{thm:sob-main} or Proposition~\ref{pro:limittheorem} by using techniques similar in spirit to \cite{KNS21}.	
\end{remark}
\begin{remark}
	Regarding the asymptotic constant in Proposition~\ref{pro:limittheorem}, note that the quantity $\bigl(\vol(\Omega)/\vol\bigl(\IB_2^d\bigr)\bigr)^{1/d}$ is related to the concept of volume ratio, which plays an important role in Banach space geometry, see, e.g.,  Szarek and Tomczak-Jaegermann \cite{ST80}. Further, as $\gamma\to\infty$ the factor $\bigl(\Gamma(1+\frac{\gamma}{d})\bigr)^{1/\gamma}$ tends to infinity, which corresponds to the fact that the covering radius of $\pnran$ is typically of larger order than $n^{-1/d}$.
\end{remark}

	\begin{remark}\label{rem:random-covering}
		On smoothly bounded domains or polytopes, exact asymptotics are also available for the covering radius of i.i.d.\ uniform random points. In \cite{RS16} a proof may be found for the case of three-dimensional polytopes but in fact it generalizes to arbitrary dimension. Even more precise results such as a limiting Gumbel law, an important extreme value distribution showing up as the limit distribution of rescaled maxima, were derived by Penrose~\cite{Pen21}. Further, the covering radius of a random point set is connected to the coupon collector's problem, which asks for the number of coupons (points) that a collector has to draw in order to obtain a complete collection (hit every set in a diameter-bounded equal volume partition of $\Omega$).
\end{remark}

\section{Open questions}
\label{sec:sob-open}

\begin{question}
	When discussing the moving least squares method, we started with general subspaces which we then replaced by the space of polynomials. It would be interesting to consider other subspaces and develop a more general formulation of our results.
\begin{center}
	Can  Theorem~\ref{thm:sob-main} be generalized to other function spaces?
\end{center}
It seems worthwhile to pass from moving least squares to general weighted least squares in order to extend our results to other function spaces where smoothness is not determined via polynomials. In particular, it could be of interest to study the integration problem in Hilbert spaces of functions on which function evaluations are continuous. These widely used spaces are called reproducing kernel Hilbert spaces and the corresponding kernel allows for an analysis of sampling methods, see for example the book \cite{BT04} by Berlinet and Thomas-Agnan or \cite[Ch.~10]{Wen04}. 
\end{question}

\begin{question}\label{que:besov}
As mentioned in Section~\ref{sec:sob-fool-int}, we could not extend our results to Besov spaces. 
The following question is left open maybe due to our own incompetence or due to the lack of a suitable seminorm. 
\begin{center}
	Does Theorem~\ref{thm:sob-main-ext} also hold for Besov spaces $B^s_{p\tau}$?
\end{center}
Similar to the proof of Theorem~23 in \cite{NT06} one may use interpolation arguments to prove an upper bound, but in this way an additional $\varepsilon>0$ appears in the exponent of the $L_{\gamma}$-norm in the case $q<p$. Therefore, new ideas seem to be required. 
\end{question}

\begin{question}\label{que:sob-convexity}
	As explained in Section~\ref{sec:sob-mls}, convexity is central to our proof technique and therefore new ideas seem necessary to answer the following question.
\begin{center}
	Can Theorem~\ref{thm:sob-main} be extended to general bounded Lipschitz domains?
\end{center}
A possible approach could consist in using a different cover adapted to the boundary of the domain, as well as to the point set used. However, caution is required to ensure \chg{independence of the constants from the point set.} 
\end{question}

\begin{question}
	We did not track the hidden constants in Theorem~\ref{thm:sob-main} (and the derived results). As the dependence of the involved constants on the dimension is important for tractability considerations, the following question comes to mind.
 \begin{center}
	 What is the dependence of the constants in Theorem~\ref{thm:sob-main} on the dimension?
 \end{center}
 We think that our proof technique leaves an exponentially growing gap with regard to this question. It appears easier to determine for fixed dimension the asymptotic constants in front of the expected radius of random information if $n$ tends to infinity. For example, one could use a more elaborate fooling function such as a smoothed version of the distance function, as suggested to us by M.~Ullrich, and employ the full strength of the limiting behaviour of random distortion given by Proposition~\ref{pro:limittheorem}. Such questions are of interest also for the minimal radius and optimal points, see, e.g., Novak~\cite{Nov20}.
\end{question}

\begin{question}\label{que:sob-weights}
	One drawback of the linear algorithm attaining the upper bound in the integration problem is that we do not know the weights it uses.
\begin{center}
	Can the weights of a suitable linear algorithm achieving the asymptotic upper bound in Theorem~\ref{thm:sob-main}b be expressed geometrically?
\end{center}
It is clear that the weights of any linear algorithm (of the form  \eqref{eq:linear-algo}) achieving the asymptotic upper bound have to depend on the point set. In the upcoming Chapter~\ref{ch:interlude} we will see that these weights can be given explicitly for Hölder spaces with small smoothness and this choice may be optimal for higher smoothness too. 
\end{question}
 
\chapter{Interlude - optimal transport and cubature rules}
\label{ch:interlude}

This chapter connects the integration problem to the concept of discrepancy, which is related to equal-weight integration, and thus links Chapter~\ref{ch:sob} with the upcoming Chapter~\ref{ch:iso}. The focus will be on the weights used by (near-)optimal cubature rules for Hölder and Lipschitz functions. For these weights a geometric expression as volumes of Voronoi cells is shown, and further that equal weights can be optimal in high dimensions for uniform random points. Along the way, we pile dirt onto heaps and use classical results from probability to study Monte Carlo.

The main results of Chapter~\ref{ch:sob} show that we can assess the quality of any point set $\Pn$ for several approximation and integration problems in terms of the distortion, an average over the distance function to $\Pn$ defined in \eqref{eq:distortion}. Much alike, both the distortion and the discrepancy, which will be introduced in Chapter~\ref{ch:iso}, measure the distance between a measure $\mu$ and a point set $\Pn$, which can be represented by a sum of Dirac measures. Distances between measures have been studied in the field of optimal transport, which gave rise to the concept of the Wasserstein (or Kantorovich) distance, a well known way to metrize the space of measures. We refer to the book of Villani~\cite{Vil03} for a comprehensive account and limit ourselves to a brief introduction in the following. 

To transport the unit mass from a Borel probability measure $\mu$ on $\IR^d$ to another, say $\nu$, one can use a transference plan or coupling $\pi$ which is a probability measure on $\IR^d\times \IR^d$ with marginals $\mu$ and $\nu$. We collect all such transference plans between $\mu$ and $\nu$ into a set denoted by $\Pi(\mu,\nu)$. Optimal transport concerns itself with finding a transference plan with minimal transportation cost, that is, attaining the infimum in
\[
\mathcal{T}_r(\mu,\nu)
:=\inf_{\pi\in \Pi(\mu,\nu)}\int_{\IR^d\times \IR^d} \|x-y\|_2^r \dd\pi(x,y),
\]
where we took the Euclidean distance raised to some power $0<r<\infty$ as a cost function.  By setting $W_r=\mathcal{T}_r^{1/r}$ for $ 1\le r<\infty$ and $W_r=\mathcal{T}_r$ for $0<r<1,$ one can define a metric on the space of probability measures with finite $r$-th moments, the Wasserstein (or Kantorovich) distance, see \cite[Thm.~7.3]{Vil03}. The case $r=1$ is especially interesting as we will see in a moment.

To prepare for the connection between Wasserstein distance and the integration problem, assume from now on that $\mu$ is absolutely continuous with respect to the Lebesgue measure and that $\nu$ is supported on a point set $\Pn=\{x_1,\ldots,x_n\}\subset \IR^d$, i.e., of the form 
\begin{equation} \label{eq:discrete}
\nu_{\Pn,a}
:=\sum_{i=1}^{n}a_i\delta_{x_i},\quad
\text{where }a=(a_1,\dots,a_n)\in\IR^n
\text{ with }\sum_{i=1}^{n}a_i=1,
\end{equation}
where for any $x\in\IR^d$ the Dirac measure of any $A\subset \IR^d$ equals one if $x\in A$ and zero else.

The optimal transportation cost between $\mu$ and $\nu_{\Pn,a}$ as above attains a minimal value which is equal to the distortion, that is, for any $0<r<\infty$, 
\begin{equation} \label{eq:transport}
	\inf_{a\in\IR^n} \mathcal{T}_r(\mu,\nu_{\Pn,a})
	= \int_{\IR^d} \min_{i=1,\dots,n} \|x-x_i\|_2^r \dd\mu(x)
\quad(=D_{\mu,\Pn,r})
\end{equation}
and the infimum over the weights is in fact attained at 
\begin{equation} \label{eq:voronoi}
\mu\big(C(x_1,\Pn)\big),\dots,\mu\big(C(x_n,\Pn)\big)
\end{equation}
where $C(x_i,\Pn)$ is the Voronoi cell of $x_i$ as in \eqref{eq:opt-quant}. A proof can be deduced from \cite[Lem.~3.1]{GL00} and will be provided for completeness in Appendix~\ref{ch:interlude-app}. The idea behind is to consider the coupling induced by the map $(x,y)\mapsto \big(x,T_{\Pn}(x)\big)$, where $T_{\Pn}$ is the quantizer from \eqref{eq:opt-quant} in Section~\ref{sec:sob-quant} mapping $x$ to the nearest point of $\Pn$.

Let us recall the integration problem on $\IR^d$. Suppose we are given a normed space of functions $F$ on $\IR^d$ such that the embedding $F\hookrightarrow C_b(\IR^d)$ holds, see Section~\ref{sec:std}. The goal is to find $\varphi\colon \IR^n\to \IR$ such that, for every $f\in F$, $A(f)=\varphi\big(f(x_1),\ldots,f(x_n)\big)$ approximates the integral ${\rm INT}_{\mu}(f)=\int_{\IR^d}f(x)\dd\mu(x)$ with respect to the absolutely continuous probability measure $\mu$ on $\IR^d$. Here, assuming absolute continuity ensures that the problem is not too easy.

As mentioned in the introduction to Chapter~\ref{ch:sob} the optimal $\varphi$ can be chosen linear, that is, the optimal sampling operator is of the form
\begin{equation}\label{eq:cubature}
	Q_{\Pn,a}\colon F\to \IR,\quad f\mapsto\sum_{i=1}^{n}a_i f(x_i)=\int_{\IR^d}f(x)\dd\nu_{\Pn,a}(x),
\end{equation}
with $\nu_{\Pn,a}$ as in \eqref{eq:discrete}. This is a cubature rule with points $\Pn=\{x_1,\ldots,x_n\}$ and weights $a=(a_1,\ldots,a_n)\in \IR^n$. Note that we mean in fact $\Pn=(x_1,\ldots,x_n)$ to be able to associate the weights to the points but out of habit use the set notation which should not cause confusion. Then the worst-case error over the unit ball $F_0$ of $F$ is
\[
e(F,\intmu, Q_{\Pn,a})
=\sup_{f\in F_0}\Big|\int_{\IR^d} f(x) \dd \mu(x)-\int_{\IR^d} f(x) \dd \nu_{\Pn,a}(x)\Big|,
\]
and can be seen as a distance between the measures $\mu$ and $\nu_{\Pn,a}$. If we consider the special case $F=\lip$, the space of Lipschitz functions $f\colon \IR^d\to\IR$, equipped with the seminorm
\[
|f|_{\lip}
:=\sup_{\substack{x,y\in\IR^d\\x\neq y}}\frac{|f(x)-f(y)|}{\|x-y\|_2},
\]
then the Kantorovich-Rubinstein duality theorem (\cite[Thm.~1.14]{Vil03}) yields that, for any $\nu_{\Pn,a}$ as in \eqref{eq:discrete},
\begin{equation} \label{eq:lip-wce}
\sup_{|f|_{\lip}\le 1} \Big|\int f(x)\dd\mu(x) - \int_{\IR^d} f(x) \dd \nu_{\Pn,a}(x)\Big|
=W_1(\mu,\nu_{\Pn,a}).
\end{equation}
This means that the worst-case error over Lipschitz functions with seminorm at most one can be expressed as the $W_1$-distance between the measure $\mu$ according to which we want to integrate and the discrete measure $\nu_{\Pn,a}$ representing the algorithm $Q_{\Pn,a}$.

It seems worth noting that the metric $W_1$ is also called the earth mover's distance for the following reason. If one imagines the density of $\mu$ to be a distribution of dirt, which one wishes to collect into heaps of size $a_i$ located at points $x_i$, and the work needed for transporting dirt is proportional to the distance it is moved, then the least amount of work is exactly the right-hand side of \eqref{eq:lip-wce}.

\bigskip

\textbf{Optimal weights.} By means of the expression of the minimal distance $W_1=\mathcal{T}_1$ in \eqref{eq:transport} one can deduce that the normalized weights minimizing the worst-case error over the unit ball of the seminormed space $\lip$ are given as in \eqref{eq:voronoi} by the $\mu$-content of the Voronoi cells. This also holds more generally as we now explain.

Equality \eqref{eq:lip-wce} extends to the Hölder classes $C^{s}(\IR^d),0<s<1$, with seminorm $|\cdot|_{C^{s}(\IR^d)}$ defined in \eqref{eq:semi-norm-hoelder}. It can be deduced for example from the more general Theorem~5 in Gruber \cite{Gru04} that 
\begin{equation} \label{eq:hoelder-wce}
\inf_{a_1,\ldots,a_n\in\IR}\sup_{|f|_{C^{s}(\IR^d)}\le 1} \Big|\int f(x)\dd\mu(x) - \sum_{i=1}^{n}a_i f(x_i)\Big|
=D_{\mu,\Pn,s},
\end{equation}
where the optimal weights are again as in \eqref{eq:voronoi}. This holds true for $s=1$ if $C^s(\IR^d)$ is replaced by $\lip$. For convenience we provide a proof in Appendix~\ref{ch:interlude-app}.

Strictly speaking, since we take the supremum over a set of seminorm at most one and not over the unit ball of a normed space, we are not in the setting considered before. However, the space $C^s(\IR^d)$ can be equipped by the norm $\|\cdot\|_{\infty}+|\cdot|_{C^s(\IR^d)}$, see the line below \eqref{eq:semi-norm-hoelder}, and one can proceed analogously for $\lip$. As the seminorm is bounded by the norm thus defined, one can get an upper bound for the radius of information of $\Pn$ for integration in these spaces. 

More can be said if $\mu$ has bounded support $D\subset \IR^d$. Then the space of Lipschitz functions $f\colon D\to \IR$ will be denoted by $\lipd$ and becomes a normed space when equipped with the norm $\|\cdot\|_{\lipd}$ which is the restriction of $\|\cdot\|_{\lip}$. For $\lipd$ and the Hölder spaces $C^s(D)$ defined in Section~\ref{sec:sob-res} we have the following known result.

\begin{proposition} \label{pro:lip}
	Let $s\in (0,1)$ and $\mu$ be a probability measure supported on a bounded set $D\subset \IR^d$. Then, for any point set $\Pn=\{x_1,\ldots,x_n\}\subset D$,
\[
\inf_{a_1,\ldots,a_n\in \IR}\sup_{\|f\|_{C^s(D)}\le 1} \Big|\int_{\IR^d} f(x)\dd\mu(x) - \sum_{i=1}^{n}a_i f(x_i)\Big|
\asymp D_{\mu,\Pn,s},
\]
where the implicit constants depend only on $D$ and the weights given in \eqref{eq:voronoi} attain the bound. The statement remains true for $s=1$ if $C^s(D)$ is replaced by $\lipd$.
\end{proposition}

Since linear algorithms are optimal for this integration problem this proposition implies that the radius of information given by the point set $\Pn$ satisfies
\[
r\big(C^s(D),\intmu,\Pn\big)
\asymp D_{\mu,\Pn,s},
\]
which extends Theorem~\ref{thm:sob-main-ext} for these spaces. An analogous result holds for $\lipd$.

Proposition~\ref{pro:lip} provides us with (near-)optimal weights and the cubature rule 
\begin{equation} \label{eq:opt-cubature}
Q_{\Pn,\mu}(f)
:=\sum_{i=1}^{n}\mu\big(C(x_i,\Pn)\big)f(x_i),
\end{equation}
which is asymptotically optimal in the sense of
\[
\chg{ e(C^s(D),\intmu,Q_{\Pn,\mu})
\le C\, r\big(C^s(D),\intmu,\Pn\big), }
\]
where the constant $C$ is independent of $\Pn$.  

In general, it is a question of interest to compute (near-)optimal weights for general point sets and in most cases they are unknown. It would be interesting to have an analogon of Proposition~\ref{pro:lip} for other classes such as Sobolev spaces, see also Question~\ref{que:sob-weights} and  Remark~\ref{rem:opt-weights} below.

Before we turn to the case of equal weights, we show that, under minor assumptions, normalized weights \chg{cannot} be much worse than general real weights. Since we were not able to locate a proof of this known fact, we provide one in Appendix~\ref{ch:interlude-app}.

\begin{proposition}\label{pro:weights}
	Let $\mu$ be a probability measure supported on a set $D\subset \IR^d$ and let $F$ be a normed space of functions with $F\hookrightarrow C_b(D)$. Suppose that the constant function $\bfone_D$ belongs to $F$ (and thus all constant functions). If $Q_{\Pn,a}$ is a cubature rule as in \eqref{eq:cubature}, then the normalized weights $a_i^*=a_i/\sum_{i=1}^{n}a_i, i=1,\ldots,n,$ satisfy
\[
\sup_{\|f\|_F\le 1} \Big|\intmu(f)-\sum_{i=1}^{n}a_i^* f(x_i)\Big|
\le (1+C\|\bfone_D\|_F) \sup_{\|f\|_F\le 1} \Big|\intmu(f)-\sum_{i=1}^{n}a_i f(x_i)\Big|.
\]
\end{proposition}

The assumptions on the function space $F$ are satisfied for all function spaces considered until now if the set $D$ is bounded. This provides some justification for studying equal weights $a_1=\cdots=a_n=1/n$. 

\begin{remark}\label{rem:opt-weights}
In the case of reproducing kernel Hilbert spaces, one can compute optimal weights explicitly by solving a linear system, see, e.g., \cite[Ch.~10.2]{NW10}. For example, this can be done for the Sobolev spaces $W^s_2(\Omega)$, where a kernel can be derived from  Novak, M.~Ullrich, Wo\'zniakowski and Zhang \cite{NUW+18}, see also Section~2.3 in \cite{KS20}. However, there does not seem to be a convenient geometric expression of near-optimal weights.
\end{remark}

In the following, we will discuss equal-weight rules in Hölder spaces and discover that these are in fact asymptotically optimal, provided the dimension is large enough. This motivates the study of discrepancy in the upcoming Chapter~\ref{ch:iso}.

\bigskip

\textbf{Equal weights. }
Define, for $\Pn=\{x_1,\ldots,x_n\}\subset \IR^d$, the equal-weight cubature rule
\begin{equation} \label{eq:qmc}
Q_{\Pn}(f)
:=Q_{\Pn,(1/n,\ldots,1/n)}
=\frac{1}{n}\sum_{i=1}^{n}f(x_i).
\end{equation}
Thus we allow only the reconstruction map $\varphi\colon \IR^n\to\IR$ with $\varphi(y)=\frac{1}{n}\sum_{i=1}^{n}y_i$ and we may interpret the worst-case error of $Q_{\Pn}$, which is 
\[
e(F,\intmu,Q_{\Pn})
= \sup_{f\in F_0}\Big| \frac{1}{n}\sum_{i=1}^{n}f(x_i)-\int_{\IR^d} f(x)\dd\mu(x)\Big|,
\]
as a measure of \chg{the} quality of information given by sampling at $\Pn$. 

At first sight, it seems that by restricting to equal weights much flexibility is lost and only special kinds of point sets may yield good equal-weight rules. This is true, for example, if one requires exact integration on subspaces as in the case of spherical designs, surveyed e.g.\ by Brauchart and Grabner in \cite{BG15}. 

However, if we are interested in asymptotically optimal equal-weight rules with an error allowed to exceed the minimal error
\[
\chg{r(F,\intmu,n)}
=\inf_{\#\Pn= n}r(F,\intmu,\Pn)
\]
by a proportional amount, then in some cases typical information is asymptotically optimal. As in Section~\ref{sec:std}, we use for each $n\in\IN$ the random point set $\pnran:=\{X_1,\ldots,X_n\}$ with random points $X_1,\ldots,X_n$ drawn independently and uniformly according to the probability measure $\mu$ on $\IR^d$. 

Given random information, the corresponding random equal-weight rule associating to each $f$ the random number $\frac{1}{n}\sum_{i=1}^{n}f(X_i)$ is the prototypical example of a Monte Carlo method, whereas the algorithm in \eqref{eq:qmc} using the deterministic points above is often called a quasi-Monte Carlo method. 
Monte Carlo methods are very popular among practioners in fields such as physics or mathematical finance where high-dimensional integrals have to be evaluated and common cubature rules based on polynomial interpolation become computationally expensive or numerically unstable. In a moment, we shall give some intuition why it is so effective in these situations. We refer to \cite{Nie92} (see also the survey \cite{Nie78}) for further information.

The law of large numbers states that, almost surely,
\[
\frac{1}{n}\sum_{i=1}^{n}f(X_i)\to \IE f(X_1)=\int_{\IR^d}f(x)\dd\mu(x) \quad \text{as }n\to\infty
\]
provided that $f(X_1)$ is integrable. Further, the central limit theorem implies that the fluctuations around this limit are asymptotically of order $n^{-1/2}$, that is, we have the convergence in distribution
\[
\sqrt{n}\,\Big(\frac{1}{n}\sum_{i=1}^{n}f(X_i) - \int_{\IR^d}f(x)\dd\mu(x) \Big)
\to \mathcal{N}\big(0,{\rm Var}_{\mu}(f)\big)\quad \text{as }n\to\infty, 
\]
to a centered Gaussian with variance ${\rm Var}_{\mu}(f)=\IE f(X_1)^2 -\big(\IE f(X_1)\big)^2$ provided that $f(X_1)$ is square-integrable.

Thus we make an error on individual functions which is of the dimension- but also smoothness-independent rate $n^{-1/2}$. For obtaining a bound on the worst-case error over a function class, one can employ empirical process theory or entropy estimates, see, e.g., van der Vaart and Wellner~\cite{VW96}. However, this necessarily increases the error, if only by a factor.

In high dimensions the rate $n^{-1/2}$ can be optimal for functions of low effective smoothness $s/d$. For example, it follows from Theorem~\ref{thm:sob-main-ext} and \eqref{eq:dist-optimal} (see also Proposition~\ref{pro:lip} for $s\in(0,1)$) that for the Hölder spaces $C^{s}(D)$, where $D\subset \IR^d$ is a bounded convex set and $s>0$, the $n$-th minimal radius of information satisfies 
\[
r\big(C^s(D),{\rm INT},n\big)\asymp n^{-s/d},
\]
where integration is with respect to the Lebesgue measure. In fact, using the Monte Carlo method above one can achieve this optimal rate if $s/d<1/2$ at least on the unit cube as the following discussion shows.

Among others, Kloeckner \cite{Klo20} studied the related topic of (Wasserstein) distances between a measure $\mu$ and its empirical measure $\mu_n:=\frac{1}{n}\sum_{i=1}^{n}\delta_{X_i}$ which corresponds to the Monte Carlo method. Theorem~1.4 in \cite{Klo20} shows that, see also related work by Fournier and Guillin~\cite{FG15} and van der Vaart and Wellner~\cite[Sec.~2.7.1]{VW96}, for any $s\ge 1$ and any probability measure $\mu$ on $[0,1]^d$, 
\[
\IE \sup_{\|f\|_{C^{s}([0,1]^d)}\le 1}\Big|\frac{1}{n}\sum_{i=1}^{n}f(X_i)-\int_{[0,1]^d} f(x)\dd\mu(x)\Big|
\lesssim 
\begin{cases}
	n^{-s/d},& \text{if }s<d/2,\\
	n^{-1/2}\log n,& \text{if }s=d/2,\\
	n^{-1/2},& \text{if }s>d/2.\\
\end{cases}
\]
Note that here for $s\in\IN$, in \chg{contrast} to our definition, the space $C^s$ consists of functions whose derivatives of order $s-1$ are Lipschitz continuous and is normed slightly differently.  This means that in the low smoothness regime $s/d<1/2$ the Monte Carlo method achieves on average the optimal rate $n^{-s/d}$. At the threshold $s/d=1/2$ an additional logarithm comes into play and for higher smoothness we arrive at the standard Monte Carlo rate of $n^{-1/2}$ which cannot be improved in general.

The special case of Lipschitz functions ($s=1$) leads to the case distinction into $d\ge 3$, where we have the optimal rate, $d=2$, where we loose a logarithm, and $d=1$, where we have the rate $n^{-1/2}$. Since Lipschitz functions are related to Wasserstein distance $W_1$, this case received particular attention, see, e.g., Weed and Bach~\cite{WB19} for some references on recent progress. 

It is no coincidence that also in the related field of optimal matching, where one tries to find a bijection between two random point sets minimizing a sum of distances, the two-dimensional case is especially interesting, see for example \cite[Ch.~4]{Tal14}. There, the author calls the worst-case error of such a Monte Carlo method a discrepancy as it measures the distance between the uniform measure and its associated empirical measure. This is related to but different from the discrepancy we will study in the following Chapter~\ref{ch:iso}.

\begin{remark}
	It is a curious phenomenon that we have a threshold as above. According to Bobkov and Ledoux~\cite{BL20}, who investigated optimal matchings in one dimension, this is connected to the fact that periodic functions of smoothness $s>1/2$ can be expanded into absolutely convergent Fourier series, which is no longer true for $s=1/2$.  
\end{remark}

\addtocontents{toc}{\protect\pagebreak}
\chapter{Isotropic discrepancy of lattice point sets} \label{ch:iso}

In this final chapter we will consider the isotropic discrepancy, a quantity related to the radius of information for certain integration problems, where we restrict to equal-weight cubature algorithms using information either at specific point sets or random ones. As motivated in Chapter~\ref{ch:interlude}, the findings here complement the study of the radius of (random) standard information in Chapter~\ref{ch:sob}. More precisely, we study the isotropic discrepancy of a special type of structured point sets, called lattice point sets, and show that in dimension $d \ge 2$ they do not attain the optimal behaviour exhibited by random constructions.

After a rather general introduction, we present in Section~\ref{sec:iso-res} the results obtained together with F.~Pillichshammer in the publications \cite{PS20} and \cite{SP21}. The subsequent Sections~\ref{sec:iso-hyp} and \ref{sec:iso-lll} contain the ideas behind the proofs and Section~\ref{sec:iso-open} ends this chapter with open questions and opportunities for subsequent work.

\section{Introduction and motivation} \label{sec:iso-intro}

We feel that the motivation for the study of istropic discrepancy, which is a special kind of discrepancy related to convex sets, benefits from a more general perspective. Thus, consider a measure space $(D,\Sigma,\mu)$ with $\mu(D)=1$ and a set system $\cA\subset\Sigma$ containing $D$ as well as the empty set. Then, for any point set $\Pn=\{x_1,\ldots,x_n\}\subset D$, the corresponding equal-weight algorithm $Q_{\Pn}\colon f\mapsto \frac{1}{n}\sum_{i=1}^{n}f(x_i)$ as in \eqref{eq:qmc} incurs an error on the indicator function $\bfone_A\colon D\to\IR$ which can be bounded via
\begin{equation} \label{eq:disc-ineq}
\Big|\frac{1}{n}\sum_{i=1}^{n}\bfone_A(x_i)-\int_D \bfone_A(x)\dd\mu(x)\Big|
\le D_{\cA}(\Pn),\quad A\in\cA,
\end{equation}
where
\begin{equation} \label{eq:disc-def}
	D_{\cA}(\Pn)
:=\sup_{A\in\cA} \Big|\frac{\#(A\cap \Pn)}{n} - \mu(A)\Big|
\end{equation}
is the discrepancy of the point set $\Pn$ with respect to the system $\cA$ (and the measure $\mu$ which is omitted in the notation). 

Analytically, the discrepancy $D_{\cA}(\Pn)$ quantifies the uniform error made by the algorithm $Q_{\Pn}$ on the set of characteristic functions belonging to $\cA$, and geometrically it compares the maximal deviation between the fraction of points of $\Pn$ belonging to such a set and its $\mu$-volume. There is a whole theory related to this notion of irregularity of distribution, see, e.g., the books by Beck and Chen~\cite{BC08}, Drmota and Tichy~\cite{DT97}, and Kuipers and Niederreiter~\cite{KN74}.

In the following, we would like to construct a space of functions built from indicator functions such that the discrepancy $D_{\cA}(\Pn)$ is equal to the quality of $Q_{\Pn}$ for this space. For this purpose, we follow Pausinger and Svane~\cite{PS15} and start with the class of simple functions 
\begin{equation} \label{eq:simple}
\Sigma(\cA)
:=\Big\{\sum_{j=1}^{m}a_j \bfone_{A_j}\colon m\in\IN \text{ and }a_j\in \IR, A_j\in \cA \text{ for all }j=1,\ldots,m,\Big\}.
\end{equation}
The complexity of a simple function $f\in\Sigma(\cA)$ is captured by the sum of its coefficients, minimized over all representations, i.e.,
\[
V_{\cA}(f)
:=\inf\Big\{\sum_{j=1}^{m}|a_j|\colon f=a_0\bfone_{D}+\sum_{j=1}^{m}a_j \bfone_{A_j} \text{ with }m, a_j \text{ and } A_j \text{ as in \eqref{eq:simple}}\Big\}
\]
which is called the $\cA$-variation of $f$.  Note that the constant part does not contribute.

By direct calculation, we can extend \eqref{eq:disc-ineq} to simple functions and obtain
\begin{equation} \label{eq:kh}
\Big| \frac{1}{n}\sum_{i=1}^{n}f(x_i)-\int_D f(x)\dd\mu(x)\Big|
\le D_{\cA}(\Pn) V_{\cA}(f),\quad f\in\Sigma(\cA).
\end{equation}
It is important to note that this bound decomposes into a part depending only on the point set and another depending only on the function. For this reason one may call \eqref{eq:kh} a Koksma-Hlawka inequality in the spirit of a theorem of the same name, for which we refer to e.g.\ \cite[Ch.~2, Thm.~5.5]{KN74}. 

To extend \eqref{eq:kh} further, we use uniform limits of simple functions to find a suitable function space containing $\Sigma(\cA)$. Let $f\colon D\to\IR$ be a measurable function such that there exists a sequence of simple functions $(f_n)_{n\in\IN}\subset \Sigma(\cA)$ converging uniformly to $f$. Then define the $\cA$-variation of $f$ by
\[
V_{\cA}(f)
:=\inf\big\{\liminf_{n\to\infty} V_{\cA}(f_n)\colon f_n\in\Sigma(\cA), f_n\to f\text{ uniformly}\big\}.
\]
If we denote by $V(\cA)$ the space of all functions with finite $\cA$-variation, then it defines a seminorm on it which is zero for the constant functions.  We obtain
\begin{equation} \label{eq:disc-wce}
	\sup_{V_{\cA}(f)\le 1}
\Big| \frac{1}{n}\sum_{i=1}^{n}f(x_i)-\int_D f(x)\dd\mu(x)\Big|
=D_{\cA}(\Pn),
\end{equation}
where the upper bound follows from \eqref{eq:kh} and for the lower bound a sequence $(A_n)_{n\in\IN}\subset \cA$ with $V_{\cA}(\bfone_{A_n})\to D_{\cA}(\Pn)$ works. 

The identity \eqref{eq:disc-wce} states that the worst-case error of any equal-weight algorithm based on a point set over the functions with $\cA$-variation at most one is equal to the discrepancy with respect to $\cA$. This should be compared with \eqref{eq:lip-wce} relating the worst-case error of a cubature rule on functions with $|f|_{\lip}\le 1$ with the Wasserstein distance. Admittedly, the set $\{f\in V(\cA)\colon V_{\cA}(f)\le 1\}$ may seem less intuitive, as it is constructed by completing simple functions in the uniform norm. Thus, we will be more concrete and specialize to the case which interests us in this chapter.

Let now $D=[0,1)^d$ be equipped with the Lebesgue measure $\vol$ and consider the set system $\cA=\cK$, the class of convex sets contained in $[0,1)^d$. Then we call 
\begin{equation} \label{eq:iso-def}
D_{\cK}(\Pn)
:= \sup_{K\in\cK} \Big|\frac{\#(K\cap \Pn)}{n} - \vol(K)\Big|
\end{equation}
the isotropic discrepancy of the point set $\Pn=\{x_1,\ldots,x_n\}$. It is also denoted by $J_n(\Pn)$ instead, which is the notation used in \cite{PS20,SP21}. Its definition dates back to Hlawka~\cite{Hla64} and its name \chg{was coined by} Zaremba~\cite{Zar70} who apparently based it on the isotropy, that is the rotation invariance, of the family of convex sets. This is in \chg{contrast} to the family \chg{$\mathcal{A}$} of boxes with sides parallel to the coordinate axes used for estimating integration errors in anisotropic function spaces of dominating mixed smoothness.

The relation \eqref{eq:disc-wce} expresses the isotropic discrepancy as a worst-case error of the corresponding equal-weight or quasi-Monte Carlo rule. In this sense, it measures the quality of the information given by a point set for a specific integration problem. It can be deduced from \cite[Thm.~3.12]{PS15} that the twice continuously differentiable functions $C^2( (0,1)^d)$ are continuously embedded into the space $V_{\cK}$ if it is equipped with the norm $\|\cdot\|_{\infty}+V_{\cK}(\cdot)$. Additionally, the class $\cK$ encompasses many other set systems used for studying worst-case errors, as for example \chg{$\mathcal{A}$ which is} related to Hardy-Krause variation and worst-case errors in spaces of mixed smoothness. Therefore, it bounds the integration error of equal-weight rules in related function spaces.

Concerning further applications, Aistleitner, Brauchart and Dick~\cite{ABD12} used an area-preserving map $\Phi\colon [0,1)^2\to \IS^2$ to relate the isotropic discrepancy to the spherical cap discrepancy, where we equip $\IS^2$ with its normalized rotation invariant measure $\sigma^{(2)}$ and consider the system $\cC$ of all spherical caps
\[
{\rm Cap}(x,t)
:= \{y\in \mathbb{S}^{2}\colon \langle x, y\rangle \ge t\},\quad x\in\mathbb{S}^{2}, t\in [-1,1],
\]
that is, intersections of the sphere with \chg{half-spaces}. Through elaborate computations, they proved in \cite[Thm.~6]{ABD12} an upper bound for the spherical cap discrepancy of the mapped point set in terms of the isotropic discrepancy of the original point set $\Pn\subset [0,1)^2$, which is of the form
\begin{equation} \label{eq:iso-sphere}
	D_{\cC}\big(\Phi(\Pn)\big)
	\le 11\, D_{\cK}(\Pn).
\end{equation}
This was then applied to Fibonacci lattice points and certain digital nets, which are both classes of point sets with low discrepancy with respect to \chg{$\mathcal{A}$}. Discussing their results, they stated on p.~1001 in \cite{ABD12}: 
\begin{quote}
``Whether $(0,m,2)$-nets and/or Fibonacci lattices achieve the optimal rate of convergence for the isotropic discrepancy is an open question.''
\end{quote}
 This served as our original motivation to study the isotropic discrepancy of more general lattice point sets as introduced below. In the above question, the optimal rate is meant to be the rate of decay of the minimal isotropic discrepancy given by
\begin{equation} \label{eq:iso-min}
D_{\cK}(n)
:=\inf_{\#\Pn= n}D_{\cK}(\Pn),
\end{equation}
where the infimum is over all point sets $\Pn\subset [0,1)^d$ with $n$ points. In \cite{Sch75} Schmidt proved the lower bound
\begin{equation} \label{eq:iso-lower}
D_{\cK}(n)
\gtrsim n^{-2/(d+1)}
\end{equation}
using a now famous technique. This bound is attained up to a power of a logarithm as shown by Stute~\cite{Stu77} for $d\ge 3$ and Beck~\cite{Bec88} for $d=2$. The proofs rely on the probabilistic method and thus one does not know of an explicit point set achieving this upper bound. Interestingly enough, the proof of Stute shows that i.i.d.\ random points with density absolutely continuous to the Lebesgue measure satisfy this bound and are thus optimal up to a logarithmic factor. Note that this cannot hold for $d=2$ as the central limit theorem yields a rate of at best $n^{-1/2}$, see Chapter~\ref{ch:interlude}.

We investigate the isotropic discrepancy of lattice-point sets, defined as follows. 
\begin{definition}\label{def:lattice}
Let $d\in\IN, d\ge 2,$ and let
\[
L=\Big\{\sum_{i=1}^{d}k_i b_i\colon k_i\in\ZZ\text{ for }i=1,\dots,d\Big\}
\]
be a $d$-dimensional lattice spanned by linearly independent basis vectors $b_1,\ldots,b_d\in \IR^d$. If the lattice $L$ contains $\ZZ^d$, it is called an integration lattice and generates the lattice point set $\cP(L):=L\cap [0,1)^d$. 
\end{definition}

The Fibonacci lattice point set is a special case which can be generated by a single vector and thus is of rank one, see, e.g., \cite[Ch.~5.3]{Nie92}. To define it, let $F_1=F_2=1$ and $F_{m}=F_{m-1}+F_{m-2}$ for $m> 2$ be the Fibonacci numbers. Then the Fibonacci lattice point set with $F_n$ points is given by
\[
\mathcal{F}_n
=\Big\{\Big(\frac{i}{F_n},\Big\{\frac{iF_{n-1}}{F_n}\Big\}\Big),\quad i=0,\dots,F_{n}-1\Big\},
\]
where $\{\cdot\}\in [0,1)$ is the fractional part. For an illustration of $\mathcal{F}_8$ with $F_8=21$ points see Figure~\ref{fig:spectraltest} in Section~\ref{sec:iso-hyp} below.

Lattice point sets are used in numerical integration for constructing good equal weight rules, see, e.g., \cite[Ch.~5]{Nie92} as well as Sloan and Joe~\cite{SJ94}. In particular, the Fibonacci lattice point set achieves the best order for discrepancy with respect to \chg{$\mathcal{A}$}. It also has optimal dispersion, which measures the size of the largest empty axis-parallel box in the cube, see Breneis and Hinrichs \cite{BH20} as well as Lachmann and Wiart \cite{LW21}. Let us mention that also for the dispersion uniform random points are the best point sets known in certain (high) dimensions, see, e.g., Litvak and Livshyts \cite{LL21}.

However, the results presented in the next section will show that the special structure of lattice point sets compared to random points is a disadvantage with respect to isotropic discrepancy.

\begin{remark}
	The Voronoi cells of a lattice point set are of equal volume, at least when it is considered on the torus $\TT^d\simeq [0,1)^d$. Therefore, equal weights are optimal for the integration of (periodic) Hölder and Lipschitz functions as Proposition~\ref{pro:lip} shows. This \chg{justifies considering} only equal weights for the integration of functions from $V_{\cK}$ when using lattice point sets.
\end{remark}

\section{Discussion of results} \label{sec:iso-res}

Among the results we present here are a lower bound on the isotropic discrepancy of lattice point sets (Theorem~\ref{thm:iso-lower}), a characterization of their isotropic discrepancy in terms of the spectral test of the underlying lattice (Theorem~\ref{thm:iso-char}) and an equivalence for the radius of information of lattice point sets for $L_q$-approximation, $1\le q\le\infty$, in the Sobolev spaces $W^s_p([0,1)^d)$ (Theorem~\ref{thm:sob-lat}).

In \cite[Thm.~1]{PS20} we derived a lower bound on the isotropic discrepancy of lattice-point sets. We present a corrected version as follows, see Remark~\ref{rem:errata}.
\begin{theorem}[{\cite[Thm.~1]{PS20}}]\label{thm:iso-lower}
Let $\cP(L)$ be an $n$-element lattice point set in $[0,1)^d$. Then 
\[
D_{\cK}\big( \cP(L)\big) 
\ge \min\Big\{\frac{1}{2\sqrt{d}+1},\frac{c_d}{n^{1/d}}\Big\},
\]
where 
\[
c_d:=\frac{\sqrt{\pi}}{2\sqrt{d}+1} \left(\Gamma\left(\frac{d}{2}+1\right)\right)^{-1/d}.
\]
If additionally $\sigma(L)\leq 1/2$, then $c_d$ may be replaced by $c_d'=c \left(\Gamma\left(\frac{d}{2}+1\right)\right)^{-1/d}$ for some $c>0$ that is independent of $d$.
\end{theorem}
Stirling's formula for the Gamma function $\Gamma(x)=\int_0^{\infty}t^{x-1}{\rm e}^{-t}\dd t, x>0$, see \eqref{eq:stirling}, implies
\[
c_d^{-1}\sqrt{\frac{\pi\,{\rm e}}{2}}\frac{1}{d}\to 1 \quad \text{as }d\to\infty.
\]

Theorem~\ref{thm:iso-lower} shows that the isotropic discrepancy of lattice point sets decays slower than the minimal isotropic discrepancy \eqref{eq:iso-min} and answers the question from \cite{ABD12} mentioned near the end of the introduction to this chapter in the negative. Therefore, the upper bound of \cite[Lem.~17]{ABD12} on the spherical discrepancy of mapped Fibonacci lattice points, which is of order $n^{-1/2}$ cannot be improved. We remark that Larcher~\cite{Lar89} has shown the lower bound of order $n^{-1/d}$ already for the class of lattice point sets of rank one by using a different approach than ours. In dimension $d\ge 3$, the lower bound of Theorem~\ref{thm:iso-lower} implies together with the discussion on the minimal isotropic discrepancy that random information prevails over information given by lattice point sets. 

We obtained the lower bound in Theorem~\ref{thm:iso-lower} using a characterization of the isotropic discrepancy of $\cP(L)$ in terms of the spectral test $\sigma(L)$ of the lattice given as follows.
\begin{definition}\label{def:spectral-dual}
The spectral test of a $d$-dimensional lattice $L$ is given by
\[
\sigma(L):=\frac{1}{\min\big\{\|h\|_2\colon h \in L^{\bot}\setminus\{0\}\big\}},
\]
where the dual lattice is given by
\[
L^{\bot}:=\left\{h \in \IR^d\colon \langle h, x \rangle \in \ZZ \text{ for all } x \in L\right\}.
\]
\end{definition}

We present the corrected version of Theorem~2 in \cite{PS20}, see again Remark~\ref{rem:errata}.

\begin{theorem}[{\cite[Thm.~2]{PS20},\cite[Thm.~1]{SP21}}]\label{thm:iso-char}
Let $\cP(L)$ be an $n$-element lattice point set in $[0,1)^d$. Then 
\[
\frac{\sigma(L)}{\sqrt{d}+\sigma(L)}
\le D_{\cK}\big(\mathcal{P}(L)\big)
\leq d \, 2^{2(d+1)} \sigma(L).
\]
If $\sigma(L)\le 1/2$, then the lower bound can be replaced by $c\,\sigma(L)$,  where $c>0$ does not depend on the dimension $d$.
\end{theorem}

To give some background, let us mention that the spectral test originated as a means to evaluate the randomness of linear congruential pseudorandom number generators, see Knuth~\cite[Ch.~3.3.4]{Knu98} and \cite[Ch.~7.2]{Nie92}. It is a well known numerical quantity to assess the coarseness of (integration) lattices and admits a convenient geometric interpretation which is developed in Section~\ref{sec:iso-hyp}, see also \cite[p.~29]{SJ94}. 

\begin{example}
	Consider $n=m^d$ for some $m\in\IN$. Then the lattice point set generated by the scaled integer lattice $(1/m)\ZZ^d$ contains $n$ points. Its dual is given by $m\ZZ^d$ and thus $\sigma\big( (1/m)\ZZ^d\big)=\|(m,0,\dots)\|_2=n^{-1/d}$. \chg{Its isotropic discrepancy appears to be unknown.}
\end{example}

\begin{example}
	Let $n=m^{d-1}$ for some $m\in\IN$. Then the lattice point set generated by $L=\ZZ\times(1/m)\ZZ^{d-1}$ has $n$ points and is contained in the hyperplane $\{0\}\times\IR^{d-1}$. Its dual is $L^{\bot}=\ZZ\times m\ZZ^{d-1}$ and $\sigma(L)=\|(1,0,\dots,0)\|_2=1$. Its isotropic discrepancy is one.
\end{example}
	
If $\sigma(L)>1/2$, then the lattice point set $\cP(L)$ is concentrated on a few hyperplanes, whereas the case $\sigma(L)\le 1/2$ corresponds to a more evenly distributed $\cP(L)$.  We have the following lower bound on the spectral test of an $n$-element lattice point set, which together with Theorem~\ref{thm:iso-char} implies Theorem~\ref{thm:iso-lower}.

\begin{proposition}[{\cite[Prop.~3]{PS20}}]\label{pro:spectral-lower}
Let $\cP(L)$ be an $n$-element lattice point set in $[0,1)^d$. Then 
\[
\sigma(L) \ge \frac{\sqrt{\pi}}{2} \left(\Gamma\left(\frac{d}{2}+1\right)\right)^{-1/d} \frac{1}{n^{1/d}}.
\]
\end{proposition}
We did not find a proof of this result in the literature and thus provide a proof in Section~\ref{sec:iso-hyp} using Minkowski's theorem which is commonly used for bounding the length of the shortest vector in a lattice.

In Proposition~\ref{pro:spectral-lower} the rate $n^{-1/d}$ can be achieved by rank-one lattices when $n$ is prime, see Dick, Larcher, Pillichshammer and Wo\'zniakowski~\cite[Lem.~2]{DLP+11} who bound a figure of merit equivalent, up to constants in $d$, to the inverse of the spectral test.

\begin{remark}\label{rem:errata}
	Theorem~\ref{thm:iso-lower} and Theorem~\ref{thm:iso-char} are the corrected versions of \cite[Thm.~1 and Thm.~2]{PS20} respectively. The upper bound in the latter has been already repaired by \cite[Thm.~1]{SP21}; details are discussed in Remark~\ref{rem:erratum-upper} in Section~\ref{sec:iso-lll}. Recently, we discovered also an error in the lower bound of \cite[Thm.~2]{PS20} which affects also \cite[Thm.~1]{PS20} without the restriction of $\sigma(L)\le 1/2$, see Remark~\ref{rem:erratum-lower} in Section~\ref{sec:iso-hyp}.
\end{remark}
\begin{remark}
	The characterization in Theorem~\ref{thm:iso-char} is true for more general point sets \chg{of the form $(L+a)\cap [0,1)^d$ for some $a\in\IR^d$ which} contain $n$ points, but one has to adjust the constants. This follows from the fact that by volume comparison the volume of the fundamental domain is roughly the inverse of the number of points. 
\end{remark}

\begin{remark}
	Since, in dimension two, the minimal spherical cap discrepancy and isotropic discrepancy are of order $n^{-3/4}$ and $n^{-2/3}$ up to logarithmic factors, respectively, the bound of \cite{ABD12} stated in \eqref{eq:iso-sphere} cannot be used to prove that the mapped Fibonacci lattice point set has optimal spherical discrepancy as the numerics suggest. However, it may be used to \chg{improve} on the current rate of $n^{-1/2}$ for the best known deterministic constructions, see also Etayo~\cite{Eta21} for such a sequence of point sets apart from the mapped Fibonacci lattice point set. Note that the upper bound on the minimal spherical cap discrepancy was obtained by Beck using the Probabilistic method, more precisely, jittered sampling, see, e.g., \cite[Thm.~24D]{BC08}. See also the author's master thesis \cite{Son19} for more information on spherical discrepancy.
\end{remark}

The following application of Theorem~\ref{thm:iso-char} characterizes the radius of information given by a lattice point set with respect to the problems of approximation and integration of Sobolev functions on the cube.

\begin{theorem}[{\cite[Thm.~2]{SP21}}]\label{thm:sob-lat}
	Let $1\le p,q\le \infty$ and \chg{$s\in\IN$ as in \eqref{eq:embedding}}. For every lattice point set $\mathcal{P}(L)$ in $[0,1)^d$ it holds that
\[
r\big(W^{s}_p([0,1)^{d})\hookrightarrow L_q([0,1)^d),\mathcal{P}(L)\big)\,
 \asymp \,\sigma(L)^{s-d(1/p-1/q)_+}.
\]
Here, the asymptotic notation conceals constants independent of the integration lattice $L$.
\end{theorem}

Note that the Sobolev space $W^s_p\big([0,1)^d\big)$ is defined via restriction, see \eqref{eq:restriction}, and an analogous result holds for the space defined on $(0,1)^d$. Theorem~\ref{thm:sob-lat} is a consequence of Theorem~\ref{thm:sob-main} applied to the Sobolev space $W^s_p([0,1)^d)$ and the following equivalence, which may be interesting on its own. 
\begin{proposition}[{\cite[Prop.~8]{SP21}}]\label{pro:distspectral}
For every lattice point set $\mathcal{P}(L)$ in $[0,1)^d$ and every $0<\gamma \le \infty$, we have
\[
\sigma(L)
\lesssim_{d,\gamma} \bigl\|\dist\bigl(\cdot,\mathcal{P}(L)\bigr)\bigr\|_{L_{\gamma}([0,1)^d)}
\lesssim_d \sigma(L).
\]  
\end{proposition}

The proof and thus the proof of Theorem~\ref{thm:sob-lat} will be given in Appendix~\ref{sec:iso-proof-dist}. The distortion of sets arising from lattices has been also studied in the context of lattice quantizers, see, e.g., Conway and Sloane \cite{CS82} as well as \cite[Ch.~8]{GL00}.

Given a sequence of lattice point sets, we can interpret Theorem~\ref{thm:sob-lat} as follows. 
\begin{corollary}\label{cor:optimal-lattices}
A sequence of lattice point sets $\big(\cP(L_k)\big)_{k\in\IN}$ is asymptotically optimal for the $L_q$-approximation of Sobolev functions from $W^s_p([0,1)^d)$ if and only if the spectral test behaves optimally, i.e., 
\[
r\big(W^{s}_p([0,1)^{d})\hookrightarrow L_q([0,1)^d),\mathcal{P}(L_{k})\big)
\asymp r\big(W^{s}_p([0,1)^{d})\hookrightarrow L_q([0,1)^d),n_k\big)
\]
if and only if 
\[
\quad\sigma(L_{n_k})\asymp n_k^{-1/d}, 
\]
where $n_k:=\#\cP(L_k)$.
\end{corollary}

By Theorem~\ref{thm:iso-char} one can replace the spectral test by isotropic discrepancy in Theorem~\ref{thm:sob-lat} and Corollary~\ref{cor:optimal-lattices}. \chg{The latter shows} that lattice point sets with small spectral test may attain the optimal rates for the integration problem in Sobolev spaces but fail to attain the optimal rate for the isotropic discrepancy. Perhaps, this may be deduced from the function spaces $W^s_p([0,1)^d)$ and $V(\cK)$ themselves.

\newpage

\section{Hyperplane coverings and lower bounds} \label{sec:iso-hyp}

We will give more details on the geometric interpretation of the spectral test $\sigma(L)$ of a $d$-dimensional integration lattice $L\subset \IR^d$ and then use it to obtain lower bounds.
\begin{figure}[ht]
	\centering
	\begin{subfigure}{0.3\textwidth}
		\centering
    \begin{tikzpicture}[scale=3.5]
		\draw[black] (0,0) rectangle (1,1);
			\foreach \i in {0,...,20}
				 \fill +({\i/21},{\i*13/21-floor(\i*13/21)}) circle ({1/(50)});	
			\end{tikzpicture}
			\qquad
	\end{subfigure}
	\begin{subfigure}{0.3\textwidth}
		\centering
    \begin{tikzpicture}[scale=3.5]
	\begin{scope}    	 
    	  \clip (0,0) rectangle (1,1);
    	  	\foreach \i in {1,...,9}	
	  \draw[shorten >=-3cm, shorten <=-3cm,black] (\i*1/65,\i*8/65)--(\i*64/65,0);
	\end{scope}	
	
	\draw[black] (0,0) rectangle (1,1);
   	\foreach \i in {0,...,20}
    	 \fill +({\i/21},{\i*13/21-floor(\i*13/21)}) circle ({1/(50)});	
    	 \draw[->,black,opacity=.8,thick] (4/65+3*8/65,4*8/65-3*1/65)--(3*1/65+3*8/65,3*8/65-3*1/65);
    \end{tikzpicture}
	\end{subfigure}
	\begin{subfigure}{0.3\textwidth}
		\centering
    \begin{tikzpicture}[scale=3.5]
              
	\begin{scope}    	 
    	  \clip (0,0) rectangle (1,1);
    	  \foreach \i in {1,...,5}
    	 \draw[shorten >=-3cm, shorten <=-3cm, black] ({\i*2/21},{\i*5/21})--({\i*5/21,\i*2/21}); 
	\end{scope}	
	
	\draw[black] (0,0) rectangle (1,1);
   	\foreach \i in {0,...,20}
    	 \fill +({\i/21},{\i*13/21-floor(\i*13/21)}) circle ({1/(50)});	
	\draw[->,black, thick,opacity=.8] (3*7/42,3*7/42)--(27.5/42,27.5/42);
	
	\draw[<->, black, opacity=.8,thick] ((18.5/42,9.5/42) --(25.5/42,16.5/42);

	\node[anchor=north, black] at (25/42,14/42) {\small $\sigma(L)$};
	
    \end{tikzpicture}
	\end{subfigure}
	\\
	\vspace{.2cm}
	\hspace{1cm}
	\begin{subfigure}{.3\textwidth}
    \begin{tikzpicture}[scale=0.21]
    \begin{scope}[yshift=.9cm]
    \clip (-8.5,-8.5) rectangle (8.5,8.5);
    
    	\foreach \i in {-10,...,10}
    	  \foreach \j in {-10,...,10} 
    	   {\fill[black!30] (\i,\j) circle ({.2});}
    		  		 
      \foreach \i in {-5,...,5}
    	  \foreach \j in {-5,...,5}
    	\fill[black] ({(\i*2+\j*3)},{(-\i*5+\j*3)}) circle ({0.4});
    		\fill[white] (0,0) circle (.25);
  
    	  \end{scope}
		
		 \end{tikzpicture}
	\end{subfigure}
	\begin{subfigure}{.3\textwidth}
    \begin{tikzpicture}[scale=0.21]
    \begin{scope}[yshift=.9cm]
    \clip (-8.5,-8.5) rectangle (8.5,8.5);
    
    	\foreach \i in {-10,...,10}
    	  \foreach \j in {-10,...,10} 
    	   {\fill[black!30] (\i,\j) circle ({.2});}
    		  		 
      \foreach \i in {-5,...,5}
    	  \foreach \j in {-5,...,5}
    	\fill[black] ({(\i*2+\j*3)},{(-\i*5+\j*3)}) circle ({0.4});
    		\fill[white] (0,0) circle (.25);
  
	\fill[black!80] (-1,-8) circle (.25);
    	  \draw[->,black!80,opacity=.8, thick] (-0.1,-0.3)--(-0.9,-7.7);  
    	  \end{scope}
		 \end{tikzpicture}
		
	\end{subfigure}
	\begin{subfigure}{.3\textwidth}
    \begin{tikzpicture}[scale=0.21]
    \begin{scope}[yshift=.9cm]
    \clip (-8.5,-8.5) rectangle (8.5,8.5);
    
    	\foreach \i in {-10,...,10}
    	  \foreach \j in {-10,...,10} 
    	   {\fill[black!30] (\i,\j) circle ({.2});}
    		  		 
      \foreach \i in {-5,...,5}
    	  \foreach \j in {-5,...,5}
    	\fill[black] ({(\i*2+\j*3)},{(-\i*5+\j*3)}) circle ({0.4});
    		\fill[white] (0,0) circle (.25);
  
    	  \fill[black!70] (3,3) circle (.25);
    	  \draw[->,black,opacity=.8,thick] (0.3,0.3)--(2.7,2.7);
    	  \end{scope}
		 \end{tikzpicture}
	\end{subfigure}
	\caption{The first row depicts a Fibonacci lattice point set $\cP(L)\subset [0,1)^d$ together with two different hyperplane coverings, whose distance is marked by an arrow. \chg{Its length is} inverse to the vector in the dual lattice $L^{\bot}$, shown in the second row, which generates the hyperplane covering. The lattice in grey depicts \chg{$\ZZ^d\supset L$}.}
	\label{fig:spectraltest}
\end{figure}
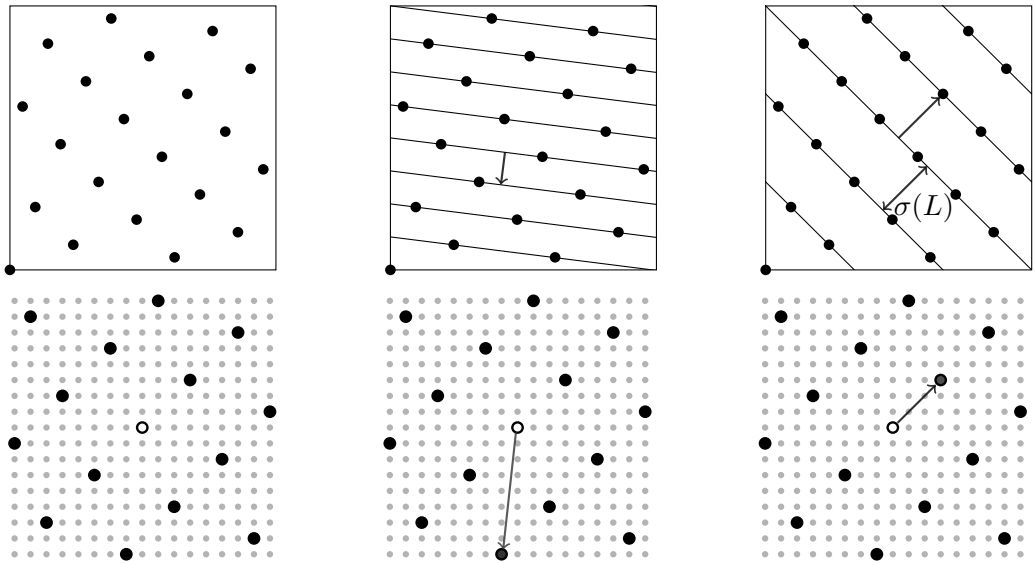

\label{loc:chstar}
Following the presentation in Hellekalek~\cite{Hel98}, we call a family $\cH$ of parallel hyperplanes a covering of $L$ if $L\subset \bigcup_{H\in\cH} H$ and every $H\in\cH$ contains some lattice point of $L$. Then the distance between any two adjacent hyperplanes in $\cH$ is equal to some number $d(\cH)\ge 0$. In fact, every primitive vector $x$ in the dual $L^{\bot}$ gives rise to a covering of $L$ by parallel hyperplanes all \chg{orthogonal} to $x$. The distance between adjacent hyperplanes in this covering is exactly $\|x\|_2^{-1}$. The spectral test is equal to the supremum $\sup_{\cH}d(\cH)$ over all coverings of $L$ by parallel hyperplanes, which is attained by the covering of hyperplanes $\cH^*=\cH^*(L)$ for which all hyperplanes are \chg{orthogonal} to the shortest vector in the dual lattice $L^{\bot}$. \chg{See Figure~\ref{fig:spectraltest} for an illustration.} Let us note that here a lattice vector is called primitive if it is non-zero and the line connecting it to the origin does not contain another lattice vector.

In the following, we will only consider the covering $\cH^*$ and use it to explain the idea behind the lower bound on the isotropic discrepancy, Theorem~\ref{thm:iso-lower}, which follows from the lower bound in Theorem~\ref{thm:iso-char} together with Proposition~\ref{pro:spectral-lower}. Detailed proofs can be found in Appendix~\ref{sec:iso-proof-lower}. The proof of the lower bound of Theorem~\ref{thm:iso-char} is divided into two cases depending on whether we suppose $\sigma(L)\le 1/2$ or not. 

In the case without any assumption on $\sigma(L)$ we will find a convex set which has no volume and contains some fraction of the points. This is achieved by intersecting $[0,1)^d$ with a hyperplane $H\in\cH^*$ containing sufficiently many points. The existence of such \chg{a hyperplane} will be a consequence of the pigeonhole principle and the fact that not too many hyperplanes of $\cH^*$ have a non-empty intersection with the unit cube as they are separated by a distance of \chg{at least} $\sigma(L)$.

\begin{remark}\label{rem:erratum-lower}
	The proof of \cite[Thm.~1]{PS20} contains a flawed estimate of the number of hyperplanes from $\cH^*$ intersecting $[0,1)^d$. There, the upper bound $\lfloor \sqrt{d}/\sigma(L)\rfloor$ was claimed but there can be in fact as many as $\lceil \sqrt{d}/\sigma(L)\rceil \le \sqrt{d}/\sigma(L)+1$. The correction performed in this thesis gives a slightly worse bound but retains the asymptotic behaviour. Here, for any $a\in\IR$, $\lceil a\rceil$ is the smallest integer not smaller than $a$. 
\end{remark}

The lower bound with the assumption $\sigma(L)\le 1/2$ relies on a recent result due to König and Rudelson~\cite{KR20} who have shown the following. 
\begin{quote}
	If an affine subspace has distance less than $1/2$ to the center of the unit cube, then the volume of their intersection is at least a positive constant only depending on the codimension of the subspace.
\end{quote}
 We apply this to the continuum of parallel hyperplanes between two adjacent hyperplanes from $\cH^*$ whose convex hull contains the center of the cube. In this way, we get a lower bound on the volume of the cube intersected with the interior of this convex hull, which contains no points of $\cP(L)$, and thus also a lower bound on the isotropic discrepancy. 

\section{Short lattice vectors and an upper bound} \label{sec:iso-lll}

For the proof of the upper bound of Theorem~\ref{thm:iso-char} we have to show that for any convex set $K\subset [0,1)^d$ the fraction of points of $\cP(L)$ contained in $K$ is not too far from its volume. For this purpose, we rely on the lattice structure to form a partition into cells of volume $1/n$ containing exactly one point. Then the discrepancy of the set $K$ depends only on cells or points near the boundary of $K$.

That kind of argument appears throughout the discrepancy literature, even though sometimes superficially (especially in high dimensions). In particular, this applies to jittered or stratified sampling where in each cell a point is uniformly distributed and which can be used to derive upper bounds on the discrepancy with respect to various set systems, see, e.g., Chapter 3 in Matou\v{s}ek's book \cite{Mat10} or \cite[Ch.~8]{BC08}. In equation (3) in the latter, an estimate on the number of intersecting cells was given, however with a hidden constant. Confer also Drmota and Tichy \cite[Sec.~2.1.2]{DT97} for another presentation. In \cite[Sec.~8]{BCC+19}, Brandolini, Chen, Colzani, Gigante and Travaglini extend this method to metric measure spaces, a key hypothesis being a bound on the measure of the neighbourhood of the boundary of the involved sets. Similar to Corollary~8.3 there, we shall derive such a bound for convex sets and give an explicit constant (see Corollary~\ref{cor:neighbourhood}). This discussion motivates a more detailed description of the proof technique of the upper bound of Theorem~\ref{thm:iso-char}; remaining proofs may be found in Appendix~\ref{sec:iso-proof-upper}. 

Consider a tiling of $\IR^d$ by disjoint translations of the following set.
\begin{definition}\label{def:fundamental}
	Let $L$ be a $d$-dimensional lattice with basis $b_1,\dots,b_d$. Then the fundamental domain or cell with respect to this basis is given by
\begin{equation} \label{eq:fundamental}
\fP:=\Big\{\sum_{i=1}^{d}\lambda_i b_i\colon 0\le \lambda_i <1\Big\}. 
\end{equation}
\end{definition}
Its definition implies that for every $x\in L$ the set $\fP_x:=\fP+x$, which is a shifted fundamental domain, contains no other point of $L$ besides $x$. Also, the sets $\fP_x$ and $\fP_y$ are disjoint for $x\neq y$ both in $L$. We have the disjoint union
$
\IR^d
=\bigcup_{x\in L}\fP_x
$.

\label{loc:det}
The volume of $\fP$ and thus any such cell is independent of the chosen basis and called the determinant $\det(L)$ of $L$. A peculiar property of $L$ being an integration lattice is that $\det(L)=1/n$, where $n=\#\cP(L)$.

For an arbitrary convex set $K\subset [0,1)^d$, its signed discrepancy 
\[
D^{\pm}(\cP(L),K)
:=\frac{\#(\cP(L)\cap K)}{n} - \vol(K)
\]
can be decomposed into the contributions made by each cell, that is,
\begin{equation} \label{eq:decomposition}
D^{\pm}(\cP(L),K)
=\frac{1}{n}\sum_{x\in L}\Big(\bfone_K(x) -n \, \vol(K\cap \fP_x)\Big).
\end{equation}

If a cell is completely contained in $K$, then it does not contribute to the sum as it contains exactly one point of $\cP(L)$ and its intersection with $K$ has exactly volume $1/n$. Also a cell not intersecting $K$ does not contribute, since then the corresponding lattice point is not in $K$ and the volume of $K$ intersected with the cell is zero. 

In order to estimate the right-hand side of \eqref{eq:decomposition}, consider the union of cells inside of $K$ and of those intersecting $K$, that is, define
\[
\Uin=\bigcup_{\substack{x\in L:\\ \fP_x\subset K}} \fP_x\quad 	\text{and}\quad
\Uout=\bigcup_{\substack{x\in L:\\ \fP_x\cap K\neq \emptyset}} \fP_x.
\]
Then we have the inclusions $\Uin\subset K \subset \Uout$. 

We get a lower bound by ignoring lattice points not belonging to $\Uin$ and thus
\[
D^{\pm}(\cP(L),K)
\ge \frac{1}{n}\sum_{x\in L\cap \Uin}\big(\bfone_K(x) - n\,\vol(\fP_x)\big)- \vol(K\setminus \Uin)
=\vol(K\setminus \Uin).
\]
Note that $\vol(\fP_x)=\frac{1}{n}$ for all $x\in L$. Similarly, we get an upper bound by adding lattice points whose cells intersect $K$ and thus
\[
D^{\pm}(\cP(L),K)
\le \frac{1}{n}\sum_{x\in L\cap \Uout}\big(\bfone_{\Uout}(x) - n\,\vol(\fP_x)\big)+ \vol(\Uout\setminus K)
=\vol(\Uout\setminus K).
\]
Taken together, they imply for the absolute value of the signed discrepancy, the discrepancy of $K$, that 
\[
D(\cP(L),K)
:=|D^{\pm}(\cP(L),K)|
\le \max\{\vol(K\setminus \Uin),\vol(\Uout\setminus K)\}.
\]
Therefore, the discrepancy is bounded in terms of the cells intersecting the boundary of $K$. In the following, we will obtain an upper bound of the volumes in the maximum by the diameter of the cell $\fP$ and bound the latter in terms of the spectral test. Taking the supremum over all convex $K\subset [0,1)^d$ then proves the upper bound of Theorem~\ref{thm:iso-char}.

The sets $K\setminus \Uin$ and $\Uout \setminus K$ are contained in $K_{\dia}^-$ and in $K_{\dia}^+$, respectively, where, for any $\varrho\in \IR$, 
\[
K_{\rho}^-:=\{x\in K\colon \dist(x,K^C)\le \rho\}\quad\text{and}\quad K_{\rho}^+:=\{x\in K^C \colon  \dist(x,K)\le \rho\}
\]
are the inner and the outer $\rho$-neighborhood, respectively. Thus, we have the bound
\[
D(\cP(L),K)
\le \max\big\{\vol\big(K_{\dia}^-\big),\vol\big(K_{\dia}^+\big)\big\},
\]
which is more effective if the diameter of the cells is small. 

We continue with the help of the following lemma.

\begin{lemma}\label{lem:neighbourhood}
For any $\rho\in [0,1]$ we have $\max\{\vol(K_{\rho}^+),\vol(K_{\rho}^-)\}\le 2^{d+3}\rho$.
\end{lemma}
The proof of this bound on the volume of a neighborhood of a convex set will be provided in Appendix~\ref{sec:iso-proof-upper}. It relies on Steiner's formula which expresses, for every $\rho> 0$, the volume of the Minkowski sum $K+\rho \IB_2=\{x\in\IR^d\colon \dist(x,K)< \rho\}$ of $K$ with a ball of radius $\rho>0$ as a polynomial in $\rho$. For the proof of Lemma~\ref{lem:neighbourhood} one needs suitable bounds on its coefficients. 

We shall apply Lemma~\ref{lem:neighbourhood} for $\rho=\dia$, where $\fP$ is the fundamental domain of a suitable basis $b_1,\ldots,b_d$ of $L$ consisting of short vectors such that the diameter is close to smallest possible. We will argue later why we can assume that $\dia \le 1$.

To find such a basis of short vectors, we use the LLL-lattice basis reduction algorithm, which was originally derived by Lenstra, Lenstra and Lov\'asz in \cite{LLL82} to factorize polynomials but then received widespread attention in many other fields. For more information we refer to Galbraith~\cite[Ch.~17]{Gal12}. The LLL-algorithm computes an LLL-reduced basis given as follows.

\begin{definition}\label{def:lll}
A basis $b_1,\ldots,b_d$ of a lattice $L$ is called LLL-reduced (with parameter $\delta=3/4$) if
\begin{enumerate}
	\item $|\mu_{i,j}|\le \frac{1}{2}$ for $1\le j\le i\le d$,
	\item $\|b_i^*\|_2^2\ge (\frac{3}{4}-\mu_{i,i-1}^2)\|b_i^*\|_2^2$ for $2\le i\le d$,
\end{enumerate}
where 
\[
b_1^*=b_1, \quad b_i^*=b_i-\sum_{j=1}^{i-1}\mu_{i,j}b_j^*\text{ for }i=2,\dots,d\quad \text{with }\mu_{i,j}:=\langle b_i, b_j^*\rangle/\|b_j^*\|_2^2
\]
is the Gram-Schmidt orthogonalization of $b_1,\ldots,b_d$. 
\end{definition}

Then we have the following lemma for the corresponding fundamental cell $\fP$ via Definition~\ref{def:fundamental}.

\begin{lemma}\label{lem:diam-bound}
The diameter of a fundamental domain $\fP$ of a $d$-dimensional lattice $L$ with respect to an LLL-reduced basis $b_1,\ldots,b_d$ as in Definition~\ref{def:lll} is bounded by
\[
\dia\le d\,2^{d-1}\sigma(L)
\]
\end{lemma}
Inserting this into Lemma~\ref{lem:neighbourhood} completes the proof of the upper bound of Theorem~\ref{thm:iso-char}. It remains to justify why it is sufficient to consider the case $\dia\le 1$. This is because in any case the isotropic discrepancy is at most one and thus we can assume that $\sigma(L) \le d^{-1} \, 2^{-2(d+1)}$, since else the upper bound is trivial. By Lemma~\ref{lem:diam-bound} this translates to
\[
\dia\le d\, 2^{d-1}\sigma(L)
\le 2^{-d-3}<1.
\]
$\hfill\square$

We note the following corollary of Lemma~\ref{lem:neighbourhood} giving a bound on the neighbourhood of the boundary of a convex set.
\begin{corollary}\label{cor:neighbourhood}
Let $K\subset [0,1)^d$ be non-empty and convex. Then, for any $\rho\in [0,1]$ we have
\[
\vol\big(\{x\in\mathbb{R}^d \colon  \dist(x,\partial K)\le \rho\}\big)\le 2^{d+4}\rho.
\]
\end{corollary}

\begin{remark}\label{rem:erratum-upper}
	As explained in \cite{SP21}, the proof of \cite[Thm.~2]{PS20} contained an erroneous version of Lemma~\ref{lem:neighbourhood} in \cite[page~5, lines~19-21]{PS20}, where the argument used in the proof of \cite[Lem.~17]{ABD12} was incorrectly extended to higher dimensions, although the asymptotic result itself remains valid with adjusted constants in $d$. 
\end{remark}

\begin{remark}\label{rem:subexp}
	The dependence of the upper bound of Lemma~\ref{lem:neighbourhood}, and thus Corollary~\ref{cor:neighbourhood}, on the dimension $d$ can be improved to be sub-exponential in $d$. Essentially, it is of order of magnitude ${\rm e}^{c d^{2/3}}$ for $c=\frac{3}{2}(2\pi)^{1/3}$. A proof is given in Appendix~\ref{sec:iso-proof-upper}.
\end{remark}

\section{Open questions}
\label{sec:iso-open}

\begin{question}
	The first question arises canonically.
\begin{center}
	Can the unit cube be replaced by a more general set in Theorem~\ref{thm:iso-char}?
\end{center}
Thinking of the geometry behind, it appears that the cube $[0,1)^d$ can be replaced by any convex set. It would be interesting how \cite[Thm.~1.1]{KR20} may be extended to more general convex sets. This is somewhat related to the famous hyperplane conjecture, see, e.g., the book of Brazitikos, Giannopoulos, Valettas and Vritsiou~\cite{BGV+14}.
\end{question}

\bigskip

\begin{question}
	The second question is about the constants involved.
\begin{center}
	What is the correct dependence on $d$ in Theorem~\ref{thm:iso-char}?
\end{center}
The exponential upper bound is essentially due to volume arguments involved and we do not know if it can be improved or other methods have to be employed.
\end{question}

\bigskip

\begin{question}
In the introduction to this chapter we related the isotropic discrepancy to a worst-case error in a function space of functions of bounded $\cK$-variation denoted by $V(\cK)$ and into which $C^2( (0,1)^d)$ is embedded. This is already a partial answer to the following question.
\begin{center}
	Can the function space $(V(\cK),\|\cdot\|_{V(\cK)})$ be characterized in terms of smoothness?
\end{center}
Note that $V(\cK)$ contains indicator functions of convex sets. In general, relating discrepancies to worst-case errors of cubature rules is of much interest and may lead to a better understanding of the best possible behaviour of $n$-point sets. For $L_2$-averaged discrepancy there is a connection to reproducing kernel Hilbert spaces, see, e.g., \cite[Ch.~9]{NW10}. On the sphere, this becomes a beautiful connection to a sum of distances, which is known from \cite{Sto73} as Stolarsky's invariance principle, see also Brauchart and Dick~\cite{BD13}.
\end{question}

\chapter*{List of symbols}\label{ch:los}
\fancyhead[CO]{\nouppercase{\textsc{List of symbols}}}
\fancyhead[CE]{\nouppercase{\textsc{List of symbols}}}
\addcontentsline{toc}{chapter}{\protect\numberline{}List of symbols}
We list here frequently occuring notation. For basic notation see the end of Chapter~\ref{ch:intro}.

\subsection*{Sets}
\begin{tabular}{ >{\centering\arraybackslash} p{.23\textwidth}p{.75\textwidth}}
	$\overline{\Omega},\partial\Omega$ & closure/boundary of a set $\Omega\subset\IR^d$\\
	$\diam(\Omega)$ & diameter \chg{of $\Omega$}, equal to $\sup_{x,y\in\Omega}\|x-y\|_2$\\
	$\dist(x,A)$ & distance of $x$ to $A$, equal to $\inf_{y\in A}\|x-y\|_2$\\
	$\vol(A)$ & $d$-dimensional Lebesgue measure of $A$\\
	$A+B$ & Minkowski sum of sets $A,B\subset\IR^d$, see Def.~\ref{def:minkowski} (p.~\pageref{def:minkowski})\\
	$\bfone_A$ & indicator function of a set $A$\\
	$f|_A$ & restriction of a function $f$ to $A$\\
\end{tabular}
\subsection*{Linear and standard information}
\begin{tabular}{ >{\centering\arraybackslash} p{.23\textwidth}p{.75\textwidth}}
	$e(S\colon F\to G, A)$ & worst-case error of $A$, see eq.~\eqref{eq:wce} (p.~\pageref{eq:wce})\\
	$r(S\colon F\to G, N_n)$ & radius of the information map $N_n$, see Def.~\ref{def:radius} (p.~\pageref{def:radius})\\
	$r(S\colon F\to G, \Lambda,n)$ & $n$-th minimal radius of information, see Def.~\ref{def:minimal-radius} (p.~\pageref{def:minimal-radius})\\
	$\lall=F'$ & \chg{space of continuous linear} functionals (p.~\pageref{loc:lall})\\
	$c_n(S\colon F\to G)$ & $n$-th Gelfand number of $S$, see Def.~\ref{def:gelfand-number} (p.~\pageref{def:gelfand-number}) \\
	$c_n(K)$ & $n$-th Gelfand width of $K$, see eq.~\eqref{eq:gelfand-width} (p.~\pageref{eq:gelfand-width}) \\
	$G_{n,m}$ & Gaussian random matrix, see Def.~\ref{def:lin-ran} (p.~\pageref{def:lin-ran}) \\
	$\mathcal{G}_{m,m-n}$ & Grassmannian (p.~\pageref{def:lin-ran}) \\
	$\enran$ & a random subspace of codimension $n$, see Def.~\ref{def:lin-ran} (p.~\pageref{def:lin-ran}) \\
	$\delta_x$ & Delta functional with $\delta_x(f)=f(x)$ (p.~\pageref{eq:info-std})\\
	$\lstd$ & function evaluations (p.~\pageref{eq:info-std}) \\
	$\pnran$ & a random point set with $n$ points, see Def.~\ref{def:std-ran} (p.~\pageref{def:std-ran}) \\
	$\intmu,{\rm INT}$ & integration functional (p.~\pageref{loc:int}) \\
\end{tabular}
\subsection*{Function spaces}
\begin{tabular}{ >{\centering\arraybackslash} p{.23\textwidth}p{.75\textwidth}}
	$F\hookrightarrow G$ & continuous embedding $\id\colon F\to G$ (p.~\pageref{loc:embedding})\\
	$C_b(D)$ & bounded continuous functions on $D$ (p.~\pageref{loc:cb})\\
	$L_q(D,\mu)$ & $q$-integrable functions on $(D,\mu)$ (p.~\pageref{eq:lqnorm})\\
	$A^{\alpha}_q(\TT^d)$ & function space defined via Fourier coefficients, see Def.~\ref{def:aaq} (p.~\pageref{def:aaq})\\
	$W^s_p(D)$ & Sobolev functions on $D$, see Def.~\ref{def:sobolev} (p.~\pageref{def:sobolev})\\
	$C^s(D)$ & Hölder continuous functions on $D$ (p.~\pageref{eq:semi-norm-hoelder})\\
	${\rm Lip}(D)$ & Lipschitz continuous functions on $D$ (p.~\pageref{eq:lip-wce})\\
	$B^s_{p\tau}(D)$ & Besov space on $D$, see Appendix~\ref{sec:sob-proof-ext} (p.~\pageref{loc:bandf})\\
	$F^s_{p\tau}(D)$ & Triebel-Lizorkin space on $D$, see Appendix~\ref{sec:sob-proof-ext} (p.~\pageref{loc:bandf})\\
\end{tabular}

\subsection*{Notation for Chapter 3}
\begin{tabular}{ >{\centering\arraybackslash} p{.23\textwidth}p{.75\textwidth}}
	$\rad(K,E)$ & radius of the intersection $K\cap E$, see eq.~\eqref{eq:radius-def} (p.~\pageref{eq:radius-def}) \\
	$\crn(K)$ & $n$-th random Gelfand width of $K$, see eq.~\eqref{eq:crn-def} (p.~\pageref{eq:crn-def}) \\
	$\elz$ & ellipsoid with semiaxes given by $\sigma$, see eq.~\eqref{eq:ellipsoid} (p.~\pageref{eq:ellipsoid}) \\
	$\elp$ & $\ell_p$-ellipsoid with semiaxes given by $\sigma$, see Def.~\ref{def:elp} (p.~\pageref{def:elp}) \\
	$D_{\sigma}$ & diagonal operator (p.~\pageref{fig:ellipsoids}) \\
	$\|\cdot\|_{p,\sigma}$ & (quasi-)norm with unit ball $\elp$, see eq.~\eqref{eq:psigma-norm} (p.~\pageref{eq:psigma-norm}) \\
	$\|\cdot\|_{p,q}$ & Lorentz (quasi-)norm, see Def.~\ref{def:lorentz}  (p.~\pageref{def:lorentz}) \\
	$\decay\big(\rad(\elp,n)\big)$ & rate of polynomial decay of minimal radii  (p.~\pageref{loc:decay-det}) \\
	$\decay\big(\rad(\elp,\enran)\big)$ & rate of polynomial decay of the random radii (p.~\pageref{loc:decay-ran}) \\
	$M^*(K)$ & mean width of $K$, see eq.~\eqref{eq:mean-width} (p.~\pageref{eq:mean-width}) \\
	$\Delta_p$ & $\ell_p$-minimization, see eq.~\eqref{eq:ellp-min} (p.~\pageref{eq:ellp-min}) \\
	$\delta_s(N_n)$ & restricted isometry constant of order $s$ of $N_n$ (p.~\pageref{eq:rip}) \\
\end{tabular}
\subsection*{Notation for Chapter 4}
\begin{tabular}{ >{\centering\arraybackslash} p{.23\textwidth}p{.75\textwidth}}
	$h_{\Pn,\Omega}$ & covering radius of $\Pn$ in $\Omega$, see eq.~\eqref{eq:covering-rad} (p.~\pageref{eq:covering-rad}) \\
	$C(x,\xi,r,\theta)$ & cone, see Definition~\ref{def:cone-condition} (p.~\pageref{def:cone-condition}) \\
	$B(x,\varrho),Q(x,\varrho)$ & ball/cube with center $x$ and radius $\varrho$, see eq.~\eqref{eq:cubes-balls} (p.~\pageref{eq:cubes-balls}) \\
	$Q_{\Pn}(x)$ &  good cube centered at $x$ with radius $r_{\Pn}(x)$, see Def.~\ref{def:small-cube} (p.~\pageref{def:small-cube}) \\
	$D_{\mu,\Pn,\gamma}$ &  distortion, see eq.~\ref{eq:distortion} (p.~\pageref{eq:distortion}) \\
	$T_{\Pn}$ & optimal quantizer, see eq.~\ref{eq:opt-quant} (p.~\pageref{eq:opt-quant}) \\
	$C(x_i,\Pn)$ & Voronoi cell of $x_i\in \Pn$, see eq.~\ref{eq:opt-quant} (p.~\pageref{eq:opt-quant}) \\
\end{tabular}
\subsection*{Notation for Chapter 5}
\begin{tabular}{ >{\centering\arraybackslash} p{.23\textwidth}p{.77\textwidth}}
	$\nu_{\Pn,a}$ & discrete measure on $\Pn$ with weights $a$, see eq.~\eqref{eq:discrete} (p.~\pageref{eq:discrete}) \\
	$Q_{\Pn,a}$ & cubature rule with points $\Pn$ and weights $a$, see eq.~\eqref{eq:cubature} (p.~\pageref{eq:cubature}) \\
	$Q_{\Pn,\mu}$ & optimal cubature rule for $\lipd$, see eq.~\eqref{eq:opt-cubature} (p.~\pageref{eq:opt-cubature})  \\
	$Q_{\Pn}$ & equal-weight cubature rule, see eq.~\eqref{eq:qmc} (p.~\pageref{eq:qmc}) \\
\end{tabular}
\subsection*{Notation for Chapter 6}
\begin{tabular}{>{\centering\arraybackslash}p{.23\textwidth}p{.77\textwidth}}
	$D_{\cA}(\Pn)$ & discrepancy of $\Pn$ with respect to $\cA$, see eq.~\eqref{eq:disc-def} (p.~\pageref{eq:disc-def}) \\
	$\cK$ & class of convex sets contained in $[0,1)^d$ (p.~\pageref{eq:iso-def}) \\
	\chg{$\mathcal{A}$} & class of axis-parallel boxes contained in $[0,1)^d$ (p.~\pageref{eq:iso-def}) \\
	$D_{\cK}(\Pn)$ & isotropic discrepancy of $\Pn$, see eq.~\eqref{eq:iso-def} (p.~\pageref{eq:iso-def}) \\
	$V(\cK)$ & functions with finite $\cK$-variation (p.~\pageref{eq:iso-def}) \\
	$D_{\cK}(n)$ & $n$-th minimal isotropic discrepancy, see eq.~\eqref{eq:iso-min} (p.~\pageref{eq:iso-min}) \\
	$\cP(L)$ & lattice point set generated by $L$, see Def.~\ref{def:lattice} (p.~\pageref{def:lattice}) \\
	$\sigma(L)$ & spectral test of $L$, see Def.~\ref{def:spectral-dual} (p.~\pageref{def:spectral-dual}) \\
	$L^{\bot}$ & dual lattice of $L$, see Def.~\ref{def:spectral-dual} (p.~\pageref{def:spectral-dual}) \\
	$\cH^*$ & hyperplane covering of maximal distance (p.~\pageref{loc:chstar}) \\
	$\fP,\fP_x$ & (translated) fundamental cell, see Def.~\ref{def:fundamental} (p.~\pageref{def:fundamental}) \\
	$\det(L)$ & determinant of $L$ or volume of $\fP$ (p.~\pageref{loc:det}) \\
\end{tabular}

\cleardoublepage
\newpage
\phantomsection
\fancyhead[CO]{\nouppercase{\textsc{\rightmark}}}
\fancyhead[CE]{\nouppercase{\textsc{\leftmark}}}
\addcontentsline{toc}{chapter}{\protect\numberline{}Bibliography}
\bibliography{thesis_lit.bib}

\newcommand{\etalchar}[1]{$^{#1}$}
\begin{thebibliography}{HNWW01}

\bibitem[ABD12]{ABD12}
C.~Aistleitner, J.~S. Brauchart, and J.~Dick.
\newblock Point sets on the sphere {$\Bbb{S}^2$} with small spherical cap
  discrepancy.
\newblock {\em Discrete Comput. Geom.}, 48(4):990--1024, 2012.

\bibitem[AF03]{AF03}
R.~A. Adams and J.~J.~F. Fournier.
\newblock {\em Sobolev spaces}.
\newblock Elsevier/Academic Press, Amsterdam, second edition, 2003.

\bibitem[AGM15]{AGM15}
S.~{Artstein-Avidan}, A.~Giannopoulos, and V.~D. Milman.
\newblock {\em Asymptotic geometric analysis. {P}art {I}}.
\newblock American Mathematical Society, Providence, RI, 2015.

\bibitem[ALdST07]{ALdST07}
R.~Arcang\'{e}li, {M.\,C.} L\'{o}pez~de Silanes, and {J.\,J.} Torrens.
\newblock An extension of a bound for functions in {S}obolev spaces, with
  applications to {$(m,s)$}-spline interpolation and smoothing.
\newblock {\em Numer. Math.}, 107(2):181--211, 2007.

\bibitem[AS16]{AS16}
N.~Alon and J.~H. Spencer.
\newblock {\em The probabilistic method}.
\newblock John Wiley \& Sons, Inc., Hoboken, NJ, fourth edition, 2016.

\bibitem[BC08]{BC08}
J.~Beck and W.~W.~L. Chen.
\newblock {\em Irregularities of distribution}.
\newblock Cambridge University Press, Cambridge, 2008.

\bibitem[BCC{\etalchar{+}}19]{BCC+19}
L.~Brandolini, W.~W.~L. Chen, L.~Colzani, G.~Gigante, and G.~Travaglini.
\newblock Discrepancy and numerical integration on metric measure spaces.
\newblock {\em J. Geom. Anal.}, 29(1):328--369, 2019.

\bibitem[BD13]{BD13}
J.~S. Brauchart and J.~Dick.
\newblock A simple proof of {S}tolarsky's invariance principle.
\newblock {\em Proc. Amer. Math. Soc.}, 141(6):2085--2096, 2013.

\bibitem[BDS{\etalchar{+}}15]{BDS+15}
{J.\,S.} Brauchart, J.~Dick, {E.\,B.} Saff, {I.\,H.} Sloan, {Y.\,G.} Wang, and
  {R.\,S.} Womersley.
\newblock Covering of spheres by spherical caps and worst-case error for equal
  weight cubature in {S}obolev spaces.
\newblock {\em J. Math. Anal. Appl.}, 431(2):782--811, 2015.

\bibitem[Bec88]{Bec88}
J.~Beck.
\newblock On the discrepancy of convex plane sets.
\newblock {\em Monatsh. Math.}, 105(2):91--106, 1988.

\bibitem[BG15]{BG15}
J.~S. Brauchart and P.~J. Grabner.
\newblock {Distributing many points on spheres: Minimal energy and designs}.
\newblock {\em J. Complexity}, 31(3):293--326, 2015.

\bibitem[BGVV14]{BGV+14}
S.~Brazitikos, A.~Giannopoulos, P.~Valettas, and B.-H. Vritsiou.
\newblock {\em Geometry of isotropic convex bodies}.
\newblock American Mathematical Society, Providence, RI, 2014.

\bibitem[BH20]{BH20}
S.~Breneis and A.~Hinrichs.
\newblock {Fibonacci lattices have minimal dispersion on the two-dimensional
  torus}.
\newblock In D.~Bilyk, J.~Dick, and F.~Pillichshammer, editors, {\em
  Discrepancy Theory}, pages 117--132. De Gruyter, Berlin/Boston, 2020.

\bibitem[BKP{\etalchar{+}}22]{BKP+20}
A.~Baci, Z.~Kabluchko, J.~Prochno, M.~Sonnleitner, and C.~Thäle.
\newblock Limit theorems for random points in a simplex.
\newblock {\em J. Appl. Probab.}, pages 1--17, 2022.

\bibitem[BL20]{BL20}
S.~G. Bobkov and M.~Ledoux.
\newblock Transport inequalities on {E}uclidean spaces for non-{E}uclidean
  metrics.
\newblock {\em J. Fourier Anal. Appl.}, 26(4):Paper No. 60, 27, 2020.

\bibitem[BT04]{BT04}
A.~Berlinet and C.~{Thomas-Agnan}.
\newblock {\em Reproducing kernel {H}ilbert spaces in probability and
  statistics}.
\newblock Kluwer Academic Publishers, Boston, MA, 2004.

\bibitem[Buc99]{Buc99}
N.~Buchmann.
\newblock {\em Fehlerabsch{\"a}tzungen von N{\"a}herungsl{\"o}sungen
  unendlicher Gleichungssysteme durch Gelfandzahlen von
  Tensorproduktoperatoren}.
\newblock PhD thesis, Carl-von-Ossietzky-Universit{\"a}t Oldenburg, 1999.

\bibitem[Car81]{Car81}
B.~Carl.
\newblock Entropy numbers of diagonal operators with an application to
  eigenvalue problems.
\newblock {\em J. Approx. Theory}, 32(2):135--150, 1981.

\bibitem[Cia78]{Cia78}
{P.\,G.} Ciarlet.
\newblock {\em {The finite element method for elliptic problems}}.
\newblock Cambridge University Press, Amsterdam, North-Holland, 1978.

\bibitem[CM17]{CM17}
A.~Cohen and G.~Migliorati.
\newblock Optimal weighted least-squares methods.
\newblock {\em SMAI J. Comput. Math.}, 3:181--203, 2017.

\bibitem[Coh04]{Coh04}
P.~Cohort.
\newblock Limit theorems for random normalized distortion.
\newblock {\em Ann. Appl. Probab.}, 14(1):118--143, 2004.

\bibitem[CRT06]{CRT06}
E.~J. Cand\`es, J.~Romberg, and T.~Tao.
\newblock Robust uncertainty principles: exact signal reconstruction from
  highly incomplete frequency information.
\newblock {\em IEEE Trans. Inform. Theory}, 52(2):489--509, 2006.

\bibitem[CS82]{CS82}
J.~H. Conway and N.~J.~A. Sloane.
\newblock Vorono\u{\i} regions of lattices, second moments of polytopes, and
  quantization.
\newblock {\em IEEE Trans. Inform. Theory}, 28(2):211--226, 1982.

\bibitem[DeV98]{DeV98}
R.~A. DeVore.
\newblock Nonlinear approximation.
\newblock In {\em Acta numerica, 1998}, volume~7 of {\em Acta Numer.}, pages
  51--150. Cambridge Univ. Press, Cambridge, 1998.

\bibitem[dG75]{DeG75}
M.~de~Guzm\'{a}n.
\newblock {\em Differentiation of integrals in {$R^{n}$}}.
\newblock Springer-Verlag, Berlin-New York, 1975.

\bibitem[DKU22]{DKU22}
M.~Dolbeault, D.~Krieg, and M.~Ullrich.
\newblock A sharp upper bound for sampling numbers in $\mathrm{L}_2$.
\newblock {\em arXiv e-prints}, arXiv:2204.12621 [math.NA], 2022.

\bibitem[DL04]{DL04}
S.~Dekel and D.~Leviatan.
\newblock Whitney estimates for convex domains with applications to
  multivariate piecewise polynomial approximation.
\newblock {\em Found.\ Comput.\ Math.}, 4(4):345--368, 2004.

\bibitem[DLPW11]{DLP+11}
J.~Dick, G.~Larcher, F.~Pillichshammer, and H.~Wo\'{z}niakowski.
\newblock Exponential convergence and tractability of multivariate integration
  for {K}orobov spaces.
\newblock {\em Math. Comp.}, 80(274):905--930, 2011.

\bibitem[Don06]{Don06b}
D.~L. Donoho.
\newblock Compressed sensing.
\newblock {\em IEEE Trans. Inform. Theory}, 52(4):1289--1306, 2006.

\bibitem[DS80]{DS80}
T.~Dupont and R.~Scott.
\newblock {Polynomial Approximation of Functions in Sobolev Spaces}.
\newblock {\em Math. Comp.}, 34(150):441--463, 1980.

\bibitem[DS93]{DS93}
R.~A. DeVore and R.~C. Sharpley.
\newblock Besov spaces on domains in {${\bf R}^d$}.
\newblock {\em Trans. Amer. Math. Soc.}, 335(2):843--864, 1993.

\bibitem[DT97]{DT97}
M.~Drmota and R.~F. Tichy.
\newblock {\em Sequences, discrepancies and applications}.
\newblock Springer-Verlag, Berlin, 1997.

\bibitem[Duc78]{Duc78}
J.~Duchon.
\newblock Sur l'erreur d'interpolation des fonctions de plusieurs variables par
  les {$D^{m}$}-splines.
\newblock {\em RAIRO Anal. Num\'{e}r.}, 12(4):325--334, vi, 1978.

\bibitem[EGO19]{EGO19}
M.~Ehler, M.~Graef, and {C.\,J.} Oates.
\newblock {Optimal Monte Carlo integration on closed manifolds}.
\newblock {\em Stat. Comput.}, 29:1203--1214, 2019.

\bibitem[ET96]{ET96}
{D.\,E.} Edmunds and H.~Triebel.
\newblock {\em Function spaces, entropy numbers, differential operators}.
\newblock Cambridge University Press, Cambridge, 1996.

\bibitem[Eta21]{Eta21}
U.~Etayo.
\newblock Spherical {C}ap {D}iscrepancy of the {D}iamond {E}nsemble.
\newblock {\em Discrete Comput. Geom.}, 66(4):1218--1238, 2021.

\bibitem[Eva10]{Eva10}
L.~C. Evans.
\newblock {\em Partial differential equations}.
\newblock American Mathematical Society, Providence, RI, second edition, 2010.

\bibitem[FG15]{FG15}
N.~Fournier and A.~Guillin.
\newblock On the rate of convergence in {W}asserstein distance of the empirical
  measure.
\newblock {\em Probab. Theory Related Fields}, 162(3-4):707--738, 2015.

\bibitem[FPRU10]{FPR+10}
S.~Foucart, A.~Pajor, H.~Rauhut, and T.~Ullrich.
\newblock The {G}elfand widths of {$\ell_p$}-balls for {$0<p\leq 1$}.
\newblock {\em J. Complexity}, 26(6):629--640, 2010.

\bibitem[FR13]{FR13}
S.~Foucart and H.~Rauhut.
\newblock {\em A mathematical introduction to compressive sensing}.
\newblock Birkh\"{a}user/Springer, New York, 2013.

\bibitem[Gal12]{Gal12}
S.~D. Galbraith.
\newblock {\em Mathematics of public key cryptography}.
\newblock Cambridge University Press, Cambridge, 2012.

\bibitem[GG84]{GG84}
A.~Y. Garnaev and E.~D. Gluskin.
\newblock The widths of a {E}uclidean ball.
\newblock {\em Soviet Math.\,Dokl.}, 30:200--204, 1984.

\bibitem[GHLP17]{GHL+17}
O.~Gu\'{e}don, A.~Hinrichs, A.~E. Litvak, and Joscha Prochno.
\newblock On the expectation of operator norms of random matrices.
\newblock In {\em Geometric aspects of functional analysis}, pages 151--162.
  Springer, Cham, 2017.

\bibitem[GL00]{GL00}
S.~Graf and H.~Luschgy.
\newblock {\em Foundations of quantization for probability distributions}.
\newblock Springer-Verlag, Berlin, 2000.

\bibitem[GLMP07]{GLM+07}
Y.~Gordon, A.~E. Litvak, S.~Mendelson, and A.~Pajor.
\newblock Gaussian averages of interpolated bodies and applications to
  approximate reconstruction.
\newblock {\em J. Approx. Theory}, 149(1):59--73, 2007.

\bibitem[GLSW06]{GLS+06}
Y.~Gordon, A.~E. Litvak, C.~Sch\"{u}tt, and E.~Werner.
\newblock On the minimum of several random variables.
\newblock {\em Proc. Amer. Math. Soc.}, 134(12):3665--3675, 2006.

\bibitem[Glu81]{Glu81}
E.~D. Gluskin.
\newblock On some finite-dimensional problems of width theory.
\newblock {\em Physis---Riv. Internaz. Storia Sci.}, 23(2):5--10, 124, 1981.

\bibitem[Glu84]{Glu84}
E.~D. Gluskin.
\newblock Norms of random matrices and diameters of finite-dimensional sets.
\newblock {\em Math.\,USSR Sbornik}, 48(1):173--182, 1984.

\bibitem[GM97]{GM97}
A.~A. Giannopoulos and V.~D. Milman.
\newblock On the diameter of proportional sections of a symmetric convex body.
\newblock {\em Internat. Math. Res. Notices}, (1):5--19, 1997.

\bibitem[GM98]{GM98}
A.~A. Giannopoulos and V.~D. Milman.
\newblock Mean width and diameter of proportional sections of a symmetric
  convex body.
\newblock {\em J. Reine Angew. Math.}, 497:113--139, 1998.

\bibitem[GMT05]{GMT05}
A.~Giannopoulos, V.~D. Milman, and A.~Tsolomitis.
\newblock Asymptotic formulas for the diameter of sections of symmetric convex
  bodies.
\newblock {\em J. Funct. Anal.}, 223(1):86--108, 2005.

\bibitem[Gor88]{Gor88}
Y.~Gordon.
\newblock On {M}ilman's inequality and random subspaces which escape through a
  mesh in {${\bf R}^n$}.
\newblock In {\em Geometric aspects of functional analysis (1986/87)}, pages
  84--106. Springer, Berlin, 1988.

\bibitem[Gri85]{Gri85}
P.~Grisvard.
\newblock {\em Elliptic problems in nonsmooth domains}.
\newblock Pitman (Advanced Publishing Program), Boston, MA, 1985.

\bibitem[Gru04]{Gru04}
P.~M. Gruber.
\newblock Optimum quantization and its applications.
\newblock {\em Adv. Math.}, 186(2):456--497, 2004.

\bibitem[Had57]{Had57}
H.~Hadwiger.
\newblock {\em Vorlesungen \"{u}ber {I}nhalt, {O}berfl\"{a}che und
  {I}soperimetrie}.
\newblock Springer-Verlag, Berlin-G\"{o}ttingen-Heidelberg, 1957.

\bibitem[Hei94]{Hei94}
S.~Heinrich.
\newblock Random approximation in numerical analysis.
\newblock In K.~D. Bierstedt and et~al., editors, {\em Functional Analysis},
  pages 123--171. Dekker, New York, 1994.

\bibitem[Hel98]{Hel98}
P.~Hellekalek.
\newblock On the assessment of random and quasi-random point sets.
\newblock In {\em Random and quasi-random point sets}, pages 49--108. Springer,
  New York, 1998.

\bibitem[HKN{\etalchar{+}}20]{HKN+20}
A.~Hinrichs, D.~Krieg, E.~Novak, J.~Prochno, and M.~Ullrich.
\newblock {On the power of random information}.
\newblock In F.~J. Hickernell and P.~Kritzer, editors, {\em Multivariate
  Algorithms and Information-Based Complexity}, pages 43--64. De Gruyter,
  Berlin/Boston, 2020.

\bibitem[HKN{\etalchar{+}}21]{HKN+21}
A.~Hinrichs, D.~Krieg, E.~Novak, J.~Prochno, and M.~Ullrich.
\newblock Random sections of ellipsoids and the power of random information.
\newblock {\em Trans. Amer. Math. Soc.}, 374(12):8691--8713, 2021.

\bibitem[HKNV21]{HKN+21b}
A.~Hinrichs, D.~Krieg, E.~Novak, and J.~Vyb\'iral.
\newblock {Lower bounds for integration and recovery in $L_2$}.
\newblock {\em arXiv e-prints}, arXiv:2108.11853 [math.NA], 2021.

\bibitem[HKV16]{HKV16}
A.~Hinrichs, A.~Kolleck, and J.~Vyb\'{\i}ral.
\newblock Carl's inequality for quasi-{B}anach spaces.
\newblock {\em J. Funct. Anal.}, 271(8):2293--2307, 2016.

\bibitem[Hla64]{Hla64}
E.~Hlawka.
\newblock Discrepancy and uniform distribution of sequences.
\newblock {\em Compositio Math.}, 16:83--91 (1964), 1964.

\bibitem[HNWW01]{HNW+01}
S.~Heinrich, E.~Novak, G.~W. Wasilkowski, and H.~Wo\'{z}niakowski.
\newblock The inverse of the star-discrepancy depends linearly on the
  dimension.
\newblock {\em Acta Arith.}, 96(3):279--302, 2001.

\bibitem[HPS21]{HPS21}
A.~Hinrichs, J.~Prochno, and M.~Sonnleitner.
\newblock {Random sections of $\ell_p$-ellipsoids, optimal recovery and Gelfand
  numbers of diagonal operators}.
\newblock {\em arXiv e-prints}, arXiv:2109.14504 [math.FA], 2021.

\bibitem[HW79]{HW79}
G.~H. Hardy and E.~M. Wright.
\newblock {\em An introduction to the theory of numbers}.
\newblock The Clarendon Press, Oxford University Press, New York, fifth
  edition, 1979.

\bibitem[Ism74]{Ism74}
R.~S. Ismagilov.
\newblock Diameters of sets in normed linear spaces and approximation of
  functions by trigonometric polynomials.
\newblock {\em Uspekhi Mat.\,Nauk}, 29(3):161--178, 1974.

\bibitem[JP22]{JP21}
M.~Juhos and J.~Prochno.
\newblock Spectral flatness and the volume of intersections of $p$-ellipsoids.
\newblock {\em J. Complexity}, 70:101617, 2022.

\bibitem[Kai97]{Kai97}
P.~C. Kainen.
\newblock {\em Utilizing Geometric Anomalies of High Dimension: When Complexity
  Makes Computation Easier}, pages 283--294.
\newblock Birkh{\"a}user Boston, Boston, MA, 1997.

\bibitem[Kas77]{Kas77}
B.~S. Kashin.
\newblock Widths of some finite-dimensional sets and classes of smooth
  functions.
\newblock {\em Izv.\,Akad.\,Nauk SSSR Ser. Mat.}, 41:334--351, 1977.

\bibitem[Klo20]{Klo20}
B.~R. Kloeckner.
\newblock Empirical measures: regularity is a counter-curse to dimensionality.
\newblock {\em ESAIM Probab. Stat.}, 24:408--434, 2020.

\bibitem[KN74]{KN74}
L.~Kuipers and H.~Niederreiter.
\newblock {\em Uniform distribution of sequences}.
\newblock John Wiley \& Sons, New York-London-Sydney, 1974.

\bibitem[KNS22]{KNS21}
D.~Krieg, E.~Novak, and M.~Sonnleitner.
\newblock Recovery of {S}obolev functions restricted to iid sampling.
\newblock {\em Math. Comp.}, 91(338):2715--2738, 2022.

\bibitem[Knu98]{Knu98}
D.~E. Knuth.
\newblock {\em {The art of computer programming. {V}ol. 2: Seminumerical
  algorithms}}.
\newblock Addison-Wesley, Reading, MA, third edition, 1998.

\bibitem[K{\"{o}}n86]{Koe86}
H.~K{\"{o}}nig.
\newblock {\em Eigenvalue distribution of compact operators}.
\newblock Birkh\"{a}user Verlag, Basel, 1986.

\bibitem[KR20]{KR20}
H.~K{\"{o}}nig and M.~Rudelson.
\newblock On the volume of non-central sections of a cube.
\newblock {\em Adv. Math.}, 360:106929, 30, 2020.

\bibitem[KS20]{KS20}
D.~Krieg and M.~Sonnleitner.
\newblock Random points are optimal for the approximation of {S}obolev
  functions.
\newblock {\em arXiv e-prints}, arXiv:2009.11275 [math.NA], 2020.

\bibitem[KS21]{KS21}
D.~Krieg and M.~Sonnleitner.
\newblock Function recovery on manifolds using scattered data.
\newblock {\em arXiv e-prints}, arXiv:2109.04106 [math.NA], 2021.

\bibitem[KU21a]{KU21}
D.~Krieg and M.~Ullrich.
\newblock Function values are enough for {$L_2$}-approximation.
\newblock {\em Found. Comput. Math.}, 21(4):1141--1151, 2021.

\bibitem[KU21b]{KU20}
D.~Krieg and M.~Ullrich.
\newblock Function values are enough for {$L_2$}-approximation: Part {II}.
\newblock {\em J. Complexity}, 66:101569, 2021.

\bibitem[K{\"{u}}h01]{Kue01}
T.~K{\"{u}}hn.
\newblock A lower estimate for entropy numbers.
\newblock {\em J. Approx. Theory}, 110(1):120--124, 2001.

\bibitem[K{\"{u}}h05]{Kue05}
T.~K{\"{u}}hn.
\newblock Entropy numbers of general diagonal operators.
\newblock {\em Rev. Mat. Complut.}, 18(2):479--491, 2005.

\bibitem[Lar89]{Lar89}
G.~Larcher.
\newblock On the distribution of the multiples of an {$s$}-tuple of real
  numbers.
\newblock {\em J. Number Theory}, 31(3):367--372, 1989.

\bibitem[Led01]{Led01}
M.~Ledoux.
\newblock {\em The concentration of measure phenomenon}.
\newblock American Mathematical Society, Providence, RI, 2001.

\bibitem[Lin85]{Lin85}
R.~Linde.
\newblock {$s$}-numbers of diagonal operators and {B}esov embeddings.
\newblock In Z.~Frol{\'i}k, V.~Sou{\v c}ek, and J.~Vin{\'a}rek, editors, {\em
  Proceedings of the 13th winter school on abstract analysis}, number~10, pages
  83--110, 1985.

\bibitem[LL21]{LL21}
A.~E. Litvak and G.~Livshyts.
\newblock New bounds on the minimal dispersion.
\newblock {\em arXiv e-prints}, arXiv:2108.10374 [math.MG], 2021.

\bibitem[LLL82]{LLL82}
A.~K. Lenstra, H.~W. Lenstra, Jr., and L.~Lov\'{a}sz.
\newblock Factoring polynomials with rational coefficients.
\newblock {\em Math. Ann.}, 261(4):515--534, 1982.

\bibitem[LMN{\etalchar{+}}20]{LMN+19}
M.~Lotz, M.~B. McCoy, I.~Nourdin, G.~Peccati, and J.~A. Tropp.
\newblock Concentration of the intrinsic volumes of a convex body.
\newblock In {\em Geometric aspects of functional analysis. {V}ol. {II}}, pages
  139--167. Springer, Cham, 2020.

\bibitem[LPT06]{LPT06}
A.~E. Litvak, A.~Pajor, and N.~{Tomczak-Jaegermann}.
\newblock Diameters of sections and coverings of convex bodies.
\newblock {\em J. Funct. Anal.}, 231(2):438--457, 2006.

\bibitem[LS81]{LS81}
P.~Lancaster and K.~Salkauskas.
\newblock Surfaces generated by moving least squares methods.
\newblock {\em Math. Comp.}, 37(155):141--158, 1981.

\bibitem[LT00]{LT00}
A.~E. Litvak and N.~{Tomczak-Jaegermann}.
\newblock Random aspects of high-dimensional convex bodies.
\newblock In {\em Geometric aspects of functional analysis}, pages 169--190.
  Springer, Berlin, 2000.

\bibitem[LvGM96]{LGM96}
G.~G. Lorentz, M.~von Golitschek, and Y.~Makovoz.
\newblock {\em {Constructive approximation: Advanced problems}}.
\newblock Springer-Verlag, Berlin, 1996.

\bibitem[LW21]{LW21}
T.~Lachmann and J.~Wiart.
\newblock On the area of empty axis-parallel boxes amidst 2-dimensional lattice
  points.
\newblock {\em arXiv e-prints}, arXiv:2109.11222 [math.NT], 2021.

\bibitem[Mat95]{Mat95}
P.~Mattila.
\newblock {\em {Geometry of sets and measures in {E}uclidean spaces: Fractals
  and rectifiability}}.
\newblock Cambridge University Press, Cambridge, 1995.

\bibitem[Mat10]{Mat10}
J.~Matou\v{s}ek.
\newblock {\em {Geometric discrepancy: An illustrated guide}}.
\newblock Springer-Verlag, Berlin, 2010.

\bibitem[Maz85]{Maz85}
{V.\,G.} Maz'ja.
\newblock {\em Sobolev spaces}.
\newblock Springer-Verlag, Berlin, 1985.

\bibitem[Mec19]{Mec19}
E.~S. Meckes.
\newblock {\em The random matrix theory of the classical compact groups}.
\newblock Cambridge University Press, Cambridge, 2019.

\bibitem[Mha10]{Mha10}
{H.\,N.} Mhaskar.
\newblock Eignets for function approximation on manifolds.
\newblock {\em Appl. Comput. Harmon. Anal.}, 29(1):63--87, 2010.

\bibitem[Mil85a]{Mil85a}
V.~D. Milman.
\newblock Almost {E}uclidean quotient spaces of subspaces of a
  finite-dimensional normed space.
\newblock {\em Proc. Amer. Math. Soc.}, 94(3):445--449, 1985.

\bibitem[Mil85b]{Mil85b}
V.~D. Milman.
\newblock {Random subspaces of proportional dimension of finite-dimensional
  normed spaces: Approach through the isoperimetric inequality}.
\newblock In {\em Banach spaces ({C}olumbia, {M}o., 1984)}, pages 106--115.
  Springer, Berlin, 1985.

\bibitem[MPT07]{MPT07}
S.~Mendelson, A.~Pajor, and N.~{Tomczak-Jaegermann}.
\newblock Reconstruction and subgaussian operators in asymptotic geometric
  analysis.
\newblock {\em Geom. Funct. Anal.}, 17(4):1248--1282, 2007.

\bibitem[MS86]{MS86}
V.~D. Milman and G.~Schechtman.
\newblock {\em Asymptotic theory of finite-dimensional normed spaces}.
\newblock Springer-Verlag, Berlin, 1986.

\bibitem[MSS15]{MSS15}
A.~W. Marcus, D.~A. Spielman, and N.~Srivastava.
\newblock Interlacing families {II}: {M}ixed characteristic polynomials and the
  {K}adison-{S}inger problem.
\newblock {\em Ann. of Math. (2)}, 182(1):327--350, 2015.

\bibitem[MY21]{MY21}
E.~Milman and Y.~Yifrach.
\newblock Regular random sections of convex bodies and the random
  quotient-of-subspace theorem.
\newblock {\em J. Funct. Anal.}, 281(7):Paper No. 109133, 22, 2021.

\bibitem[Nie78]{Nie78}
H.~Niederreiter.
\newblock Quasi-{M}onte {C}arlo methods and pseudo-random numbers.
\newblock {\em Bull. Amer. Math. Soc.}, 84(6):957--1041, 1978.

\bibitem[Nie92]{Nie92}
H.~Niederreiter.
\newblock {\em Random number generation and quasi-{M}onte {C}arlo methods}.
\newblock Society for Industrial and Applied Mathematics (SIAM), Philadelphia,
  PA, 1992.

\bibitem[NOU16]{NOU16}
S.~Nitzan, A.~Olevskii, and A.~Ulanovskii.
\newblock Exponential frames on unbounded sets.
\newblock {\em Proc. Amer. Math. Soc.}, 144(1):109--118, 2016.

\bibitem[Nov88]{Nov88}
E.~Novak.
\newblock {\em Deterministic and stochastic error bounds in numerical
  analysis}.
\newblock Springer-Verlag, Berlin, 1988.

\bibitem[Nov20]{Nov20}
E.~Novak.
\newblock Algorithms and complexity for functions on general domains.
\newblock {\em J. Complexity}, 61:101458, 11, 2020.

\bibitem[NSU22]{NSU20}
N.~Nagel, M.~Sch{\"a}fer, and T.~Ullrich.
\newblock {A new upper bound for sampling numbers}.
\newblock {\em Found. Comp. Math.}, 22:445--468, 2022.

\bibitem[NT06]{NT06}
E.~Novak and H.~Triebel.
\newblock {Function spaces in Lipschitz domains and optimal rates of
  convergence for sampling}.
\newblock {\em Constr. Approx.}, 23:325--350, 2006.

\bibitem[NUWZ18]{NUW+18}
E.~Novak, M.~Ullrich, H.~Wo\'{z}niakowski, and S.~Zhang.
\newblock Reproducing kernels of {S}obolev spaces on {$\Bbb R^d$} and
  applications to embedding constants and tractability.
\newblock {\em Anal. Appl. (Singap.)}, 16(5):693--715, 2018.

\bibitem[NW08]{NW08}
E.~Novak and H.~Wo\'{z}niakowski.
\newblock {\em {Tractability of multivariate problems. {V}olume {I}: {L}inear
  information}}.
\newblock European Mathematical Society (EMS), Z\"{u}rich, 2008.

\bibitem[NW10]{NW10}
E.~Novak and H.~Wo\'{z}niakowski.
\newblock {\em Tractability of multivariate problems. {V}olume {II}: {S}tandard
  information for functionals}.
\newblock European Mathematical Society (EMS), Z\"{u}rich, 2010.

\bibitem[NW12]{NW12}
E.~Novak and H.~Wo\'{z}niakowski.
\newblock {\em Tractability of multivariate problems. {V}olume {III}:
  {S}tandard information for operators}.
\newblock European Mathematical Society (EMS), Z\"{u}rich, 2012.

\bibitem[NWW04]{NWW04}
{F.\,J.} Narcowich, {J.\,D.} Ward, and H.~Wendland.
\newblock {Sobolev bounds on functions with scattered zeros, with applications
  to radial basis function surface fitting}.
\newblock {\em Math. Comp.}, 74(250):743--763, 2004.

\bibitem[Pag98]{P98}
G.~Pag\`es.
\newblock A space quantization method for numerical integration.
\newblock {\em J. Comput. Appl. Math.}, 89(1):1--38, 1998.

\bibitem[Pen21]{Pen21}
M.~D. Penrose.
\newblock Random euclidean coverage from within.
\newblock {\em arXiv e-prints}, arXiv:2101.06306 [math.PR], 2021.

\bibitem[Pie80]{Pie80}
A.~Pietsch.
\newblock {\em Operator ideals}.
\newblock North-Holland Publishing Co., Amsterdam-New York, 1980.

\bibitem[Pie87]{Pie87}
A.~Pietsch.
\newblock {\em Eigenvalues and {$s$}-numbers}.
\newblock Cambridge University Press, Cambridge, 1987.

\bibitem[Pie07]{Pie07}
A.~Pietsch.
\newblock {\em History of {B}anach spaces and linear operators}.
\newblock Birkh\"{a}user Boston, Inc., Boston, MA, 2007.

\bibitem[Pin85]{Pin85}
A.~Pinkus.
\newblock {\em {$n$}-widths in approximation theory}.
\newblock Springer-Verlag, Berlin, 1985.

\bibitem[PS15]{PS15}
F.~Pausinger and A.~M. Svane.
\newblock A {K}oksma-{H}lawka inequality for general discrepancy systems.
\newblock {\em J. Complexity}, 31(6):773--797, 2015.

\bibitem[PS20]{PS20}
F.~Pillichshammer and M.~Sonnleitner.
\newblock A note on isotropic discrepancy and spectral test of lattice point
  sets.
\newblock {\em J. Complexity}, 58:101441, 7, 2020.

\bibitem[PT86]{PT86}
A.~Pajor and N.~{Tomczak-Jaegermann}.
\newblock Subspaces of small codimension of finite-dimensional {B}anach spaces.
\newblock {\em Proc. Amer. Math. Soc.}, 97(4):637--642, 1986.

\bibitem[RS16]{RS16}
A.~Reznikov and {E.\,B.} Saff.
\newblock The covering radius of randomly distributed points on a manifold.
\newblock {\em Int. Math. Res. Not. IMRN}, 2016(19):6065--6094, 2016.

\bibitem[RS22]{RG20}
C.~Richter and E.~{Saor\'{\i}n G\'{o}mez}.
\newblock On the {M}onotonicity of the {I}soperimetric {Q}uotient for
  {P}arallel {B}odies.
\newblock {\em J. Geom. Anal.}, 32(1):15, 2022.

\bibitem[Ryc99]{Ryc99}
{V.\,S.} Rychkov.
\newblock On restrictions and extensions of the {B}esov and
  {T}riebel-{L}izorkin spaces with respect to {L}ipschitz domains.
\newblock {\em J. London Math. Soc. (2)}, 60(1):237--257, 1999.

\bibitem[Sch75]{Sch75}
W.~M. Schmidt.
\newblock Irregularities of distribution. {IX}.
\newblock {\em Acta Arith.}, 27:385--396, 1975.

\bibitem[Sch84]{Sch84}
C.~Sch\"{u}tt.
\newblock Entropy numbers of diagonal operators between symmetric {B}anach
  spaces.
\newblock {\em J. Approx. Theory}, 40(2):121--128, 1984.

\bibitem[Sch14]{Sch14}
R.~Schneider.
\newblock {\em Convex bodies: the {B}runn-{M}inkowski theory}.
\newblock Cambridge University Press, Cambridge, expanded edition, 2014.

\bibitem[SJ94]{SJ94}
I.~H. Sloan and S.~Joe.
\newblock {\em Lattice methods for multiple integration}.
\newblock The Clarendon Press, Oxford University Press, New York, 1994.

\bibitem[SK87]{SK87}
I.~H. Sloan and P.~J. Kachoyan.
\newblock {Lattice methods for multiple integration: Theory, error analysis and
  examples}.
\newblock {\em SIAM J. Numer. Anal.}, 24(1):116--128, 1987.

\bibitem[Son19]{Son19}
M.~Sonnleitner.
\newblock Discrepancy and numerical integration on spheres.
\newblock Master's thesis, Johannes Kepler University Linz, 2019.
\newblock available at \texttt{https://epub.jku.at/urn:nbn:at:at-ubl:1-30468}.

\bibitem[SP21]{SP21}
M.~Sonnleitner and F.~Pillichshammer.
\newblock On the relation of the spectral test to isotropic discrepancy and
  {$L_q$}-approximation in {S}obolev spaces.
\newblock {\em J. Complexity}, 67:101576, 9, 2021.

\bibitem[Spr00]{Spr00}
F.~Sprengel.
\newblock A class of periodic function spaces and interpolation on sparse
  grids.
\newblock In {\em Proceedings of the {I}nternational {C}onference on {F}ourier
  {A}nalysis and {A}pplications ({K}uwait, 1998)}, volume~21, pages 273--293,
  2000.

\bibitem[ST80]{ST80}
S.~Szarek and N.~{Tomczak-Jaegermann}.
\newblock On nearly {E}uclidean decomposition for some classes of {B}anach
  spaces.
\newblock {\em Compositio Math.}, 40(3):367--385, 1980.

\bibitem[Ste71]{Ste71}
{E.\,M.} Stein.
\newblock {\em Singular integrals and differentiability properties of
  functions}.
\newblock Princeton University Press, Princeton, New Jersey, 1971.

\bibitem[Sto73]{Sto73}
K.~B. Stolarsky.
\newblock Sums of distances between points on a sphere. {II}.
\newblock {\em Proc. Amer. Math. Soc.}, 41:575--582, 1973.

\bibitem[Stu77]{Stu77}
W.~Stute.
\newblock Convergence rates for the isotrope discrepancy.
\newblock {\em Ann. Probability}, 5(5):707--723, 1977.

\bibitem[Suh79]{S79}
{A.\,G.} Suharev.
\newblock Optimal formulas of numerical integration for some classes of
  functions of several variables.
\newblock {\em Dokl. Akad. Nauk SSSR}, 246(2):282--285, 1979.

\bibitem[SW08]{SW08}
R.~Schneider and W.~Weil.
\newblock {\em Stochastic and integral geometry}.
\newblock Springer-Verlag, Berlin, 2008.

\bibitem[Tal94]{Tal94}
M.~Talagrand.
\newblock Sharper bounds for {G}aussian and empirical processes.
\newblock {\em Ann. Probab.}, 22(1):28--76, 1994.

\bibitem[Tal14]{Tal14}
M.~Talagrand.
\newblock {\em {Upper and lower bounds for stochastic processes: Modern methods
  and classical problems}}.
\newblock Springer, Heidelberg, 2014.

\bibitem[Tem18]{Tem18}
V.~Temlyakov.
\newblock {\em Multivariate approximation}.
\newblock Cambridge University Press, Cambridge, 2018.

\bibitem[Tri92]{Tri92}
H.~Triebel.
\newblock {\em Theory of function spaces. {II}}.
\newblock Birkh\"{a}user Verlag, Basel, 1992.

\bibitem[Tri11]{Tri11}
H.~Triebel.
\newblock {\em {Fractals and spectra: Related to Fourier analysis and function
  spaces}}.
\newblock Birkh\"{a}user Verlag, Basel, 2011.

\bibitem[TWW88]{TWW88}
J.~F. Traub, G.~W. Wasilkowski, and H.~Wo\'{z}niakowski.
\newblock {\em Information-based complexity}.
\newblock Academic Press, Inc., Boston, MA, 1988.

\bibitem[Ull20]{Ull20}
M.~Ullrich.
\newblock On the worst-case error of least squares algorithms for {$L_2
  $}-approximation with high probability.
\newblock {\em J. Complexity}, 60:101484, 6, 2020.

\bibitem[Uni72]{WSU72}
Wayne~State University.
\newblock {Classroom {N}otes: Every convex function is locally Lipschitz}.
\newblock {\em Amer. Math. Monthly}, 79:1121--1124, 1972.

\bibitem[{van}17]{van17}
R.~{van Handel}.
\newblock On the spectral norm of {G}aussian random matrices.
\newblock {\em Trans. Amer. Math. Soc.}, 369(11):8161--8178, 2017.

\bibitem[{van}18]{van18}
R.~{van Handel}.
\newblock Chaining, interpolation, and convexity.
\newblock {\em J. Eur. Math. Soc. (JEMS)}, 20(10):2413--2435, 2018.

\bibitem[Vap98]{Vap98}
{V.\,N.} Vapnik.
\newblock {\em Statistical learning theory}.
\newblock John Wiley \& Sons, Inc., New York, 1998.

\bibitem[Ver06]{Ver06}
R.~Vershynin.
\newblock Isoperimetry of waists and local versus global asymptotic convex
  geometries.
\newblock {\em Duke Math. J.}, 131(1):1--16, 2006.

\bibitem[Vil03]{Vil03}
C.~Villani.
\newblock {\em Topics in optimal transportation}.
\newblock American Mathematical Society, Providence, RI, 2003.

\bibitem[vW96]{VW96}
A.~W. {van der Vaart} and J.~A. Wellner.
\newblock {\em Weak convergence and empirical processes}.
\newblock Springer-Verlag, New York, 1996.

\bibitem[Vyb12]{Vyb12}
J.~Vyb\'{\i}ral.
\newblock Average best {$m$}-term approximation.
\newblock {\em Constr. Approx.}, 36(1):83--115, 2012.

\bibitem[WB19]{WB19}
J.~Weed and F.~Bach.
\newblock Sharp asymptotic and finite-sample rates of convergence of empirical
  measures in {W}asserstein distance.
\newblock {\em Bernoulli}, 25(4A):2620--2648, 2019.

\bibitem[Wen01]{Wen01}
H.~Wendland.
\newblock Local polynomial reproduction and moving least squares approximation.
\newblock {\em IMA J. Numer. Anal.}, 21(1):285--300, 2001.

\bibitem[Wen05]{Wen04}
H.~Wendland.
\newblock {\em Scattered data approximation}.
\newblock Cambridge University Press, Cambridge, 2005.

\bibitem[WS93]{WS93}
{Z.\,M.} Wu and R.~Schaback.
\newblock Local error estimates for radial basis function interpolation of
  scattered data.
\newblock {\em IMA J. Numer. Anal.}, 13(1):13--27, 1993.

\bibitem[Yuk08]{Yuk08}
{J.\,E.} Yukich.
\newblock Limit theorems for multi-dimensional random quantizers.
\newblock {\em Electron. Commun. Probab.}, 13:507--517, 2008.

\bibitem[Zad82]{Zad82}
P.~L. Zador.
\newblock Asymptotic quantization error of continuous signals and the
  quantization dimension.
\newblock {\em IEEE Trans. Inform. Theory}, 28(2):139--149, 1982.

\bibitem[Zar70]{Zar70}
{S.\,C.} Zaremba.
\newblock La discr\'{e}pance isotrope et l'int\'{e}gration num\'{e}rique.
\newblock {\em Ann. Mat. Pura Appl. (4)}, 87:125--135, 1970.

\end{thebibliography}

\clearpage
\appendix
\chapter{Additional material for Chapter 3} \label{ch:ell-app}
\section{Gelfand numbers of diagonal operators} \label{sec:ell-gelfand}

In what follows, we collect useful knowledge about the Gelfand numbers of diagonal operators. At the end of this section we will state and prove Corollary~\ref{cor:min-order} from which the expression in \eqref{eq:min-order} on the rate of polynomial decay of the minimal radius for an $\ell_p$-ellipsoid with polynomial semiaxes is deduced. 

Let $0< p,q\le \infty$ and $\sigma=(\sigma_{j})_{j\in\IN}$ with $\sigma_{1}\ge \sigma_{2}\ge \cdots \ge 0$. To this sequence we associate the diagonal operator
\[
	D_{\sigma}\colon\ell_{p}\to\ell_{q},
	\quad x=(x_{j})_{j\in\IN}\mapsto (\sigma_{j}x_{j})_{j\in\IN},
\]
which, for any $m\in\IN$, can be considered as an operator from $\ell_p^m$ to $\ell_q^m$. Recall that
\begin{equation}\label{eq:radiusgelfand}
	\rad(\elp,n)
	= c_{n+1}(D_{\sigma}\colon \ell_{p}^{m}\to\ell_{2}^{m})\quad\text{for all }0\le n< m,
\end{equation}
and an analogous result holds if 2 is replaced by $q$ and we measure the radius in $\|\cdot\|_q$.  Although we will need only the case $q=2$, we state the following result in a more general form. 
\begin{proposition}\label{pro:gelfandtail}
Let $0< q \le p\le \infty$ and $\sigma_{1}\ge \sigma_{2}\ge \cdots\ge 0$. Then, for any $1\le n\le m$, we have
\[
	c_{n}(D_{\sigma}\colon\ell_{p}^{m}\to\ell_{q}^{m})=\Big(\sum_{j=n}^{m}\sigma_{j}^{r}\Big)^{1/r},
\]
where $\frac{1}{r}=\frac{1}{q}-\frac{1}{p}$ if $q<p$ and $r=\infty$ if $q=p$.
\end{proposition}
A proof can be found in \cite[Sec.~11.11]{Pie80} for the case $q\ge 1$ but is in fact also valid if $q<1$. In addition to Proposition~\ref{pro:gelfandtail}, we have for all $0< p,q\le \infty$ that
\[
	c_{1}(D_{\sigma}\colon \ell_{p}^{m}\to\ell_{q}^{m})
	= \|D_{\sigma}\colon \ell_{p}^{m}\to\ell_{q}^{m}\|
	= \Big(\sum_{j=1}^{m}\sigma_{j}^{r}\Big)^{1/r},
\]
where $\frac{1}{r}=(\frac{1}{q}-\frac{1}{p})_{+}$. This shows that $\|\sigma\|_{r}<\infty$ is necessary to ensure that the operators $D_{\sigma}\colon \ell_{p}^{m}\to\ell_{q}^{m}, m\in\IN,$ are uniformly bounded and explains the condition $\lambda>(\frac{1}{2}-\frac{1}{p})_+$ in Section~\ref{sec:ell-poly}. 

All of the above extends to the infinite-dimensional case in a canonical way and we now state an equivalence between the Lorentz quasi-norm of an infinite sequence $\sigma$ and the Gelfand numbers of the corresponding diagonal operators. It is taken from Buchmann~\cite{Buc99} and one implication can be extracted from Linde \cite[Thm.~5]{Lin85}. The technique behind makes use of well-known results on the Gelfand numbers of the identity operator from $\ell_p$ to $\ell_q$ by, to name a few, Garnaev and Gluskin~\cite{GG84}, Gluskin~\cite{Glu81,Glu84}, Ismagilov~\cite{Ism74} and Kashin~\cite{Kas77}.
\begin{proposition}\label{pro:lorentzeq}
Let $1\le p,q\le \infty$ and $r>0$ with $\frac{1}{r}>(\frac{1}{q}-\frac{1}{p})_{+}$ as well as $0< t\le\infty$. Then
\[
	\sigma\in\ell_{r,t}
	\quad\Leftrightarrow\quad \big(c_{n}(D_{\sigma}\colon \ell_{p}\to\ell_{q})\big)_{n\in\IN}\in \ell_{u,t},
\]
where 
\begin{enumerate}
	\item  if $1\le q\le p\le\infty$, then $\frac{1}{u}=\frac{1}{r}+\frac{1}{p}-\frac{1}{q}$,
	\item if $1\le p<q\le 2$, then $\frac{1}{u}=
		\begin{cases}
			\frac{p^*}{2r}&\text{if }\frac{1}{r}<\frac{1}{p^*}\frac{1/p-1/q}{1/p-1/2},\\
			\frac{1}{r}	+\frac{1}{p}-\frac{1}{q}&\text{if }\frac{1}{r}>\frac{1}{p^*}\frac{1/p-1/q}{1/p-1/2},
		\end{cases}$
	\item if $1\le p<2<q\le \infty$, then $\frac{1}{u}=
		\begin{cases}
			\frac{p^*}{2r}&\text{if }\frac{1}{r}<\frac{1}{p^*},\\
			\frac{1}{r}	+\frac{1}{p}-\frac{1}{q}&\text{if }\frac{1}{r}>\frac{1}{p^*},
		\end{cases}$
	\item if $2\le p<q\le\infty$, then $\frac{1}{u}=\frac{1}{r} $.
\end{enumerate}
\end{proposition}

By means of equality \eqref{eq:radiusgelfand}, the above Propositions \ref{pro:gelfandtail} and \ref{pro:lorentzeq} apply to the minimal radius for $\ell_p$-ellipsoids. Note that in the latter proposition some cases are missing; for example, if $q=2$, there is a gap for $\frac{1}{r}=\frac{1}{p^*}$. In this case, we can deduce \chg{the following corollary from Theorem~\ref{thm:ell-main}.} 

\begin{corollary}\label{cor:minrad}
For all $n\in\IN$,
\[
c_n(D_{\sigma}\colon \ell_p\to\ell_2)
	\lesssim
	\begin{cases}
		n^{-1/2}\sup\limits_{k\le j\le \infty}\sigma_{j}\sqrt{\log(j)+1} &\text{if }p=1,\\
		\sqrt{p^*}n^{-1/2} \Big(\sum\limits_{j=k}^{\infty}\sigma_{j}^{p^{*}}\Big)^{1/p^{*}} &\text{if }1< p\le \infty,
	\end{cases}
\]
where $k\asymp \frac{n}{p^*}$ for $p>1$, while $k\asymp n$ for $p=1$. In particular, $c_{n}(D_{\sigma}\colon \ell_{p}\to\ell_{2})\in \ell_{2,\infty}$ if $1<p\le \infty$ and $\sigma\in \ell_{p^*}$.
\end{corollary}
\begin{proof}
	We use \eqref{eq:radiusgelfand} to state Theorem~\ref{thm:ell-main} for Gelfand numbers. Then, for each $m\in\IN$, let $D_{\sigma}^m$ be the restriction of the operator $D_{\sigma}$ to the first $m$ coordinates. This gives 
\[
c_n(D_{\sigma}\colon \ell_p^m\to\ell_q^m)
=c_n(D_{\sigma}^m\colon \ell_p\to\ell_q),
\]
and further, by continuity and H\"older's inequality,
\[
\big|c_n(D_{\sigma}^m\colon \ell_p\to\ell_q)-c_n(D_{\sigma}\colon \ell_{p}\to\ell_{q})\big|
\le \Big(\sum_{j=m+1}^{\infty}\sigma_j^s\Big)^{1/s},
\]
with $\frac{1}{s}=(\frac{1}{2}-\frac{1}{p})_+$. Letting $m\to\infty$ for each $n\in\IN$ completes the proof.
\end{proof}

Let us note that bounding the Gelfand numbers of operators into $\ell_2$ via $M^*$-estimates has been done before, e.g., in \cite{PT86}.

As a corollary to Proposition~\ref{pro:lorentzeq} we obtain the following result on the order of decay of the minimal radius of $\ell_p$-ellipsoids with polynomially decaying semiaxes. For $p\ge 2$ it also follows from Proposition~\ref{pro:gelfandtail} which shows that $\rad(\elp,n) \lesssim n^{-\lambda+1/2-1/p}$ for all $m>n$ with a matching lower bound for $m$, say, larger than $2n$.

\begin{corollary}\label{cor:min-order}
Let $1\le p\le \infty$. If $\sigma_j=j^{-\lambda}$, $j\in\IN$, for some $\lambda>(\frac{1}{2}-\frac{1}{p})_+$, then
\begin{equation*}
\decay\big(\rad(\elp,n)\big)
=	\begin{cases}
	\lambda\cdot\frac{p^{*}}{2}&\text{if } 1\le p<2 \text{ and }\lambda<\frac{1}{p^{*}},\\
	\lambda+\frac{1}{p}-\frac{1}{2}&\text{otherwise.}
	\end{cases}
\end{equation*}
\end{corollary}
\begin{proof}
	We only prove the first case since the other case does not require additional ideas. To show that $\decay\big(\rad(\elp,n)\big)\ge \lambda p^*/2$, it is sufficient by \eqref{eq:radiusgelfand} to find $C>0$ such that, for all $\varrho<\lambda p^*/2$ large enough,
	\begin{equation}\label{eq:decay}
	c_{n,m}
	:=c_{n}(D_{\sigma}\colon \ell_{p}^{m}\to\ell_{q}^{m})
	\le C n^{-\varrho}\quad \text{for all }m\in\IN \text{ and }1\le n\le m.
\end{equation}
	This is satisfied if the sequence of Gelfand numbers $c_{n}(D_{\sigma}\colon \ell_{p}\to\ell_{q})\geq c_{n,m}$ ($n\in\IN$) belongs to $\ell_{u,\infty}$ with $u=1/\varrho$. By Proposition~\ref{pro:lorentzeq}, this holds if $\sigma\in \ell_{r,\infty}$ with a certain $r>1/\lambda$, which is true by assumption. 

	For the other inequality we assume that \eqref{eq:decay} holds for some $\varrho>\lambda p^*/2$. Choosing $m$ large enough compared to $n$, see the proof of Corollary~\ref{cor:minrad}, we deduce from \eqref{eq:decay} that for some $\varrho>\lambda p^*/2$ and every $n\in\IN$, 
	\begin{equation}\label{eq:decay-inf}
		c_n
	:=c_{n}(D_{\sigma}\colon \ell_{p}\to\ell_{q})
	\le 2C n^{-\varrho}.
	\end{equation}
	For every $r< 1/\lambda$ and $0<t<\infty$, it follows from $\sigma\not\in \ell_{r,t}$ and Proposition~\ref{pro:lorentzeq} that
	\[
	\sum_{n=1}^{\infty}\frac{1}{n}c_{n}^t n^{tp^*/2r}=\infty,
	\]
	which implies $c_n n^{p^*/2r}\log^{2/t} n\to \infty$, contradicting \eqref{eq:decay-inf}.
\end{proof}

\section{Proofs for the convex case}\label{sec:ell-proof-convex}

In this section we complete the proof ideas for Sections~\ref{sec:ell-rounding} and \ref{sec:ell-lower}, in particular the upper bound of Theorem~\ref{thm:ell-main} and the lower bound of Proposition~\ref{pro:lower-general}.

We start with the proof of Proposition~\ref{pro:mstar-estimate} on the mean width of a rounded $\ell_p$-ellipsoid which is then used to derive Theorem~\ref{thm:ell-main}. The following elementary estimate for $\ell_q$-norms of structured Gaussian random vectors will be useful.
\begin{lemma}\label{lem:khintchine-gaussian}
	Let $k\in\IN$ and $1\le q<\infty$. If $b=(b_{j})_{j=1}^k\in\IR^k$ and $X=(b_j g_j)_{j=1}^k$ with independent standard Gaussian random variables $g_1,\dots,g_k$, then
\[
	\gamma_{1}\|b\|_{q}
	\le \IE \|X\|_{q}
	\le \gamma_{q}\|b\|_{q},
	\quad \text{where }\quad
	\gamma_{q}
	:=\big(\IE|g_{1}|^{q}\big)^{1/q}
	\asymp \sqrt{q}.
\]
Further, 
\[
	\IE	\|X\|_{\infty}
	\asymp \sup_{1\le j\le k}b_{j}^{*}\sqrt{\log(j)+1},
\]
where $(b_j^*)_{j=1}^k$ is the non-increasing rearrangement of $(|b_j|)_{j=1}^k$ (see Definition~\ref{def:lorentz}).
\end{lemma}
\begin{proof}
For $1\le q<\infty$ the upper bound follows from Jensen's inequality and the lower bound follows from $\|\IE X'\|_{q}\le \IE\|X'\|_{q}=\IE\|X\|_{q}$, where $X'=(|b_j g_j|)_{j=1}^k$. The asymptotics for $q=\infty$ are taken from van Handel \cite[Lem.~2.3 and 2.4]{van17}.
\end{proof}

\begin{proof}[Proof of Proposition~\ref{pro:mstar-estimate}]
Fix $m\in\IN$, $0\le k< m$ and $\varrho>0$. We shall use the representation \eqref{eq:mstargaussian}, which translates to
\[
M^{*}(\elp\cap\varrho\IBo_2^m)
=\frac{1}{c_m}\IE\sup_{t\in\elp\cap\varrho\IBo_2^m}\sum_{j=1}^{m}t_j g_j,
\]
where $c_m\asymp \sqrt{m}$. In order to bound the supremum, fix some $x\in\IR^{m}$ and $y\in \elp \cap \varrho\IBo_{2}^{m}$. Then it follows from Hölder's inequality that, for every $0\leq k < m$,
\[
	\langle x,y\rangle
	\le \sum_{j=1}^{k}|x_{j}y_{j}|+\sum_{j=k+1}^{m}|x_{j}y_{j}|
	\le \varrho\Big(\sum_{j=1}^{k}x_{j}^{2}\Big)^{1/2}+\Big(\sum_{j=k+1}^{m}\sigma_{j}^{p^{*}}|x_{j}|^{p^{*}}\Big)^{1/p^{*}},
\]
where the first, and consequently the third, sum is empty if $k=0$.
Combining this estimate with Lemma \ref{lem:khintchine-gaussian}, we obtain that if $p>1$,
\begin{align*}
	\IE\sup_{y\in \elp \cap \varrho\IBo_{2}^{m}}\sum_{j=1}^{m}y_{j}g_{j}
	\le \varrho\, a_{k}+\gamma_{p^{*}}\Big(\sum_{j=k+1}^{m}\sigma_{j}^{p^{*}}\Big)^{1/p^{*}},
\end{align*}
where we used that $a_k=\IE\Big(\sum_{j=1}^{k}g_i^2\Big)^{1/2}$, see \eqref{eq:asymptotic ak}. By the previously stated asympotics for $a_{k}$ and for $\gamma_{p^{*}}$ in Lemma \ref{lem:khintchine-gaussian}, we obtain the statement for $p>1$.

If $p=1$, then we deduce from Lemma \ref{lem:khintchine-gaussian} that, for some suitable $C>0$,
\[
	\IE\,\sup_{y\in \elp\cap \varrho\IBo_{2}^{m}} \sum_{j=1}^{m}y_jg_{j}
	\le \varrho\sqrt{k}+C\sup_{k+1\leq j \leq m}\sigma_{j}\sqrt{\log(j)+1}.
\]
This completes the proof.
\end{proof}

\begin{proof}[Proof of Theorem \ref{thm:ell-main}]
For any $\varrho>0$, it follows from Gordon's $M^*$-estimate (Proposition \ref{pro:escape}) applied to the convex body $\elp\cap \varrho \IBo_2^m$ that, with probability as claimed, a random subspace $\enran$ of codimension $n$ satisfies
\[
	\rad(\elp\cap\varrho \IBo_{2}^{m}, \enran)	
	=\sup_{x\in \elp \cap \varrho \IBo_{2}^{m}\cap \enran}\|x\|_{2}
	\le 2\frac{a_{m}}{a_{n}}M^{*}(\elp\cap\varrho \IBo_{2}^{m}).
\]
We start with the case $p>1$. Inserting the bound obtained in Proposition \ref{pro:mstar-estimate}, we obtain a constant $C>0$ such that, for any $0\le k< m$ and $1\leq n< m$,
\[
	\rad(\elp\cap\varrho \IBo_{2}^{m}, \enran)	
	\le C \frac{\sqrt{p^*}}{\sqrt{n}}\Big(\varrho\sqrt{k}+\Big(\sum_{j=k+1}^{m}\sigma_{j}^{p^{*}}\Big)^{1/p^{*}}\Big).
\]
In accordance with relation \eqref{eq:rounding}, we seek to choose $\varrho$ suitably small and $k$ large such that this radius is bounded by $\varrho$. 

First, let $\frac{n}{p^*}>4C^2$. Setting $\varrho:=\frac{1}{\sqrt{k}}\Big(\sum_{j=k+1}^{m}\sigma_{j}^{p^{*}}\Big)^{1/p^{*}}$ with $k:=c\frac{n}{p^*}$, where the constant $c>0$ is chosen (sufficiently small) such that $k\in \IN$ with $1\le k<m$ and 
\[
	\rad(\elp\cap\varrho \IBo_{2}^{m}, \enran)	
	<\varrho
\]
and so in particular that	
\[
	\rad(\elp, \enran)	
	<\varrho
\]
for all $m\in \IN$ and $1\le n <m$. Noting that 
\[
	\varrho
	\lesssim \frac{\sqrt{p^*}}{\sqrt{n}}\Big(\sum_{j=k+1}^{m}\sigma_{j}^{p^{*}}\Big)^{1/p^{*}},
\]
proves the result for $p>1$ in this case. If $\frac{n}{p^*}\le 4C^2$, then let $k=0$ and let $\varrho>0$ be large enough such that $\elp\subset \varrho \IBo_2^m$ and thus 
\[
\rad(\elp,\enran)
\le C\frac{\sqrt{p^*}}{\sqrt{n}}\Big(\sum_{j=1}^{m}\sigma_{j}^{p^{*}}\Big)^{1/p^{*}}.
\]
In both cases, $k+1\asymp \frac{n}{p^*}$. The proof for $p=1$ is carried out analogously. 
\end{proof}

Next, we prove the general version of the lower bound in Proposition~\ref{pro:lower}.

\begin{proof}[Proof of Proposition~\ref{pro:lower-general}]
By Lemma \ref{lem:large-coord}, with probability at least $1-\varepsilon$, we find $x\in \IR^m$ with
\[
x_1^2 \ge 1-\frac{n}{\varepsilon m},\quad
\|x\|_2=1,\quad\text{and}\quad
G_{n,m} x=0.
\]
We estimate
\[
	\Big(\sum_{j=1}^{m}\frac{|x_j|^p}{\sigma_j^p}\Big)^{1/p}
	\le \frac{1}{\sigma_1}+ \Big(\sum_{j=2}^{m}\frac{|x_j|^p}{\sigma_j^p}\Big)^{1/p}
	\le \frac{1}{\sigma_1}+ \frac{1}{\sigma_m}\Big(\sum_{j=2}^{m}|x_j|^p\Big)^{1/p}
\]
and by means of Hölder's inequality, we obtain
\[
	\frac{1}{\sigma_m}\Big(\sum_{j=2}^{m}|x_j|^p\Big)^{1/p}
	\le \frac{m^{1/p-1/2}}{\sigma_m}\Big(\sum_{j=2}^{m}x_j^2\Big)^{1/2}
	= \frac{m^{1/p-1/2}}{\sigma_m}\big(1-x_1^2\big)^{1/2}.
\]
Since $1-x_1^2\le \frac{n}{\varepsilon m}$, we have
\[
	\Big(\sum_{j=1}^{m}\frac{|x_j|^p}{\sigma_j^p}\Big)^{1/p}
	\le \frac{1}{\sigma_1}+1
\]
if $n\le \varepsilon m^{2p^*}\sigma_m^2$. In this case, we can normalize such that $\tilde{x}:=x/(1+\frac{1}{\sigma_1})$ satisfies
\[
\tilde{x}\in \elp,\quad
G_{n,m} \tilde{x}=0,\quad\text{and}\quad
\|\tilde{x}\|_2=\frac{\sigma_1}{1+\sigma_1},
\]
which completes the proof.
\end{proof}

\section{Proofs for the non-convex case}\label{sec:ell-proof-quasi}

We continue with the results leading to the proof of Theorem~\ref{thm:ell-main-quasi}. It is a special case of Theorem~\ref{thm:ell-main-quasi-general} from Section~\ref{sec:ell-sparse}. The proof relies on Lemma~\ref{lem:nonlin-app} which is proven first.

\begin{proof}[Proof of Lemma~\ref{lem:nonlin-app}]
Approximating every $x\in\elp$ by the $s$-sparse vector consisting of its $s$ largest entries we obtain
\[
\sigma_s(x)_q
\le\Big(\sum_{j=s+1}^{m}(x_j^*)^q\Big)^{1/q},
\]
where $(x_j^*)_{j=1}^m$ is the non-increasing rearrangement as in Definition~\ref{def:lorentz}. If $q=\infty$, then take the maximum norm instead and analogously for $p$ in the following. Recall that in this thesis we set $a/\infty=0$ for every $a\in\IR$. 

Using that $x\in\elp$, we can write
\[
1
\ge \sum_{j=1}^{m}\frac{|x_j|^p}{\sigma_{j}^p}
= \sum_{j=1}^{m}\frac{(x_j^*)^p}{\sigma_{\pi(j)}^p}
\ge \sum_{j=1}^{k}\frac{(x_k^*)^p}{\sigma_{\pi(j)}^p},
\]
for any $1\leq k \leq m$, where $\pi:\{1,\ldots,m\}\to\{1,\ldots,m\}$ is a suitable permutation. Thus,
\[
(x_k^*)^p
\le \Big(\sum_{j=1}^{k}\frac{1}{\sigma_{\pi(j)}^p}\Big)^{-1}
\le \Big(\sum_{j=1}^{k}\frac{1}{\sigma_{j}^p}\Big)^{-1}
\asymp k^{-\lambda p-1},
\]
where the notation $\asymp$ (and $\lesssim$ below) means in this proof that the implicit constants do not depend on $k,s$ or $m\ge 2s$ but may depend on $p,q$ and $\lambda$.  Inserting this bound above yields
\[
\sigma_s(x)_q
\lesssim\Big(\sum_{k=s+1}^{m}k^{-\lambda q-q/p}\Big)^{1/q}
\asymp s^{-\lambda+1/q-1/p}
\]
as $\lambda>(1/q-1/p)_+$. 

The lower bound is achieved by a vector on the boundary of $\elp$ having its support on the first $2s$ coordinates which are all equal to 
\[
\Big(\sum_{i=1}^{2s}\sigma_i^{-p}\Big)^{-1/p}
\asymp s^{-\lambda-1/p}.
\]
\end{proof}

\begin{proof}[Proof of Theorem~\ref{thm:ell-main-quasi-general}]
	Let $m\in\IN$, $1\le n<m$, $N_n:=n^{-1/2}G_{n,m}$ and set $r=\min\{1,q\}$. By Proposition~\ref{pro:radius-general-+}, see also Remark~\ref{rem:ell-quasi},
\begin{equation} \label{eq:radq-upper}
\rad_q(\elp,\ker G_{n,m})
\le 2^{(1/q-1)_+} \sup_{x\in \elp}\|x-\Delta_r(N_n x)\|_q,
\end{equation}
where we specified $\varphi(y)=\Delta_r(n^{-1/2}y)=$ arg min $\|z\|_r$ subject to $N_n z=n^{-1/2}y$ for $y=G_{n,m} x$. 

We follow the proof of \cite[Thm.~3.2]{FPR+10} in order to obtain an upper bound. To this end, let $D\in (0,\infty)$ be large enough such that
\[
D/2>{\rm e} \quad\text{and}\quad \frac{D/2}{1+\log(D/2)}>C_1,
\]
where $C_1\in(0,\infty)$ is the constant from Proposition~\ref{pro:rip-gaussian} with $\delta=1/3$. Also choose $s\ge 1$ as the largest integer smaller than $ n/D\log({\rm e}m/n)$ such that
\[
\frac{n}{2D\log({\rm e}m/n)}
\le s
< \frac{n}{D\log({\rm e}m/n)}.
\]
Then, if $n\ge D\log({\rm e}m/n)$ it holds that $n>C_1 (2s)\log({\rm e}m/2s)$. This can be seen as follows. Put $t=2s$ such that $t/n<2/D<1/{\rm e}$. As the function $x\mapsto x\log(x)$ is decreasing on $[0,1/{\rm e}]$, we obtain
\[
n\ge\frac{D}{2}t \log({\rm e}m/n)
= \frac{D}{2}t \log({\rm e}m/t)
+ \frac{D}{2}n (t/n) \log({\rm e}t/n)
> \frac{D}{2}t \log({\rm e}m/t)
-n \log(D/2)
\]
and therefore
\[
n\ge\frac{D/2}{1+\log(D/2)}t\log({\rm e}m/t)
>C_1\,t\log({\rm e}m/t).
\]
Then, by Proposition~\ref{pro:rip-gaussian} the matrix $N_n$ satisfies $\delta_{2s}(N_n)\le \delta$ with probability at least $1-2\exp(-C_2 n)$. Assume from now on that $N_n$ is replaced by a realization satisfying this.

Let $x\in\elp$ be arbitrary. Then $v:=x-\Delta_r(N_n x)$ belongs to the kernel and we will apply the restricted isometry property to a decomposition of $v$ in order to bound $\|v\|_q$. Decompose the set $\{1,\dots,m\}$ into disjoint sets $S_1,\dots,S_M$ which are of size $s$, except for possibly $S_M$, such that $|v_i|\ge |v_j|$ for all $i\in S_{k-1}$ and $j\in S_k$, where $k\ge 2$. Write $v=\sum_{k=1}^{M}v_{S_k}$, where $v_{S_k}$ is a vector whose coordinates coincide with $v_i$ if $i\in S_k$ and with zero else. Then $r\le 2$ and Hölder's inequality imply 
\begin{equation} \label{eq:v-hoelder}
\|v_{S_k}\|_2\le s^{1/2-1/r}\|v_{S_{k-1}}\|_r,\quad k=2,\dots,M,
\end{equation}
which will be used a few lines below. Together with Hölder's inequality we deduce from the lower bound in the definition \eqref{eq:rip} of the restricted isometry property for the $s$-sparse (and also $2s$-sparse) vector $v_{S_k}$ that
\[
\|v\|_q^r
\le \Big\| \sum_{k=1}^{M} v_{S_k}\Big\|_q^r
\le \sum_{k=1}^{M} \|v_{S_k}\|_q^r
\le  \sum_{k=1}^{M}(s^{1/q-1/2}\|v_{S_k}\|_2)^r
\le \Big(\frac{s^{1/q-1/2}}{\sqrt{1-\delta}}\Big)^r \sum_{k=1}^{M}\|N_nv_{S_k}\|_2^r.
\]
Since $v\in \ker N_n$, it holds that $N_n v_{S_1}=-\sum_{k=2}^{M}N_n v_{S_k}$ and further, since $r\le 2$,
\[
\|N_n v_{S_1}\|_2^r
=\Big\|\sum_{k=2}^{M}N_n v_{S_k}\Big\|_2^r
=\Big(\sum_{k=2}^{M}\|N_n v_{S_k}\|_2^2\Big)^{r/2}
\le \sum_{k=2}^{M}\|N_n v_{S_k}\|_2^r.
\]
Thus, we can continue our estimate with
\[
\|v\|_q^r
\le 2 \Big(\frac{s^{1/q-1/2}}{\sqrt{1-\delta}}\Big)^r \sum_{k=2}^{M}\|N_nv_{S_k}\|_2^r.
\]
Applying now the upper bound in the restricted isometry property of \eqref{eq:rip} yields that
\[
\|v\|_q^r
\le 2 \Big(\frac{\sqrt{1+\delta}}{\sqrt{1-\delta}}s^{1/q-1/2}\Big)^r \sum_{k=2}^{M}\|v_{S_k}\|_2^r.
\]
Using the inequality \eqref{eq:v-hoelder} and that $\delta=1/3$ yields
\[
\|v\|_q
\le 2^{1/r+1/2} s^{1/q-1/r}\Big(\sum_{k=1}^{M}\|v_{S_k}\|_r^r\Big)^{1/r}
= 2^{1/r+1/2}s^{1/q-1/r}\|v\|_r
\]
and inserting now for $s$ yields for some constant $C>0$ depending on $r$ and $C_1$ that
\[
\|v\|_q\le C \,\Big(\frac{\log({\rm e}m/n)}{n}\Big)^{1/r-1/q} \|v\|_r.
\]
After inserting for $v=x-\Delta_r(N_n x)$, we can apply Proposition~\ref{pro:rip-sparse} to get a bound in terms of the error of best $s$-term approximation of $x$. Then taking the supremum shows that there exists a constant $C>0$, depending on $p,q$ and $C_1$, such that
\[
\sup_{x\in \elp }\|x-\Delta_r(N_n x)\|_q
\le C\, \Big(\frac{\log({\rm e}m/n)}{n}\Big)^{1/r-1/q}\sup_{x\in \elp }\sigma_s(x)_r.
\]
Until now we considered a specific realization of $N_n$ and therefore this holds with probability at least $1-2\exp(-C_2 n)$ for the random matrix $N_n$.

With Lemma~\ref{lem:nonlin-app} the proof of Theorem~\ref{thm:ell-main-quasi-general} is complete if we can show that $s\le m/2$. Indeed, since the function $n\mapsto n/\log({\rm e}m/n)$ is increasing for $1\le n\le m$, we have $s\le m/D < m/2{\rm e}$.
\end{proof}

\chapter{Additional material for Chapter 4} \label{ch:sob-app}

We provide the proofs for the statements claimed in Chapter~\ref{ch:sob} and in particular follow the ideas given in Sections~\ref{sec:sob-mls} and \ref{sec:sob-fool-int} to complete the proof of Theorem~\ref{thm:sob-main} and its extension, Theorem~\ref{thm:sob-main-ext}. 

First, in Section~\ref{sec:sob-proof-discuss}, we provide the proofs for the asymptotics of the optimal distortion \eqref{eq:dist-optimal} and for Example~\ref{ex:big-hole}.

In Section~\ref{sec:sob-proof-poly} we continue with technical preparations regarding the approximation of Sobolev functions by polynomials (Lemma~\ref{lem:sob-poly}) and their extension from bounded Lipschitz domains to the whole space (Lemma~\ref{lem:sob-extension}). We also give a proof of the known fact that a bounded convex domain has Lipschitz boundary (Lemma~\ref{lem:convex-lip}). 

Section~\ref{sec:sob-proof-ge} contains the proof of Theorem~\ref{thm:sob-main} for the case $q\ge p$. As a preparation for the upper bound, Lemma~\ref{lem:wendland} is deduced from \cite[Thm.~4.7]{Wen04} and the fact that \chg{the} interior cone condition is preserved by the closure (Lemma~\ref{lem:closure-CC}) is proven. The lower bound relies on finding a suitable large hole which is provided by Lemma~\ref{lem:ballinball}.

In Section~\ref{sec:sob-proof-le} we give the proof of Theorem~\ref{thm:sob-main} for the case $q<p$. For this, we commence by proving the local cone condition (Lemma~\ref{lem:local-CC}) and the error bound of moving least squares applied to a cube (Lemma~\ref{lem:local-estimate}). Afterwards, Lemmas~\ref{lem:sob-trivial} and \ref{lem:inf-min} justify the Definition~\ref{def:small-cube} of the smallest good cube. Then the existence of an efficient covering (Proposition~\ref{pro:good-cover}) and of a hole of proportional size in every smallest good cube (Lemma~\ref{lem:good-hole}) are proven. With these preparations at hand the proof of the upper and lower bound in the case $q<p$ is given. 

Finally, in Section~\ref{sec:sob-proof-ext} we prove the extension, Theorem~\ref{thm:sob-main-ext}, of our results to more general function spaces. There, we shall also provide a definition of Besov and Triebel-Lizorkin spaces.

In this appendix we will make use of the asymptotic notations $\lesssim,\gtrsim,\asymp$ as given in Definition~\ref{def:asymp-std}. Additionally, whenever an arbitrary but fixed function $f$ occurs, the implicit constants will not depend on it.

\section{Proofs for Section~4.2}\label{sec:sob-proof-discuss}

Let us prove the asymptotics for the optimal points on a bounded set $\Omega\subset\IR^d$. 
\begin{proof}[Proof of \eqref{eq:dist-optimal}]
	Let $n\in\IN$ and $\Pn=\{x_1,\dots,x_n\}\subset\IR^d$. Since 
\begin{equation} \label{eq:dist-norms}
	\|\dist(\cdot,\Pn)\|_{L_{\infty}(\Omega)}\ge \vol(\Omega)^{1/\gamma}\|\dist(\cdot,\Pn)\|_{L_{\gamma}(\Omega)}\text{ for every }0<\gamma<\infty,
\end{equation}
it suffices to prove the upper bound for $\gamma=\infty$ and the lower bound for $0<\gamma<\infty$.

For the upper bound we construct an $\varepsilon$-net, with $\varepsilon>0$ to be fixed later, by iteratively selecting $y_1\in\Omega$, $y_2\in\Omega\setminus B(y_1,\varepsilon)$ and $y_{k+1}\in\Omega$ not contained in $\bigcup_{j=1}^k B(y_j,\varepsilon)$ until the process stops after finitely many steps. Then we are left with $N$ balls $B(y_1,\varepsilon),\dots,B(y_N,\varepsilon)$ which cover $\Omega$ and the balls $B(y_1,\varepsilon/2),\dots,B(y_N,\varepsilon/2)$ are disjoint.

If $\varepsilon<1$, we have the volume estimate
\[
\vol\big(\Omega+B(0,1/2)\big)
\ge \vol\Big(\bigcup_{j=1}^N B(y_j,\varepsilon/2)\Big)
= N (\varepsilon/2)^d \vol(\IB_2^d).
\]
We refer to Definition~\ref{def:minkowski} in Appendix~\ref{sec:iso-proof-upper} for the sum of sets.

If we choose, for sufficiently large $n$, 
\[
\varepsilon=2\,\vol(\IB_2^d)^{-1/d}\vol\big(\Omega+B(0,1/2)\big)^{1/d}n^{-1/d},
\]
it follows that $N\le n$ and therefore $P_N=\{y_1,\dots,y_N\}$ has at most $n$ points and every point of $\Omega$ is not more than $\varepsilon$ from the nearest point of $P_N$. This proves the upper bound. For completeness, adding arbitrary points ensures that $\#P_N=n$.

For the proof of the lower bound, let $0<\gamma<\infty$ and fix 
\[
\delta=\big(\vol(\Omega)/2\,\vol(\IB_2^d)\big)^{1/d} n^{-1/d}\quad \text{ such that }\quad n\,\vol\big(B(0,\delta)\big)= \vol(\Omega)/2.
\]
Then for any $\Pn=\{x_1,\dots,x_n\}\subset \IR^d$ the set 
\[
\Omega_{\Pn}:=\Omega\setminus \bigcup_{i=1}^n B(x_i,\delta)
\]
has volume at least $\vol(\Omega)/2$ and $\dist(x,\Pn)^{\gamma}\ge \delta^{\gamma}$ for each $x\in \Omega_{\Pn}$. Integrating over $\Omega_{\Pn}$ yields
\[
\|\dist(\cdot,\Pn)\|_{L_{\gamma}(\Omega)}^{\gamma}
\ge \int_{\Omega_{\Pn}}\dist(x,\Pn)^{\gamma}\dd x
\gtrsim n^{-\gamma/d},
\]
concluding the proof.
\end{proof}

Next, we provide the proof for the claimed necessary condition in Example~\ref{ex:big-hole} for the optimality of a sequence of point sets.

\begin{proof}[Proof of condition \eqref{eq:condition-largest-hole}]
Assume that for each $n$ the ball $B_n:=B(y_n,r_n)$ with $y_n\in\Omega$ and $r_n>0$ does not contain any points of $\Pn$ and that the points of $\Pn$ cover $\Omega\setminus B_n$ nicely, i.e., the covering radius of $\Pn$ in $\Omega\setminus B_n$ is of order $n^{-1/d}$. Then we have
\[
\int_{\Omega}\dist(x,\Pn)^{\gamma}~{\rm d}x \,\ls\, n^{-\gamma/d}+r_n^{\gamma+d}\quad \text{for all }0<\gamma<\infty.
\]
Rearranging, this means that if $r_n\lesssim n^{-1/d+1/(\gamma+d)}$, then $\|\dist(\cdot,\Pn)\|_{L_{\gamma}(\Omega)}\lesssim n^{-1/d}$ for the sequence $(\Pn)_{n\in\IN}$.   

To show that this condition is also necessary, note that $\dist(x,\Pn) \ge r_n/2$ for every $x\in\Omega \cap B(y_n,r_n/2)$ and $n\in\IN$. The interior cone condition satisfied by $\Omega$ implies that there is a constant $c>0$ depending only on $\Omega$ such that $\vol\big(\Omega \cap B(y_n,r_n/2)\big)\ge c\, r_n$, and thus
\[
\int_{\Omega}\dist(x,\Pn)^{\gamma}~{\rm d}x \,\gs\, r_n^{\gamma+d}.
\]
Letting $1\le q<p \le \infty$ and $\gamma$ as above, we obtain from Corollary \ref{cor:sob-char} that $(\Pn)_{n\in\IN}$ is asymptotically optimal for $L_q$-approximation on $W_p^s(\Omega)$ if and only if condition \eqref{eq:condition-largest-hole}, that is,
\begin{equation*}
r_n\ls n^{-1/d+1/(\gamma+d)}
\end{equation*}
is satisfied.
\end{proof}
\section{Approximation and extension of Sobolev functions} \label{sec:sob-proof-poly}

In the following, let $\Omega\subset \IR^d$ be a domain as well as $1\le p\le \infty$ and $s\in\IN$ \chg{as in \eqref{eq:embedding}}. Recall from Definition~\ref{def:sobolev} that the Sobolev norm of a function $f\in W^s_p(\Omega)$ is given by
\[
\|f\|_{W^s_p(\Omega)}=\Big(\sum_{|\alpha|\leq s} \|D^{\alpha}f\|_{L_p(\Omega)}^p\Big)^{1/p},
\]
where the sum is over all multi-indices $ \alpha\in \mathbb{N}_0^d $ with sum of entries $ |\alpha|=\alpha_1+\ldots+\alpha_d $ at most $s$.  The seminorm
\begin{equation} \label{eq:sob-seminorm}
 |f|_{W^s_p(\Omega)}:= \Big(\sum_{|\alpha|= s} \|D^{\alpha}f\|_{L_p(\Omega)}^p\Big)^{1/p}
\end{equation}
is zero on polynomials of degree at most $s-1$ and will be used to bound the error by polynomial approximation. 

Variants of the following result about optimal polynomial approximation on cubes are well known, see, e.g., the references provided in \cite[Sec.~1.3.12]{Nov88}. The proof given here is not adapted from \cite[Lem.~9]{KS20}, where we stated it originally, but from the similar \cite[Lem.~4]{KNS21}.
\begin{lemma}\label{lem:sob-poly}
	For any $1\le p\le \infty$ and $s\in \IN$ \chg{as in \eqref{eq:embedding}}, there is a constant $\cpoly>0$ such that the following holds.  For any $0< \varrho \le 1$, any cube $Q(x,\varrho)$ with center $x\in\IR^d$ and radius $ \varrho $, and any $f\in W^s_p(\IR^d)$, there is a polynomial $\pi$ of degree at most $s$ such that
\[
 \sup_{x\in Q(x,\varrho)}\vert(f - \pi)(x)\vert \le \cpoly \,  \varrho ^{s-d/p} \vert f \vert_{W_p^s(Q(x,\varrho))}.
\]
\end{lemma}
\begin{proof}
We use an affine map $T$ from $Q=Q(x,\varrho)$ to the reference cube $Q_0=Q(0,1)$ with $T(y)=\varrho^{-1}(y-x)$ for every $y\in Q$. Applied to the function $f\circ T^{-1}\colon Q_0\to\IR$, the generalized Poincar\'e inequality from \cite[1.1.11]{Maz85} implies that there exists a constant $c_1>0$ independent of $f$ and $Q$ such that we find a polynomial $\pi$ of degree at most $s-1$ with
	\begin{equation}\label{eq:poincare}
	\|f\circ T^{-1} - \pi\|_{W^s_p(Q_0)}
	\, \le \, c_1\, |f\circ T^{-1}|_{W^s_p(Q_0)}.
\end{equation}
A change of variables, together with the chain rule and integration by substitution, gives another constant $c_2> 0$ such that
 \begin{equation}\label{eq:cov}
	|f\circ T^{-1}|_{W^s_p(Q_0)}
	\, \le \, c_2\,\varrho^{s-d/p}\,|f|_{W^s_p(Q)}.
\end{equation}
The continuous embedding of $W^s_p(Q_0)$ into $C_b(Q_0)$, see the introduction to Chapter~\ref{ch:sob}, yields $c_3>0$ with
	\[
	\Vert f - \pi' \Vert_{L_\infty(Q)}
	\,=\, \|f\circ T^{-1} - \pi \|_{L_{\infty}(Q_0)}
	\,\le\,c_3\,\|f\circ T^{-1} - \pi \|_{W^s_p(Q_0)},
	\]
	where $\pi'=\pi\circ T$ is again a polynomial of degree at most $s-1$.  Combining this with the estimates \eqref{eq:poincare} and \eqref{eq:cov} yields 
	\[
	\Vert f - \pi' \Vert_{L_\infty(Q)}
	\,\le\,\cpoly\,\varrho^{s-d/p}|f|_{W^s_p(Q)},
	\]
	where $\cpoly:=c_1\,c_2\,c_3$ is independent of $f$ and $Q$. The continuity of $f-\pi'$ concludes the proof.
\end{proof}

The reason that Lemma~\ref{lem:sob-poly} is stated for functions defined on the whole space $\IR^d$ is that we can extend the functions from $W^s_p(\Omega)$ to $W^s_p(\IR^d)$, when $\Omega$ is a bounded convex domain, which is the content of well-known Sobolev extension theorems. See, e.g., the historical comments in \cite[Sec.~1.1.19]{Maz85}. To state a version suitable for our needs, we introduce the notion of a bounded Lipschitz domain, which is \chg{a popular} regularity condition under which Sobolev embedding and extension theorems hold. The following definition was used in \cite{KS20} and taken from \cite{NT06}.

\begin{definition}\label{def:bounded-lipschitz}
A bounded Lipschitz domain $\Omega$ is a domain with the following property. There are points $x_1,\ldots,x_N\in\partial\Omega$ on the boundary and radii $r_1,\ldots,r_N>0$ such that $\partial\Omega$ is covered by the balls $B(x_1,r_1),\ldots,B(x_N,r_N)$ and
\[
\Omega\cap B(x_i,r_i)= \Omega_i\cap B(x_i,r_i), \quad i=1,\ldots,N,
\]
where $\Omega_i$ is a suitable rotation of a special Lipschitz domain in $\mathbb{R}^d$. Here, a special Lipschitz domain in $\mathbb{R}^d$, $d\geq 2$, is the collection of all points $x=(x',x_d)$ with $x'\in\mathbb{R}^{d-1}$ such that
\[
h(x')<x_d<\infty,
\]
where $h\colon\mathbb{R}^{d-1}\to \IR$ is some Lipschitz function, i.e., there exists a constant $C>0$ with $|h(x')-h(y')|\leq C\|x'-y'\|_2$ for all $x',y'\in\mathbb{R}^{d-1}$. 
\end{definition}
The following extension theorem is taken from Stein \cite[Sec.~VI.3]{Ste71}.
\begin{lemma} \label{lem:sob-extension}
Let $\Omega \subset \IR^d$ be a bounded Lipschitz domain, $1\le p\le \infty$ and $s\in \IN$.  Then there is a bounded linear operator ${\rm ext}\colon W^s_p(\Omega) \to W^s_p(\IR^d)$ with ${\rm ext}(f)\vert_\Omega = f$ for all $f\in W^s_p(\Omega)$.
\end{lemma}

Note that Stein uses the notion of a minimally smooth domain which, however, entails bounded Lipschitz domains. Further, the following lemma shows that a bounded convex domain is a bounded Lipschitz domain. As we could not find a suitable proof of this intuitive fact, we provide one ourselves. A similar result may be found in Dekel and Leviatan \cite[Lem.~2.3]{DL04}.
\begin{lemma}\label{lem:convex-lip}
Every bounded convex domain is a bounded Lipschitz domain.
\end{lemma}

\begin{proof}
If $\Omega\subset \mathbb{R}^d$ is a bounded convex domain, we find a ball $B(x_0,r)\subset \Omega$. As $\partial\Omega$ is compact, there are points $x_1,\ldots, x_N$ such that the balls $B(x_i,r/2)$, $i\le N$, cover $\partial\Omega$.  For all $i\le N$ we show that 
\begin{equation}\label{eq:Lip}
 \Omega\cap B(x_i,r/2)=\Omega_i\cap B(x_i,r/2),
\end{equation} 
where $\Omega_i$ is a rotation of a special Lipschitz domain in $\mathbb{R}^d$.  Applying a suitable rotation (and translation), we may assume that $x_0=0$ and that $x_i=(0,\hdots,0,a)$ for some $a\le -r$. Consider the open ball $B'=B^{d-1}(0,r)$ in $\IR^{d-1}$.  For $x'\in B'$, we define the set $A(x') = \{x_d\in\IR \colon (x',x_d)\in \Omega\} $.  Since $\Omega$ is convex and open, $A(x')$ is an open interval.  Moreover, $A(x')$ is non-empty since $0\in A(x')$.  We define $h(x')$ to be the infimum of $A(x')$.  The convexity of $\Omega$ implies that the function $h\colon B'\to \IR$ is convex.  Since every convex function on a convex domain in $\IR^{d-1}$ is Lipschitz on every compact subset of the domain, see \cite{WSU72}, the function $h$ is Lipschitz on the $d-1$-dimensional ball $B^*=B^{d-1}(0,r/2)$. This Lipschitz continuity carries over to the whole $\IR^{d-1}$ if we set $h(\lambda x')=h(x')$ for all $x' \in \partial B^*$ and $\lambda \ge 1$ (thereby redefining $h$ on $B'\setminus B^*$).  It remains to note that for every $x=(x',x_d)\in B(x_i,r/2)$ it holds that
\[
 x \in \Omega 
 \quad\Leftrightarrow\quad 
 x_d \in A(x')
 \quad\Leftrightarrow\quad 
 x_d > h(x'),
\]
proving \eqref{eq:Lip} for the special Lipschitz domain $\Omega_i=\{(x',x_d)\colon x_d>h(x')\}$.
\end{proof}

Moreover, let us mention that a bounded Lipschitz domain satisfies an interior cone condition. This seems to be well known and is used for example in \cite{NT06}, see also \cite[4.11]{AF03}. However we could not find a proof of this fact. To convince ourselves, we gave a proof of it in \cite[Lem.~5]{KS20} which we will not repeat here.

\section{The proof of the characterization if $q\ge p$} \label{sec:sob-proof-ge}

Let $\Omega\subset\IR^d$ be a bounded convex domain and $\Pn=\{x_1,\dots,x_n\}\subset\Omega$. Further, let the parameters $p,q,s$ be as in the statement of Theorem~\ref{thm:sob-main} with $q\ge p$ and fix \chg{$ r\in (0,1]$} such that $\Omega$ contains a closed ball of radius $r$.  By Lemma~\ref{lem:convex-cone} the set $\Omega$ satisfies an interior cone condition as in Definition~\ref{def:cone-condition} with radius $r>0$ and some angle $\theta\in (0,\pi/2)$. In the following, all constants are allowed to depend on the parameters mentioned and in particular on the domain $\Omega$ and its dimension $d$. 

For the proof of the upper bound of Theorem~\ref{thm:sob-main} we apply Lemma~\ref{lem:wendland} to the point set $\Pn$ and the closure $\overline{\Omega}$. To this end, we first deduce Lemma~\ref{lem:wendland} and then give a proof of the fact that the closure of a set inherits the interior cone condition.

\begin{proof}[Proof of Lemma~\ref{lem:wendland}]
	We will show that for any point set $\Pn\subset K$ there is a subset $X\subset \Pn$ with $h_{X, K}\le 2 h_{\Pn, K}$ and separation distance $q_X\ge \frac 12 h_{\Pn,K}$. Then the condition \eqref{eq:quasi-uniform} is satisfied for $X$ and we can apply \cite[Thm.~4.7]{Wen04} to the subset $X$. We then set the functions $u_i$ for every $x_i\in \Pn \setminus X$ equal to the zero function. 

 Take $y_1=x_1\in \Pn$ and iteratively choose $y_i \in \Pn\setminus \bigcup_{j<i} B(y_j,h_{\Pn, K})$ until this set difference is empty. In this way, we obtain a subset $X$ of $\Pn$ which satisfies $q_X\ge \frac 12 h_{\Pn, K}$. Moreover, for any $x\in  K$, there is some $x_i\in \Pn$ with $\Vert x- x_i \Vert_2 \le h_{\Pn, K}$. Since $\Pn\setminus \bigcup_{y \in X} B(y,h_{\Pn, K})=\emptyset$, there is some $y \in X$ with $\Vert x_i - y \Vert_2 \le h_{\Pn, K}$. The triangle inequality gives $\Vert x - y\Vert_2 \le 2\, h_{\Pn, K}$ and therefore $h_{X, K}\le 2\, h_{\Pn, K}$.
\end{proof}

The following proof ensures that the closure $\overline{\Omega}$ satisfies an interior cone condition with $r$ and $\theta$ as above.

\begin{proof}[Proof of Lemma~\ref{lem:closure-CC}]
	Assume that $\Omega\subset\IR^d$ satisfies an interior cone condition with radius $r$ and angle $\theta$. For each limit point $x$ in the closure $\overline{\Omega}$ consider a sequence $(x_n)_{n\in\IN}\subset \Omega$ converging to it and let $C(x_n):=C(x_n,\xi(x_n),r,\theta)\subset \Omega$, where $n\in\IN,$ be the corresponding sequence of cones given by the interior cone condition. As the sphere is a compact metric space when equipped with the geodesic distance, the sequence $\xi(x_n)\in\IS^{d-1}$ of directions contains a convergent subsequence. Upon passing to this subsequence we may assume that $\xi(x_n)$ is convergent itself and define $\xi(x)$ to be its limit point. Then the cone $C(x):=C(x,\xi(x),r,\theta)$ is contained in $\overline{\Omega}$. To give a proof by contradiction, suppose that this does not hold, i.e., there is a point $y\in C(x)$ with positive distance from $\overline{\Omega}$. The convergence of $\xi(x_n)\to \xi(x)$ and of $x_n\to x$ implies that every point of $C(x)$ will be arbitrarily close to some point from $C(x_n)$ whenever $n$ is large enough. This is a contradiction to $y\in C(x)$ having a positive distance from $\overline{\Omega}$ and thus of every $C(x_n)$, where $n\in\IN$.
\end{proof}

After these preparations, we prove the already known part of our main result. 
\begin{proof}[Proof of the upper bound of Theorem~\ref{thm:sob-main} in the case $q\ge p$]
	
	Let $m>s$ be an integer and let $c_0,c_1,c_2>0$ be as in Lemma~\ref{lem:wendland}. In order to apply it to the point set $\Pn=\{x_1,\dots,x_n\}$ and $\overline{\Omega}$ we need to have $ h_{\Pn,\overline{\Omega}}\le c_1  r $. Without loss of generality we can assume that this is satisfied as the upper bound becomes trivial if this does not hold; compare with Lemma~\ref{lem:sob-trivial} in Appendix~\ref{sec:sob-proof-le} below.  

We consider the linear algorithm
\[
S_{\Pn}\colon W_p^s(\Omega) \to L_q(\Omega), \quad S_{\Pn}(f) = \chg{\sum_{i=1}^n} f(x_i) u_i,
\]
with $u_i$ provided by Lemma~\ref{lem:wendland} for the compact set $ \overline{\Omega} $.

For technical reasons we need a suitable covering of the domain by pieces of radius of order $h_{\Pn,\Omega}$. In order to find such a covering, consider the balls $B(y,h_{\Pn,\Omega})$, $y\in \Omega$.  By compactness we can extract a finite subset $I\subset \Omega$ such that the balls with $y\in I$ form a covering of $\overline{\Omega}$. Because of a well-known finite covering lemma due to G.~Vitali we can select from this finite covering a subset of pairwise disjoint balls $B_i:=B(y_i,h_{\Pn,\Omega})$ such that the balls $3B_i:=B(y_i,3h_{\Pn,\Omega})$, where $i=1,\dots,N$, cover $\overline{\Omega}$.

Let $f\in W_p^s(\Omega)$ with $\Vert f \Vert_{W_p^s(\Omega)} \le 1$.  By Lemma~\ref{lem:sob-extension}, we can replace $f$ with its extension and assume that $f\in W_p^s(\IR^d)$ with $\Vert f \Vert_{W_p^s(\IR^d)} \le C_1$, where $C_1>0$ is independent of $f$. Fix some $i\le N$ and consider the cube $Q_i$ centered at $y_i$ with radius $(3+c_2)h_{\Pn,\Omega}$.  By Lemma~\ref{lem:sob-poly}, we find a polynomial $\pi_i$ of degree at most $s$ such that
\begin{equation*}
 \sup_{y \in Q_i} \big\vert (f-\pi_i)(y)\big\vert \le \cpoly\, h_{\Pn,\Omega}^{s-d/p} \vert f \vert_{W_p^s(Q_i)}.
\end{equation*}
For each $y\in \Omega_i:= \Omega\cap 3B_i$, we note that $u_i(y)=0$ for $x_i\not\in Q_i$, and obtain
\begin{align}
\nonumber
\big\vert (f - S_{\Pn} f)(y) \big\vert &= \big\vert (f - \pi_i)(y) - S_{\Pn}(f-\pi_i)(y) \big\vert \\
 \label{eq:polyapp}
 &\le \big\vert (f- \pi_i) (y)\big\vert + \Big\vert \sum_{i=1}^n (f-\pi_i)(x_i) u_i(y) \Big\vert 
 \le C_2\, h_{\Pn,\Omega}^{s-d/p} \vert f \vert_{W_p^s(Q_i)},
\end{align}
where $C_2=(1+c_0) \cpoly$. In particular, this yields
\[
\Vert f - S_{\Pn}(f) \Vert_{L_\infty(\Omega)} 
 \le C_2\,  h_{\Pn,\Omega}^{s-d/p} \vert f \vert_{W_p^s(\IR^d)},
\]
which proves the case $q=\infty$. For $p \le q<\infty$, we use $ \Omega_i\subset Q_i $ and \eqref{eq:polyapp} to get
\[
	\Vert f - S_{\Pn}(f) \Vert_{L_q(\Omega)}^q 
	\le \sum_{i=1}^N \int_{\Omega_i} \big\vert \big(f - S_{\Pn}(f)\big)(y) \big\vert^q {\rm d} y
 \le C_2^q\sum_{i=1}^N \, h_{\Pn,\Omega}^{sq-dq/p} \vert f \vert_{W_p^s(Q_i)}^q \vol(\Omega_i),
\]
and thus, as a continuation,
\[
\Vert f - S_{\Pn}(f) \Vert_{L_q(\Omega)}^q
 \le C_3\, h_{\Pn,\Omega}^{sq-dq/p+d} \sum_{i=1}^N \vert f \vert_{W_p^s(Q_i)}^q
 \le C_3\, h_{\Pn,\Omega}^{sq-dq/p+d} \Big(\sum_{i=1}^N \vert f \vert_{W_p^s(Q_i)}^p\Big)^{q/p}.
\]
It remains to bound the sum on the right-hand side by a constant. A volume argument shows that there exists $M\in\IN$ such that for all $\Pn\subset\Omega$ every $x\in \IR^d$ is contained in at most $ M\in \mathbb{N} $  of the $ N $ cubes $Q_i$. To see this, note that the balls $B_i\subset Q_i$ are disjoint and of radius $h_{\Pn,\Omega}$. Therefore, 
\begin{equation}\label{eq:norms-efficient}
 \sum_{i=1}^N \vert f \vert_{W_p^s(Q_i)}^p
 = \sum_{\vert \alpha \vert =s} \int_{\IR^d} \Big( \vert D^\alpha f(x) \vert^p \sum_{i=1}^N  \mathbf{1}_{Q_i}(x)\Big) {\rm d} x
 \le M\, \vert f \vert_{W_p^s(\IR^d)}^p 
 \le M C_1^p
\end{equation}
and we arrive at the desired inequality. 
\end{proof}

As described in Section~\ref{sec:sob-fool-int}, the lower bound will be proven with the help of a fooling function $f_\ast$ from the unit ball of $W_p^s(\Omega)$ which has a large $L_q$-norm and is supported in the largest hole of the point set. 

\bigskip

\noindent\textit{Proof of the lower bound of Theorem~\ref{thm:sob-main} in the case $q\ge p$.} 
Let $\Pn\subset\Omega$ be arbitrary and choose $ x_0\in \Omega $ such that 
\begin{equation} \label{eq:large-hole}
\dist(x_0,\Pn)\ge r_0 := \min\{ r ,h_{\Pn,\Omega}/2\}.
\end{equation}
 Then the ball $ B(x_0,r_0)$ does not contain any point of $ \Pn $ and is a large hole in which we can place a fooling function. However, this ball need not be contained in $\Omega$ . To remedy this, we will find a ball in the intersection $\Omega\cap B(x_0,r_0)$ with radius proportional to $r_0$. This is the content of Lemma~\ref{lem:ballinball} below, for which we need the following special case of \cite[Lem.~3.7]{Wen04}.

\begin{lemma}\label{lem:ballincone}
Every cone $C(x,\xi,r,\theta)$ contains a closed ball of radius $c_{\theta}r$ with $c_{\theta}:=\frac{\sin \theta}{1+\sin\theta}$.
\end{lemma}
With this geometric fact we obtain the following simple but useful consequence.
\begin{lemma}\label{lem:ballinball}
	Let $ \Omega\subset \mathbb{R}^d $ satisfy an interior cone condition with parameters $ r $ and $ \theta $. If $ B(x,\varrho)$ is a ball with center $ x\in\Omega $ and radius $ 0<\varrho\leq r $, there is a ball $ B(y,c_{\theta}\varrho)$ contained in $\Omega\cap B(x,\varrho)$ with $c_{\theta} $ as in Lemma \ref{lem:ballincone}.
\end{lemma}
\begin{proof}
	By the interior cone condition, there is a cone with apex $x$, radius $\varrho$ and angle $\theta$ such that its interior is contained in $ \Omega\cap B(x,\varrho) $. Now Lemma \ref{lem:ballincone} completes the proof.
\end{proof}

Returning to the proof of the lower bound of Theorem~\ref{thm:sob-main}, consider again the empty ball $B(x_0,r_0)$, where $r_0$ is as in \eqref{eq:large-hole}. By Lemma \ref{lem:ballinball} we find a ball $B(y, \varrho)$ with $ \varrho := c_\theta r_0, $ which is contained in $ \Omega\cap B(x_0,r_0) $.  Now, take a smooth non-negative function $ \varphi $ supported in $ B(0,1) $ with $ \varphi(0)=1 $ and consider the function $x\mapsto f_*(x)= \varphi\big(\varrho^{-1}(x-y)\big)$, which is supported in $B(y,\varrho)$.  One can easily compute the scaling properties
\begin{equation}\label{eq:scaling-LB}
	\|f_*\|_{L_q(\Omega)}\asymp \varrho^{d/q}\qquad 
	\text{and} \qquad 	\|f_*\|_{W^s_p(\Omega)} \ls \varrho^{-s+d/p}
\end{equation}
by taking into account that $\varrho \le 1$ and \chg{$s\ge d/p$}, see also the proof of \cite[Thm.~4.11]{HKN+20}. Here, the implicit constants are independent of $\varrho$, but depend on the choice of $\varphi$. Replacing $f_*$ by its normalization  $ f_*/\|f_*\|_{W^s_p(\Omega)}$, it satisfies 
	\[ 
	\|f_*\|_{W^s_p(\Omega)}\le 1, \quad 
	f_*|_{P_n}=0
  \quad\text{and}\quad
  \|f_*\|_{L_q(\Omega)} \gs h_{\Pn,\Omega}^{s-d(1/p-1/q)}. 
	\]
This follows after inserting for $\varrho$ in \eqref{eq:scaling-LB} and noting that $\varrho=c_{\theta}\min\{ r , h_{\Pn,\Omega}/2\}$ as well as that $h_{\Pn,\Omega}$ is bounded from above. Since $f_*\ge 0$, this completes the proof also for the integration problem. 
	$\hfill \square$

\section{The proof of  the characterization if $q< p$} \label{sec:sob-proof-le}

Let $\Omega,\Pn,p,q,s,r$ and $\theta$ as in the previous section but now with $q<p$.  Before we give the proof of Theorem~\ref{thm:sob-main}a, we prove the auxilary results Lemma~\ref{lem:local-CC}, Lemma~\ref{lem:local-estimate} and Proposition~\ref{pro:good-cover} from Section~\ref{sec:sob-mls}. 

For the proof of Lemma~\ref{lem:local-CC} we shall need that a full-dimensional convex set satisfies an interior cone condition. This has been stated in Lemma~\ref{lem:convex-cone} and is a consequence of \cite[Prop.~11.26]{Wen04} which, in fact, holds for more general sets, which are star-shaped with respect to a ball. Here, a set is called star-shaped with respect to a ball if the line connecting any point in this ball with any point in the set itself is fully contained in it. In particular, every convex set is star-shaped with respect to any ball inside it.

\begin{proof}[Proof of Lemma~\ref{lem:local-CC}]
	Let $x\in\Omega$ and $0<\varrho\le r$. By Lemma~\ref{lem:convex-cone}, the set $\Omega$ contains a cone with radius $ \varrho $, apex $x$ and angle $\theta$.  Clearly, this cone is contained in $\overline{\Omega \cap B(x, \varrho )}$, which in turn is a subset of $A(x, \varrho ):=\overline{\Omega\cap Q(x, \varrho )}$.  By Lemma~\ref{lem:ballincone}, there is a closed ball of radius $c_{\theta} \varrho $ in this cone and thus in $A(x, \varrho )$.  The proof is finished if we apply Lemma~\ref{lem:convex-cone} to the convex set $A(x, \varrho )$ since its diameter is at most $2\varrho\sqrt{d}$ and therefore we can set $\theta'=2\arcsin(c_{\theta}/4\sqrt{d})$.
\end{proof} 

Together with Lemma~\ref{lem:ballinball} from Section~\ref{sec:sob-proof-ge} we can deduce the proof of the local approximation result which gave rise to the notion of a good cube.

\begin{proof}[Proof of Lemma~\ref{lem:local-estimate}]
	 Let $Q=Q(x,\varrho)$ for some $x\in\Omega$ and $0<\varrho\le r$. By Lemma~\ref{lem:local-CC} the set $ A(x,\varrho):=\overline{\Omega \cap Q}$ satisfies an interior cone condition with radius $ \varrho '=c_\theta  \varrho $ and angle~$\theta'$.  
	
	Let $m\in \IN$ be the smallest integer greater than $s$ and let $c_0,c_1>0$ be as in Lemma~\ref{lem:wendland} for the parameters $\theta'$ and $m$.  We can assume without loss of generality that $c_1<1$.  We set 
\begin{equation} \label{eq:cgood}
\cgood:=c_{\theta} c_{\theta'} c_1/2,
\end{equation}
where $c_{\theta'}=\frac{\sin\theta'}{1+\sin\theta'}$ is from Lemma~\ref{lem:ballincone}.  Then assumption \eqref{eq:good-cube} implies that every ball $B(y,2\cgood  \varrho )$, $y\in \Omega\cap Q$, contains a point of $\Pn$.  By Lemma \ref{lem:ballinball} applied to $A(x,\varrho)$ every ball $ B(x,2c_{\theta'}^{-1}\cgood  \varrho ), x\in \Omega\cap Q $, contains a ball $ B(y,2\cgood  \varrho )\subset A(x,\varrho)$  and therefore a point of $\Pn\cap Q$.  We thus have
\[
h_{\Pn\cap Q,\Omega\cap Q}
=\sup_{y\in \Omega\cap Q}\dist(y,\Pn\cap Q) 
\leq 2c_{\theta'}^{-1} \cgood   \varrho  
= c_1  \varrho '.
\]
Thus, we may apply Lemma~\ref{lem:wendland} to the point set $\Pn\cap Q$ within the compact set $A(x,\varrho)$ and obtain bounded continuous functions $u_i\colon \Omega \cap Q \to \mathbb{R}$ for $i=1,\dots,n$ with 
\[
 \sum_{i=1}^n |u_i(y)|\leq c_0 \qquad\text{and}\qquad \pi(y)=\sum_{i=1}^n \pi(x) u_i(y)
 \] 
for all $\pi\in\mathcal{P}_m^d$ and $y\in \Omega\cap Q$, where we set $u_i=0$ for $x_i \in \Pn\setminus Q$.  For any $f\in W^s_{p}(\IR^d)$, by Lemma \ref{lem:sob-poly} there is a polynomial $\pi\in \mathcal{P}_m^d$ with
\[
\sup_{y\in Q} \big\vert (f-\pi)(y)\big\vert 
\,\le\, \cpoly  \varrho ^{s-d/p}|f|_{W^s_{p}(Q)}.
\]
Here, we used that $\varrho\le r\le 1$. Similar to the proof of the upper bound of Theorem~\ref{thm:sob-main} in the case $q\ge p$ in Section~\ref{sec:sob-proof-ge}, we get for all $y\in \Omega\cap Q$ that
\begin{align*}
\Big\vert f(y) - \sum_{i=1}^n f(x_i) u_i(y) \Big\vert
& = \Big\vert (f-\pi)(y) - \sum_{i=1}^n (f-\pi)(x_i) u_{i}(y) \Big\vert\\
& \,\le\, \big\vert (f-\pi)(y) \big\vert + c_0\, \max_{x_i\in \Pn\cap Q} \big\vert (f-\pi)(x_i) \big\vert\\
& \,\le\, (1+c_0)\cpoly\,  \varrho ^{s-d/p}|f|_{W^s_{p}(Q)},
\end{align*}
as it was to be proven.
\end{proof}

We state and prove the following justification for $r_{\Pn}(x)\le r$ for all $x\in\Omega$, which ensures the integrity of Definition~\ref{def:small-cube}.
\begin{lemma} \label{lem:sob-trivial}
Let $\Pn\subset\Omega$ be a non-empty finite point set. If, for some $x\in\Omega$, we can not find $\varrho\in(0,r)$ such that \eqref{eq:good-cube}, i.e.,
\[
\sup_{y\in \Omega\cap Q(x, \varrho )} \dist(y,\Pn) \le \cgood  \varrho 
\]
holds, then $\|\dist(\cdot,\Pn)\|_{L_{\gamma}(\Omega)}\ge c$ for any $0<\gamma\le \infty$ and some constant $c>0$ independent of the point set $\Pn$.
\end{lemma}
\begin{proof}
Because of \eqref{eq:dist-norms} in Appendix~\ref{sec:sob-proof-discuss} it is sufficient to consider finite $\gamma$. By assumption we find $y\in\Omega$ with $\dist(y,\Pn)> \cgood r/2$, that is the ball $B(y,\cgood r/2)$ is empty of $\Pn$. By Lemma~\ref{lem:ballinball} we find a ball $B(z,r')$ of radius $r'=c_{\theta}\cgood r/2$ contained in $\Omega\cap B(y,\cgood r/2)$. Hence, the distance function satisfies $\dist(x,\Pn)\ge r'/2$ for all $x\in B(z,r'/2)$. Raising this inequality to the power of $\gamma$ and integrating over $\Omega$ yields the desired result.
\end{proof}

The following lemma is used to show that $Q_{\Pn}(x)$ is a good cube for every $x\in\Omega$.
\begin{lemma} \label{lem:inf-min}
Let $x\in\Omega$. Then 
\[
\sup_{y\in\Omega\cap Q(x,r_{\Pn}(x))}\dist(y,\Pn) \le \cgood\,r_{\Pn}(x).
\]
\end{lemma}
\begin{proof}
	Take a sequence $(\varrho_n)_{n\in\IN}\subset (0,r)$ converging from above to $r_{\Pn}(x)$ such that $Q(x,\varrho_n)$ satisfies \eqref{eq:good-cube}, that is,
\[
 \sup_{y\in \Omega\cap Q(x,\varrho_n)} \dist(y,\Pn) \le \cgood\, \varrho_n. 
\]
As $Q\big(x,r_{\Pn}(x)\big)\subset Q(x,\varrho_n)$ for every $n$, taking the infimum completes the proof.
\end{proof}

Before we come to the main part of the proof of the upper bound of Theorem~\ref{thm:sob-main} in the case $q<p$, we need to validate the covering by good cubes given in Proposition~\ref{pro:good-cover}. As mentioned, we will make use of a Besicovitch-type covering result in \cite[Ch.~1,Thm.~1.1]{DeG75}. Remark (2) directly below it allows us to consider open cubes instead of closed ones. For convenience we rephrase the statement here.

\bigskip

\begin{quote}
Consider a bounded set $D$ in $\IR^d$ and for each point $x\in D$ a cube $Q\big(x,r(x)\big)$ centered at this point with some radius $r(x)>0$. Then one can choose a (possibly finite) sequence among the given cubes which is a cover of $A$, can be distributed in finitely many families, each consisting of pairwise disjoint cubes, and every point of $\IR^d$ is contained in at most a constant many cubes of the sequence. These constants depend only on the dimension.
\end{quote}

\bigskip

Similar to the proof we provided in \cite{KS20} to find a suitable covering, the idea behind the proof is to iteratively subtract from $D$ a cube of maximal size. 

\begin{proof}[Proof of Proposition~\ref{pro:good-cover}]
	For each point set $\Pn$, we will apply \cite[Ch.~1,Thm.~1.1]{DeG75} to the bounded convex domain $\Omega$ and the collection $Q_{\Pn}(x)$ of small good cubes. In this way, we obtain a sequence $y_1,y_2,\dots$ of points belonging to $\Omega$ such that (i), (ii) and (iii) in Proposition~\ref{pro:good-cover} are satisfied with possibly infinite $N$. The proof is complete if we can show that the sequence is indeed finite. This is what we show in the following using a volume argument.

	Consider the families $\cQ_j, j=1,\dots,M_1$ of pairwise disjoint cubes into which the set $Q_{\Pn}(y_i),i\in\IN,$ is distributed. We show that an arbitrary such family $\cQ_{j_0}$ is finite and thus the whole collection $Q_{\Pn}(y_i),i\in\IN$. All the cubes $Q\in \cQ_{j_0}$ are pairwise disjoint and contained in the neighborhood $\Omega_r:=\{y\in\IR^d\colon \exists x \text{ with } \|x-y\|_{\infty}\le r\}$. It is sufficient to show that there is a constant $c>0$ such that $\vol(Q)\ge c$ for all of these cubes since then
\[
\#\cQ_{j_0}\, c\le \sum_{Q\in \cQ_{j_0}}\vol(Q)\le \vol(\Omega_r)
\]
implies that the family has to be of finite cardinality and that the sequence $y_1,y_2,\dots$ is in fact finite and of length $N$ for some $N\in\IN$. 

We will show that in fact that there exists $c>0$ such that $\vol\big(Q_{\Pn}(x)\big)\ge c$ for every $Q_{\Pn}(x)$ with $x\in\Omega$. For this, note that any such good cube contains at least two points of $\Pn$.  In particular, for any $x\in\Omega$,
\[
0<\frac{q_{\Pn}}{\sqrt{d}}<r_{\Pn}(x)\le r,
\]
where $q_{\Pn}=\frac{1}{2}\min_{i\neq j}\|x_i-x_j\|_2$ is the separation distance of $\Pn=\{x_1,\dots,x_n\}$. This gives a uniform lower bound of the form $\vol\big(Q_{\Pn}(x)\big)=\big(2\,r_{\Pn}(x)\big)^d\ge c:= (2\,q_{\Pn}/\sqrt{d})^d$. 
\end{proof}

The following proof is similar to that of \cite[Prop.~2]{KS20}.
\begin{proof}[Proof of Lemma~\ref{lem:good-hole}]
Let $x\in \Omega$ and $Q_{\Pn}(x)=Q\big(x,r_{\Pn}(x)\big)$. By definition, $r_{\Pn}(x)$ is the infimum over all $\varrho\in (0,r)$ such that \eqref{eq:good-cube} holds. In particular, we find a point $y\in \Omega\cap Q(x,r_{\Pn}(x)/2)$ with $\dist(y,\Pn)\ge \cgood r_{\Pn}(x)/4$ meaning that the ball $B(y,\cgood r_{\Pn}(x)/4)\subset Q_{\Pn}(x)$ is empty of $\Pn$. By Lemma~\ref{lem:ballinball} applied to $\Omega$ we obtain a ball of radius $\chole r_{\Pn}(x)$ with $\chole:=c_{\theta}c/4$ which is contained in $\Omega\cap B(y,\cgood r_{\Pn}(x)/4)$. As $\cgood<1$ (see \eqref{eq:cgood}), this ball is also contained in $\Omega\cap Q_{\Pn}(y)$. 
\end{proof}

Let us conclude this section with the proof of our main result.

\begin{proof}[Proof of the upper bound of Theorem~\ref{thm:sob-main} in the case $q<p$]
	We choose $y_i\in\Omega$, $r_i:=r_{\Pn}(y_i)$, $Q_i=Q(y_i,r_i)$ for $1\le i \le N$ as in Proposition~\ref{pro:good-cover}. To get a well-defined algorithm, we will make the covering disjoint by letting
 \[
 \Omega_i
 :=(Q_i\cap\Omega)\setminus \bigcup_{j<i} \Omega_j
 \]
 for all $i=1,\dots, N$. That is, $\Omega$ is the disjoint union of the sets $\Omega_i\subset Q_i$. 
 
 Then, for any $x_j\in \Pn$ and $y\in \Omega_i$, we define $u_j(y)$ according to Lemma~\ref{lem:local-estimate}, applied to the cube $Q_i$.  This yields bounded functions $u_j\colon \Omega\to \IR$ and a linear algorithm 
 \[
 S_{\Pn}\colon W^s_{p}(\Omega) \to L_q(\Omega), \quad S_{\Pn}(f)=\sum_{j=1}^n f(x_j) u_j.
 \]
 We proceed by showing that this algorithm which uses function samples at the point set $\Pn$ satisfies the required bound on the worst-case error. To this end, let $f\in W^s_{p}(\Omega)$ with $\Vert f\Vert_{W^s_{p}(\Omega)}\le 1$.  As in the proof for the case $q\ge p$, using Lemma~\ref{lem:sob-extension}, we may assume that $f\in W^s_{p}(\IR^d)$ with $\|f\|_{W^s_{p}(\IR^d)}\le C$. With the help of Lemma~\ref{lem:local-estimate} we derive the estimate
 \begin{equation*}
	 \Vert f - S_{\Pn}(f) \Vert_{L_q(\Omega_i)}^q 
 \lesssim  r_i^{(s-d/p)q+d} |f|_{W^s_{p}(Q_i)}^q,\quad i=1,\dots,N. 
 \end{equation*}
 Using this we obtain
 \[
 \Vert f - S_{\Pn}(f) \Vert_{L_q(\Omega)}^q 
 \le \sum_{i=1}^{N} \Vert f - S_{\Pn}(f) \Vert_{L_q(\Omega_i)}^q 
 \lesssim \sum_{i=1}^{N} r_i^{(s-d/p)q+d} |f|_{W^s_{p}(Q_i)}^q. 
 \]
 If $p=\infty$, we use that $|f|_{W^s_{\infty}(Q_i)}\le |f|_{W^s_{\infty}(\IR^d)}\le C$ to obtain
 \[
 \Vert f - S_{\Pn}(f) \Vert_{L_q(\Omega)}^q 
 \lesssim \sum_{i=1}^{N} r_i^{sq+d}. 
 \]
 And if $p<\infty$, we have by means of $(s-d/p)q+d=(\gamma+d)(1-q/p)$ and H\"older's inequality that
 \begin{equation*} 
	 \Vert f - S_{\Pn}(f) \Vert_{L_q(\Omega)}^q 
 \ls\, \sum_{i=1}^N r_i^{(\gamma+d)(1-q/p)}|f|_{W^s_{p}(Q_i)}^q
 \,\le\, \bigg(\sum_{i=1}^N r_i^{\gamma+d}\bigg)^{1-q/p} 
 \bigg(\sum_{i=1}^N |f|_{W^s_{p}(Q_i)}^p \bigg)^{q/p}.
 \end{equation*} 
 Since the cubes $Q_i$ form an efficient covering, we can proceed as in \eqref{eq:norms-efficient} to obtain that the second factor in the previous estimate is bounded by a constant. We will show that the first factor is bounded by $\|\dist(\cdot,\Pn)\|_{L_{\gamma}(\Omega)}^{qs}$. 

 To this end, let $B_i:=B(z_i,\chole r_i)\subset \Omega\cap Q_i$ be the ball of radius $\chole r_i$ given by Lemma~\ref{lem:good-hole} for each $i=1,\ldots,N$. Write, for every $\gamma\in (0,\infty)$, 
 \begin{equation} \label{eq:first-factor}
\sum_{i=1}^N r_i^{\gamma+d} 
=c \sum_{i=1}^{N} r_i^{\gamma}\vol(B_i'),
 \end{equation}
where $B_i':=B(z_i,\chole r_i/2)$ is the ball of half the radius which is concentric with $B_i$ and $c>0$ depends on $d$ and $\chole$. Since $B_i$ is empty of $\Pn$, the distance function satisfies the lower bound $\dist(x,\Pn)\ge \chole r_i/2$ for every $x\in B_i'$. Therefore, integrating over $B_i'$ shows that
\begin{equation} \label{eq:dist-lower-int}
(\chole r_i/2)^{\gamma}\vol(B_i')\le \int_{B_i'}\dist(x,\Pn)^{\gamma}\dd x.
\end{equation}
Combining the estimates \eqref{eq:first-factor} and \eqref{eq:dist-lower-int} yields
\[
\sum_{i=1}^N r_i^{\gamma+d} 
\lesssim \sum_{i=1}^{N} \int_{B_i'}\dist(x,\Pn)^{\gamma}\dd x
\le \int_{\Omega} \Big(\dist(x,\Pn)^{\gamma}\sum_{i=1}^{N}\mathbf{1}_{Q_i}(x)\Big)\dd x.
\]
Again the efficiency of the cover $Q_i,i=1,\ldots,N$, yields the required bound. Noting that $(1-q/p)=qs/\gamma$, this completes the proof.
\end{proof}

We give now the proof of the corresponding lower bound. For this purpose we use a fooling function supported in a disjoint subset of the holes given by Proposition~\ref{pro:good-cover} together with Lemma~\ref{lem:good-hole}.

\begin{proof}[Proof of the lower bound of Theorem~\ref{thm:sob-main} in the case $q<p$]
	To begin, choose a cover by good cubes $Q_i=Q_{\Pn}(y_i)=Q(y_i,r_i)$, where $i=1,\ldots,N$, as given by Proposition~\ref{pro:good-cover}. Let, for every $i=1,\ldots,N$, the ball $B_i:=B(z_i,\chole r_i)\subset \Omega\cap Q_i$ be as in Lemma~\ref{lem:good-hole}. 
	
	Then, for every $j=1,\ldots, M_1$, we collect the balls $B_i,i=1,\ldots,N$, such that $Q_i$ belongs to the family $\cQ_j$ into the family $\cB_j$. Thus, each ball $B_i, i=1,\ldots,N,$ belongs to exactly one of the families $\cB_j, j=1,\ldots,M_1$ and the families themselves are pairwise disjoint. To select a suitable family, note that \chg{since} a sum over some family must be at least the average over the sums over all families, there must be an index $1\le j_0\le M_1$ with  
\begin{equation} \label{eq:max-avg}
	\sum_{\substack{1\le i\le N\\ B_i\in \cB_{j_0}}} r_i^{\gamma+d}
	\ge \frac{1}{M_1}\sum_{j=1}^{M_1}\sum_{\substack{1\le i\le N\\ B_i\in \cB_{j}}}r_i^{\gamma+d}
	=\frac{1}{M_1}\sum_{i=1}^{N}r_i^{\gamma+d}.
\end{equation}
We shall use this bound later for estimating the norm of the fooling function we consider. Consider the disjoint balls belonging to the family $\cB_{j_0}$, each of which will contain part of the support of the fooling function. Let $I_0\subset\{1,\ldots,N\}$ be the subset of indices with $B_i\in \cB_{j_0}$.  To define the fooling function we use atomic decompositions as discussed by Triebel \cite[Ch.~13]{Tri11}, see Section~\ref{sec:sob-fool-int}.
	
Let $K>s$ be an integer and $Q_0=Q(0,1/2)$ be the cube of sidelength $1$ which is centered at the origin. Further, let $\psi\in C^{\infty}(\IR^d)$ be non-negative with $\mathrm{supp}\, \psi\subset Q_0$ such that $\psi(0)>0$ and $\Vert D^\alpha \psi \Vert_\infty \le 1$ for all $\alpha\in\IN_0^d$ with $|\alpha|\le K$. The exact choice of the function $\psi$ is not important for our argument. Then according to Definition~13.3 in \cite{Tri11} the function 
\[
\psi_{\nu m}\colon \IR^d\to\IR,\quad \psi_{\nu m}(x)
=2^{-\nu(s-d/p)}\psi(2^{\nu}x-m),\quad x\in\IR^d,\quad\text{where }\nu\in \IN_0, \ m\in \ZZ^d,
\]
is an $(s,p)_{K,-1}$-atom supported in $Q_{\nu m}$, the cube with center $2^{-\nu}m$ and sidelength $2^{-\nu}$.  

Choose, for every $i\in I_0$, a point $m_i\in \ZZ^d$ and $\nu_i\in \IN_0$ with $2^{-\nu_i}\le  \sqrt{d}\,\chole r_i\le 2^{-\nu_i+1}$ such that the dyadic cube $Q_i^*=Q_{\nu_i m_i}$ is contained in $B_i$. Define the fooling function
\[
f_*:=\sum_{i\in I_0}\lambda_i \psi_i,\quad \text{where }\lambda_i:=2^{-\nu_i \beta} \text{ with }\beta=(\gamma+d)/p,
\]
which is supported in $\Omega$ and satisfies $f_*|_{\Pn}=0$. The choice $\lambda_i$ will become clear in a moment. 

To estimate $\|f_*\|_{W^s_p(\Omega)}$ from above, we use Theorem 13.8 in \cite{Tri11}, which yields a constant $C>0$ depending only on $d,s,p$ such that
\begin{equation} \label{eq:atomic}
\|f_*\|_{W^s_p(\IR^d)}^p
\le C\, \|\lambda\|_{f_{pp}}^p,\quad
\text{where }\|\lambda\|_{f_{pp}}
:=\bigg(\int_{\IR^d}\Big(\sum_{i\in I_0} |\lambda_i 2^{\nu_i d/p}\mathbf{1}_{Q_i^*}(x)|^{p}\Big)\dd x\bigg)^{1/p}.
\end{equation}
Here, the symbol $f_{pp}$ denotes a (quasi-)normed space of sequences which are indexed by $\nu\in\IN_0, m\in\ZZ^d$, see \cite[Def.~13.5]{Tri11}. In the sum in \eqref{eq:atomic}, for every $x\in\IR^d$, only one summand is not equal to zero and since $\mathrm{vol}(Q_i^*)=2^{-\nu_i d}$, we have
\[
\|f_*\|_{W^s_p(\IR^d)}^p
\lesssim \sum_{i\in I_0}\int_{\IR^d}(2^{-\nu_i})^{\beta p-d}\mathbf{1}_{Q_i^*}(x)\dd x
= \sum_{i\in I_0}(2^{-\nu_i})^{\beta p}.
\]
As $\beta p=\gamma +d$ and $2^{-\nu_i}\asymp r_i$ and $\|f_*\|_{W^s_p(\Omega)}=\|f_*\|_{W^s_p(\IR^d)}$, we arrive at
\[
\|f_*\|_{W^s_p(\Omega)}^p
\lesssim \sum_{i\in I_0} r_i^{\gamma+d}.
\]
Since the balls $B_i\in \cB_{j_0}$ containing the cubes $Q_i^*$ are pairwise disjoint, a substitution shows that	
\[
\|f_*\|_{L_q(\Omega)}^q
=\sum_{i\in I_0}\int_{Q_i^*}(2^{-\nu_i})^{\beta q}\psi_i(x)^q\dd x
=\sum_{i\in I_0}(2^{-\nu_i})^{\beta q+sq-d q/p+d}\int_{Q}\psi(x)^q\dd x.
\]
Since $\gamma=s(1/q-1/p)^{-1}$ we have $\beta q+sq = \gamma +d q/p$ and thus
\[
\|f_*\|_{L_q(\Omega)}^q
\asymp\sum_{i\in I_0} r_i^{\gamma+d}.
\]
From this it follows that if we replace $f_*$ by the normalized function $f_*/\|f_*\|_{W^s_p(\Omega)}$, it satisfies
\begin{equation} \label{eq:fooling-subset}
\|f_*\|_{L_q(\Omega)}
\gtrsim \Big(\sum_{i\in I_0} r_i^{\gamma+d}\Big)^{1/q-1/p}.
\end{equation}
We now employ the bound \eqref{eq:max-avg} for the specific choice of the family $\cB_{j_0}$ to estimate
\begin{equation} \label{eq:fooling-avg}
\sum_{i\in I_0} r_i^{\gamma+d}
\ge \frac{1}{M_1}\sum_{i=1}^{N}r_i^{\gamma+d}
\ge \frac{1}{2^dM_1}\sum_{i=1}^{N}r_i^\gamma \vol(\Omega\cap Q_i),
\end{equation}
where $Q_i$ is as above. By definition, we have $\dist(x,\Pn)\le \cgood r_i$ for every $x\in \Omega\cap Q_i$. Integrating this inequality yields
\begin{equation} \label{eq:dist-upper-int}
\int_{\Omega\cap Q_i}\dist(x,\Pn)^{\gamma}\dd x
\le \cgood^{\gamma} r_i^{\gamma} \vol(\Omega\cap Q_i),\quad i=1,\ldots,N.
\end{equation}
Thus, combining the bounds \eqref{eq:fooling-avg} and \eqref{eq:dist-upper-int} gives
\[
\sum_{i\in I_0} r_i^{\gamma+d}
\gtrsim \sum_{i=1}^{N} \int_{\Omega\cap Q_i}\dist(x,\Pn)^{\gamma}\dd x
\ge \int_{\Omega}\dist(x,\Pn)^{\gamma}\dd x
\]
since the cubes $Q_i,i=1,\dots,N$, cover $\Omega$. Inserting this into \eqref{eq:fooling-subset} yields the lower bound for the $L_q$-norm of the fooling function and thus the lower bound for the $L_q$-approximation problem. Since $f_{\ast}\ge 0$, also the lower bound for the integration problem is proven.

\end{proof}

\section{Other function spaces} \label{sec:sob-proof-ext}

In this section we give a proof of Theorem~\ref{thm:sob-main-ext}. For the convenience of the reader we will give a definition of Triebel-Lizorkin spaces and also Besov spaces on bounded Lipschitz domains. We will mostly follow \cite{NT06}. For more details see, e.g., Triebel \cite{Tri92} or DeVore and Sharpley \cite{DS93}.

\bigskip

\textbf{The definition of Besov and Triebel-Lizorkin spaces. }We first define the Besov space $B^s_{pq}(\IR^d)$ and the Triebel-Lizorkin space $F^s_{pq}(\IR^d)$ of real-valued functions on $\IR^d$ with the help of a suitable dyadic resolution of unity of the Fourier domain. The Fourier transform maps a function $\psi\colon \IR^d\to\IR$ belonging to $S(\IR^d)$, the Schwartz space of rapidly decreasing smooth functions, to the function $\widehat{\psi}$ defined by
\[
\widehat{\psi}(\xi)
=F(\psi)(\xi)
:=(2\pi)^{-d/2}\int_{\IR^d}{\rm e}^{-i\langle x,\xi\rangle}\psi(x)\dd x,\quad \xi\in\IR^d.
\]
The inverse $F^{-1}$ of $F$ acts the same way but with $-i$ replaced by $i$. 

Choose $\psi\in S(\IR^d)$ with
\[
\psi(x)=1\quad \text{if}\quad \|x\|_2\le 1\qquad \text{and}\qquad \psi(x)=0 \quad \text{if}\quad \|x\|_2\ge \frac{3}{2}.
\]
If we put $\psi_0:=\psi$ and 
\[
\psi_j(x)
:=\psi(2^{-j}x)-\psi(2^{-j+1}x)\quad\text{for }x\in\IR^d \text{ and }j\in\IN,
\]
then the system $(\psi_j)_{j\in\IN_0}$ forms a dyadic resolution of unity, that is $\sum_{j=1}^{\infty}\psi_j(x)=1$ for all $x\in\IR^d$. For any tempered distribution $f\in S'(\IR^d)$ its Fourier and inverse Fourier transform may be defined and the function $F^{-1}(\psi_j \widehat{f})$ can be evaluated pointwise. Thus, the following definition makes sense.

\label{loc:bandf}
Let $s\in\IR$ and $0<q\le\infty$. For $0<p\le \infty$ define the Besov space $B^s_{pq}(\IR^d)$ as the collection of all $f\in S'(\IR^d)$ such that
\[
\|f\|_{B^s_{pq}(\IR^d)}^{\psi}
:=\Big(\sum_{j=0}^{\infty}2^{jsq}\|F^{-1}(\psi_j \widehat{f})\|_{L_p(\IR^d)}^q\Big)^{1/q}
\]
is finite. If $q=\infty$ we modify as usual by taking the $\ell_{\infty}$-norm. For $0<p<\infty$ the Triebel-Lizorkin space $F^{s}_{pq}(\IR^d)$ is given as the collection of all $f\in S'(\IR^d)$ with finite
\[
\|f\|_{F^s_{pq}(\IR^d)}^{\psi}
:=\Big\|\Big(\sum_{j=0}^{\infty}2^{jsq}|F^{-1}(\psi_j \widehat{f})(\cdot)|^q\Big)^{1/q}\Big\|_{L_p(\IR^d)}.
\]
The expression $\|\cdot\|_{A^s_{pq}(\IR^d)}^{\psi}$, where $A$ stands for either $B$ or $F$, is a quasi-norm and becomes a norm if both $p,q\ge 1$. A different choice of $\psi$ yields an equivalent (quasi-)norm, which is why $\psi$ is omitted in the notation. 

On a bounded Lipschitz domain $\Omega\subset\IR^d$ one can define the space $A^s_{pq}(\Omega)$ via restriction as in \eqref{eq:restriction}.  By definition, the space $A^s_{pq}(\Omega)$ contains distributions but an embedding theorem shows that if $s>d/p$ we have the continuous embedding $A^s_{pq}(\Omega)\hookrightarrow C_b(\Omega)$. Further, extension of these functions onto $\IR^d$ is possible if $s>d/p$. Thus, Lemma~\ref{lem:sob-extension} extends to these spaces. For further information consult for example Triebel's book~\cite{Tri92} and see also Rychkov~\cite{Ryc99} for the extension theorem.

\bigskip

\textbf{The proof of Theorem~\ref{thm:sob-main-ext}. } Let us first note that the analogues of Corollary~\ref{cor:sob-char} and Corollary~\ref{cor:sob-ran} immediately follow if we can extend Theorem~\ref{thm:sob-main}, which is what we prove in the following.

For the space $C^s(\Omega)$ with $s\in\IN$ the extension of Theorem~\ref{thm:sob-main}a is already included in Section~\ref{sec:sob-proof-ge}. Namely, the upper bound is immediate from the continuous embedding $C^s(\Omega) \hookrightarrow W_\infty^s(\Omega)$. For this note that all norms on the finite-dimensional space of multi-indices are equivalent. The lower bound holds since our fooling functions $f_\ast$ for $W_\infty^s(\Omega)$ are smooth and thus contained also in $C^s(\Omega)$ with an equivalent norm.  

For the remaining cases, we follow the lines of the proof of Theorem~\ref{thm:sob-main} as given in the previous Sections~\ref{sec:sob-proof-ge} and \ref{sec:sob-proof-le}. We will discuss the necessary changes and refrain from copying the proof.

Let $0<p,q,\tau \le \infty$ and $s\in\IR$ with $s>d/p$. We first give the proof of the upper bound of Theorem~\ref{thm:sob-main-ext} in the remaining cases, where we replace $W^s_p(\Omega)$ by
\begin{itemize}
	\item $C^s(\Omega)$ and $s\not\in\IN$ if $p=\infty$,
	\item $F^{s}_{p\tau}(\Omega)$ if $p<\infty$.
\end{itemize}
Note that it suffices to consider the case $\tau=\infty$ since we have the continuous embedding $ F^s_{p\tau_1}(\Omega)\hookrightarrow F^s_{p\tau_2}(\Omega) $ for $ \tau_1\le \tau_2 $.  Everywhere in the proof of Theorem~\ref{thm:sob-main} we replace $| f |_{W_p^s(\Omega)}$ by the following (quasi-)seminorms.

In the case $p=\infty$ and $s\not\in\IN$, we use the seminorm $|f|_{C^s(\Omega)}$ as defined in \eqref{eq:semi-norm-hoelder}.

If $p<\infty$, we use a (quasi-)seminorm which is defined via the averaged means 
\[
(d_t^{M,\Omega}f)(x) :=
t^{-d}\int_{V^M_{\Omega}(x,t)}\bigl|(\Delta_{h,\Omega}^M f)(x)\bigr|\,\text{d} h,
\]
where $M:=\lfloor s+1 \rfloor$, $\Delta_{h,\Omega}^M$ is an $M^{\rm{th}}$-order difference operator restricted to $\Omega$ and $V^M_{\Omega}(x,t)$ is the set of directions $h\in\IR^d$ of length less than $t>0$ with $x+a h\in\Omega$ for all $0\le a \le M$. More precisely, 
\[
\Delta_{h,\Omega}^M(f)
:=
\begin{cases}
	\Delta_{h}^M(f)&\text{if }x+lh\in\Omega\text{ for }l=0,\ldots,M,\\
	0&\text{otherwise,}
\end{cases}
\]
where, for every $x,h\in\IR^d$,
\[
\Delta_{h}^M(f)
= \Delta_{h}^{1}\Delta_{h}^{M-1}(f) \quad\text{with}\quad \Delta_{h}^{1}(f):=f(x+h)-f(x).
\]

Choosing then in the case $p<\infty$ the (quasi-)seminorm
\[
|f|_{F^s_{p\infty}(\Omega)} :=
 \bigg\| \sup_{0\le t \le 1}  \frac{(d_t^{M,\Omega}f)(\cdot)}{t^s} \bigg\|_{L_p(\Omega)},
\]
it is known that then
\[
 F^s_{p\infty}(\Omega)=\{f\in L_{\infty}(\Omega)\colon |f|_{F^s_{p\infty}(\Omega)}<\infty\}
\]
with
$
\|\cdot \|_{L_{\max\{p,1\}}(\Omega)}+|\cdot|_{F^s_{p\infty}(\Omega)}
$
being an equivalent quasi-norm, see Proposition 6 in \cite{NT06} and set $ u=1$ as well as $r=\infty $. The same is true for $\Omega=\IR^d$. 

It is readily verified that these (quasi-)seminorms have the following scaling property.  If $T\colon\IR^d\to\IR^d$ is of the form $T(y)= \varrho ^{-1} (y-x)$ with $ \varrho \le 1$ and $x\in\IR^d$ and $f\in F^s_{p\infty}(\Omega)$ or $f\in C^s(\Omega)$, then 
\begin{equation}\label{eq:scaling}
|f\circ T^{-1}|_{F^s_{p\infty}(Q_0)} \le  \varrho ^{s-d/p} |f|_{F^s_{p\infty}(Q)},
 \qquad
 |f\circ T^{-1}|_{C^s(Q_0)} =  \varrho ^{s} |f|_{C^s(Q)},\quad  \text{respectively.}
\end{equation}
As mentioned above the extension theorem in Lemma~\ref{lem:sob-extension} holds without changes for the spaces $F^s_{p\tau}(\Omega)$ and $C^s(\Omega)$, where we note that $C^s(\Omega)=B^s_{\infty\infty}(\Omega)$ for $s\not\in\IN$.  If we use \cite[Cor.~11]{NT06} (for $p<\infty$) and \cite[Thm.~6.1]{DS80} (for $p=\infty$) instead of \cite[Lem.~1.1.11]{Maz85} and the scaling properties \eqref{eq:scaling}, we see that Lemma~\ref{lem:sob-poly} concerning polynomial approximation on cubes remains valid under the modifications of Theorem~\ref{thm:sob-main-ext}. Thus, also Lemma~\ref{lem:local-estimate} defies our modifications.

To complete the proof of the upper bound in both cases $q\ge p$ and $q<p$, it only remains to note that the seminorm $|\cdot |_{F^s_{p\infty}(\Omega)}$ behaves equally well with respect to an efficient covering. Analogous to \eqref{eq:norms-efficient}, we have
\begin{equation}\label{eq:norms-efficient-fspace}
\sum_{i=1}^N |f|^p_{F^s_{p\infty}(Q_i)}
= \int_{\mathbb{R}^d} \left(\sup_{0\le t \le 1}  \frac{(d_t^{M,Q_i}f)(x)}{t^s}\right)^p
\sum_{i=1}^N \bfone_{Q_i}(x){\rm d}x
\le c_7 |f|^p_{F^s_{p\infty}(\mathbb{R}^d)}
\le c_8
\end{equation}
since $d_t^{M,Q_i}f(x)\le d_t^{M,\mathbb{R}^d}f(x)$ for every $x\in\mathbb{R}^d$. With these preparations at hand, we may copy the proof of the upper bound in both cases $q\ge p$ and $q<p$. 

Let us now discuss the lower bound. In case of $q\ge p$ we can use the same fooling function supported in a large hole. The scaling properties \eqref{eq:scaling-LB} for $C^s(\IR^d)$ require a straightforward computation, and for $F^s_{p\tau}$ they may be obtained e.g.~from Proposition~2.3.1/1 in Edmunds and Triebel~\cite{ET96} and the translation-invariance of the (quasi-)seminorm. 

Let us therefore consider the case $q>p$ where we used an atomic decomposition which is also valid for Besov and Triebel-Lizorkin spaces. If, for $p<\infty$, we want to replace $W^s_p$ by $F^s_{p\tau}$, we have to substitute the bound \eqref{eq:atomic} by
\[
\|f_*\|_{F^s_{p\tau}(\IR^d)}^p
\le C\, \|\lambda\|_{f_{p\tau}}^p,\quad
\text{where }\|\lambda\|_{f_{p\tau}}
:=\bigg(\int_{\IR^d}\Big(\sum_{i\in I_0} |\lambda_i 2^{\nu_i d/p}\mathbf{1}_{Q_i^*}(x)|^{\tau}\Big)^{p/\tau}\dd x\bigg)^{1/p}.
\]
\chg{Here,} $\|\cdot\|_{f_{p\tau}}$ is another sequence quasi-norm. Otherwise the proof can remain unchanged.

In the case $q<p=\infty$, where we replace $W^s_p$ by $C^s$, we use the convention that $a/\infty=0$ for any $a\in\IR$, whence $\beta=0$ and $\lambda_i=1$ for all $i\in I_0$. In this case, the fooling function is equal to 
\[
f_{*}
:=\sum_{i\in I_0}2^{-s\nu_i}\psi(2^{\nu_i}x-m_i).
\]
The norm of this function can be bounded as follows. If $s\in\IN$, one can use the disjoint supports of the summands to see that $\|f_{*}\|_{C^s(\IR^d)}\le \sup_{x\in\IR^d}|\psi(x)|\le 1$, and if $s\not\in\IN$ we use that Theorem~13.8 in \cite{Tri11} is also valid for Besov spaces and the bound 
\[
\|f_{*}\|_{C^s(\IR^d)}
\le C\,\|\lambda\|_{b_{\infty\infty}},\quad
\text{where }\|\lambda\|_{b_{\infty\infty}}:=\max_{i\in I_0}|\lambda_i|
= 1
\]
instead of $\eqref{eq:atomic}$. Here we used that if $s\not\in\IN$, then $C^s(\IR^d)=B_{\infty\infty}^s(\IR^d)$ with equivalent norms. The remainder of the proof of the lower bound requires only obvious modifications.

Note that also Theorem~\ref{thm:sob-main}b extends to the spaces since $\psi$ is non-negative. See also the discussion in Section~\ref{sec:sob-fool-int}.

\chapter{Additional material for Chapter 5}\label{ch:interlude-app}
\fancyhead[CO]{\nouppercase{\textsc{\leftmark}}}
\fancyhead[CE]{\nouppercase{\textsc{\leftmark}}}

In the following, we prove equality \eqref{eq:transport} between the minimal Wasserstein distance to a measure supported on a point set and the distortion, identity \eqref{eq:hoelder-wce} for the worst-case error for Hölder functions, the extension to the normed spaces given in Proposition~\ref{pro:lip} and Proposition~\ref{pro:weights} stating that weights can be assumed to be normalized.

The following proof is inspired by the one for \cite[Lem.~3.4]{GL00}. Recall that $\mu$ is a Borel probability measure on $\IR^d$ absolutely continuous with respect to the Lebesgue measure and that its distortion is given by \eqref{eq:distortion}. 
\vspace{0cm}
\begin{proof}[Proof of \eqref{eq:transport}]
For the upper bound we will exhibit an explicit transference plan between the absolutely continuous measure $\mu$ and the discrete measure $\mu_{\Pn,a}$ given by \eqref{eq:discrete}. To this end, let $T_{\Pn}\colon \IR^d\to\IR^d$ be the quantizer from \eqref{eq:opt-quant} in Section~\ref{sec:sob-quant} which maps $x$ to the closest point of $\Pn$ and is of the form
\[
T_{\Pn}=\sum_{i=1}^{n}x_i \bfone_{C(x_i,\Pn)},
\]
where $C(x_i,\Pn)$ is as in \eqref{eq:voronoi}. 

Define the map \chg{$G_{\Pn}\colon \IR^d\to \IR^d\times \IR^d$} by $G_{\Pn}(x)=\big(x,T_{\Pn}(x)\big)$. Then the pushforward measure $\pi_{\Pn}:=\mu\circ G_{\Pn}^{-1}$ satisfies, for any Borel sets $A,B\subset\IR^d$,
\begin{align*}
	\pi_{\Pn}(A\times \IR^d) 
	&= \mu(x\in\IR^d\colon x\in A, T_{\Pn}(x)\in \IR^d)
=\mu(A)\\
\pi_{\Pn}(\IR^d\times B) 
&= \mu(x\in\IR^d\colon T_{\Pn}(x)\in B)
=\nu_{\Pn}(B),
\end{align*}
where $\nu_{\Pn}:=\sum_{i=1}^{n}\mu\big(C(x_i,\Pn)\big)\delta_{x_i}$. Thus, $\pi_{\Pn}$ has marginals $\mu$ and $\nu_{\Pn}$ and belongs to $\pi(\mu,\nu_{\Pn})$.  Then, by \eqref{eq:dist-quant}, the distortion is equal to
\begin{equation*}
	\int_{\IR^d} \|x-T_{\Pn}(x)\|^r \dd\mu(x)
=\int_{\IR^d} \|x-y\|^r \dd\pi_{\Pn}(x,y)
\ge T_r(\mu,\nu_{\Pn})
\ge \inf_{a_1,\ldots,a_n\in\IR} T_r(\mu,\nu_{\Pn,a}),
\end{equation*}
which completes the upper bound.

\newpage

For the lower bound we note that, for every $a_1,\dots,a_n\in\IR$ and every transference plan $\pi\in\pi(\mu,\nu_{\Pn,a})$,
\begin{align*}
\int_{\IR^d\times\IR^d} \|x-y\|^r \dd\pi(x,y)
&\ge \int_{\IR^d\times \Pn} \|x-y\|^r \dd\pi(x,y)\\
&\ge \int_{\IR^d\times \Pn}\min_{i=1,\dots,n} \|x-x_i\|^r \dd\pi(x,y)
= D_{\mu,\Pn,r}
\end{align*}
Taking the infimum over all couplings and then over all weights completes the proof.
\end{proof}

Next, we deduce the expression for the minimal worst-case error of weighted algorithms for Hölder and Lipschitz functions from the proof of Theorem~5 in \cite{Gru04}.
\begin{proof}[Proof of \eqref{eq:hoelder-wce}]
Let $s\in (0,1)$. For the lower bound consider the fooling function $f_{\ast}:=\dist(\cdot,\Pn)^{s}$ which vanishes on $\Pn$ and has integral $\intmu(f_{\ast})=D_{\mu,\Pn,s}$.  It is a consequence of the triangle and Hölder's inequality that $|f_{\ast}|_{C^{s}(\IR^d)}\le 1$. The method of fooling functions described in Section~\ref{sec:sob-fool-int} then establishes the lower bound.

For the upper bound we note that the infimum over all weights can be bounded from above if we insert the optimal cubature rule from \eqref{eq:opt-cubature}, and then 
\begin{align*}
\inf_{a_1,\dots,a_n\in\IR}\sup_{|f|_{C^{s}}\le 1}\Big| \intmu(f)-\sum_{i=1}^{n}f(x_i) a_i\Big|
&\le \sup_{|f|_{C^{s}}\le 1}\Big| \intmu(f)-Q_{\Pn,\mu}(f) \Big|\\
&= \sup_{|f|_{C^{s}}\le 1}\Big| \sum_{i=1}^{n}\int_{C(x_i,\Pn)}f(x)-f(x_i)\dd \mu(x) \Big|
\end{align*}
Using the triangle inequality and that $|f|_{C^{s}(\IR^d)}\le 1$ implies \chg{$|f(x)-f(x_i)|\le \|x-x_i\|_{2}^{s}$}, we see that this bounded from above by the distortion $ D_{\mu,\Pn,s} $ completing the proof of the identity for $s\in (0,1)$. The modifications necessary for $s=1$ and $\lip$ are straightforward.
\end{proof}

\begin{proof}[Proof of Proposition~\ref{pro:lip}]
	Let $s\in (0,1)$. For the lower bound we again choose the fooling function $f_{\ast}:=\dist(\cdot,\Pn)^s$ which satisfies $|f_{\ast}|_{C^s(D)}\le 1$. To normalize, we divide by $1+\|f_{\ast}\|_{\infty}=1+h_{\Pn,D}^s$, which yields the lower bound
\[
e(C^s(D),\intmu,Q_{\Pn,a})
\ge \frac{1}{1+h_{\Pn,D}^s}\int_{D}\dist(x,\Pn)^s\dd\mu(x),
\]
where we can replace $C^s(D)$ by $\lipd$ for $s=1$. To have a constant independent of the point set $\Pn$, and also of $s$, one can use that
\[
h_{\Pn,D}^{s}
\le \max\{1,h_{\Pn,D}\}
\le \max\{1,\diam(D)\}.
\]

For the proof of the upper bound we use the just proven identity \eqref{eq:hoelder-wce} and that optimal weigths are given by \eqref{eq:voronoi}. It remains to note that the infimum over arbitrary real weights is bounded from above by the infimum over weights summing to one and that the worst-case error over functions with $|f|_{C^s(\IR^d)}\le 1$ is at least as large as the worst-case error over all functions with $\|f\|_{C^s(\IR^d)}\le 1$, and analogously for $s=1$ and $\lipd$.
\end{proof}

Let us conclude this appendix with the following proof.

\begin{proof}[Proof of Proposition~\ref{pro:weights}]
	Let $f\in F$ with $\|f\|_F\le 1$ be arbitrary and $Q_{\Pn,a}$ a cubatrue rule with points $\Pn=\{x_1,\ldots,x_n\}\subset D$ and weights $a=(a_1,\ldots,a_n)$. Then the cubature rule $Q_{\Pn,a^*}$ using the normalized weights $a^*$ with $a_i^*=a_i/\sum_{i=1}^{n}a_i, i=1,\ldots,n,$ satisfies
\[
\Big|\sum_{i=1}^{n}a_i^* f(x_i)-\intmu(f)\Big|
\le \Big|\sum_{i=1}^{n} a_i f(x_i) - \intmu(f)\Big|
+\Big|\sum_{i=1}^{n} (a_i^*-a_i)f(x_i)\Big|. 
\]
The first term on the right-hand side is bounded by $e(F,\intmu,Q_{\Pn,a})$ and the second term is bounded in absolute value by
\[
\Big|\sum_{i=1}^{n} (a_i^*-a_i)f(x_i)\Big|
\le C\,\Big|1-\sum_{i=1}^{n}a_i\Big|
\le C\,\|\bfone_D\|_F\, e(F,\intmu,Q_{\Pn,a})
\]
as the normalized function $\bfone_D/\|\bfone_D\|_F$ belongs to the unit ball of $F$.
\end{proof}
\cleardoublepage
\phantomsection
\chapter{Additional material for Chapter 6}\label{ch:iso-app}
\fancyhead[CO]{\nouppercase{\textsc{\rightmark}}}
\fancyhead[CE]{\nouppercase{\textsc{\leftmark}}}

In this appendix we provide the missing proofs for Chapter~\ref{ch:iso}. In Section~\ref{sec:iso-proof-lower} we prove the lower bound of its characterization in terms of the spectral test (Theorem~\ref{thm:iso-char}) and the lower bound on the spectral test (Proposition~\ref{pro:spectral-lower}), thus completing the proof of the lower bound in Theorem~\ref{thm:iso-lower}. 

Section~\ref{sec:iso-proof-upper} contains the proofs of Lemmas~\ref{lem:neighbourhood} and \ref{lem:diam-bound} to fill in the gaps of the proof sketch for the upper bound of Theorem~\ref{thm:iso-char} given in Section~\ref{sec:iso-lll}. Additionally, we shall give the proof behind the asymptotics in Remark~\ref{rem:subexp}.  Finally, Section~\ref{sec:iso-proof-dist} contains the proof of Proposition~\ref{pro:distspectral}.

Throughout, let $L$ be a $d$-dimensional lattice and $\cP(L)=L\cap [0,1)^d$ be the corresponding lattice point set. Further, let $\cH^*$ be a hyperplane covering of $L$ as in Section~\ref{sec:iso-hyp}, where the distance between adjacent hyperplanes is maximal and equal to the spectral test $\sigma(L)$.

\section{Lower bounds}\label{sec:iso-proof-lower}

\begin{proof}[Proof of the lower bound of Theorem~\ref{thm:iso-char}]
By the pigeonhole principle we find a hyperplane from $\cH^*$ which contains sufficiently many points of $\cP(L)$. Since the unit cube $[0,1)^d$ has diameter $\sqrt{d}$, it can be intersected by no more than $\sqrt{d}/\sigma(L)+1$ hyperplanes from $\cH^*$. If all of them contained strictly less than 
\[
n_{-}
:=\frac{\sigma(L)}{\sqrt{d}+\sigma(L)}n
\]
points of the lattice point set $\cP(L)$, this would be a contradiction to the fact that there are at most $\sqrt{d}/\sigma(L)+1$ of them and any point of $\cP(L)$ must lie on one of these hyperplanes. Therefore, we may find at least $n_{-}$ points of $\cP(L)$ lying on some hyperplane. This implies the lower bound.

Suppose now that additionally $\sigma(L)\leq 1/2$. 
In order to find a suitable pair of hyperplanes such that their convex hull contains the center $(1/2,\dots,1/2)$ of the cube $[0,1)^d$, consider the one-dimensional space orthogonal to all hyperplanes in $\cH^*$ which is spanned by some $h\in L^{\bot}$. The rays emanating from the center of the cube into the opposite directions $\pm h$ hit \chg{a hyperplane} of $\cH^*$ at distance at most $\sigma(L)$ from the center of the cube. In this way, we get a pair of adjacent hyperplanes $H_1,H_2\in\cH^*$ sandwiching the center of the cube with possibly one of them containing the center. Denote the collection of all hyperplanes which lie between $H_1$ and $H_2$ by $\widetilde{\cH}$ such that the interior of the convex hull of $H_1$ and $H_2$ satisfies 
\[
\widetilde{C}
={\rm int}\,{\rm conv}(H_1\cup H_2)=\bigcup_{H\in\widetilde{\cH}} H.
\]
Then $C:=\widetilde{C}\cap [0,1)^d$ does not contain any point from $\cP(L)$. Since all hyperplanes in $\widetilde{\cH}$ have distance at most $\sigma(L)\leq 1/2$ from the center, \cite[Thm.~1.1]{KR20} yields a constant independent of the dimension such that
\begin{equation} \label{eq:slab-lower}
\vol(C)
\geq \sigma(L) \inf_{H\in \widetilde{\cH}} \text{vol}_{d-1}(H\cap [0,1)^d)
\ge c\, \sigma(L).
\end{equation}
Consequently, the isotropic discrepancy of $\cP(L)$ is at least $c\,\sigma(L)$. 
\end{proof}

We briefly explain how the lower bound on the isotropic discrepancy in Theorem~\ref{thm:iso-lower} follows. If $\sigma(L)>1/2$, then we use the fact that the function $f\colon x\mapsto x/(\sqrt{d}+x), x\ge 1/2,$ is increasing and satisfies $f(1/2)=1/(2\sqrt{d}+1)$. In the other case we apply the lower bound on the spectral test in Proposition~\ref{pro:spectral-lower} which we prove next.

Before doing that, let us recall the content of Minkowski's fundamental theorem, which states the following: 
\begin{quote}
Let $L$ be a lattice in $\IR^d$. Then any convex set in $\IR^d$ which is symmetric with respect to the origin and with volume greater than $2^d \det(L)$ contains a non-zero lattice point of $L$.
\end{quote}
See, e.g., \cite[Thm.~447]{HW79} in the book by Hardy and Wright.

\begin{proof}[Proof of Proposition~\ref{pro:spectral-lower}]
	By the definition of the spectral test we need an upper bound on the shortest vector in the dual lattice $L^{\bot}$. To this end, we will apply Minkowski's theorem to $L^{\bot}$ and the ball $r\IB_2^d$ of radius $r>0$. According to Sloan and Kachoyan \cite[Sec.~3 and 4]{SK87} (see also \cite[Thm.~5.30]{Nie92}) we have $\det(L^{\bot})=n$. The volume of $r\IB_2^d$ is 
\[
\vol(r\IB_2^d)
=r^d \frac{\pi^{d/2}}{\Gamma(\frac{d}{2}+1)}.
\]
Hence, by Minkowski's theorem, if
\[
r^d \frac{\pi^{d/2}}{\Gamma(\frac{d}{2}+1)} > 2^d \det(L^\bot) =2^d n,
\]
i.e., if
\[
r > \frac{2}{\sqrt{\pi}} \big(\Gamma(\frac{d}{2}+1)\big)^{1/d}  n^{1/d}=:\widetilde{r}(d,n)
\]
then $r\IB_2^d$ contains a non-zero point from $L^\bot$. In other words, $L^\bot$ contains a non-zero lattice point of length at most $\widetilde{r}(d,n)$ and therefore the length of the shortest vector of $L^{\bot}$ is at most $\widetilde{r}(d,n)$ and the spectral test at least $\widetilde{r}(d,n)^{-1}$.
\end{proof}

\section{The upper bound for the characterization}\label{sec:iso-proof-upper}

In the following, we introduce useful concepts from convex geometry, see, e.g., the book of Schneider~\cite{Sch14} for an introduction. 
\begin{definition}\label{def:minkowski}
Given non-empty $A,B\subseteq \mathbb{R}^d$ we define the Minkowski addition and the Minkowski difference by
\[
A+B:=\bigcup_{b\in B} (A+b)\quad \text{and}\quad A\div B:=\bigcap_{b\in B} (A-b),\text{ respectively,}
\]
where $A\pm b=\{a\pm b\colon a\in A\}$ for every $b\in B$.
\end{definition}
Then, for all $\rho> 0$,
\begin{align*}
K+\rho \IB_2^d&=\{x\in\mathbb{R}^d\colon \dist(x,K)<\rho\},\\
K\div\rho \IB_2^d&=\{x\in K\colon \dist\big(x,K^C\big)\ge \rho\}.
\end{align*}
We define a family of convex parallel sets by
\[
K_{\rho}:=
\begin{cases}
K+\rho \IB_2^d& \text{for }\rho\ge 0,\\
K\div (-\rho)\IB_2^d & \text{for } \rho < 0.
\end{cases}
\]
The largest $\rho>0$ such that $K\div \rho \IB_2^d\neq\emptyset$ is given by the inradius of $K$, which is defined by $r(K):=\sup\{\rho\ge 0\colon x+\rho \IB_2^d\subseteq K \text{ for some } x\in \mathbb{R}^d\}$. As a consequence, $\vol(K_{-r(K)})=0$ and $\rho<-r(K)$ implies $K_{\rho}=\emptyset$. For $\rho> 0$ we have $K_{\rho}=K+\rho \IB_2^d$ and $K_{-\rho}=K\div \rho \IB_2^d$. Further, $K_0=K$. 

Comparing the definitions we see that for any $\rho\ge 0$
\begin{equation}\label{eq:volumedifference}
\vol(K_{\rho}^+)=\vol(K_{\rho})-\vol(K) \quad\text{and}\quad \vol(K_{\rho}^-)=\vol(K)- \vol(K_{-\rho}).
\end{equation}

We will use Steiner's formula (see, e.g., Schneider \cite[eq.~(4.8)]{Sch14}) stating that, for every $\rho\ge 0$,
\begin{equation}\label{eq:steiner}
\vol(K+\rho \IB_2^d)=\sum_{j=0}^d \binom{d}{j}W_j(K)\rho^j,
\end{equation}
where $W_j(K)$ is the $j$-th quermassintegral of $K$. As a mixed volume, it is \chg{monotone with respect to} set inclusion, i.e., it satisfies $W_j(K_1)\le W_j(K_2)$ for $j=0,\ldots, d$, whenever $K_1\subseteq K_2$ are convex bodies. Note that $W_0(K)=\vol(K)$ and $d\, W_1(K)$ is the surface area of $K$.

We shall need a result noted by Hadwiger in his book \cite[Eq.~(30), page 207]{Had57}; compare also to \cite[Prop.~2.6]{RG20} by Richter and Saor\'in G\'omez who give additional references. 

\begin{lemma}\label{lem:diffvol}
The function $v(\rho):=\vol(K_{\rho})$ is differentiable on $(-r(K),\infty)$ and its derivative satisfies $v'(\rho)=d\, W_1(K_{\rho})$.
\end{lemma}

From this we derive the following inequality, which will enable us to prove Lemma~\ref{lem:neighbourhood}.

\begin{lemma}\label{lem:outervsinner}
For all $\rho\ge 0$ we have $\vol(K_{\rho}^+)\ge \vol(K_{\rho}^-).$
\end{lemma}
\begin{proof}
Using \eqref{eq:volumedifference}, this inequality can be written as $v(\rho)-v(0)\ge v(0)-v(-\rho)$. Suppose first that $0<\rho\le r(K)$. Lemma \ref{lem:diffvol} and the mean value theorem yield some $\rho_1\in (0,\rho)$ and $\rho_2\in (-\rho,0)$ such that
\[
\frac{v(\rho)-v(0)}{\rho}=v'(\rho_1) \quad \text{and}\quad \frac{v(0)-v(-\rho)}{\rho}=v'(\rho_2).
\]
Since $K_{\rho_2}\subseteq K_{\rho_1}$ and the quermassintegral $W_1(\cdot)$ is \chg{monotone}, we have $v'(\rho_1)\ge v'(\rho_2)$. This completes the proof in this case.

If $\rho=0$, we have equality by definition, and if $\rho>r(K)$, the monotonicity of the volume yields $v\big(r(K)\big)\le v(\rho)$, and thus from the previously established case $\rho=r(K)$ it follows that
\[
v(\rho)-v(0)\ge v\big(r(K)\big)-v(0)\ge v(0)-v\big(-r(K)\big)=v(0)-v(-\rho)
\]
since $v\big(-r(K)\big)=v(-\rho)=0$. By means of 
 \eqref{eq:volumedifference} this completes the proof.
\end{proof}

\begin{proof}[Proof of Lemma~\ref{lem:neighbourhood}]
If $\rho=0$, there is nothing to show and thus assume $\rho\in (0,1]$ from now on. Lemma \ref{lem:outervsinner} implies that $\max\{\vol(K_{\rho}^+),\vol(K_{\rho}^-)\}= \vol(K_{\rho}^+)$ and it remains to estimate the latter. By Steiner's formula \eqref{eq:steiner} we have
\[
\vol(K_{\rho}^+)=\vol(K_{\rho})-\vol(K)=\sum_{j=1}^d \binom{d}{j} W_j(K) \rho^j.
\]
The monotonicity of the quermassintegrals yields
\[
\vol(K_{\rho}^+)\le \sum_{j=1}^d \binom{d}{j} W_j([0,1]^d) \rho^j
\]
with equality for $K=[0,1]^d$. According to e.g.\ Lotz, McCoy, Nourdin, Peccati and Tropp \cite[Ex.~6.1.3]{LMN+19} it is a classical fact that for $j=0,1,\ldots, d$ the $j$-th intrinsic volume $V_j([0,1]^d)$ of the unit cube equals $\binom{d}{j}$. The relation to the quermassintegrals is given by $\binom{d}{j}W_j([0,1]^d)=\kappa_jV_{d-j}([0,1]^d),$ where $\kappa_j$ is the $j$-dimensional volume of the unit ball of $(\mathbb{R}^j,\|\cdot\|_2)$. Together with the symmetry of the binomial coefficients, this implies $W_j([0,1]^d)=\kappa_j$. Therefore, as $\rho\le 1$, we have
\[
\vol(K_{\rho}^+)\le \rho\sum_{j=1}^d \binom{d}{j}\kappa_j \rho^{j-1} \le \rho\sum_{j=1}^d\binom{d}{j}\kappa_j.
\]
Using the fact that $\kappa_j\le \kappa_5=8\pi^2/15\le 2^3$ for every $ j\in \IN$ and that $\sum_{j=0}^d \binom{d}{j}=2^d$ completes the proof of Lemma~\ref{lem:neighbourhood}.
\end{proof}

We finish the proof of the upper bound of Theorem~\ref{thm:iso-char} with the following proof of the estimate on the diameter of a cell with respect to an LLL-reduced basis.

\begin{proof}[Proof of Lemma~\ref{lem:diam-bound}]
	Let $\fP=\{\sum_{i=1}^{d} \lambda_i b_i\colon 0\le \lambda_i<1, i=1,\ldots,d\}$ be the fundamental cell with respect to an LLL-reduced basis $b_1,\ldots,b_d$ of $L$ as given by Definition~\ref{def:lll}. Then, from the properties of a reduced basis it can easily be deduced, see, e.g., \cite[Lem.~17.2.8]{Gal12}, that
\begin{enumerate}
\item[(a)] $\Vert b^{\ast}_j\Vert_2^2\leq 2^{i-j}\Vert b^{\ast}_{i}\Vert_2^2$ for $1\leq j \leq i \leq d$ and
\item[(b)] $\Vert b_i\Vert_2^2\leq 2^{d-1}\Vert b^{\ast}_i\Vert_2^2$ for $1\leq i\leq d$.
\end{enumerate}
Together these estimates imply
\[
\Vert b^{\ast}_d\Vert_2\geq 2^{-(d-1)/2}\max_{1\leq i\leq d}\Vert b^{\ast}_i\Vert_2\geq 2^{-d+1}\max_{1\leq i \leq d}\Vert b\Vert_2.
\]
Using this, we bound the diameter of $\fP$ by
\[
\dia
\leq \sum_{i=1}^d \Vert b_i\Vert_2 
\leq d \max_{1\leq i \leq d} \Vert b_i\Vert_2
\leq d \, 2^{d-1} \, \Vert b_d^{\ast}\Vert_2.
\]
By the Gram-Schmidt algorithm we have 
\[
\Vert b_d^{\ast}\Vert_2 = \left\Vert b_d - \sum_{i=1}^{d-1} \mu_{d,j}b^{\ast}_j\right\Vert_2.
\]
That is, the length of the last vector in the Gram-Schmidt orthogonalization is the length of the projection of the vector $b_d$ onto the orthogonal complement of the subspace span$\{b^{\ast}_1,\dots,b^{\ast}_{d-1}\}=\text{span}\{b_1,...,b_{d-1}\}$ spanned by the other basis vectors. This is exactly the distance between two adjacent hyperplanes of the family of parallel hyperplanes
\[
k\, b_d+\text{span}\{b_1,...,b_{d-1}\},\quad k\in\mathbb{Z},
\]
which covers the entire lattice $L$. Therefore, since this distance cannot be larger than the spectral test, we have
\[
d\, 2^{d-1}\, \sigma(L)
\geq d\, 2^{d-1} \, \Vert b^{\ast}\Vert_2 \geq \text{diam}(\fP).
\]
\end{proof}
This completes the proof of the upper bound of Theorem~\ref{thm:iso-char}.

We give now the proof behind Remark~\ref{rem:subexp}. It was suggested by us by an anonymous reviewer when the paper \cite{SP21} was submitted and improved our previous asymptotic estimates. The improvement of Remark~\ref{rem:subexp} can be achieved because the above upper estimate of the sum 
\[
\sum_{j=1}^d\binom{d}{j}\kappa_j \quad \text{with the volumes of the unit balls } \kappa_j=  \frac{\pi^{j/2}}{\Gamma(1+j/2)}
\]
can be replaced with the following asymptotics
\begin{equation} \label{eq:subexp}
\sum_{j=1}^d {d \choose j} \frac{\pi^{j/2}}{\Gamma(1+j/2)} = {\rm e}^{\frac{3}{2}(2\pi)^{1/3}d^{2/3}+\mu(d)},
\end{equation}
 where $\vert \mu(d)\vert \le c\, d^{1/3}$ for some constant $c>0$ and all $d$.

 \begin{proof}[Proof of \eqref{eq:subexp}]
Setting
\[
	a_{j}:={d \choose j} \frac{\pi^{j/2}}{\Gamma(1+j/2)} \quad \text{for } j=1,\ldots,d,
\]
it holds that $a_{j}\le a_{j+2}$ if and only if
\[
	(j+1)(j+2)^{2}\le 2\pi (d-j)(d-j-1).
\]
From this one can deduce that there exists an absolute constant $C>0$ such that for each $d$ the maximal index $j_{d}$, i.e., $a_{j_{d}}=\max_{1\le j \le d}a_{j}$, satisfies $j_{d}= (2\pi)^{1/3}d^{2/3} + \mu_{0}(d)$  with $\vert \mu_{0}(d)\vert\le  C\, d^{1/3}$.

Using Stirling's formula for the Gamma function stating that
\begin{equation} \label{eq:stirling}
\Gamma(x)=\sqrt{\frac{2\pi}{x}} \left(\frac{x}{{\rm e}}\right)^x {\rm e}^{\mu_{1}(x)},\quad
\text{where }0< \mu_{1}(x) < \frac{1}{12 x}\text{ for all }x>0,
\end{equation}
and the asymptotics
\[
d(d-1)\cdots (d-j+1)=d^{j}{\rm e}^{-(j^{2}/2d)\big(1+\mu_{2}(j/d)\big)}\quad \text{for } j=1,\ldots,d,
\]
where $\mu_{2}(x)\to 0$ if $x\to 0$, we compute 
\[
	a_{j}=\frac{\sqrt{2}}{\pi j}\frac{d^{j}(2\pi)^{j/2}}{j^{3j/2}}{\rm e}^{3j/2-(j^{2}/2d)\big(1+\mu_{2}(j/d)\big)-\mu_{1}(j)-\mu_{1}(j/2)}\quad \text{for } j=1,\ldots,d. 
\]
Inserting $j=j_{d}$ gives
\[
	a_{j_{d}}={\rm e}^{\frac{3}{2}(2\pi)^{1/3}d^{2/3}+\mu_{3}(d)}
\]
with $\vert \mu_{3}(d)\vert\le C'\, d^{1/3}$ for some $C'>0$ and all $d$. Finally, using $a_{j_{d}}\le \sum_{j=1}^{d} a_{j}\le d\, a_{j_{d}}$ proves the statement.
\end{proof}

\section{Proof of Proposition 6.13}\label{sec:iso-proof-dist}

	H\"older's inequality implies 
\[
\|\dist\big(\cdot,\mathcal{P}(L)\big)\|_{L_{\gamma}([0,1)^d)}\le \|\dist\big(\cdot,\mathcal{P}(L)\big)\|_{L_{\infty}([0,1)^d)},
\]
and therefore it suffices to prove the lower bound for $\gamma \in (0,\infty)$  and the upper for $\gamma=\infty$.

We start with the proof of the lower bound and let $\gamma\in (0,\infty)$ be arbitrary. Take a hyperplane covering $\cH^*$ of $L$ as in the beginning of Section~\ref{sec:iso-hyp} with distance $\sigma(L)$ between adjacent hyperplanes. For $t\in (0,1/2)$ consider the set
\[
A_t:=\bigg\{x\in [0,1)^d \ \colon \ \dist\bigg(x,\bigcup_{H\in\cH^*}H\bigg)\ge t \, \sigma(L)\,\bigg\}.
\]
As the family $\cH^*$ covers the lattice $L$, we have $\mathcal{P}(L)\subseteq \bigcup_{H\in\cH^*}H$, and thus for $x\in A_t$ it holds that $\dist\big(x,\mathcal{P}(L)\big)\ge t\, \sigma(L)$. Taking powers and integrals on both sides yields
\[
\int_{[0,1)^d}\dist\big(x,\mathcal{P}(L)\big)^{\gamma}\,{\rm d}x\ge  t^{\gamma} \, \sigma(L)^{\gamma} \, \vol(A_t).
\]

For establishing the lower bound it suffices to find $t_d>0$ such that $\vol(A_{t_d})\ge 1/2$. To this end, we show that for some $t_d>0$ the complement $B_{t}:=[0,1)^d\backslash A_{t}$ satisfies $\vol(B_{t_d})\le 1/2$. 

For any $t>0$, we first decompose the set $B_t$ into the disjoint union $B_t=\bigcup_{H\in\cH^*} S_t(H),$
where we let 
\[
S_t(H):=\{x\in [0,1)^d \ \colon \ \dist(x,H)< t\, \sigma(L)\}.
\]
Consequently, we have $\vol(B_t)=\sum_{H\in\cH^*} \vol\big(S_t(H)\big)$. 

For any $t>0$, at most $\sqrt{d}/\sigma(L)+2$ of the sets $S_t(H), H\in\cH^*,$ are non-empty, and thus only finitely many terms of the sum are non-zero. This is because the cube $[0,1)^d$ has diameter $\sqrt{d}$ and can therefore be intersected by no more than $\sqrt{d}/\sigma(L)$ hyperplanes contained in $\cH^*$. The volume of a set $S_t(H)$ is bounded by its width, which is at most $2t\,\sigma(L)$ times the quantity $\sup_H \vol_{d-1}(H\cap [0,1)^d)$, where the supremum is extended over all hyperplanes $H$ in $\IR^d$. Since $[0,1)^d$ is bounded, this supremum is bounded by some constant only depending on the dimension, call it $v_d$. This implies, since $H\in\cH^*$ can be arbitrary,
\[
\vol(B_t)\le (\sqrt{d}/\sigma(L)+2)\, 2t \,\sigma(L)\, v_d=\big(2\sqrt{d}+4\sigma(L)\big)\,v_d\, t.
\]
Using that $\sigma(L)\le \sqrt{d}$ we can choose $t_d=(12\sqrt{d}v_d)^{-1}$ such that $\vol(B_{t_d})\le 1/2$. This completes the proof of the lower bound.

We turn to the proof of the upper bound for which we have to find $C_d>0$ such that 
\[
\|\dist\big(\cdot,\mathcal{P}(L)\big)\|_{L_{\infty}([0,1)^d)}\le C_d\,\sigma(L).
\]
If we choose 
\[
r_0:=\frac{1}{2} \|\dist\big(\cdot,\mathcal{P}(L)\big)\|_{L_{\infty}([0,1)^d)},
\]
then there exists a ball $B(y,r_0)$ with center $y\in [0,1)^d$ and radius $r_0>0$ which is empty of points from $\mathcal{P}(L)$. We now use the fact that by Lemma~\ref{lem:convex-cone} in Section~\ref{sec:sob-mls} the cube $[0,1)^d$ satisfies an interior cone condition. Then Lemma~\ref{lem:ballinball} implies that there exists a ball $B(z,r_0')$ contained in $ B(y,r_0)\cap [0,1)^d$ with $r_0'= u_d r_0$, where the quantity $u_d>0$ only depends on $d$. 

Let $\fP$ be the fundamental parallelotope with respect to a LLL-reduced basis of $L$ (see Section~\ref{sec:iso-lll}) and fix $x\in L$ such that the translate $\fP_x=x+P$ contains the center $z$ of the ball. By Lemma~\ref{lem:diam-bound}, we have $\dia\le d\, 2^{d-1}\sigma(L)$. 

Thus, if $\dia\le r_0'$, then we must have the inclusions $\fP_x\subset B(z,r_0')\subset [0,1)^d$ and consequently $x\in \mathcal{P}(L)\cap B(z,r_0')$. This is a contradiction to $B(y,r_0)\cap \cP(L)=\emptyset$. Therefore, we must have $\dia>r_0'\ge u_d r_0$ and 
\[
\|\dist\big(\cdot,\mathcal{P}(L)\big)\|_{L_{\infty}([0,1)^d)}\le C_d\, \sigma(L),
\]
where $C_d:=d\,2^{d}u_d^{-1}$. This completes the proof of the upper bound of Proposition~\ref{pro:distspectral}.$\hfill \qed$

\end{document}